\documentclass[12pt]{articlefedericosmallerlabels}

\usepackage{amsfonts}
\usepackage{amsmath}
\usepackage{amssymb}
\usepackage{amsthm}
\usepackage{upref}
\usepackage[colorlinks,final,backref=page,hyperindex]{hyperref}
\usepackage{mathtools}
\usepackage{float}

\usepackage{xcolor}

\newcommand{\black}[1]{\textcolor{black}{#1}}
\newcommand{\brown}[1]{\textcolor{brown}{#1}}
\newcommand{\green}[1]{\textcolor{darkgreen}{#1}}
\definecolor{darkgreen}{rgb}{0,.5,0}
\definecolor{brown}{rgb}{0.5,0.3,0}

\usepackage{url}
\usepackage{amsxtra}
\usepackage[final]{rotating}
\usepackage{array}
\usepackage[all]{xy}
\SelectTips{cm}{}
\usepackage{epic}

\newtheorem{theorem}{Theorem}[section]
\newtheorem{proposition}[theorem]{Proposition}
\newtheorem{lemma}[theorem]{Lemma}
\newtheorem{corollary}[theorem]{Corollary}
\newtheorem*{theorem*}{Theorem}

\newtheorem{problem}[theorem]{Problem}

\theoremstyle{definition}
\newtheorem{definition}[theorem]{Definition}

\newtheorem{example}[theorem]{Example}

\newtheorem{remark}[theorem]{Remark}

\theoremstyle{remark}
\newtheorem*{associativity}{Associativity}
\newtheorem*{unitality}{Unitality}
\newtheorem*{naturality}{Naturality}
\newtheorem*{compatibility}{Compatibility}
\newtheorem*{multiplicativity}{Multiplicativity}
\newtheorem*{warning}{Warning}

\renewcommand{\1}{\mathbf{1}}

\newcommand{\Set}{\mathsf{Set}}

\newcommand{\Vect}{\mathsf{Vec}} 

\newcommand{\Kc}{\mathcal{K}}
\newcommand{\Kcb}{\overline{\Kc}}

\newcommand{\rA}{\mathrm{A}} 
\newcommand{\rbA}{\overline{\mathrm{A}}} 
\newcommand{\rBS}{\mathrm{BS}} 
\newcommand{\rWBS}{\mathrm{WBS}} 

\newcommand{\rF}{\mathrm{F}} 
\newcommand{\rG}{\mathrm{G}} 
\newcommand{\rSC}{\mathrm{SC}} 
\newcommand{\rSG}{\mathrm{SG}} 
\newcommand{\rHG}{\mathrm{HG}} 
\newcommand{\rSHG}{\mathrm{SHG}} 

\newcommand{\rGamma}{\mathrm{\Gamma}} 
\newcommand{\rH}{\mathrm{H}} 
\newcommand{\rK}{\mathrm{K}}
\newcommand{\rL}{\mathrm{L}} 
\newcommand{\rM}{\mathrm{M}} 
\newcommand{\rPi}{\mathrm{\Pi}}
\newcommand{\rbPi}{\overline{\mathrm{\Pi}}}
\newcommand{\rP}{\mathrm{P}}
\newcommand{\rQ}{\mathrm{Q}}
\newcommand{\rW}{\mathrm{W}}

\newcommand{\rGP}{\mathrm{GP}} 
\newcommand{\rHGP}{\mathrm{HGP}} 

\newcommand{\rBF}{\mathrm{BF}} 
\newcommand{\rSF}{\mathrm{SF}} 
\newcommand{\rbGP}{\overline{\rGP}} 
\newcommand{\rbHGP}{\overline{\rHGP}} 

\newcommand{\rbbGP}{\overline{\overline{\rGP}}} 
\newcommand{\rbbA}{\overline{\overline{\rA}}} 
\newcommand{\wbbA}{\overline{\overline{\wA}}} 

\newcommand{\wP}{\mathbf{P}} 
\newcommand{\wQ}{\mathbf{Q}}
\newcommand{\wH}{\mathbf{H}} 
\newcommand{\wK}{\mathbf{K}}
\newcommand{\wL}{\mathbf{L}} 
\newcommand{\wM}{\mathbf{M}} 
\newcommand{\wG}{\mathbf{G}} 
\newcommand{\wPi}{\mathbf{\Pi}} 
\newcommand{\wA}{\mathbf{A}} 
\newcommand{\wSC}{\mathbf{SC}} 
\newcommand{\wGP}{\mathbf{GP}} 
\newcommand{\wbGP}{\overline{\mathbf{GP}}} 
\newcommand{\wbbGP}{\overline{\overline{\mathbf{GP}}}} 

\newcommand{\wbA}{\overline{\mathbf{A}}} 
\newcommand{\wbPi}{\overline{\mathbf{\Pi}}} 
\newcommand{\wBS}{\mathbf{BS}} 
\newcommand{\wWBS}{\mathbf{WBS}} 
\newcommand{\wHGP}{\mathbf{HGP}} 
\newcommand{\wHG}{\mathbf{HG}} 
\newcommand{\wSHG}{\mathbf{SHG}} 
\newcommand{\wSF}{\mathbf{SF}} 
\newcommand{\wW}{\mathbf{W}} 

\newcommand{\wF}{\mathbf{F}} 

\DeclareMathOperator{\apode}{\textsc{s}} 
\DeclareMathOperator{\sw}{\mathrm{sw}} 
\DeclareMathOperator{\cone}{\mathrm{cone}} 
\DeclareMathOperator{\tubes}{\mathrm{tubes}} 
\DeclareMathOperator{\ess}{\mathrm{ess}} 

\newcommand{\vanish}[1]{}
\let\onto=\twoheadrightarrow
\let\into=\hookrightarrow

\let\map=\xrightarrow

\newdir{ (}{*{ }*!/-5pt/@^{(}}
\newdir{|(}{*{ }*!/-5pt/@_{(}}
\newcommand{\xyinc}{\ar@{{ (}->}}
\newcommand{\xyrinc}{\ar@{{|(}->}}
\newcommand{\xyonto}{\ar@{->>}}
\newcommand{\xytwo}{\ar@{<->}}

\newcommand{\unsim}{\mathord{\sim}}

\newcommand{\abs}[1]{\lvert#1\rvert} 

\newcommand{\id}{\mathrm{id}}

\newcommand{\inc}{\mathrm{inc}} 
\newcommand{\cut}{\mathrm{cut}}
 
\newcommand{\low}{\mathrm{low}}
\newcommand{\rank}{\mathrm{rank}}
\newcommand{\espan}{\mathrm{span}}
\newcommand{\simple}[1]{#1^{\prime}}
\newcommand{\irred}{\mathrm{irred}} 
\newcommand{\supp}{\mathrm{supp}} 

\newcommand{\qand}{\quad\text{and}\quad}

\let\Kb=\Bbbk
\newcommand{\Xb}{\mathbb{X}}
\newcommand{\Rb}{\mathbb{R}}

\newcommand{\Zb}{\mathbb{Z}}
\newcommand{\Nb}{\mathbb{N}}
\newcommand{\Cb}{\mathbb{C}}

\newcommand{\wa}{\mathfrak{a}}
\renewcommand{\wp}{\mathfrak{p}}
\newcommand{\wq}{\mathfrak{q}}
\renewcommand{\wr}{\mathfrak{r}}
\newcommand{\Nc}{\mathcal{N}}
\newcommand{\Bc}{\mathcal{B}}
\newcommand{\Fc}{\mathcal{F}}
\newcommand{\Pc}{\mathcal{P}}

\newcommand{\B}{\mathcal{B}}
\newcommand{\C}{\mathcal{C}}

\newcommand{\G}{\mathcal{G}}
\renewcommand{\H}{\mathcal{H}}
\newcommand{\N}{\mathcal{N}}

\newcommand{\adjustheight}{23pt}

\newcommand{\gzero}{
\raisebox{-3pt}{
\rule{0pt}{\adjustheight}
{\color{blue}
\xymatrix@C-10pt{
*{\bullet} \ar@{-}[r]_(-0.07){\black{a}}_(1.0){\black{b}} & 
*{\bullet} \ar@{-}[r]_(1.07){\black{c}} & *{\bullet} \ar@(ul,ur)@{-} 
}
}
}
}

\newcommand{\gone}{
\raisebox{-3pt}{
\rule{0pt}{\adjustheight}
{\color{blue}
\xymatrix@C-10pt{
*{\bullet} \ar@{-}[r]_(-0.07){\black{a}}_(1.0){\black{b}} & 
*{\bullet} \ar@(ul,ur)@{-}  \ar@{}[r]_(1.07){\black{c}} & *{\bullet} \ar@(ul,ur)@{-} 
}
}
}
}

\newcommand{\gtwo}{
\raisebox{-3pt}{
\rule{0pt}{\adjustheight}
{\color{blue}
\xymatrix@C-10pt{
*{\bullet}  \ar@{}[r]_(-0.07){\black{a}}_(1.0){\black{b}} & 
*{\bullet} \ar@(ul,ur)@{-} \ar@{-}[r]_(1.07){\black{c}} & *{\bullet} \ar@(ul,ur)@{-} 
}
}
}
}

\newcommand{\gthree}{
\raisebox{-3pt}{
\rule{0pt}{\adjustheight}
{\color{blue}
\xymatrix@C-10pt{
*{\bullet} \ar@(ul,ur)@{-} \ar@{}[r]_(-0.07){\black{a}}_(1.0){\black{b}} & 
*{\bullet} \ar@{-}[r]_(1.07){\black{c}} & *{\bullet} \ar@(ul,ur)@{-} 
}
}
}
}

\newcommand{\gfour}{
\raisebox{-3pt}{
\rule{0pt}{\adjustheight}
{\color{blue}
\xymatrix@C-10pt{
*{\bullet} \ar@{-}[r]_(-0.07){\black{a}}_(1.0){\black{b}} & 
*{\bullet}   \ar@{}[r]_(1.07){\black{c}} & *{\bullet} \ar@(ul,u)@{-} \ar@(u,ur)@{-}
}
}
}
}

\newcommand{\gfive}{
\raisebox{-3pt}{
\rule{0pt}{\adjustheight}
{\color{blue}
\xymatrix@C-10pt{
*{\bullet} \ar@(ul,ur)@{-} \ar@{}[r]_(-0.07){\black{a}}_(1.0){\black{b}} & 
*{\bullet} \ar@(ul,ur)@{-} \ar@{}[r]_(1.07){\black{c}} & *{\bullet} \ar@(ul,ur)@{-} 
}
}
}
}

\newcommand{\gsix}{
\raisebox{-3pt}{
\rule{0pt}{\adjustheight}
{\color{blue}
\xymatrix@C-10pt{
*{\bullet}  \ar@{}[r]_(-0.07){\black{a}}_(1.0){\black{b}} & 
*{\bullet} \ar@(ul,u)@{-} \ar@(u,ur)@{-} \ar@{}[r]_(1.07){\black{c}} & *{\bullet} \ar@(ul,ur)@{-} 
}
}
}
}

\newcommand{\gseven}{
\raisebox{-3pt}{
\rule{0pt}{20pt}
{\color{blue}
\xymatrix@C-10pt{
*{\bullet}  \ar@{}[r]_(-0.07){\black{a}}_(1.0){\black{b}} & 
*{\bullet}  \ar@(ul,ur)@{-}  \ar@{}[r]_(1.07){\black{c}} & *{\bullet} \ar@(ul,u)@{-} \ar@(u,ur)@{-}
}
}
}
}

\newcommand{\geight}{
\raisebox{-3pt}{
\rule{0pt}{20pt}
{\color{blue}
\xymatrix@C-10pt{
*{\bullet}  \ar@(ul,ur)@{-} \ar@{}[r]_(-0.07){\black{a}}_(1.0){\black{b}} & 
*{\bullet}    \ar@{}[r]_(1.07){\black{c}} & *{\bullet} \ar@(ul,u)@{-} \ar@(u,ur)@{-}
}
}
}
}

\newcommand{\wzero}{
\raisebox{+6pt}{
\rule{0pt}{\adjustheight}
{\color{blue}
\xymatrix@C-10pt{
*{\bullet} \ar@{-}[r]_(-0.07){\black{1}}_(1.0){\black{2}} & 
*{\bullet} \ar@{-}[r]_(1.07){\black{3}} & *{\bullet} 
}
}
}
}

\newcommand{\wone}{
\raisebox{-3pt}{
\rule{0pt}{\adjustheight}
{\color{blue}
\xymatrix@C-10pt{
*{\bullet} \ar@{-}[r]_(-0.07){\black{1}}_(1.0){\black{2}} & 
*{\bullet}[r]_(1.07){\black{3}} & *{\bullet}
}
}
}
}

\newcommand{\wtwo}{
\raisebox{-3pt}{
\rule{0pt}{\adjustheight}
{\color{blue}
\xymatrix@C-10pt{
*{\bullet} \ar@{-}[r]_(-0.07){\black{1}}_(1.0){\black{2}} & 
*{\bullet}[r]_(1.07){\black{3}} & *{\bullet}
}
}
}
}

\newcommand{\wthree}{
\raisebox{-3pt}{
\rule{0pt}{\adjustheight}
{\color{blue}
\xymatrix@C-10pt{
*{\bullet} \ar@{-}[r]_(-0.07){\black{1}}_(1.0){\black{2}} & 
*{\bullet}[r]_(1.07){\black{3}} & *{\bullet}
}
}
}
}

\newcommand{\wfour}{
\raisebox{-3pt}{
\rule{0pt}{\adjustheight}
{\color{blue}
\xymatrix@C-10pt{
*{\bullet} \ar@{-}[r]_(-0.07){\black{1}}_(1.0){\black{2}} & 
*{\bullet}[r]_(1.07){\black{3}} & *{\bullet}
}
}
}
}

\begin{document}
\title{Hopf monoids and generalized permutahedra}
\author{
\textsf{Marcelo Aguiar\footnote{\noindent \textsf{Cornell University; maguiar@math.cornell.edu}}}
\and
\textsf{Federico Ardila\footnote{\noindent \textsf{San Francisco State University, Mathematical Sciences Research Institute, U. de Los Andes; federico@sfsu.edu.
\newline 
Aguiar was supported by NSF grants DMS-0600973 and DMS-1001935. Ardila was supported by NSF CAREER grant DMS-0956178, research grants DMS-0801075 and DMS-1600609, and grant 
DMS-1440140 (MSRI).}}}
}\date{}
\maketitle

\begin{abstract}
Generalized permutahedra are a family of polytopes with a rich combinatorial structure and strong connections to optimization. 
We prove that they are the universal family of polyhedra with a certain Hopf algebraic structure. Their antipode is remarkably simple: the antipode of a polytope is the alternating sum of its faces. 
Our construction provides a unifying framework to organize numerous combinatorial structures, including graphs, matroids, posets, set partitions, linear graphs, hypergraphs, simplicial complexes, building sets, and simple graphs. 
We highlight three applications: \\
1.  We obtain uniform proofs of numerous old and new results about the Hopf algebraic and combinatorial structures of these families. In particular, we give the optimal formula for the antipode of graphs, posets, matroids, hypergraphs, and building sets, and we answer questions of Humpert--Martin and Rota.\\
2. We show that the reciprocity theorems of Stanley and Billera--Jia--Reiner on chromatic polynomials of graphs, order polynomials of posets, and BJR-polynomials of matroids are instances of the same reciprocity theorem for generalized permutahedra. \\
3.  We explain why the formulas for the multiplicative and compositional inverses of power series are governed by the face structure of permutahedra and associahedra, respectively, answering a question of Loday.\\
Along the way, we offer a combinatorial user's guide to  Hopf monoids.
\end{abstract}

\addtocontents{toc}{\protect\setcounter{tocdepth}{1}}
\tableofcontents

\section{{Introduction}}\label{s:intro}

%


\noindent \textbf{\textsf{Hopf monoids and generalized permutahedra.}} Joyal \cite{joyal98}, Joni and Rota \cite{joni82:_coalg}, Schmitt \cite{schmitt93:_hopf}, Stanley \cite{stanley07cca} and others taught us that to study combinatorial objects, it is often useful to endow them with algebraic structures. 
Aguiar and Mahajan's \emph{Hopf monoids in species} \cite{am} provide a particularly useful framework to study many important combinatorial families. 

Edmonds \cite{edmonds70}, Lov\'asz \cite{lovasz09}, Postnikov \cite{postnikov09}, Stanley \cite{stanley72:_order}, and others taught us that to study combinatorial objects, it is often useful to build a polyhedral model for them. Generalized permutahedra (also known as polymatroids and equivalent to submodular functions) are a particularly useful family of polytopes which model many combinatorial families. 

The main idea of this article is to bring together these two points of view. We endow the family of generalized permutahedra with a Hopf algebraic structure: the \emph{Hopf monoid of generalized permutahedra} $\rGP$. In fact we show that, in a precise sense, generalized permutahedra are the only family of polytopes which supports such a Hopf structure. We then use this framework to unify classical results, discover new results, and answer open questions about numerous combinatorial families of interest. 
We highlight three areas of application.

\bigskip

\noindent \textbf{\textsf{Application A. Antipodes and combinatorial formulas.}}
Many families of combinatorial objects have natural operations of \emph{merging} two disjoint objects into one, and \emph{breaking} an object into two disjoint parts. Under some hypotheses, these operations give a product and coproduct in a Hopf monoid or algebra. In this paper we will consider many such Hopf structures: graphs, matroids, posets, set partitions, simplicial complexes, building sets, and simple graphs, to name a few.

Any Hopf monoid has an \emph{antipode map} $\apode$, which is analogous to the inverse map in a group. The antipode is given by a very large alternating sum, generally involving lots of cancellation. A fundamental and highly nontrivial question is to give a cancellation-free formula for this antipode. Let's see a few examples of antipodes:


\bigskip

\noindent 
Graphs $\wG$:
 \begin{figure}[H]
\flushleft
\includegraphics[scale=.5]{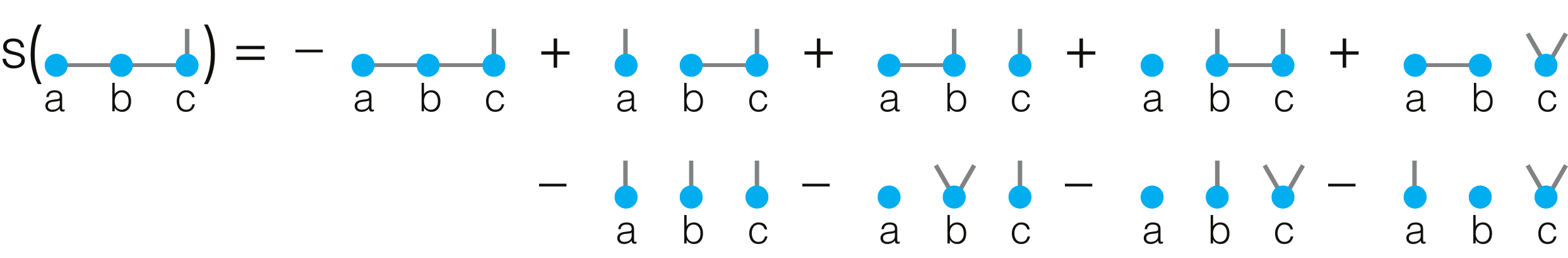} \qquad
\end{figure}

\noindent 
Matroids $M$:

 \begin{figure}[H]
\flushleft
\includegraphics[scale=.5]{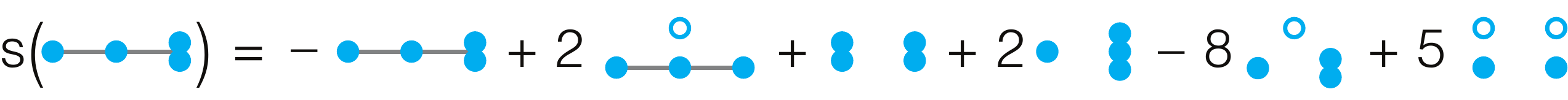} \qquad
\end{figure}

\noindent 
Posets $P$:

 \begin{figure}[H]
\flushleft
\includegraphics[scale=.5]{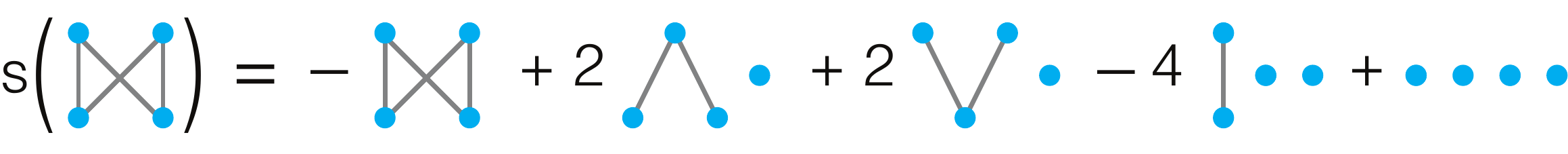} \qquad
\end{figure}

\noindent 
Partitions $\wPi$:

 \begin{figure}[H]
\flushleft
\includegraphics[scale=.5]{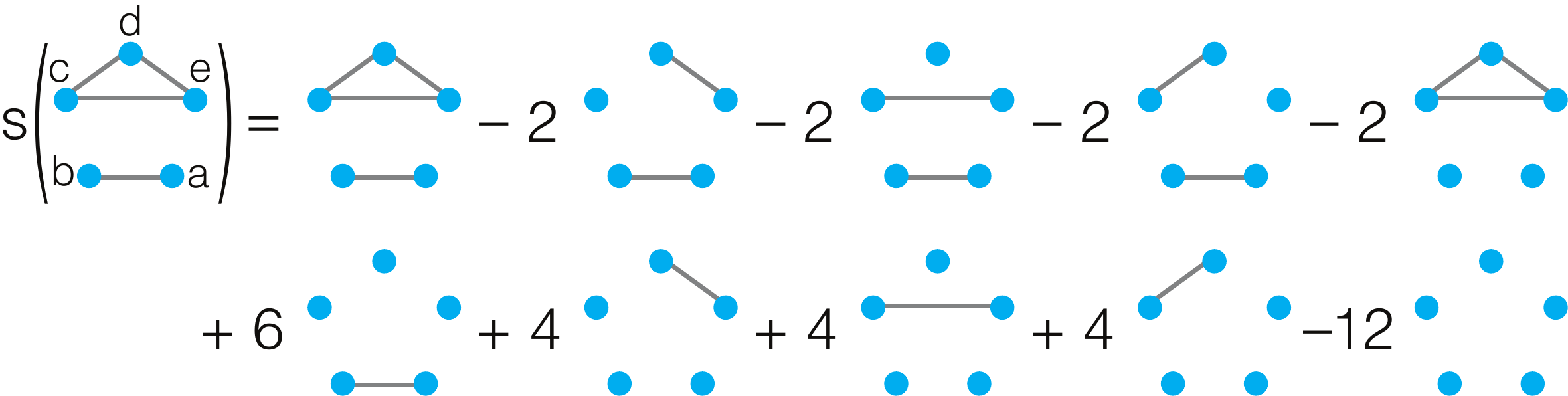} \qquad
\end{figure}

\bigskip

\noindent 
Paths $\wF$:

 \begin{figure}[H]
\flushleft
\includegraphics[width=\textwidth]{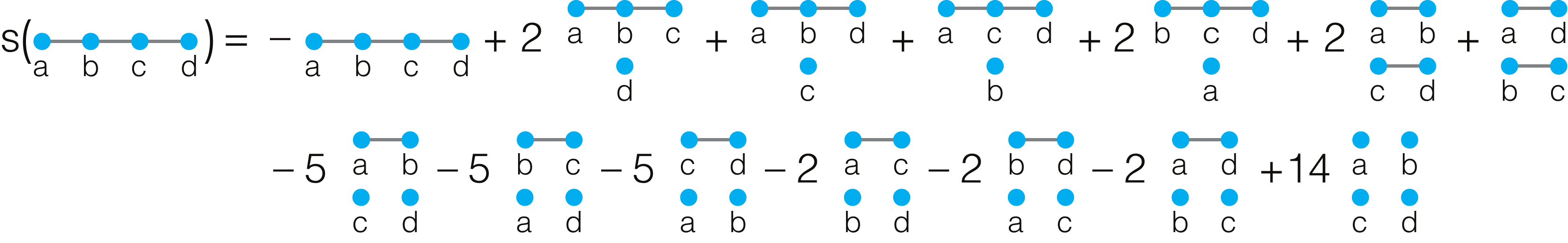} \qquad
\end{figure}



These formulas result from simplifying alternating sums of 13, 75, 75, 541, and 75 terms, respectively. We provide a uniform explanation, showing that each one of these formulas has a  polyhedron controlling it.

We prove that the Hopf monoids of these five structures $\rG, \rM, \rP, \rPi, \rF$ are related to the Hopf monoid $\rGP$ of generalized permutahedra -- a setting that is ideal for the question at hand, because it provides the geometric and topological structure necessary to understand the cancellation completely. The following is one of our main theorems.

\begin{theorem}\label{t:antipode}
The antipode of the Hopf monoid $\rGP$ is given by the following \textbf{cancellation-free} and \textbf{grouping-free} formula: If $\wp \in \Rb I$ is a generalized permutahedron, then
\[
\apode_I(\wp) = (-1)^{|I|}\sum_{\wq \textrm{ face of } \wp} (-1)^{\dim \wq} \, \wq.
\]
\end{theorem}

From this formula, it becomes straightforward to interpret the five formulas above, and many others. In fact, the polyhedral point of view leads to the construction of other natural Hopf monoids. This framework allows us to compute for the first time the antipode of many Hopf structures of interest. 

Our results in this direction and other earlier results are summarized in the table below. Each row mentions a combinatorial family of objects, the polyhedra modeling them, the first construction of a Hopf structure (algebra or monoid) on this family, and the first cancellation-free computation of the antipode. Some entries indicate that other researchers obtained similar results independently and essentially simultaneously.

\begin{table}[h]
\centering
\begin{tabular}{|l|l|l|l|}
\hline
objects & polytopes & Hopf structure & antipode  \\
\hline
\hline
set partitions & permutahedra & \brown{Joni-Rota} & \brown{Joni-Rota}
 \\
\hline
paths & associahedra & \brown{Joni-Rota}, \green{new} & \brown{Haiman-Schmitt}, \green{new} \\
\hline
graphs & graphic zonotopes & \brown{Schmitt} & \green{new}+\brown{Humpert-Martin} \\
\hline
matroids & matroid polytopes & \brown{Schmitt} & \green{new} \\
\hline
posets & braid cones & \brown{Schmitt} & \green{new} \\
\hline
submodular fns & generalized permutahedra & \brown{Derksen-Fink} & \green{new} \\
\hline
hypergraphs & hypergraphic polytopes &  \green{new} & \green{new} \\
\hline
simplicial cxes & \green{ simplicial cx polytopes} &\brown{Benedetti et. al.} & \brown{Benedetti et. al.} \\
\hline
building sets & nestohedra & \green{new}+\brown{Gruji\'c} & \green{new} \\ 
\hline
simple graphs & graph associahedra & \green{new} & \green{new} \\
\hline
\end{tabular}
\caption{The main combinatorial Hopf structures and antipode theorems in this paper.\label{table:antipodes}}
\end{table}

\noindent All the earlier Hopf structures listed above are Hopf algebras. They all  have Hopf monoids that specialize to them; for details, see Section \ref{ss:Fock} or \cite[Part III]{am}.

\bigskip

\noindent \textbf{\textsf{Application B. Reciprocity theorems.}}
Consider the following polynomials: the \emph{chromatic polynomial} $\chi_g$ of a graph $g$,  the \emph{Billera--Jia--Reiner polynomial} $\chi_m$ of a matroid $m$, and
the \emph{strict order polynomial} $\chi_p$ of a poset $p$. 
These polynomials are determined by the following properties which hold for $n \in \Nb$: \\
\noindent $\bullet$ $\chi_g(n)$ = number of proper vertex $n$-colorings of $g$.\\
\noindent $\bullet$ $\chi_p(n)$ = number of strictly order preserving $n$-labellings of $p$.\\
\noindent $\bullet$ $\chi_m(n)$ = number of $n$-weightings of $m$ under which $m$ has a unique maximum basis. \\
In each case, it is true -- but not clear a priori -- that a polynomial exists with those properties. 

There is no reason to expect that plugging in negative values into these polynomials should have any combinatorial significance. However, these polynomials satisfy the following \emph{combinatorial reciprocity theorems}. For $n \in \Nb$:

\noindent $\bullet$ $|\chi_g(-n)|$ = number of compatible pairs of an $n$-coloring and an acyclic orientation of $g$.\\
\noindent $\bullet$ $|\chi_p(-n)|$ = number of weakly order preserving $n$-labellings of $p$.\\
\noindent $\bullet$ $|\chi_m(-n)|$ = number of pairs of an $n$-weighting $w$ of $m$ and a $w$-maximum basis. 

We will see that these are three instances of the same general result: Any character $\zeta$ in a Hopf monoid gives rise to a polynomial invariant $\chi_e(n)$ for each element $e$ of the monoid. Furthermore, this polynomial satisfies a reciprocity rule that gives a combinatorial interpretation of $|\chi_e(-n)|$ for $n \in \Nb$. The three statements above are straightforward consequences of this general theory. In fact, they are special cases of the same theorem for generalized permutahedra under the inclusions of $\wG, \wM,$ and $\wP$ into $\wGP$.
Closely related results were obtained by Billera, Jia, and Reiner in \cite{billera06}.

\bigskip

\noindent \textbf{\textsf{Application C. Inversion of formal power series.}}
The left panel of
Figure \ref{f:permassoc} shows the first few \emph{permutahedra}: a point $\pi_1$, a segment $\pi_2$, a hexagon $\pi_3$, and a truncated octahedron $\pi_4$. There is one permutahedron in each dimension, and every face of a permutahedron is a product of permutahedra.
The right panel shows the first few \emph{associahedra}: a point $\wa_1$, a segment $\wa_2$, a hexagon $\wa_3$, and a three-dimensional associahedron $\wa_4$. There is one associahedron in each dimension, and every face of an associahedron is a product of associahedra.

\begin{figure}[h]
\begin{center}
\includegraphics[scale=.1]{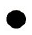} \quad
\includegraphics[scale=.3]{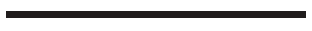}  \quad
\includegraphics[scale=.3]{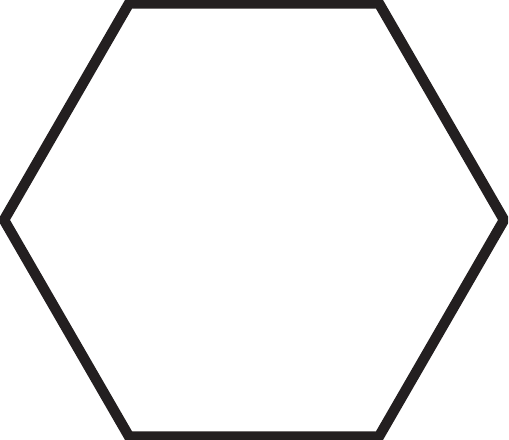} \quad
\includegraphics[scale=.35]{Figures/perm} \qquad \qquad \quad
\includegraphics[scale=.1]{Figures/point} \quad
\includegraphics[scale=.3]{Figures/line}  \quad
\includegraphics[scale=.3]{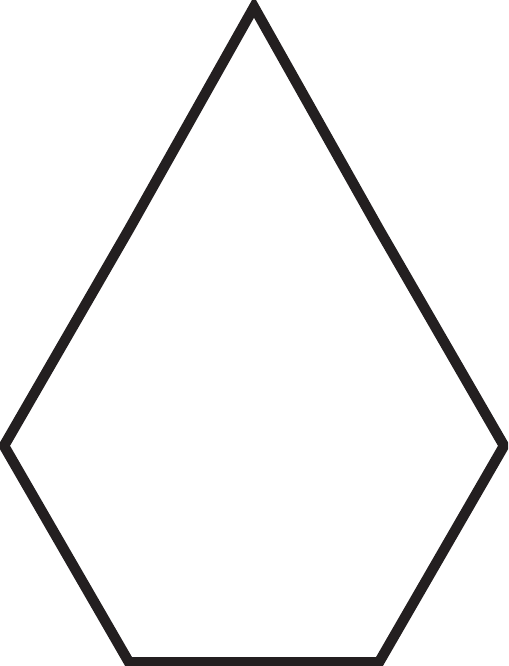} \quad
\includegraphics[scale=.2]{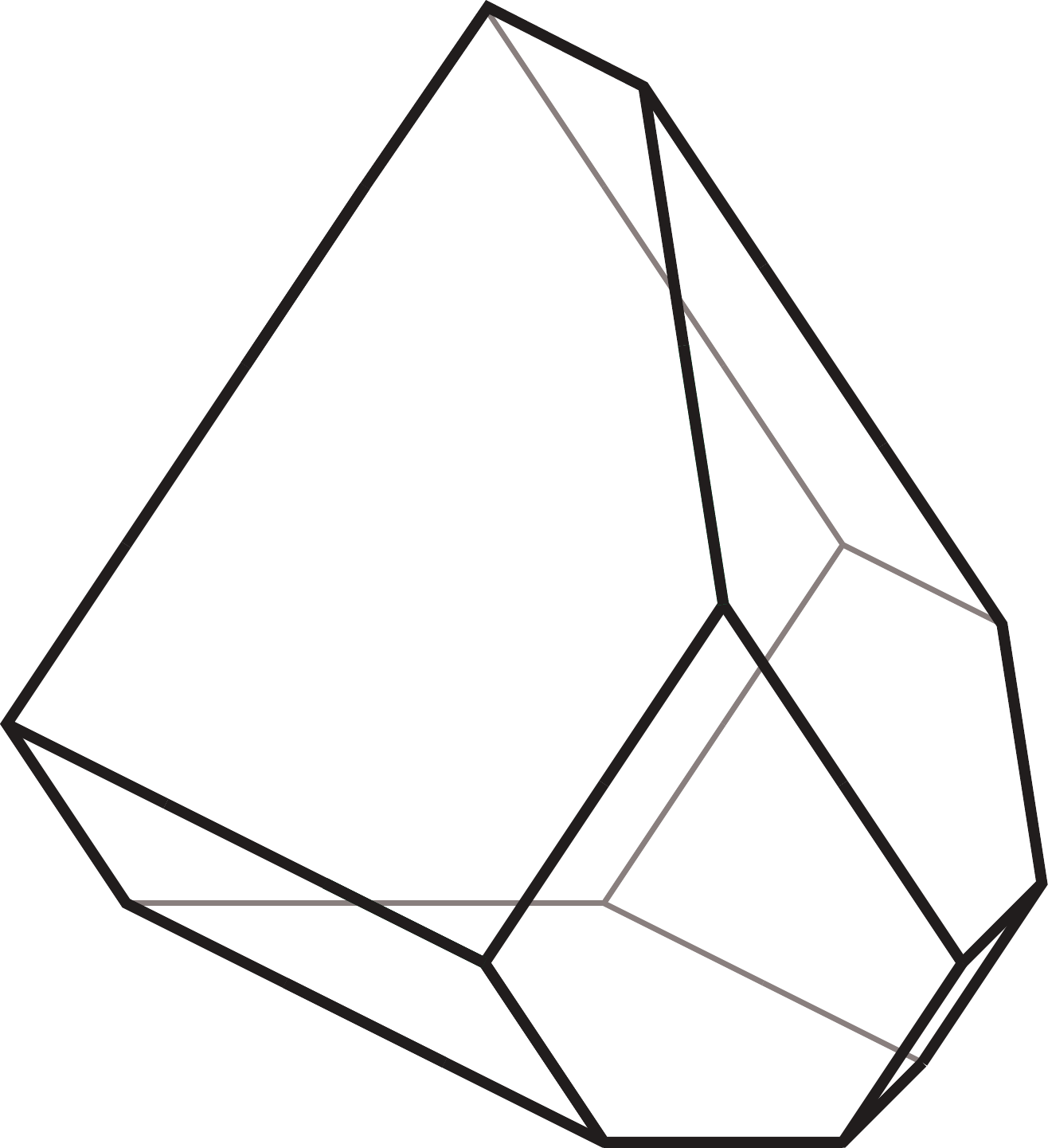} 
\caption{Left: The permutahedra $\pi_1, \pi_2, \pi_3, \pi_4$. Right: 
the associahedra $\wa_1, \wa_2, \wa_3 ,\wa_4$.\label{f:permassoc}}
\end{center} 
\end{figure}

\vspace{-.2cm}

\noindent C1. \textsf{Multiplicative Inversion.} 
Consider formal power series
\[
A(x) = \sum_{n \geq 0} a_n \frac{x^n}{n!} \quad \textrm{ and } \quad B(x) = \sum_{n \geq 0} b_n \frac{x^n}{n!} \quad \textrm{ such that } \quad A(x)B(x)=1,
\]
assuming for simplicity $a_0=1$. The first few coefficients of $B(x)=1/A(x)$ are:
\begin{eqnarray*}
b_1 &=& -a_1 \\
b_2 &=& -a_2+2a_1^2 \\
b_3 &=& -a_3+6a_2a_1-6a_1^3 \\
b_4 &=& -a_4+8a_3a_1+6a_2^2-36a_2a_1^2+24a_1^4  
\end{eqnarray*}

What do these numbers count? The face structure of permutahedra tells the full story: for example, the formula for $b_4$ comes from  the faces of the permutahedron $\pi_4$:
1 truncated octahedron {$\pi_4$}, 
 {8} hexagons {$\pi_3 \times \pi_1$} and {6} squares {$\pi_2 \times \pi_2$}, 
 {36} segments {$\pi_2 \times \pi_1 \times \pi_1$}, and 
 {24} points {$\pi_1 \times \pi_1 \times \pi_1 \times \pi_1$}. The signs in the formula are given by the dimensions of the faces.

\bigskip

\noindent  C2. \textsf{Compositional Inversion.} 
Consider formal power series
\[
C(x) = \sum_{n \geq 1} c_{n-1} x^n \quad \textrm{ and } \quad D(x) = \sum_{n \geq 1} d_{n-1}x^n \textrm{ such that } \quad C(D(x))=x,
\]
assuming for simplicity $c_0=1$. The first few coefficients of $D(x)=C(x)^{\langle -1 \rangle}$ are:
{\begin{eqnarray*}
d_1 &=& -c_1 \\
d_2 &=& -c_2+2c_1^2 \\
d_3 &=& -c_3+5c_2c_1-5c_1^3 \\
d_4 &=& -c_4+6c_3c_1+3c_2^2-21c_2c_1^2+14c_1^4
\end{eqnarray*}
}
What do these numbers count? Now it is the face structure of associahedra that tells the full story: for example, the formula for $d_4$ comes from  the faces of the associahedron $\wa_4$:
 {1} three-dimensional associahedron {$\wa_4$}, 
 {6} pentagons {$\wa_3 \times \wa_1$} and {3} squares {$\wa_2 \times \wa_2$}, 
 {21} segments {$\wa_2 \times \wa_1 \times \wa_1$}, and 
 {14} points {$\wa_1 \times \wa_1 \times \wa_1 \times \wa_1$}. The signs in the formula are given by the dimensions of the faces.
 
\bigskip
 
The problem of inverting power series is classical. Combinatorial formulas for the coefficients of $b_n$ and $d_n$ above and combinatorial formulas for the face enumeration of permutahedra and associahedra have been known for a long time; and these formulas do coincide. 
However, our treatment seems to be the first to truly explain the geometric  connection. We derive these inversion formulas in a unified fashion, directly from the combinatorial and topological structure of these polytopes. In the case of Lagrange inversion and associahedra, this answers a 2005 question of Loday \cite{loday05}.

\subsection{Outline}

The paper is roughly divided into four parts. Part 1 and the sections labeled \emph{Preliminaries} (Sections 2, 4, 8, and 16) contain general results that are prerequisites to the other parts. Parts 2, 3, and 4 are also interconnected in several ways, but we attempted to make the exposition of each one of them as self-contained as possible, so they can mostly be read independently of each other.

\subsubsection{Part 1: The Hopf monoid $\rGP$ and its antipode. (Sections 2-7) } The first part establishes the foundational definitions, examples, and results. 
In Section \ref{s:hopf} we define Hopf monoids and state some key general results. To illustrate the ubiquity of Hopf monoids in combinatorics, Section \ref{s:G,M,P}  provides five examples which we will return to throughout the paper: set partitions, paths, graphs, matroids, and posets. Section \ref{s:prelimGP} defines generalized permutahedra and collects the basic discrete geometric facts that will be important to us. Section \ref{s:GP} shows that generalized permutahedra $\rGP$ have the structure of a Hopf monoid, and Section \ref{s:universality} shows that, in a precise sense, generalized permutahedra are the universal family of polytopes supporting such an algebraic structure. Finally, in Section \ref{s:antipode} we prove one of our main results: that the antipode of a polytope $\wp$ is the alternating sum of its faces.

\subsubsection{Part 2: Permutahedra, associahedra, and inversion (Sections 8-11) } 
The second part reveals the relationship between Hopf monoids and algebraic operations on power series. In Section \ref{s:prelimcharacters} we introduce the characters of a Hopf monoid and show how they assemble into a group of characters. 
Section \ref{s:Pi} shows that the group of characters of a Hopf monoid of permutahedra $\rbPi$ is isomorphic to the group of invertible power series under multiplication.
Section \ref{s:F} shows that the group of characters of a Hopf monoid of associahedra $\rF$ is isomorphic to the group of invertible power series under composition.
Section \ref{s:inversion} then uses these results to and the antipode of $\wGP$ to give a unified geometric topological explanation for the formulas to compute the inverse of a power series under multiplication and composition.

\subsubsection{Part 3: Characters, polynomial invariants, and reciprocity (Sections 12-18) } 
The third part is more combinatorial in nature. We begin with Section \ref{s:SF} which shows a bijection between generalized permutahedra and submodular functions. This partially explains why generalized permutahedra appear in so many different settings: submodular functions model situations where a very natural diminishing property holds. This allows us to realize graphs, matroids, and posets as submonoids $\rG, \rM, \rP$ of the Hopf monoid of generalized permutahedra in the next three sections. In Sections \ref{s:G}, \ref{s:M}, and \ref{s:P} respectively we recall how graphs, matroids, and posets are modeled in $\rSF$ via their cut functions, rank functions, and order ideal indicator functions, and in $\rGP$ via their graphic zonotopes, matroid polytopes, and poset cones. This allows us to compute the antipodes of $\rG, \rM,$ and $\rP$ for the first time. The antipode of $\rG$ was also computed by Humpert and Martin, and we prove their conjectures on characters of complete graphs.

We then turn to combinatorial invariants and reciprocity theorems. Section \ref{s:pol-inv} shows that any character on a Hopf monoid gives rise to a polynomial invariant of the objects of study, and that invariant satisfies a reciprocity theorem. Section \ref{s:basic} carries out this construction for the simplest non-trivial character of $\rGP$. 
This allows us to explain in Section \ref{s:reciprocity} how two classical theorems of Stanley on graphs and posets and a theorem of Billera-Jia-Reiner on matroids are really the same theorem.

\subsubsection{Part 4: Hypergraphs and hypergraphic polytopes (Sections 19-24) } 

The fourth and final part focuses on a subfamily of generalized permutahedra which inherits the Hopf algebraic structure from $\rGP$ while containing additional combinatorial structure; we call them hypergraphic polytopes. 
We introduce and characterize this family of polytopes $\rHGP$ in Section \ref{s:HGP}, answering a question of Rota.  
Section \ref{s:HG} introduces a Hopf monoid structure on hypergraphs $\rHG$ which generalizes $\rG$ and is isomorphic to $\rHGP$. We then use generalized permutahedra to study several interesting submonoids of $\rHG$. 
Section \ref{s:SC} recasts Benedetti et. al.'s Hopf structure on simplicial complex $\rSC$ geometrically, thus explaining the mysterious similarity between the antipodes for $\wG$ and $\wSC$. 
In Section \ref{s:BS}, nestohedra give building sets the structure of a Hopf monoid $\rBS$ and control its antipode. 
This gives rise to a new Hopf monoid of graphs $\rW$ in Section \ref{s:W}.
Finally, Sections \ref{s:Pirevisited} and \ref{s:Frevisited} show how $\rW$ contains the Hopf monoids $\rPi$ and $\rF$ of set partitions (and permutahedra) and paths (and associahedra) respectively, giving rise to some new enumerative consequences.

\subsection{Conventions}
We work over  a field $\Kb$ of characteristic $0$. We use the font $\rH$ for Hopf monoids in set species,  $\wH$ for Hopf monoids in vector species, and $H$ for Hopf algebras.

\bigskip
\bigskip

\begin{LARGE}
\noindent 
\textsf{PART 1: The Hopf monoid $\wGP$ and its antipode. 
}
\end{LARGE}

\section{{Preliminaries 1: A user's guide to Hopf monoids in species}}\label{s:hopf}

The theory of Hopf monoids in species developed by Aguiar and Mahajan in \cite{am} provides a useful algebraic setting to study many families of combinatorial objects of interest, as follows. The families under study have operations of \emph{merging} two disjoint objects into one, and \emph{breaking} an object into two disjoint parts. Under some hypotheses, these operations define a product and coproduct in a Hopf monoid. One can then use the general theory to obtain numerous combinatorial consequences. In Sections \ref{s:hopf}, \ref{s:prelimcharacters}, and \ref{s:pol-inv} we outline the most relevant combinatorial features of this theory. 
Our exposition is self-contained;  the interested reader may find
more details on some of these constructions in~\cite{am}.

\subsection{{Set species}} 

We begin by reviewing Joyal's notion of set species~\cite{bergeron98:_combin,joyal81:_une}. This is a  framework, rooted in category theory, used to systematically study combinatorial families and the relationships between them.

\begin{definition}\label{d:setsp}
A \emph{set species} $\rP$ consists of the following data.
\begin{itemize}
\item For each finite set $I$, a set $\rP[I]$.
\item For each bijection $\sigma:I\to J$, a map $\rP[\sigma]:\rP[I]\to \rP[J]$. These should be such that $\rP[\sigma\circ\tau]=\rP[\sigma]\circ\rP[\tau]$ and $\rP[\id]=\id$.
\end{itemize}
\end{definition}

\noindent It follows that each map $\rP[\sigma]$ is invertible, with inverse $\rP[\sigma^{-1}]$.
Sometimes we refer to an element $x\in\rP[I]$ as a \emph{structure} (of species 
$\rP$) on the set $I$.

In the examples that interest us, $\rP[I]$ is the set of all combinatorial structures of a certain kind that can be constructed on the ground set $I$. For each bijection $\sigma:I\to J$, the map $\rP[\sigma]$ takes each structure on $I$ and relabels its ground set to $J$ according to $\sigma$.

\begin{example}\label{eg:speciesL}
Define a set species $\rL$ as follows. For any finite set $I$,
$\rL[I]$ is the set of all linear orders on $I$. If $\ell$ is a linear order on $I$
and $\sigma:I\to J$ is a bijection, then $\rL[\sigma](\ell)$ is the linear order on $J$ for which $j_1<j_2$ if $\sigma^{-1}(j_1)<\sigma^{-1}(j_2)$ in $\ell$. If we regard $\ell$ as a list of the elements of $I$, then 
$\rL[\sigma](\ell)$ is the list obtained by replacing each $i\in I$ for $\sigma(i)\in J$.

For instance, $\rL[\{a,b,c\}]=\{abc,\,bac,\,acb,\,bca,\,cab,\,cba\}$
and if $\sigma:\{a,b,c\}\to\{1,2,3\}$ is given by $\sigma(a)=1, \sigma(b)=2, \sigma(c)=3$,  
then $\rL[\sigma]:\rL[\{a,b,c\}]\to\rL[\{1,2,3\}]$ is given by $\sigma(abc)=123, \sigma(acb)=132, \sigma(bac)=213, \sigma(bca)=231,\sigma(cab)=312, \sigma(cba)=321$.
\end{example}

\begin{definition}\label{d:morsetsp}
A \emph{morphism} $f:\rP\to\rQ$ between set species $\rP$ and $\rQ$ is a collection of maps
$
f_I:\rP[I]\to\rQ[I]
$
which satisfy the following \emph{naturality} axiom: for each bijection $\sigma:I\to J$,
$
f_J\circ \rP[\sigma] = \rQ[\sigma]\circ f_I.
$
\end{definition}

\begin{example}
An automorphism of the Hopf monoid $\rL$ of linear orders is given by the reversal maps $\mathrm{rev}_I: \rL[I] \to \rL[I]$ defined by $\mathrm{rev}_I(a_1a_2\ldots a_i) = a_i\ldots a_2a_1$ for each linear order on $I$ written as a list $a_1a_2\ldots a_i$. 
\end{example}

\subsection{{Hopf monoids in set species}}\label{ss:hopf-set}

A \emph{decomposition} of a finite set $I$ is a finite sequence $(S_1,\ldots,S_k)$
of pairwise disjoint subsets of $I$ whose union is $I$. In this situation, we write
\[
I=S_1\sqcup\cdots\sqcup S_k.
\]
Note that $I=S\sqcup T$ and $I=T\sqcup S$ are distinct decompositions of $I$
(unless $I=S=T=\emptyset$). A \emph{composition} is a decomposition where all parts
are non-empty.


\begin{definition}\label{d:hopfset}
A \emph{connected Hopf monoid in set species} consists of the following data.
\begin{itemize}
\item A set species $\rH$ such that the set $\rH[\emptyset]$ is a singleton.
\item For each finite set $I$ and each decomposition $I=S\sqcup T$, product and coproduct maps 
\[
\rH[S]\times\rH[T] \map{\mu_{S,T}}\rH[I]
\qand
\rH[I]\map{\Delta_{S,T}} \rH[S]\times\rH[T]
\]
\end{itemize}
satisfying the naturality, unitality, associativity, and compatibility axioms below.

Before stating those axioms, we discuss some terminology and notation. The collection of maps $\mu$ (resp.\ $\Delta$) is called the \emph{product} (resp.\ the \emph{coproduct}) of the Hopf monoid $\rH$. 
Fix a decomposition $I=S\sqcup T$. For $x\in \rH[S]$, $y\in \rH[T]$, and $z \in \rH[I]$ we write
\[
(x,y) \map{\mu_{S,T}} x \cdot y
\qand
z \map{\Delta_{S,T}} (z|_S,  z/_S).
\]
We call $x \cdot y \in \rH[I]$ the \emph{product} of $x$ and $y$, 
$z|_S\in\rH[S]$  the \emph{restriction of $z$ to $S$}
and $z/_S\in\rH[T]$ the \emph{contraction of $S$ from $z$}.
%
%
Finally, we call the element $1 \in \rH[\emptyset]$ 
the \emph{unit} of $\rH$.

In most combinatorial applications, the product keeps track of how we merge two disjoint structures $x$ on $S$ and $y$ on $T$ into a single structure $x \cdot y$ on $I$, according to a suitable combinatorial rule. The coproduct keeps track of how we break up a structure $z$ on $I$ into a structure $z|_S$ on $S$ and a structure $z/_S$ on $T$. Section \ref{s:G,M,P} features five important examples.

The axioms are as follows.

\begin{naturality}
For each decomposition $I=S\sqcup T$ each bijection $\sigma:I\to J$, and any choice of 
 $x\in\rH[S]$, $y\in\rH[T]$, and $z\in\rH[I]$, we have
\begin{gather*}
\rH[\sigma](x\cdot y) = \rH[\sigma|_S](x)\cdot \rH[\sigma|_T](y),\\ 
\rH[\sigma](z)|_S = \rH[\sigma|_S](z|_S), \qquad 
\rH[\sigma](z)/_S = \rH[\sigma|_T](z/_S).
\end{gather*}
For combinatorial families, this means that the relabelling maps respect the merging and breaking operations.
\end{naturality}

\begin{unitality}
For each $I$ and $x\in\rH[I]$, we must have
\[
x\cdot 1 = x = 1\cdot x,
\qquad
x|_I = x = x/_\emptyset.
\]
Combinatorially this means that the merging and breaking of structures is trivial when the decomposition of the underlying set $I$ is trivial; here $1$ represents the unique (and usually trivial) structure on the empty set.
\end{unitality}

\begin{associativity}
For each decomposition $I=R\sqcup S\sqcup T$, and any
$x\in\rH[R]$, $y\in\rH[S]$, $z\in\rH[T]$, and $w\in\rH[I]$, we must have 
\begin{gather*}
x\cdot(y\cdot z) = (x\cdot y)\cdot z, \\
(w|_{R\sqcup S})|_R=w|_R,
\qquad
(w|_{R\sqcup S})/_R=(w/_R)|_S,
\qquad
w/_{R\sqcup S}=(w/_R)/_S.
\end{gather*}
Combinatorially, this means that the merging of three combinatorial structures on $R,S,T$ into one structure on $I$ is well defined, and the breaking up of a single structure on $I$ into three structures on $R,S,T$ is well-defined. By induction, merging and breaking up are then also well-defined for decompositions of $I$ into more than three parts.
\end{associativity}

\begin{compatibility}
Fix decompositions $S\sqcup T=I=S'\sqcup T'$,
and consider the pairwise intersections $A:=S\cap S',\ B:=S\cap T',\ C:=T\cap S',\ D:=T\cap T'$
as illustrated below. In this situation, for any $x\in\rH[S]$ and $y\in\rH[T]$, we must have
\[
 (x\cdot y)|_{S'}  = x|_A\cdot y|_C 
 \qand
 (x\cdot y)/_{S'} = x/_A\cdot y/_C.
\]
\begin{equation}\label{e:4sets}
\begin{gathered}
\begin{picture}(100,90)(20,0)
  \put(50,40){\oval(100,80)}
  \put(0,40){\dashbox{2}(100,0){}}
  \put(45,55){$S$}
  \put(45,15){$T$}
\end{picture}
\quad
\begin{picture}(100,90)(10,0)
  \put(50,40){\oval(100,80)}
  \put(50,0){\dashbox{2}(0,80){}}
  \put(20,35){$S'$}
  \put(70,35){$T'$}
\end{picture}
\quad
\begin{picture}(100,90)(0,0)
\put(50,40){\oval(100,80)}
  \put(0,40){\dashbox{2}(100,0){}}
  \put(50,0){\dashbox{2}(0,80){}}
  \put(20,55){$A$}
  \put(70,55){$B$}
  \put(20,15){$C$}
  \put(70,15){$D$}
\end{picture}
\end{gathered}
\end{equation}
Combinatorially, this says that ``merging then breaking" is the same as "breaking then merging". If we start with structures $x$ and $y$ on $S$ and $T$, we can merge them into a structure $x \cdot y$ on $I$, and then break the result into structures on $S'$ and $T'$. We can also break $x$ (resp. $y$) into two structures  on $A$ and $B$ (resp. $C$ and $D$), and merge the resulting pieces into structures on $S'$ and $T'$. These two procedures should give the same answer.

\end{compatibility}
This completes the definition of connected Hopf monoid
in set species. In the cases that interest us, naturality and unitality are immediate and associativity is very easy; usually the only non-trivial condition to be checked is compatibility.
\end{definition}

%

\begin{definition}\label{d:morhopf}
A \emph{morphism} $f:\rH\to\rK$ between Hopf monoids $\rH$ and $\rK$ is a morphism of species which preserves products, restrictions and contractions; that is, we have 
\begin{tabular}{ll}
\quad $f_J\bigl(\rH[\sigma](x)\bigr) = \rK[\sigma]\bigl(f_I(x)\bigr)$ &
for all bijections $\sigma:I\to J$ and all $x\in\rH[I]$, \\ 
\quad $f_I(x\cdot y) = f_S(x)\cdot f_T(y)$  & for all $I = S \sqcup T$ and all $x\in\rH[S]$, $y\in\rH[T]$,\\
\quad $f_S(z|_S)=f_I(z)|_S$, \,\,  $f_T(z/_S)=f_I(z)/_S$  & for all $I = S \sqcup T$ and all $z\in\rH[I]$.
\end{tabular}\\
Units are preserved because of connectedness.
\end{definition}

Suppose $\rH$ is a Hopf monoid.
Note that if $I=S\sqcup T$ is a decomposition, then $I=T\sqcup S$ is another.
Therefore, any $x\in\rH[S]$ and $y\in\rH[T]$ give rise to two products
 $x\cdot y$ and $y\cdot x$. Similarly, any $z\in\rH[I]$ gives rise to two 
 coproducts $(z|_S, z/_S)$ and $(z|_T, z/_T)$.  

\begin{definition}\label{d:comm}
A Hopf monoid $\rH$ 
is \emph{commutative} if $x\cdot y =y\cdot x$ for any $I=S\sqcup T$,
$x\in\rH[S]$ and $y\in\rH[T]$. 
It is \emph{cocommutative} if $(z|_S, z/_S) = (z/_T, z|_T)$ for any $I=S\sqcup T$ and $z\in\rH[I]$; it is enough to check that $z/_S=z|_T$  for any $I=S\sqcup T$ and $z\in\rH[I]$.
\end{definition}

\begin{example}\label{eg:hopfL}
We now define a Hopf monoid on the species $\rL$  of linear orders of Example~\ref{eg:speciesL}. To this end, we define the operations of
\emph{concatenation} and \emph{restriction}. Let $I=S\sqcup T$.
If $\ell_1=s_1\ldots s_i$
is a linear order on $S$ and $\ell_2=t_1\ldots t_j$ is a linear order on $T$,
their concatenation is the following linear order on $I$:
\[
\ell_1\cdot \ell_2:=s_1\ldots s_i\, t_1\ldots t_j.
\]
Given a linear order $\ell$ on $I$, the restriction 
 $\ell\,|_S$ is the list consisting of the elements of $S$ written in the order
 in which they appear in $\ell$.

The product (merging) and coproduct (breaking) of the Hopf monoid $\rL$ are defined by
\begin{align*}
\rL[S]\times\rL[T] & \map{\mu_{S,T}} \rL[I] & \rL[I]  & \map{\Delta_{S,T}} \rL[S]\times\rL[T]\\
(\ell_1,\ell_2) & \map{\phantom{\mu_{S,T}}} \ell_1\cdot\ell_2
& \ell  & \map{\phantom{\Delta_{S,T}}} (\ell |_S,\ell |_T).
\end{align*}

Given linear orders $\ell_1$ on $S$ and $\ell_2$ on $T$, the compatibility axiom in Definition~\ref{d:hopfset} comes from the fact that the concatenation of $\ell_1|_A$ and $\ell_2|_C$ agrees with the restriction to $S'$ of $\ell_1\cdot \ell_2$.
The verification of the remaining axioms is similar.

By definition, $\ell/_S=\ell\,|_T$, so $\rL$ is cocommutative. 
\end{example}

Given a Hopf monoid $\rH$ we define the \emph{co-opposite} Hopf monoid $\rH^{cop}$ by preserving the product and reversing the coproduct:
if $\Delta_{S,T}(z) = (z|_S, z/_S)$ in $\rH$, then $\Delta_{S,T}(z) = (z/_T, z|_T)$ in $\rH^{cop}$. One easily verifies this is also a Hopf monoid.

The main characters of this paper are the Hopf monoids of generalized permutahedra discussed in Section \ref{s:GP}. We consider many other Hopf monoids throughout the paper; additional examples are given in~\cite[Chapters~11--13]{am}.

\subsection{{Vector species}}

All vector spaces and tensor products below are over a fixed field $\Kb$.

A \emph{vector species} $\wP$ consists of  the following data.
\begin{itemize}
\item For each finite set $I$, a vector space $\wP[I]$.
\item For each bijection $\sigma:I\to J$, a linear map $\wP[\sigma]:\wP[I]\to \wP[J]$. 
\end{itemize}
These are subject to the same axioms as in Definition~\ref{d:setsp}.
Again, these axioms imply that every such map $\wP[\sigma]$ is invertible. 
A morphism of vector species $f:\wP\to\wQ$ is a collection of linear maps
$f_I:\wP[I]\to\wQ[I]$ satisfying the naturality axiom of Definition~\ref{d:morsetsp}.

\subsection{{Hopf monoids in vector species}}\label{ss:Hopfmonoidinvectorspecies}

\begin{definition}
A \emph{connected Hopf monoid in vector species} is a vector species $\wH$
with $\wH[\emptyset]=\Kb$ that is
equipped with linear maps
\[
\wH[S]\otimes\wH[T] \map{\mu_{S,T}}\wH[I]
\qand
\wH[I]\map{\Delta_{S,T}} \wH[S]\otimes\wH[T]
\]
for each decomposition $I=S\sqcup T$, subject to the same axioms 
as in  Definition~\ref{d:hopfset}.
\end{definition}


We employ similar notations as for Hopf monoids in set species; namely,
\[
\mu_{S,T}(x\otimes y)=x\cdot y
\qand
\Delta_{S,T}(z)=\displaystyle\sum z|_S\otimes z/_S,
\]
the latter being a variant of \emph{Sweedler's notation} for Hopf algebras.
In general, $\sum z|_S\otimes z/_S$ stands for a
tensor in $\wH[S]\otimes\wH[T]$; individual elements $z|_S$ and $z/_S$
may not be defined. However, in most combinatorial situations that interest us, $\Delta_{S,T}(z) = z|_S \otimes z/_S$ is a pure tensor defined combinatorially by the breaking operation.

A \emph{morphism of Hopf monoids in vector species} is a morphism of vector species
which preserves products, coproducts, and the unit, as in Definition \ref{d:morhopf}.

\subsection{{Linearization}}

Consider the \emph{linearization functor}
\[
\Set \longrightarrow \Vect,
\]
which sends a set to the vector space with basis the given set.
Composing a set species $\rP$ with the linearization functor
gives a vector species, which we denote $\wP$. 
If $\rH$ is a Hopf monoid in set species, then its linearization $\wH$ is a Hopf monoid
in vector species.

Most, but not all, of the Hopf monoids considered in this paper are in set species.
The linearization functor allows us to regard them as Hopf monoids in vector species also.
%

\begin{remark}
The category of vector species carries a symmetric monoidal structure.
In any symmetric monoidal category one may consider the notion of \emph{Hopf monoid}. A Hopf monoid in vector species is a Hopf monoid in this categorical sense. For more details about this point of view, and a discussion of set
species versus vector species, see~\cite[Chapter 8]{am}.
\end{remark}

\subsection{{Higher products and coproducts}}\label{ss:high}

Let $\wH$ be a Hopf monoid in vector species. The following is a consequence of the associativity axiom. For any
decomposition $I=S_1\sqcup\cdots\sqcup S_k$ with $k\geq 2$, there are unique maps
\begin{equation}\label{e:iterated-mu}
\wH[S_1]\otimes\cdots\otimes\wH[S_k] \map{\mu_{S_1,\ldots,S_k}} \wH[I], \qquad \wH[I] \map{\Delta_{S_1,\ldots,S_k}} \wH[S_1]\otimes\cdots\otimes\wH[S_k]
\end{equation}
obtained by respectively iterating the product maps $\mu$ or the coproduct maps $\Delta$ in any meaningful way. 
As we mentioned when discussing the associativity axiom in Section \ref{ss:hopf-set}, these maps are well-defined; we refer to them as the \emph{higher products and coproducts} of $\wH$. 

For $k=1$, we define $\mu_I$ and $\Delta_I$ to be the identity map $\id: \wH[I] \to \wH[I]$.
For $k=0$, the only set with a decomposition into $0$ parts is the empty set, and in that case we let
$\mu_{(\,)}: \Kb \map{} \wH[\emptyset]$ and 
$\Delta_{(\,)}: \wH[\emptyset] \map{} \Kb$ 
 be the linear maps that send $1$ to $1$.

\medskip

When $\wH$ is the linearization of a Hopf monoid $\rH$ in set species,
we have higher (co)products 
\[
\mu_{S_1,\ldots,S_k}(x_1,\ldots,x_k)
= x_1\cdot \ldots \cdot x_k \in \rH[I] , \qquad \Delta_{S_1,\ldots,S_k}(z) = (z_1,\ldots,z_k)
\]
whenever $x_i\in\rH[S_i]$ for $i=1,\ldots,k$,
and $z\in\rH[I]$, respectively. We refer to $z_i\in\rH[S_i]$ as the \emph{$i$-th minor}
of $z$ corresponding to the decomposition $I=S_1\sqcup\cdots\sqcup S_k$; it is obtained from $z$ by combining restrictions and contractions in any
 meaningful way.

\subsection{{The antipode and the antipode problem}}\label{ss:antipode}

A \emph{composition} of a finite set $I \neq \emptyset$ is an ordered decomposition $I=S_1\sqcup\cdots\sqcup S_k$ in which each subset $S_i$ is nonempty;  we write
\[
(S_1,\ldots,S_k) \vDash I.
\]
In other words, a composition is a decomposition whose parts are non-empty.

\begin{definition}\label{d:antipode}
Let $\wH$ be a connected Hopf monoid in vector species.
The \emph{antipode} of $\wH$ is the collection of maps 
\[
\apode_I:\wH[I]\to \wH[I],
\]
one for each finite set $I$, given by $\apode_\emptyset=\id$ and
\begin{equation}\label{e:takeuchi}
\apode_I\ =\ \sum_{(S_1,\ldots,S_k) \vDash I \atop k\geq 1} (-1)^k\, \mu_{S_1,\ldots,S_k}\circ\Delta_{S_1,\ldots,S_k} \qquad \textrm{ for } I \neq \emptyset.
\end{equation}
\end{definition}

The right hand side of~\eqref{e:takeuchi} involves the higher (co)products of~\eqref{e:iterated-mu}. 
Since a composition of $I$ can have at most $\abs{I}$ parts, the sum is finite.
We refer to~\eqref{e:takeuchi} as \emph{Takeuchi's formula}. 
For alternate formulas and axioms
defining the antipode of a Hopf monoid, see~\cite[Section 8.4]{am}.

As explained, for example, in \cite{am, montgomery93:_hopf}, Hopf monoids may be regarded as a generalization of groups. In this context, the antipode of a Hopf monoid generalizes the inverse function $g \mapsto g^{-1}$ of a group. For this reason, the antipode is a central part of the structure of a Hopf monoid, and the following is a fundamental problem. 

\begin{problem}\label{prob:antipode}\cite[Section 8.4.2]{am}
Find an explicit, cancellation-free formula for the antipode of a given Hopf monoid.
\end{problem}

If $\wH$ is the linearization of a Hopf monoid in set species $\rH$, the sum in~\eqref{e:takeuchi} takes place in the vector space $\wH[I]$ with basis $\rH[I]$.
The antipode problem asks for an understanding of the structure constants
of $\apode_I$ on this basis. 

\begin{remark}
The number of terms in Takeuchi's formula~\eqref{e:takeuchi} is the \emph{ordered Bell number} $\omega(n) \approx n!/2(\log 2)^{n+1}$; the first few terms in this sequence are $1, 1, 3, 13, 75, 541, 4683, 47293,$ $545835$. \cite{good75} Their rapid growth makes this equation impractical, even for moderate values of $n$. To solve the Antipode Problem \ref{prob:antipode}, one needs further insight into the Hopf monoid in question.
\end{remark}

\subsection{{Properties of the antipode}}\label{ss:antipodeproperties}

The following properties of the antipode follow from general results for Hopf monoids in monoidal categories; see \cite[Prop.~1.22.(iii), Cor.~1.24, Prop.~1.16]{am} for proofs. 
\begin{proposition}\label{p:antipode}\emph{(The antipode reverses products and coproducts)}
Let $\wH$ be a Hopf monoid and $I=S\sqcup T$ a decomposition. Then
\begin{gather}
\label{e:apode-mu}
\apode_I(x\cdot y) = \apode_{T}(y)\cdot\apode_{S}(x)
\text{ whenever $x\in\wH[S]$ and $y\in\wH[T]$,}\\
\label{e:apode-delta}
\Delta_{S,T}\bigl(\apode_I(z)\bigr) = 
\sum \apode_S(z/_T)\otimes \apode_T(z|_T)
\text{ whenever $z\in\wH[I]$.}
\end{gather}
\end{proposition}

Iterating Proposition \ref{p:antipode}, we obtain the analogous result for higher products and coproducts, which we now state. 

Consider a decomposition $I = S_1 \sqcup \cdots \sqcup S_k$, and write $F=(S_1, \ldots, S_k)$. 
Let 
the \emph{switch map} $\sw_F: \wH[S_1] \otimes \cdots \otimes \wH[S_k] \rightarrow \wH[S_k] \otimes \cdots \otimes \wH[S_1]$ reverse the entries; 
that is, $\sw_F(x_1 \otimes \cdots \otimes  x_k) = x_k \otimes \cdots \otimes x_1$ whenever $x_i \in \wH[S_i]$ for $1 \leq i \leq k$. Let 
\[
\mu_F = \mu_{S_1,\ldots, S_k}, \qquad
\Delta_F = \Delta_{S_1,\ldots, S_k}, \qquad
\apode_F = \apode_{S_1} \otimes \cdots \otimes \apode_{S_k},
\]
where $\apode_{S_i}$ is the antipode map on $S_i$ for $1 \leq i \leq k$.

\begin{proposition}\emph{(The antipode reverses higher products and coproducts)} \label{p:antipode1.5}
Let $\wH$ be a Hopf monoid and let $F=(S_1, \ldots, S_k)$ be a decomposition, so $I=S_1 \sqcup  \cdots \sqcup S_k$. 
Let the reverse decomposition be $-F = (S_k, \ldots, S_1)$.
Then
\begin{equation}\label{e:apodereverses}
\apode_I \mu_F = \mu_{-F} \apode_{-F} \sw_F , \qquad 
\Delta_F \apode_I = \apode_F \sw_{-F}  \Delta_{-F}. 
\end{equation}
\end{proposition}

\begin{proof}
For $k=2$, this is a restatement of Proposition \ref{p:antipode}. For $k \geq 2$, this is the result of iterating Proposition \ref{p:antipode}.
\end{proof}

\begin{proposition}\label{p:antipode2}
Let $\wH$ be a Hopf monoid, $I$ a finite set, and $x\in\wH[I]$. 
If $\wH$ is either commutative or cocommutative, then 
\begin{equation}\label{e:apode-inv}
\apode_I^2(x)=x.
\end{equation}
If $\wH$ is commutative, then $\wH$ and its co-opposite $\wH^{cop}$ share the same antipode.

If $f:\wH\to\wK$ is a morphism of Hopf monoids, then
\begin{equation}\label{e:apode-mor}
f_I\bigl(\apode_I(x)\bigr) = \apode_I\bigl(f_I(x)\bigr).
\end{equation}
\end{proposition}
\begin{proof} See~\cite[Prop.~1.16, Prop.~1.22, Cor.~1.24]{am}.
\end{proof}

\begin{example}\label{eg:antipodeL}
Consider the Hopf monoid $\wL$ of linear orders in vector species. Problem \ref{prob:antipode} asks for an explicit expression for $\apode_I(\ell)$, where $\ell$ is a linear order on a finite set $I$. Takeuchi's formula~\eqref{e:takeuchi} yields a very large alternating sum of linear orders, but
many cancellations take place. It turns out that only one term survives:
\[
\apode_I(i_1 i_2\ldots i_n) \ =\  (-1)^{n}\, i_n\ldots i_2 i_1.
\]
In other words, up to a sign, the antipode simply reverses the linear order.

Here is a simple proof. When $I$ is a singleton, this follows
readily from~\eqref{e:takeuchi}. When $\abs{I}\geq 2$, 
Proposition \ref{p:antipode1.5} tells us that the antipode reverses products, which implies that
$\apode_I(i_1 i_2\ldots i_n) = \apode_{\{i_n\}}(i_n)\cdot \cdots \cdot \apode_{\{i_1\}}(i_1) \ =\  (-i_n) \cdots (-i_1) =  (-1)^{n}\, i_n\ldots i_2 i_1,$, as desired.
\end{example}

For more complicated Hopf monoids, obtaining such an explicit description for the antipode is often difficult. It requires understanding the cancellations that occur in a large alternating sum indexed by combinatorial objects; the antipode problem is therefore of a clear combinatorial nature. 

Several instances of the antipode problem
are solved in~\cite[Chapters 11--12]{am}. In Section~\ref{s:antipode} of
this paper we offer a unified framework that solves this problem for many other Hopf monoids of interest, as outlined in Table \ref{table:antipodes}. We describe a few consequences of these formulas in Sections~\ref{s:inversion},~\ref{s:pol-inv}, ~\ref{s:reciprocity}, and \ref{s:Frevisited}.

\subsection{From Hopf monoids to Hopf algebras.}\label{ss:Fock}

All our results on Hopf monoids have counterparts at the level of Hopf algebras, thanks to the \emph{Fock functor}\footnote{In fact this is only one of four Fock functors; see \cite[Sections 15.2, 17]{am}.} $\Kcb$ that takes a Hopf monoid on vector species $\wH$ to a graded Hopf algebra $H$ while preserving most of the structures that interest us. We will not use Hopf algebras in this paper, but we include a brief discussion for the benefit of readers who are used to working with them; we  assume familiarity with Hopf algebras in this subsection.

Given a connected Hopf monoid on set species $\rH$, let $\rH[n] = \rH[\{1, \ldots, n\}]$ for $n \in \Nb$ and consider the coinvariant vector space
\[
H = \bigoplus_{n \geq 0} \wH[n]_{S_n} := \bigoplus_{n \geq 0} \textrm{ span} \{\textrm{isomorphism classes of elements of } \rH[I] \textrm{ for } |I|=n \}
\]
where objects $h_1 \in \rH[I_1]$ and $h_2 \in \rH[I_2]$ are said to be \emph{isomorphic} if there exists a bijection $\sigma: I_1 \rightarrow I_2$ such that $\sigma(h_1)=h_2$. 
We denote the isomorphism class of $h \in \rH[I]$ by $[h]$.

The product and coproduct of $\rH$ endow the graded vector space $H$ with the structure of a graded Hopf algebra. It has the natural unit and counit. The product and coproduct are
\[
[h_1] \cdot [h_2] = [h_1 \cdot h_2^{+k_1}], \qquad 
\Delta([h]) = \sum_{[n] = S \sqcup T} [h|_S] \otimes [h/_S]
\]
for $h_1 \in \rH[k_1]$, $h_2 \in \rH[k_2]$, and $h \in \rH[n]$, where $h_2^{+k_1}=\sigma^{+k_1}(h_2) \in \rH[\{k_1+1, \ldots, k_1+k_2\}]$ is the image of $h_2$ under the order-preserving bijection $\sigma^{+k_1}: [k_2] \to \{k_1+1, \ldots, k_1+k_2\}$.

\begin{theorem}\cite[Proposition 3.50, Theorem 15.12]{am}
If $\rH$ is a Hopf monoid in species then $\Kcb(\rH)$ is a graded Hopf algebra. Furthermore, the Fock functor $\Kcb$ maps the antipode of $\rH$ to the antipode of $\Kcb(\rH)$.
\end{theorem}

%
%

\section{$\rG, \rM, \rP, \rPi, \rF$: {Graphs, matroids, posets, set partitions, paths}}\label{s:G,M,P}

In this section, we illustrate the previous definitions with five examples of Hopf monoids built from graphs, matroids, posets, set partitions, and paths. Some of these and many others appear in~\cite[Chapter~13]{am}.
Important ideas leading to
these constructions are due to Joni and Rota~\cite{joni82:_coalg}, Schmitt~\cite{schmitt93:_hopf}, and many others; additional references are given below.

\subsection{$\rG$: {Graphs}}\label{ss:graphs}

A \emph{graph with vertex set $I$} consists of a multiset of edges.
Each edge is a subset of $I$ of cardinality $1$ or $2$; in the former case
we call it a \emph{half-edge}. 

Let $\rG[I]$ denote the set of all graphs with vertex set $I$. One may use
a bijection $\sigma:I\to J$ to relabel the vertices of a graph $g\in\rG[I]$
and turn it into a graph $\rG[\sigma](g)\in\rG[J]$. Thus, $\rG$ is a species, which is a Hopf monoid with the following structure. 

Let $I=S\sqcup T$ be a decomposition. 

\noindent $\bullet$ 
The product of two graphs $g_1\in\rG[S]$ and $g_2\in\rG[T]$ is
simply the disjoint union of the two. Thus, an edge of $g_1\cdot g_2$ is an edge
of either $g_1$ or $g_2$. 

\noindent $\bullet$ 
To define the coproduct of a graph  $g\in\rG[I]$, we let 
the restriction $g|_S\in\rG[S]$ be the \emph{induced subgraph} on $S$, which consists of the edges and half-edges whose ends are in $S$. 
The contraction $g/_S\in\rG[T]$ is the graph on $T$ given by
 all edges incident to $T$, where an edge $\{t,s\}$ in $g$ joining $t \in T$ and $s \in S$ becomes a half-edge $\{t\}$ in $g/_S$.

\noindent The Hopf monoid axioms are easily verified.

%


An example follows. 
Let
$I=\{a,b,c,x,y\}, S=\{x,y\},$ and $T=\{a,b,c\}$.

 \begin{figure}[h]
\centering
If $g = $
\includegraphics[scale=.5]{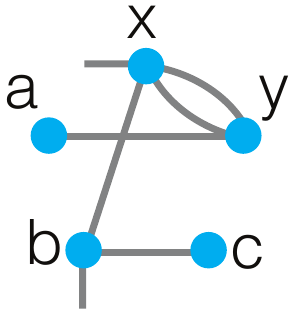} 
\qquad then \qquad $g|_S = $ 
\includegraphics[scale=.5]{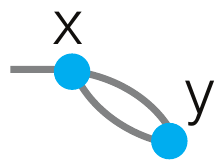}
\qquad and \qquad $g/_S = $
\includegraphics[scale=.5]{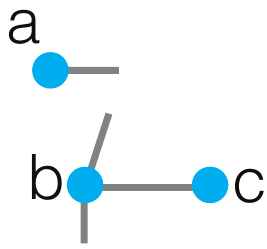}
\,\, .
\end{figure}


 \begin{example}\label{eg:antipode-graph} 
 Consider  the antipode of the Hopf monoid of graphs $ \apode_I:\wG[I]\to\wG[I]$, 
 where $\wG$ is the linearization of $\rG$. For a graph on $3$ vertices,
Takeuchi's formula~\eqref{e:takeuchi} returns an alternating sum of $13$ graphs
on the same vertex set, corresponding to the $13$ compositions of a $3$-element set. An explicit but cumbersome calculation shows the following:

 \begin{figure}[h]
\centering
\includegraphics[scale=.5]{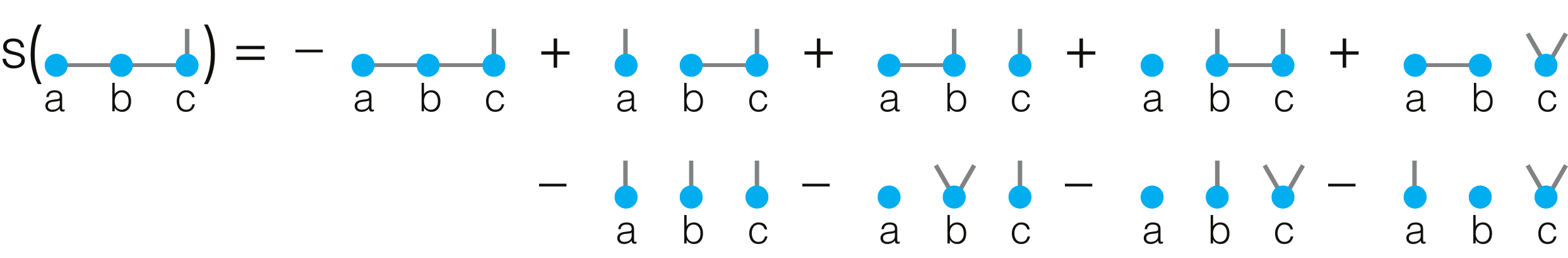} 
\end{figure}

Cancellations took place which resulted in a cancellation-free and grouping-free sum of only $9$ graphs. 
The antipode problem for the Hopf monoid $\wG$ (Section~\ref{ss:antipode}) asks for to understand this cancellation and interpret the terms of the resulting formula. This problem is solved and its consequences are explored in Section~\ref{s:G}. It was also solved by Humpert and Martin \cite{humpert12} for the corresponding Hopf algebra, and we prove their related conjectures in Theorem \ref{conj:hm}.
\end{example}

 \subsection{$\rM$: {Matroids}}\label{ss:matroid}

A \emph{matroid with ground set $I$} is a non-empty collection of subsets of $I$, called \emph{bases}, which satisfy the \emph{basis exchange axiom}: if $A$ and $B$ are bases and $a \in A - B$, there exists $b \in B-A$ such that $A - a \cup b$ is a basis.

Subsets of bases are called \emph{independent sets}. Matroids abstract the notion of independence, and arise naturally in many fields of mathematics. Three key examples are the following.
\begin{itemize}
\item 
\emph{(Linear matroids)} If $I$ is a set of vectors spanning a vector space $V$, the collection of subsets of $I$ which are bases of $V$ is a matroid. 
\item 
\emph{(Graphical matroids)} 
If $I$ is the set of edges of a connected graph $g$, the collection of (edge sets of) spanning trees of $g$ is a matroid.
\item
\emph{(Algebraic matroids)}
If $I$ is a set of elements that generate a field extension $\mathbb{K}$ over $\mathbb{F}$, the collection of subsets of $I$ which are transcendence bases of $\mathbb{K}$ over $\mathbb{F}$  is a matroid. 
\end{itemize} 

Matroids have important notions of restriction (or deletion) and contraction, which simultaneously generalize natural geometric and graph-theoretic operations. We now provide the basic definitions; for details  on these and other notions related to matroids, we refer the reader to~\cite{oxley92:_matroid,welsh76:_matroid}. 

Consider a matroid $m$ on $I$ and an element $e \in I$. Then the \emph{deletion} $m \backslash_e$ and \emph{contraction} $m/_e$ are matroids on $I - e$ defined by:
\begin{eqnarray*}
m \backslash_e &=& \{B \in m \, : \, e \notin B\} \\
m/_e &=& \{B-e \, : \, B \in m, e \in B\} 
\end{eqnarray*}
These operations commute when they are well defined; that is, we have  $m/_e/_f = m/_f/_e, \newline m\backslash_e \backslash_f = m \backslash_f \backslash_e$, and $m/_e \backslash_f = m \backslash_f /_e$ for $e \neq f$ in $I$.

Therefore, if $J=\{j_1, \ldots, j_k\} \subseteq I$, we may define the \emph{deletion} and \emph{contraction} of $J$ in $m$ are defined iteratively: $m \backslash_J = m \backslash_{j_1}\backslash_{j_2} \cdots \backslash_{j_k}$ and $m /_J = m /_{j_1}/_{j_2} \cdots /_{j_k}$, respectively. Finally, we define the \emph{restriction} of $m$ to $S$ to be the deletion of $E-S$ in $m$; that is, $m|_S = m\backslash_{I-S}$

 Let $\rM[I]$ be the set of matroids with ground set $I$. Again, $\rM$ is a species, which we now turn into
 a Hopf monoid.

 Let $I=S\sqcup T$ be a decomposition.

\noindent $\bullet$ For any matroids $m_1\in\rM[S]$ and $m_2\in\rM[T]$ we define their product $m_1\cdot m_2\in\rM[I]$ to be their direct sum $m_1 \oplus m_2 = \{B_1 \sqcup B_2 \, : \, B_1 \in m_1, B_2 \in m_2\}$, which is indeed a matroid. \cite{oxley92:_matroid}

\noindent $\bullet$  
To define the coproduct, we employ the notions of restriction and contraction of matroids, and define $\Delta_{S,T}(m) = (m|_S, m/_S)$.

\noindent 
The Hopf monoid axioms
 boil down to familiar properties relating direct sums, restriction, and contraction of matroids. \cite{oxley92:_matroid}

The (linearization of the) Hopf monoid $\rM$ is discussed in~\cite[Section~13.2]{am}. 
 The crucial idea of assembling these matroid operations into an algebraic structure goes back to Joni and Rota~\cite[Section~XVII]{joni82:_coalg}
and Schmitt~\cite[Section~15]{schmitt94:_incid_hopf}. In fact the terms restriction, contraction, and minor, which we employ for an arbitrary Hopf monoid in set species, originate in this example.

 \begin{example}\label{eg:antipode-matroid} 
Following Schmitt \cite{schmitt94:_incid_hopf}, we consider the antipode of the Hopf algebra of matroids ${M}$, where isomorphic matroids are identified.\footnote{The identification of isomorphic matroids is not essential; we only do it to follow the convention of \cite{schmitt94:_incid_hopf}.}
 Let $m$ be the matroid of rank 2 on $\{a,b,c,d\}$  whose only non-basis is $\{c,d\}$. Takeuchi's formula~\eqref{e:takeuchi} expresses the antipode $\apode(m)$ as an alternating sum of $73$ matroids, but after extensive cancellation, one obtains:
 
 \begin{figure}[h]
\centering
\includegraphics[scale=.5]{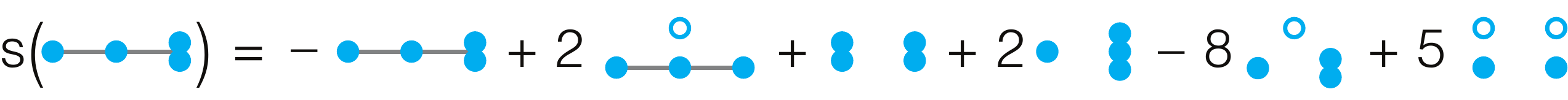} 
\end{figure}
\noindent where we are representing matroids by affine diagrams \cite{stanley11:_ec1};   points represent elements and the following represent dependent sets: three points on a line, two points above each other, and one hollow point.
The antipode problem \ref{prob:antipode} for $M$, or more strongly for the Hopf monoid $\wM$,  asks for an understanding of this cancellation. This problem is solved in Section~\ref{s:M}.
\end{example}

 \subsection{$\rP$: {Posets}}\label{ss:poset}

A \emph{poset} $p$ on a finite set $I$ is a relation $p \subseteq I \times I $, denoted $\leq$, which is 

\noindent $\bullet$
\emph{(Reflexive)}: $x \leq x$ for all $x \in p$,

\noindent $\bullet$
\emph{(Antisymmetric)}: $x \leq y$ and $y \leq x$ imply $x = y$  for all $x,y \in p$,

\noindent $\bullet$
\emph{(Transitive)}: $x \leq y$ and $y \leq z$ imply $x \leq z$  for all $x,y,z\in p$.
Let $p$ be a poset on $I=S\sqcup T$.
The restriction of $p$ to $S$ is the induced poset on $S$:
\[
p|_S := p\cap(S\times S).
\]
We say $S$ is a \emph{lower set} or \emph{order ideal} of $p$ if no element of $T$ is less than an element of $S$.


Let $\rP[I]$ denote the set of all posets on $I$
and $\wP[I]$ its linearization; that is, the vector space with basis $\rP[I]$. 
We turn $\wP$ into a Hopf monoid in
vector species as follows.

Let $I=S\sqcup T$ be a decomposition.
 
\noindent $\bullet$  
The product $\mu_{S,T}: \wP[S] \otimes \wP[T]  \to  \wP[I]$ is given by
\[
\mu_{S,T}(p_1\otimes p_2) =  p_1 \sqcup p_2.
\]
where $p_1 \sqcup p_2$ is the disjoint union of the sets $p_1\subseteq S\times S$ and
$p_2\subseteq T\times T$; this poset has no relations
between the elements of $S$ and $T$.

\noindent $\bullet$  
The coproduct $\Delta_{S,T}:\wP[I] \to  \wP[S] \otimes \wP[T]$ is given by 
\[
\Delta_{S,T}(p) = \begin{cases}
p|_S\otimes p|_T & \text{if $S$ is a lower set of $p$,}\\
0 & \text{otherwise.}
\end{cases}
\]
\noindent One easily verifies the axioms. 

The Hopf monoid $\wP$ is commutative and cocommutative.
This Hopf monoid is discussed in~\cite[Section~13.1]{am}.
Important work of Malvenuto~\cite{malvenuto94:_produit}
and of Schmitt~\cite[Section~16]{schmitt94:_incid_hopf}
are at the root of this construction.
Gessel~\cite{gessel84:_multip_p_schur} is also relevant.

\begin{warning}
The vector species $\wP$ is the linearization of the set species $\rP$.
Note, however, that  $\Delta_{S,T}:\wP[I]\to\wP[S]\otimes\wP[T]$ does not 
always send the basis $\rP[I]$ to $\rP[S]\times\rP[T]$. The Hopf monoid
structure of $\wP$ is not linearized.
\end{warning}

 \begin{example}\label{eg:antipode-poset} 
Following Schmitt \cite{schmitt94:_incid_hopf} again, we consider the antipode of the Hopf algebra of posets ${P}$, where isomorphic posets are identified.\footnote{As with matroids, this identification of isomorphic posets is not essential.} For the poset $p$ of four elements shown,  Takeuchi's alternating sum of $73$ posets simplifies to:

 \begin{figure}[h]
\centering
\includegraphics[scale=.5]{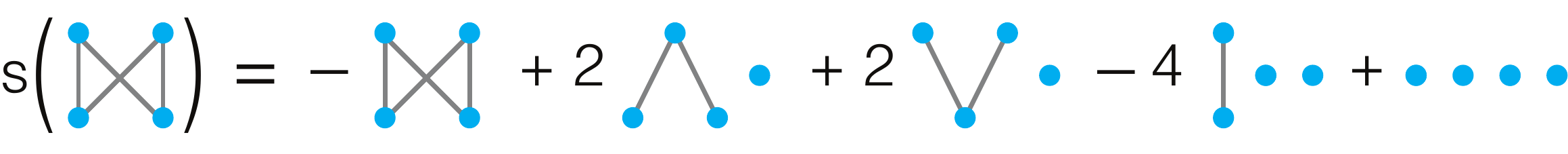} \qquad
\label{f:poset} 
\end{figure}
The antipode problem \ref{prob:antipode} for posets is solved in Section~\ref{s:P}.
\end{example}

\subsection{$\rPi$: {Set partitions}}\label{ss:Pi}

A \emph{set partition} $\pi$ of a finite set $I$ is an unordered collection $\pi = \{\pi_1, \ldots, \pi_k\}$ of pairwise disjoint non-empty subsets of $I$ that partition $I$; that is, 
\[
\pi_1 \cup \cdots \cup \pi_k = I, \qquad \pi_i \cap \pi_j = \emptyset \, \textrm{ for } \, i \neq j.
\]
The sets $\pi_1, \ldots, \pi_k$ are called the \emph{parts} of $\pi$. To simplify the notation, when we write down a set partition, we remove the brackets around each one of its parts. For instance, we will write the partition $\{\{a,b\}, \{c,d,e\}\}$ as $\{ab, cde\}$.

Let $\rPi[I]$ denote the set of all set partitions on $I$. One may use a bijection $\sigma: I \to J$ to relabel the parts of a partition of $I$, thus turning it into a partition of $J$. Therefore $\rPi$ is a set species, which becomes a Hopf monoid with product and coproduct operations given by the following rules for merging and breaking set partitions.

Let $I = S \sqcup T$ be a decomposition. 

\noindent $\bullet$
The product of two set partitions $\pi \in\rPi[S]$ and $\rho \in\rPi[T]$ is their disjoint union $\pi \cup \rho \in \rPi[I]$. 

\noindent $\bullet$
The coproduct of a set partition $\pi \in \rPi[I]$ is $(\pi|_S, \pi|_T) \in \rPi[S] \times \rPi[T]$ where $\pi|_J$ denotes the restriction of $\pi$ to a subset $J \subseteq I$, namely,  $\pi|_J = \{\pi^i \cap J \, : \, \pi^i \in \pi\}$.

\noindent 
One easily checks the axioms of a Hopf monoid. 
Since $\pi/_S = \pi|_T$ by definition, $\rPi$ is cocommutative. 

 For example, if $I = \{a,b,c,d,e\},$ $S = \{a,b,d\}$, $T=\{c,e\}$ and
\[
\pi=\{ab, cde\} \quad \textrm{ then } \quad  \pi|_S = \{ab, d\} \quad \textrm{ and } \quad \pi/_S =  \{ce\},
\]
so $\Delta_{S,T}(\{ab, cde\}) = (\{ab, d\}, \{ce\})$, whereas $\mu_{S,T}(\{ab, d\}, \{ce\}) = \{ab, ce, d\}.$

\begin{example}\label{ex:antipodePi}
For the antipode $\apode(\{ab, cde\})$, Takeuchi's formula~\eqref{e:takeuchi} returns an alternating sum of $530$ set partitions, which simplifies to:
 \begin{figure}[h]
\centering
\includegraphics[scale=.5]{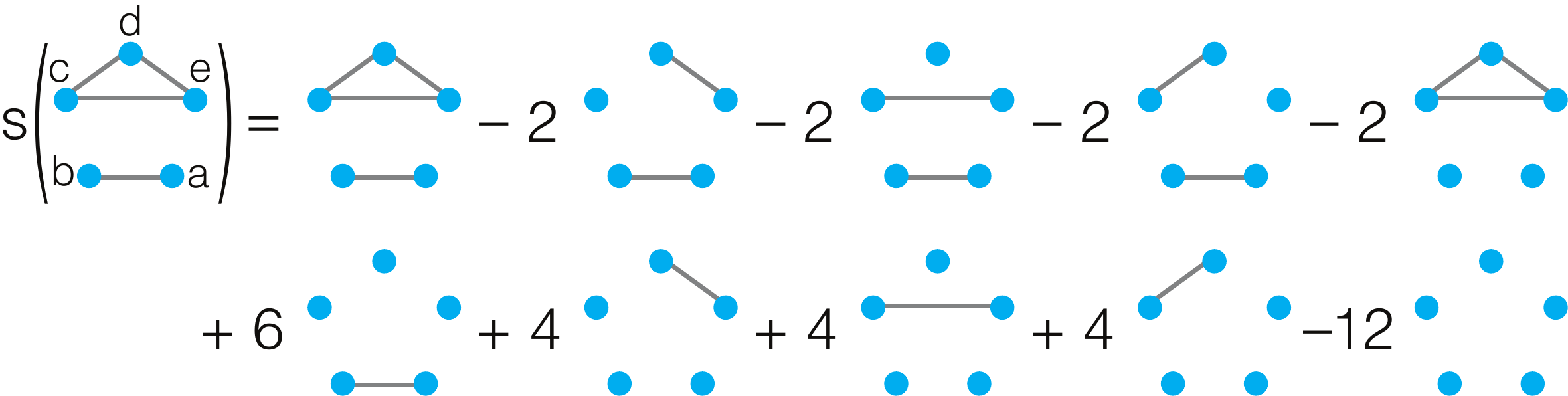} \qquad
\end{figure}

\noindent
where we represent a partition of $I$ by a graph on $I$ whose edges are the pairs of elements in the same block.

In Sections \ref{s:inversion} and \ref{s:Pirevisited} we explain this cancellation, and see that $\rPi$ is closely related to the permutahedron, symmetric functions,  and the problem of computing the multiplicative inverse of a formal power series.
\end{example}

\subsection{$\rF$: {Paths and Fa\`a di Bruno}}\label{ss:F}

We now consider \emph{paths} whose vertex set is a given finite set $I$. We will write a path $s$ as a word, listing the vertices in the order they appear in the path. 
A \emph{set of paths} $\alpha$ on $I$ is a decomposition $I = I_1 \sqcup \cdots \sqcup I_k$ together with a choice of a path $s_i$ on each $I_i$ with $1 \leq i \leq k$. We denote this set of paths by ${s_1 | \cdots | s_k}$. We must keep in mind that a word and its reverse represent the same path, and the order of the paths is irrelevant. For instance, ${ac|bde} = {ca|bde} = {bde|ac} = {bde|ca}$. Figure \ref{f:bundle} shows an example. 

\begin{figure}[h]
\centering
\includegraphics[scale=.5]{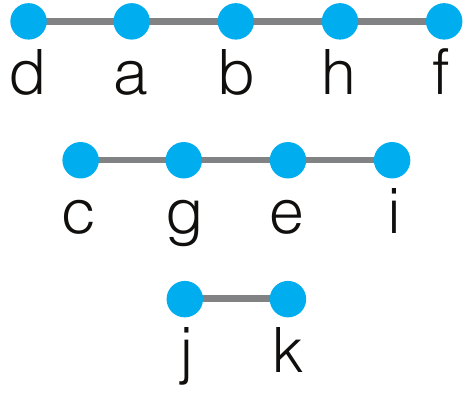}
\caption{The set of paths ${dabhf|cgei|jk}$.\label{f:bundle}}
\end{figure}

Let $\rF[I]$ denote the collection of sets of paths on $I$. We turn the species $\rF$ into a Hopf monoid with the following additional structure. 

\newpage

Let $I = S \sqcup T$ be a decomposition. 

\noindent $\bullet$ The product of two sets of paths $\alpha_1\in\rF[S]$ and $\alpha_2\in\rF[T]$ is simply their disjoint union, which we denote $\alpha_1 \sqcup \alpha_2 \in \rF[I]$. 

\noindent $\bullet$ The coproduct of a set of paths $\alpha = {s_1 | \cdots | s_k} \in \rF[I]$ is $(\alpha|_S, \alpha/_S) \in \rF[S] \times \rF[T]$, defined as follows. 
In the restriction $\alpha|_S$, each path $s_i$ that intersects $S$ gives rise to a single path $s_i \cap S$ on the set $I_i \cap S$, whose order is inherited from $s_i$. 
In the contraction $\alpha/_S$, each path $s_i$ that intersects $T$ gives rise to paths which partition $I_i \cap T$, corresponding to the connected components of $s_i$ when restricted to $T$. 

\noindent The Hopf monoid axioms are easily verified. We call this the \emph{Fa\`a di Bruno Hopf monoid}.


For example, if $I = \{a,b,c,d,e,f,g,h,i,j,k\}$, $S=\{b,e,f\}$, $T =\{a,c,d,g,h,i,j,k\}$ and
\[
\alpha = {dabhf | gcei | jk}, \quad  \textrm{ then } \quad 
\alpha|_S = {bf | e} \quad \textrm{ and } \quad \alpha/_S = {da | h | gc | i | jk}
\]
so $\Delta_{S,T}({dabhf | gcei | jk}) = ({bf | e}\, , \, {da | h | gc | i | jk})$ and 
$\mu_{S,T}({bf | e}, {da | h | gc | i | jk}) = {bf | e | da | h | gc | i | jk}$.

\begin{example}
For the antipode $\apode({abcd})$, Takeuchi's formula~\eqref{e:takeuchi} returns an alternating sum of $73$ sets of paths which, after cancellation, gives
 \begin{figure}[h]
\centering
\includegraphics[width=\textwidth]{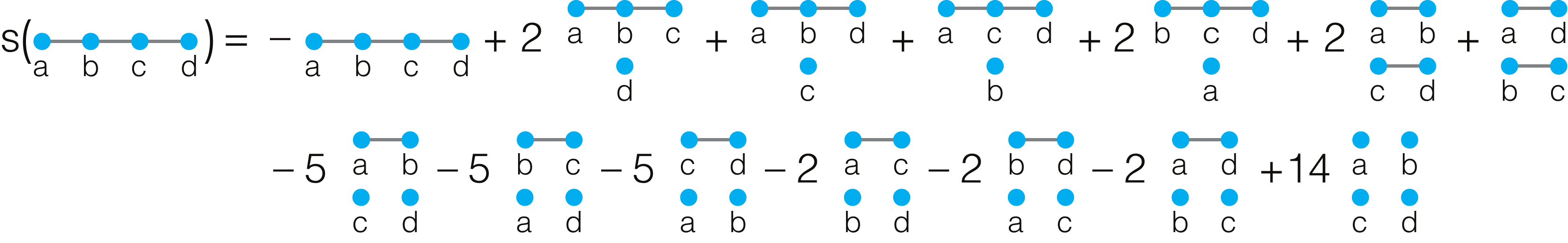} \qquad
\end{figure}

In Section \ref{s:Frevisited} we solve the antipode problem \ref{prob:antipode} for $\rF$. In particular, we explain why every coefficient in this formula is a Catalan number, and the number of terms is also a Catalan number. 
We will also see that $\rF$ is closely related to the associahedron, the Fa\`a di Bruno Hopf algebra, parenthesizations, and the problem of computing the compositional inverse of a formal power series.
\end{example}

\section{{Preliminaries 2: Generalized permutahedra}}\label{s:prelimGP}

The permutahedron is a natural polytopal model of the set of permutations of a finite set. We are interested in its deformations, known as \emph{generalized permutahedra} or \emph{polymatroids}. This family of polytopes is special enough to welcome combinatorial analysis, and general enough to model many  combinatorial families of interest. It is also precisely the family of polytopes which are amenable to the algebraic techniques of this paper, as Section \ref{s:universality} will show.

We now recall some basic facts about permutahedra and their deformations. These results and other background on polytopes can be found in~\cite{edmonds70, fujishige05, postnikov09,schrijver03, ziegler}.

 \subsection{{Permutahedra}}\label{ss:p}

Given a finite set $I$, let  $\Rb I$ be the real vector space with distinguished basis $I$. For each element $i \in I$, we will let $e_i$ denote $i$ when we wish to distinguish the element $i \in I$ from the vector $e_i \in \Rb I$. In $\Rb I$ we identify the vector $\sum_{i\in I} a_ie_i$ with the $I$-tuple $(a_i)_{i \in I}$ for $a_i \in \Rb$, so 
\[
\Rb I = \{(a_i)_{i \in I}\, : \, a_i \in \Rb\} = \left\{\sum_{i \in I} a_i e_i \, : \, a_i \in \Rb\right\}.
\]

Throughout this section, let $n := |I|$. The \emph{standard permutahedron} $\Rb I$ is the convex hull of the $n!$ permutations of $[n]:= \{1, \ldots, n\}$ in $\Rb I$; that is, 
\[
\pi_I:= \text{conv} \left\{\, (a_i)_{i \in I} \, : \, \{a_i\}_{i \in I} = [n] \, \right\} \subseteq\Rb I
\]
where $n = \abs{I}$. We let $\pi_n := \pi_{[n]}$ denote the standard permutahedron in $\Rb[n]$.

For example, $\pi_{\{a,b,c\}}$ is a regular hexagon lying
on the plane 
$
x_a+x_b+x_c=6
$, and $\pi_{\{a,b,c,d\}}$ is a truncated octahedron on the hyperplane
$
x_a+x_b+x_c+x_d = 10
$ 
in $\Rb{\{a,b,c,d\}},$ as shown in Figure \ref{f:permutahedra}. It is a useful exercise to work out how the 24 vertices of $\pi_{\{a,b,c,d\}}$ match the permutations of $[4]$.

\begin{figure}[h]
\centering
\includegraphics[scale=.5] {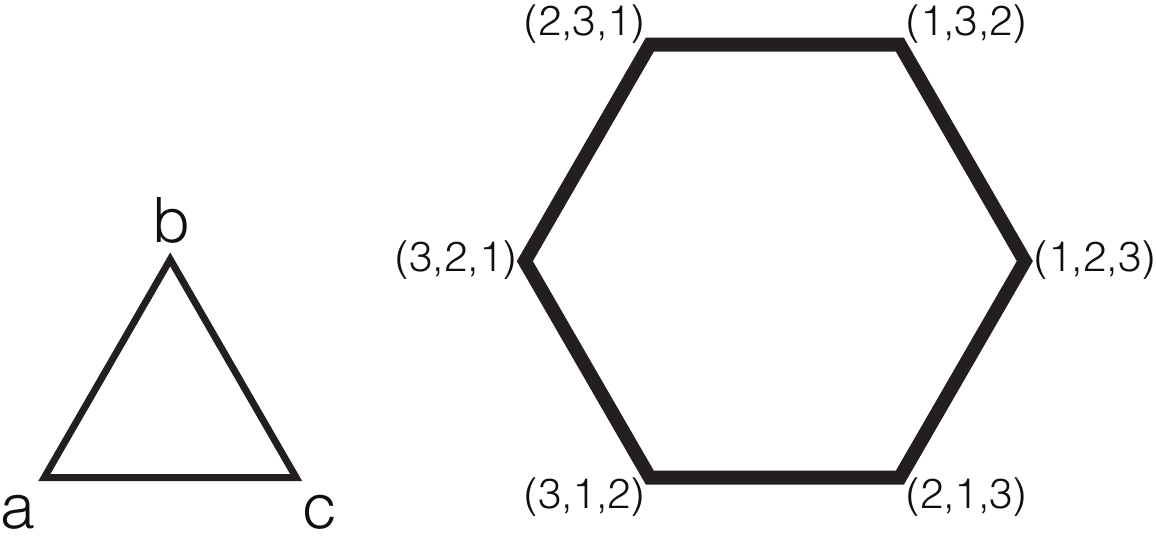} \qquad \qquad  \qquad
 \includegraphics[scale=.4]{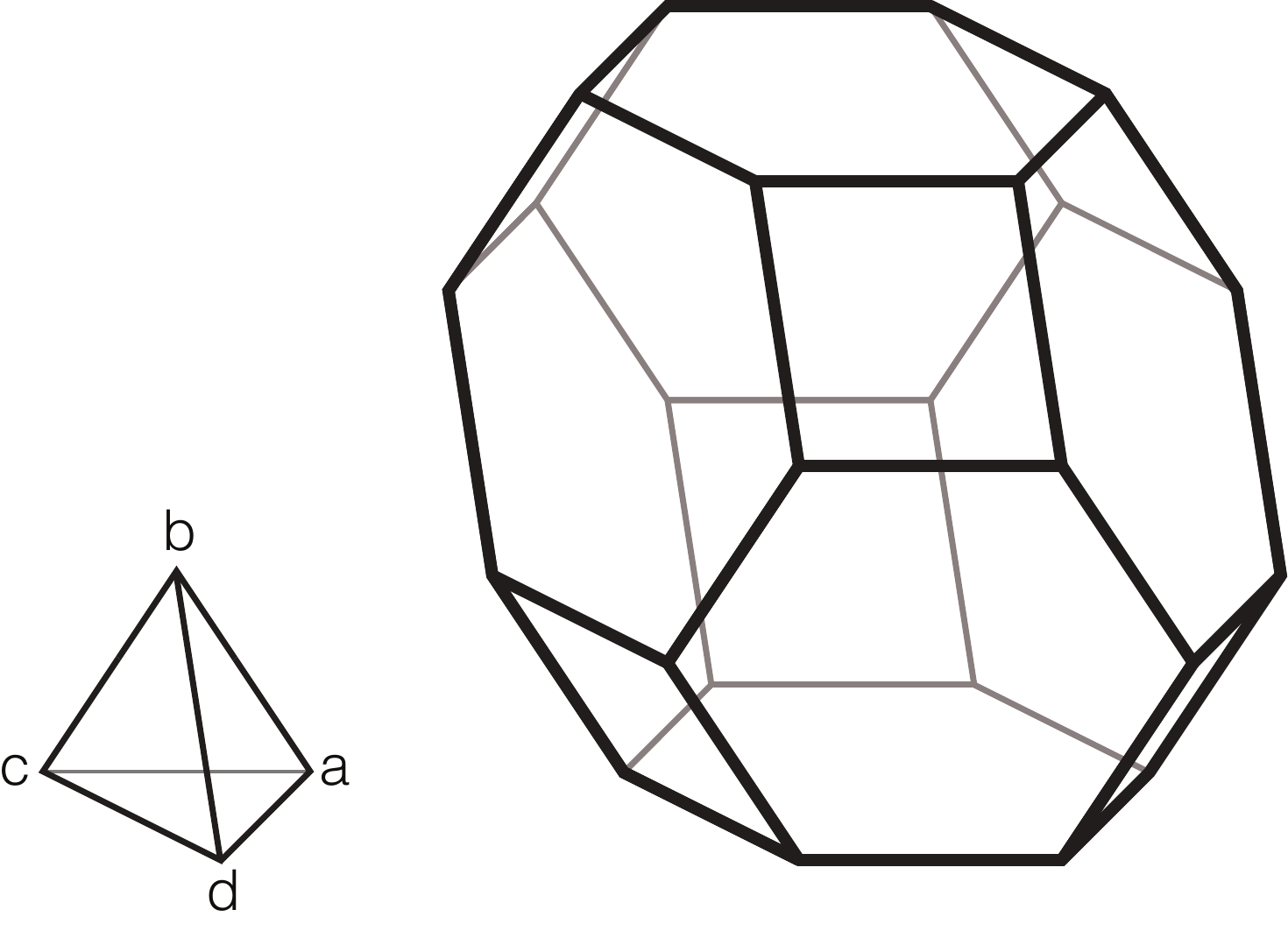} 
\caption{The standard simplex and the standard permutahedron in dimensions 3 and 4.} \label{f:permutahedra}
 \end{figure}

More generally, the permutahedron $\pi_I$ is a convex polytope of dimension $n-1$.
It can be described as the set of solutions
$(x_i)_{i\in I}\in \Rb I$ to the following system of (in)equalities:
\begin{align*}
 \sum_{i\in I}x_i &= \abs{I} + (\abs{I}-1) + \cdots + 2 + 1, \\
  \sum_{i\in A}x_i & \leq \abs{I} + (\abs{I}-1) + \cdots + (\abs{I}-\abs{A}+1)\,\,\text{ for all } \emptyset \subsetneq A \subsetneq I.
\end{align*}

The permutahedron has a very nice facial structure. We write $\wq \leq \wp$ when $\wq$ is a non-empty face of $\wp$.

\begin{itemize}
\item  (Dimension 0)
The vertices of $\pi_I$ are the $n!$ permutations of $[n]$ in $\Rb I$. 
\item  (Dimension 1)
There is an edge between two vertices $x$ and $y$ if and only if they can be obtained from each other by swapping the positions of the numbers $r$ and $r+1$ for some $r$; that is, $y=(r, r+1)\circ x$.
Thus $x_i=r$ and $x_j=r+1$ become $y_i=r+1$ and $y_j=r$ for some $i$ and $j$, while $y_k = x_k$ for $k \neq i, j$.
The edge joining $x$ and $y$ is a parallel translate of the vector $e_i-e_j$.
\item (Dimension $n-2$) The $2^n-2$ facets of $\pi_I$ are given by the defining inequalities above.
\item  (Arbitrary dimension)
The $(n-k)$-dimensional faces of $\pi_I$ are in bijection with the compositions of $I$ into $k$ parts.
For each composition $(S_1,\ldots,S_k)$, where $S_i \neq \emptyset$ for all $i$ and $I = S_1 \sqcup \cdots \sqcup S_k$, the corresponding face $\pi_{S_1, \ldots, S_k}$ has as vertices the permutations $x \in \Rb I$  such that the entries $\{x_i \, : \, i \in S_1\}$ are the largest $|S_1|$ numbers in $[n]$, the numbers
$\{x_i \, : \, i \in S_2\}$ are the next largest $|S_2|$ numbers in $[n]$, and so on. This description shows that the face $\pi_{S_1, \ldots, S_k}$ is a parallel translate of the product of permutahedra $\pi_{S_1} \times \cdots \times \pi_{S_k}$ in $\Rb S_1 \times \cdots \times \Rb S_k \cong \Rb I$.
\item (Face containment) From the previous  it follows that $\pi_{S_1, \ldots, S_k} \leq \pi_{T_1, \ldots, T_l}$ if and only if $(S_1,\ldots,S_k)$ refines $(T_1,\ldots,T_l)$ in the sense that each  $T_i$ is a union of consecutive $S_j$s.
\end{itemize}


There is another convenient representation of the permutahedron. The \emph{Minkowski sum} of polytopes $\wp$ and $\wq$ in $\Rb I$ is defined to be 
\[
\wp+\wq := \{p + q \, : \, p \in \wp,\,  q \in \wq\} \subseteq \Rb I
\]
A \emph{zonotope} is a Minkowski sum of segments. 
If we let $\Delta_{ij}$ be the segment connecting $e_i$ and $e_j$ in $\Rb I$, then we can represent the standard permutahedron as the zonotope:
\begin{equation}\label{e:Pizonotope}
\pi_I =  
 \sum_{i \neq j \in I} \Delta_{ij} + \sum_{i \in I} e_i.
\end{equation}
Note that the summand $\sum_{i \in I} e_i$ is simply a translation by the vector $(1, \ldots, 1)$.

%
%
%

\subsection{{Normal fans of polytopes}}\label{ss:nf}

Let $(\Rb I)^*$ denote the dual vector space to $\Rb I$. 
We may naturally identify
\[
(\Rb I)^* = \Rb^I\ :=\ \{\text{functions $y:I\to \Rb$}\},
\]
where an element $y$ of $\Rb^I$ acts as a linear functional on $\Rb I$ by $y(\sum_{i \in I} a_i e_i) = \sum_{i \in I} a_i y(i)$ for $\sum_{i \in I} a_i e_i \in \Rb I$. We call such a $y$ a \emph{direction}. 

Let $\{\1_i \, : \, i \in I\}$ be the basis of $\Rb^I$ dual to the basis $\{e_i \, : \, i \in I\}$ of $\Rb I$. We write $\1_S = \sum_{i \in S} \1_i$ for each subset $S \subseteq I$, so
\[
\1_S(x) = \sum_{i \in S} x_i, \qquad \textrm{for } x=(x_i)_{i \in I} \in \Rb I.
\]
We also write $\1 = \1_I$. These functions will play a crucial role.


For a (possibly unbounded) polyhedron $\wp \in \Rb I$ and a direction $y \in \Rb^I$ in the dual space, we denote the maximum face of $\wp$ in the direction of $y$, or \emph{$y$-maximum face} of $\wp$, by 
\[
\wp_y := \{p \in \wp \, : \, y(p) \geq y(q) \textrm{ for all } q \in \wp\}.
\]
For each face $\wq$ of $\wp$, we define the (open and closed) \emph{normal cone}s to be 
\begin{eqnarray*}
\Nc^o_{\wp}(\wq) &:=& \{y \in \Rb^I \,: \, \wp_y = \wq\}, \\
\Nc_{\wp}(\wq) &=& \overline{\Nc^o_{\wp}(\wq)} := \{y \in \Rb^I \,: \, \wq \textrm{ is a face of } \wp_y\},
\end{eqnarray*}
respectively. Note that $\dim \Nc_{\wp}(\wq) = |I|- \dim \wq$ and that $\wq_1$ is a face of $\wq_2$ if and only if  $\Nc_{\wp}(\wq_2)$ is a face of $\Nc_{\wp}(\wq_1)$.

The \emph{normal fan} $\Nc_\wp$ of $\wp \subseteq \Rb^I$ is the polyhedral fan consisting of the normal cones $\Nc_{\wp}(\wq)$ for all faces $\wq$ of $\wp$. Its support is the cone of directions with respect to which $\wp$ is bounded above. \cite[Chapter 2]{sturmfels96}

\subsection{{The braid arrangement and generalized permutahedra}}\label{ss:gp}

The normal fan $\Nc_{\pi_I}$ of the permutahedron $\pi_I$ is the set of faces of the  \emph{braid arrangement} in $\Rb^I$, which consists of the ${n \choose 2}$ hyperplanes
\[
\Bc_I: \qquad y(i) = y(j),  \qquad i, j \in I, \, i \neq j. 
\]
The faces of the braid arrangement $\Bc_I$ are also in bijection with the  compositions of $I$. The composition $I=S_1\sqcup \cdots \sqcup S_k$ gives rise to the face $\Bc_{S_1, \ldots, S_k}$ 
of directions $y \in \Rb^I$ such that $y(i)=y(j)$ if $i, j \in S_a$ and $y(i) \geq y(j)$ if $i \in S_a, \, j \in S_b, \, a<b$. For example, for the composition $I = \{b, f\} \sqcup \{a,c,e\} \sqcup \{d\}$ of $\{a,b,c,d,e,f\}$ we have
\[
\Bc_{bf, ace, d} = \{y \in \Rb^{\{a,b,c,d,e,f\}} \, : \, y(b) = y(f) \geq y(a) = y(c) = y(e) \geq y(d)\}
\]
For a direction $y$ in the relative interior of $\Bc_{S_1, \ldots, S_k}$, the $y$-maximum face is $(\pi_I)_y = \pi_{S_1, \ldots, S_k}$.

Recall that a fan $\mathcal{G}$ is a \emph{coarsening} of $\mathcal{F}$
(or $\mathcal{F}$ is a  \emph{refinement} of  $\mathcal{G}$) if every cone of $\mathcal{F}$ is contained in a cone of $\mathcal{G}$ or, equivalently, if every cone of $\mathcal{G}$ is a union of cones of $\mathcal{F}$. We are now ready to define our main object of study. 

\begin{proposition}
A \emph{generalized permutahedron} $\wp \subseteq\Rb I$ is a polyhedron whose normal fan $\Nc_\wp$ is a coarsening of the braid arrangement $\Bc_I = \Nc_{\pi_I}$ in $\Rb^I$.
\end{proposition}

It will sometimes be useful to allow the deformations of a permutahedron to become unbounded. Equivalently, we wish to allow the support of the normal fan to be smaller than the ambient space $\Rb^I$.

\begin{definition}\label{def:GP}
An \emph{extended generalized permutahedron} $\wp \subseteq\Rb I$ is a polyhedron whose normal fan $\Nc_\wp$ is a coarsening of a subfan of the braid arrangement $\Bc_I = \Nc_{\pi_I}$ in $\Rb^I$.
\end{definition}

Informally speaking, a \emph{generalized permutahedron} is a deformation of the standard permutahedron which can be carried out in at least two equivalent ways: \\
$\bullet$ by moving the vertices while preserving the edge directions, and \\
$\bullet$ by translating the facets without allowing them to move past vertices. \\
Figure \ref{f:genperm} illustrates this informal description; for details on these points of view, see \cite{prw08}.

\begin{figure}[h]
\centering
\includegraphics[scale=.5]{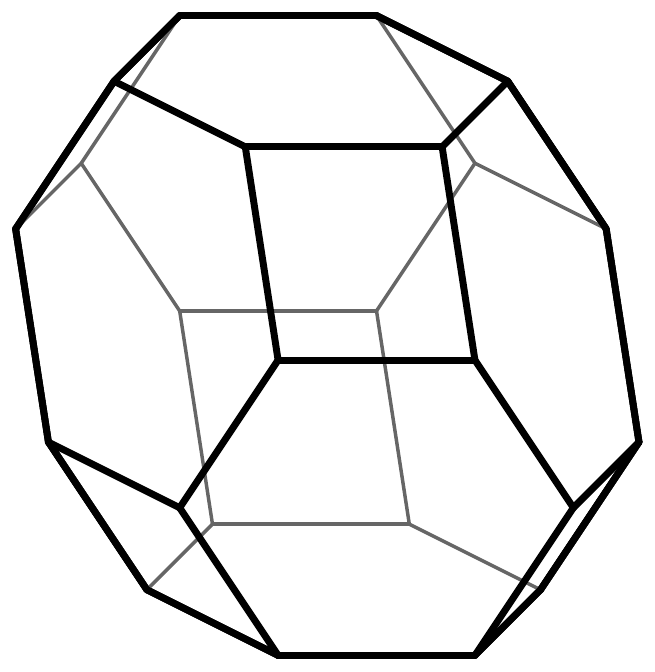} \qquad
\includegraphics[scale=.5]{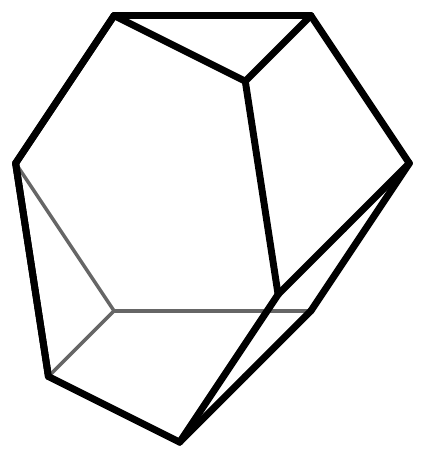}  \quad
\includegraphics[scale=.5]{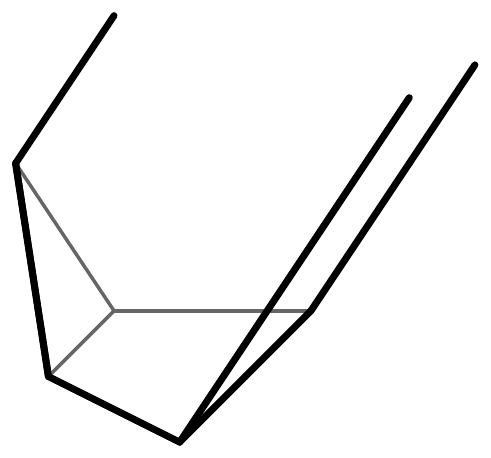} \quad
\includegraphics[scale=.5]{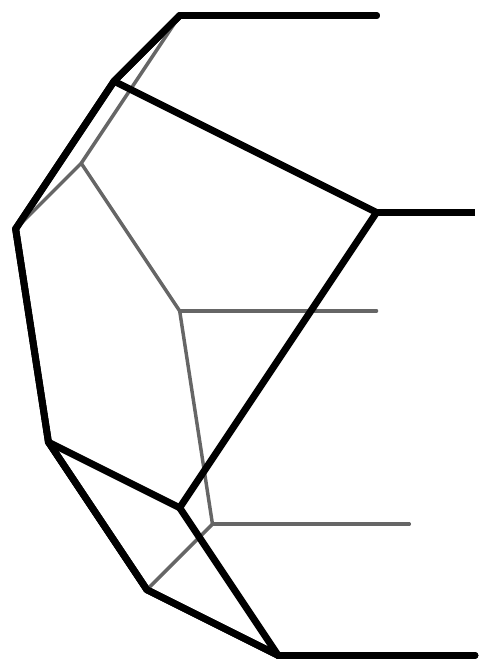} \quad
\includegraphics[scale=.5]{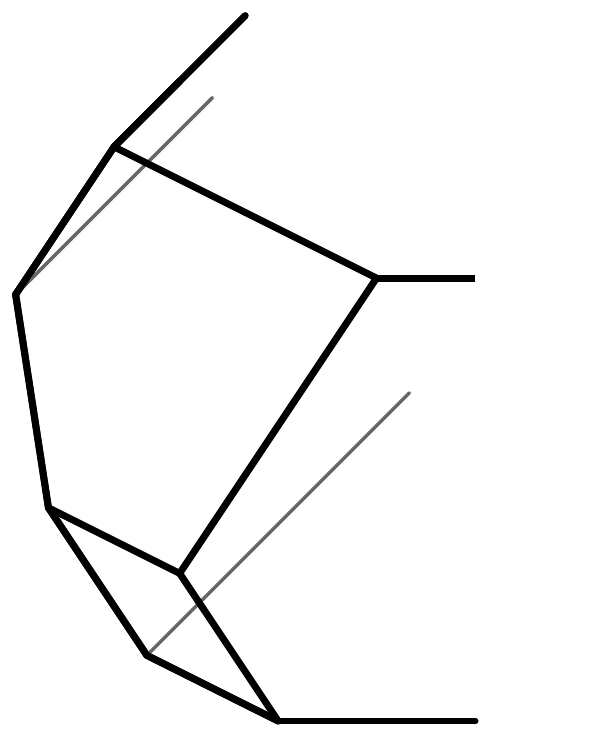} \quad
\caption{The standard permutahedron in $\Rb^4$ and four of its deformations.\label{f:genperm}}
\end{figure}

The faces of a generalized permutahedron $\wp$ in $\Rb I$ can also be labeled by compositions of $I$. 
Consider one such composition $I = S_1\sqcup \cdots \sqcup S_k$. By definition, if $\wp$ is bounded above in some direction in the open cone $\Bc^\circ_{S_1, \ldots, S_k}$, then it is bounded above in \textbf{every} direction in $y \in \Bc^\circ_{S_1, \ldots, S_k}$ and, furthermore, the $y$-maximal face is always the same. We denote that face
\[
\wp_{S_1, \ldots, S_k}:= \wp_y \,\, \textrm{ for any } \,\,  y \in \Bc^\circ_{S_1, \ldots, S_k}.
\]
Every face of $\wp$ arises in this way from a composition. Note, however, that this correspondence is no longer bijective: different compositions may lead to the same face, and some compositions may have no corresponding face.

\begin{remark}
In the language that we use, generalized permutahedra were introduced by Postnikov in \cite{postnikov09}. Up to translation, they are equivalent to \emph{polymatroids}, which were defined earlier by Edmonds. \cite{edmonds70}. As we will explain in Section \ref{s:SF}, they are also equivalent to \emph{base polyhedra of submodular functions}  \cite{edmonds70, fujishige05, schrijver03}, although that is not clear from the outset. We slightly generalize these definitions to allow for unbounded polyhedra; these were also considered by Fujishige \cite{fujishige05} and Derksen and Fink \cite{derksenfink10}. 
\end{remark}

\section{$\rGP$: {The Hopf monoid of generalized permutahedra}}\label{s:GP}

\subsection{{The Hopf monoid $\rGP$ of generalized permutahedra}}\label{ss:GP} 

We now introduce combinatorial and algebraic operations that give generalized permutahedra the structure of a Hopf monoid in set species. We focus on \emph{bounded} polytopes for the moment, and treat the unbounded case in Section \ref{ss:GP+}. 

%
%

\begin{proposition} \label{p:product}
Let $I = S \sqcup T$ be a decomposition. If $\wp \subseteq\Rb S$ and $\wq \subseteq\Rb T$ are bounded generalized permutahedra, then $\wp \times \wq \subseteq\Rb I$ is a bounded generalized permutahedron.
\end{proposition}

\begin{proof}
Since the braid arrangements $\Bc_S$ and $\Bc_T$ refine $\Nc_\wp$ and $\Nc_\wq$ respectively, their product $\Bc_S \times \Bc_T$ refine $\Nc_\wp \times \Nc_\wq = \Nc_{\wp \times \wq}$. The result then follows from the fact $\Bc_{S \sqcup T}$ refines $\Bc_S \times \Bc_T$.
%
\end{proof}

\begin{proposition}\cite[Thm. 3.15]{fujishige05}\label{p:face}
Let $\wp \subset \Rb I$ be a generalized permutahedron and $I = S \sqcup T$ a decomposition. 
By definition, the linear functional $\1_S(x) = \sum_{s \in S} x_s$ is maximized at the face $\wp_{S,T}$ of $\wp$.
Then there exist generalized permutahedra $\wp|_S \subset\Rb S$ and $\wp/_S \subset\Rb T$ such that 
\[
\wp_{S,T} = \wp|_S \times \wp/_S.
\]
We call $\wp|_S$ and $\wp/_S$ the \emph{restriction} and \emph{contraction} of $\wp$ with respect to $S$, respectively.
\end{proposition}

\begin{proof}
This result is known and can be proved directly; since it is a straightforward consequence of the proof of Theorem \ref{t:submod-gp}, we delay the proof until then.  
%
%
%
\end{proof}

\begin{figure}[h]
\centering
\includegraphics[scale=.7]{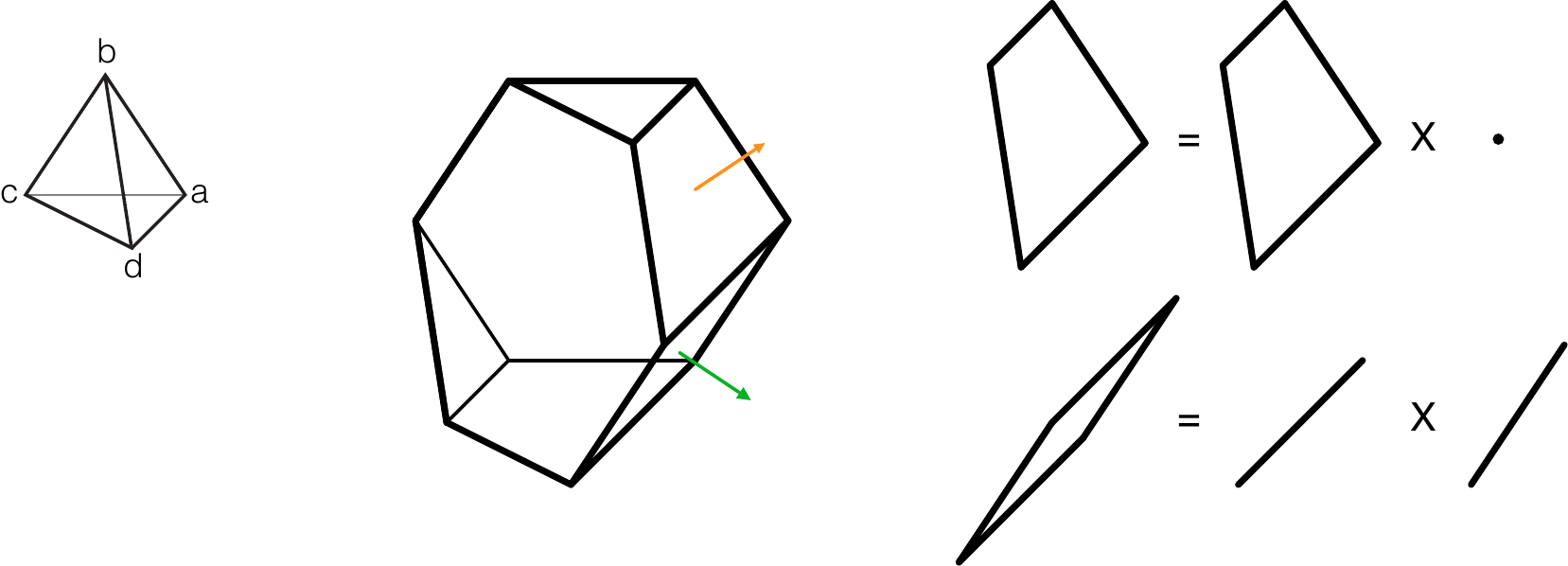}\qquad 
\caption{A generalized permutahedron $\wp$ and its faces $\wp_{abd} = \wp|_{abd} \times \wp/_{abd} \subset \Rb^{\{a,b,d\}} \times \Rb^{\{c\}}$ and $\wp_{ad} = \wp|_{ad} \times \wp/_{ad} \subset \Rb^{\{a,d\}} \times \Rb^{\{b,c\}}$. \label{f:genperm}}
\end{figure}


\begin{theorem}\label{t:GPisHopf}
Let $\rGP[I]$ denote the set of bounded generalized permutahedra on $I$. Define a product and coproduct as follows. 

Let $I = S \sqcup T$ be a decomposition.
 \\
$\bullet$
If $\wp \in \rGP[S]$ and $\wq \in \rGP[T]$, define their product to be
\[
\wp \cdot \wq := \wp \times \wq \in \rGP[I].
\]
$\bullet$
If $\wp \in \rGP[I]$ then define its coproduct to be
\[
\Delta_{S,T}(\wp) = (\wp|_S, \wp/_S),
\]
where the restriction $\wp|_S$ and contraction $\wp/_S$ are defined in Proposition \ref{p:face}.

These operations turn the set species $\rGP$ into a Hopf monoid.
\end{theorem}

\begin{proof}
The previous two propositions prove that the operations are well-defined, so it remains to check the axioms of a Hopf monoid. We verify the two non-trivial ones: the associativity of the coproduct and the compatibility of the product and coproduct.

\smallskip

\noindent \emph{Coassociativity.}
For any decomposition $I = R \sqcup S \sqcup T$ and any generalized permutahedron $\wp \subseteq \Rb I$ we need to show that
\[
(\wp|_{R\sqcup S})|_R=\wp|_R,
\qquad
(\wp|_{R\sqcup S})/_R=(\wp/_R)|_S,
\qquad
\wp/_{R\sqcup S}=(\wp/_R)/_S.
\]

Notice that
\begin{align*}
(\wp_{R \sqcup S, T})_{R, S \sqcup T} &=
(\wp|_{R \sqcup S} \times \wp/_{R \sqcup S})_{R, S \sqcup T} = 
(\wp|_{R \sqcup S})_{R,S} \times \wp/_{R \sqcup S} = 
(\wp|_{R \sqcup S})|_R \times (\wp|_{R \sqcup S})/_R \times \wp/_{R \sqcup S}
\\
(\wp_{R, S \sqcup T})_{R \sqcup S, T} &=
(\wp|_R \times \wp/_R)_{R \sqcup S, T} =
\wp|_R \times (\wp/_R)_{S, T} =
\wp|_R \times (\wp/_R)|_S \times (\wp/_R)/_S
\end{align*}
so it suffices to prove the equality of these two polytopes. To do this, we use the following general fact about polytopes \cite{gruenbaum67:_convex}: for any polytope $\wp \in \Rb I$ and any directions $v, w \in \Rb^I$, we have 
\begin{equation}\label{eq:faceofaface}
(\wp_v)_w = \wp_{v + \lambda w} \quad \textrm{ for any small enough }\lambda > 0.
\end{equation}
 Therefore we have 
\begin{eqnarray*}
(\wp_{R \sqcup S , T})_{R, S \sqcup T} &=& (\wp_{\1_{R\sqcup S}})_{\1_R} = \wp_{\1_{R \sqcup S} + \lambda \1_R} = \wp_{R,S,T}\\
(\wp_{R, S \sqcup T})_{R \sqcup S, T} &=& (\wp_{\1_R})_{\1_{R \sqcup S}} = \wp_{\1_R + \lambda \1_{R \sqcup S}} = \wp_{R,S,T}
\end{eqnarray*}
since both $\1_{R \sqcup S} + \lambda \1_R$ and $\1_R + \lambda \1_{R \sqcup S}$ are in the open face $\Bc^\circ_{R,S,T}$ of the braid arrangement.

\smallskip

\noindent \emph{Compatibility.}
Fix decompositions $S \sqcup T = I = S' \sqcup T'$ and let $A = S \cap S', B = S \cap T', C = T \cap S', D = T \cap T'$, as in \ref{e:4sets} . Let $\wp \subseteq \Rb S$ and $\wq \subseteq \Rb T$ be generalized permutahedra. We need to verify that
\[
 (\wp \cdot \wq)|_{S'}  = \wp|_A\cdot \wq|_C 
 \qand
 (\wp \cdot \wq)/_{S'} = \wp/_A\cdot \wq/_C,
\]
This follows from the following computation:
\[
 (\wp \cdot \wq)|_{S'}  \times  (\wp \cdot \wq)/_{S'} =  (\wp \cdot \wq)_{\1_{S'}} = (\wp \cdot \wq)_{\1_A + \1_C} = 
\wp_{\1_A} \cdot \wq_{\1_C} = \wp|_A \cdot \wp/_A \cdot \wq|_C \cdot \wq/_C = (\wp|_A \cdot \wq|_C) \cdot (\wp/_A \cdot \wq/_C.
\] 
The proof of compatibility is now complete.
\end{proof}

The following strengthening of Proposition \ref{p:face} will be useful. 

\begin{proposition} \label{p:face2}
Let $\wp \subset \Rb I$ be a generalized permutahedron and let $I = S_1 \sqcup \cdots \sqcup S_k$ be a composition. Let $y \in \Rb^I$ be a direction contained in the open face $F^\circ = \Bc^\circ_{S_1, \ldots, S_k}$ of the braid arrangement. Then the $y$-maximal face $\wp_y$ depends only on $F$; we also denote it by $\wp_F$ or $\wp_{S_1, \ldots S_k}$. Polytopally, we can express this face as a product 
\[
\wp_y = \wp_F = \wp_{S_1, \ldots S_k} =  \wp_1 \times \cdots \times \wp_k
\]
of generalized permutahedra $\wp_i \subset \Rb S_i$ for $1 \leq i \leq k$. Hopf-theoretically, we can express it as
\[
\wp_F = \mu_F \Delta_F (\wp)
\]
where $\mu_F = \mu_{S_1, \ldots, S_k}$ and $\Delta_F = \Delta_{S_1, \ldots, S_k}$.
\end{proposition}

\begin{proof}
Since $\wp$ is a generalized permutahedron, the braid arrangement $\Bc_I$ refines its normal fan $\Nc(\wp)$, so $\wp_y$ depends only on the face $F$ of $\Bc_I$ containing $y$. Therefore it suffices to compute $\wp_y$ for one such $y$; we choose $y = \1_{S_1} + \lambda_2 \1_{S_1 \sqcup S_2} + \cdots + \lambda_k \1_{S_1 \sqcup \cdots \sqcup S_k}$ for scalars $1 >> \lambda_2 >> \cdots >> \lambda_k > 0$. The
 result now follows from (\ref{eq:faceofaface}) by  iterating Proposition \ref{p:face}. The Hopf-theoretic statement then follows from the coassociativity of $\rGP$ in Theorem \ref{t:GPisHopf}, which implies that $\Delta_{S_1, \ldots, S_k}(\wp) = (\wp_1, \ldots, \wp_k)$.
\end{proof}

\begin{remark}
In the language of polymatroids, equivalent definitions of restriction and contraction were given in \cite{edmonds70} and a similar Hopf structure was defined in \cite{derksenfink10}. We emphasize the polytopal perspective, which allows us to obtain many new results.
\end{remark}

%
%
%
%
%

\subsection{{The Hopf monoids $\rbGP$ and $\rbbGP$ of equivalence classes of generalized permutahedra}}\label{s:rbbGP}

There are two quotients of the Hopf monoid $\rGP$ that are often useful to consider. Say that two generalized permutahedra $\wp$ and $\wq$ in $\Rb I$ are \emph{normally equivalent} if they have the same normal fan:
 \[
\wp\equiv\wq \iff \Nc_\wp=\Nc_\wq.
\]
Say they are \emph{quasinormally equivalent} if they are normally equivalent up to the action of the symmetric group on the underlying vector spaces: for $\wp, \wq \in \rGP[I]$
\[
\wp \equiv' \wq \iff \wp\equiv w(\wq) \textrm{ for some } w \in S_I.
\]
where $w \in S_I$ acts naturally on $\Rb I$ by $w(\sum_{i \in I} a_i e_i) = \sum_{i \in I} a_i e_{w(i)}$.\footnote{Notice that this is very close to the definition of the Hopf algebra $GP = \Kcb(\wGP)$.}
These notions of equivalence are stronger than  combinatorial equivalence, which is the property of having isomorphic face posets. 


Let $\rbGP[I]$ and $\rbbGP[I]$ denote the quotients of $\rGP[I]$ by these equivalence relations, respectively.
Since the operations of $\rGP$ are defined in terms of normal fans, the Hopf monoid structure of $\rGP$ descends to $\rbGP$ and $\rbbGP$ via morphisms:
$\rGP \onto \rbGP \onto \rbbGP$.
The sets $\rbGP[I]$ and $\rbbGP[I]$ of normal and quasinormal equivalence classes of generalized permutahedra in $\Rb^I$ is finite, since there are finitely many coarsenings of subfans of the braid arrangement.


\subsection{{The Hopf monoid $\wGP_+$  of possibly unbounded generalized permutahedra}} \label{ss:GP+}

To allow unbounded polytopes, we need to work with Hopf monoids in vector species. Propositions \ref{p:product} and (when $\wp_{S,T}$ exists) \ref{p:face} still hold, and imply the following result.

\begin{theorem}
Let $\rGP_+[I]$ be the set of extended generalized permutahedra on $I$
and let $\wGP_+[I]$ denote its linearization. Define a product and coproduct as follows. 

Let $I = S \sqcup T$ be a decomposition.\\
$\bullet$
If $\wp \in \rGP_+[S]$ and $\wq \in \rGP_+[T]$, define their product to be
\[
\wp \cdot \wq := \wp \times \wq \in \rGP_+[I]
\]
$\bullet$
If $\wp \in \rGP_+[I]$ then define its coproduct to be
\[
\Delta_{S,T}(\wp) = 
\begin{cases}
\wp|_S \otimes \wp/_S & \textrm{ if $\wp$ is bounded in the direction of $\1_S$}, \\
0 & \textrm{ otherwise},
\end{cases}
\]
where the restriction $\wp|_S$ and contraction $\wp/_S$ are defined in Proposition \ref{p:face}.

These operations turn the vector species $\wGP_+$ into a Hopf monoid.
\end{theorem}

Similarly, we let $\overline{\wGP}_+$ and $\overline{\wbGP}_+$ be the Hopf monoids of extended generalized permutahedra modulo normal equivalence and quasinormal equivalence, respectively.

\section{Universality of $\rGP$}\label{s:universality}

The previous section showed that the family of generalized permutahedra is a natural polyhedral setting for Hopf theory: it has natural product and coproduct operations that turn it into a Hopf monoid. 
In a sense, it is the only such setting: we now show that generalized permutahedra are the only polytopes for which these operations give a Hopf monoid.

\begin{theorem} \emph{(Universality Theorem)}
Suppose $\rP$ is a connected Hopf monoid on set species whose elements are polytopes, and whose operations are defined as in Theorem \ref{t:GPisHopf}; that is, for any decomposition $I = S \sqcup T$:\\
$\bullet$
If $\wp \in \rP[S]$ and $\wq \in \rP[T]$ are polytopes, their product is the polytopal product:
\[
\wp \cdot \wq := \wp \times \wq \in \rP[I]
\]
$\bullet$
If $\wp \in \rP[I]$ then its maximal face in the direction $\1_S$ factors as $\wp_{\1_S} = \wp|_S \times \wp/_S =: \wp_{S,T}$ for  polytopes $\wp|_S \in \rP[S]$ and $\wp/_S \in \rP[T]$, and the coproduct of $\wp$ is
\[
\Delta_{S,T}(\wp) = (\wp|_S, \wp/_S).
\]
Then every polytope in $\rP$ is a generalized permutahedron, and $\rP$ is a submonoid of $\rGP$.
\end{theorem}

\begin{proof}
First notice that for any polytope $\wp \in \rP[I]$ we have 
\[
\wp_{\1_I} = \wp|_I \times \wp/_I = \wp \times 1 = \wp
\]
using counitality and the connectedness of $\rP$. It follows that $\1_I(x) = \sum_{i \in I} x_i$ is constant for $x \in \wp$. Therefore $\wp$  is not full-dimensional, and the direction $\1_I$ is in the lineality space of its normal fan $\Nc_\wp$.

As observed in the proof of Theorem \ref{t:GPisHopf}, the coassociativity for $\rP$ implies that 
\[
(\wp_{R \sqcup S , T})_{R, S \sqcup T} =
(\wp_{R, S \sqcup T})_{R \sqcup S, T} 
\]
Let us call this polytope $\wp_{R,S,T}$. Then by (\ref{eq:faceofaface}) we have
\[
\wp_{R,S,T} = \wp_{\1_{R \sqcup S} + \lambda \1_R} = \wp_{\1_R + \lambda \1_{R \sqcup S}}
\]
for any small enough $\lambda > 0$. It follows that $\1_{R \sqcup S} + \lambda \1_R$ and $\1_R + \lambda \1_{R \sqcup S}$ are both in the normal cone $\Nc_{\wp}(\wp_{R,S,T})$ corresponding to the face $\wp_{R,S,T}$ of $\wp$. Since that cone is closed, we may take the limit $\lambda \rightarrow 0$ and obtain that $\1_R$ and $\1_{R \sqcup S}$ (and, as we already knew, $\1_{R \sqcup S \sqcup T}$) are in $\Nc_{\wp}(\wp_{R,S,T})$. Therefore the whole braid cone $\Bc_{R,S,T}= \cone\{\1_R, \1_{R \sqcup S}, \1_{R \sqcup S \sqcup T}\}$ is in the normal cone $\Nc_{\wp}(\wp_{R,S,T})$.

We now use higher coassociativity to carry out the same argument for any composition $I = S_1 \sqcup \cdots \sqcup S_k$. The higher coproduct $\Delta_{S_1, \ldots, S_k}(\wp) = (\wp_1, \ldots, \wp_k)$ may be computed by iterating the coproduct maps in any of $(k-1)!$ meaningful ways. If we define $\wp_{S_1, \ldots, S_k} := \wp_1 \times \cdots \times \wp_k$ then we obtain $(k-1)!$ expressions for this face of $\wp$; one of them is
\[
\wp_{S_1, \ldots, S_k} = (\cdots((\wp_{S_1 \sqcup \cdots \sqcup S_{k-1}, S_k})_{S_1 \sqcup \cdots \sqcup S_{k-2}, S_{k-1} \sqcup S_k})_{\ldots})_{S_1, S_2 \sqcup \cdots \sqcup S_k}.
\]
This implies that the direction
$\1_{S_1 \sqcup \cdots \sqcup S_{k-1}} + \lambda_1 \1_{S_1 \sqcup \cdots \sqcup S_{k-2}} + \cdots +  \lambda_{k-2} \1_{S_1}$ 
is contained in the
normal cone $\Nc_\wp(\wp_{S_1, \ldots, S_k})$ 
for any $\lambda_1 >> \lambda_2 >> \cdots >> \lambda_{k-2}>0$. By sending $\lambda_{k-2}, \lambda_{k-3}, \ldots, \lambda_1 \rightarrow 0$ in that order, we obtain $\1_{S_1 \sqcup \cdots \sqcup S_{k-1}} \in \Nc_\wp(\wp_{S_1, \ldots, S_k})$. By computing the coproduct in different orders, we similarly obtain $\1_{S_1 \sqcup \cdots \sqcup S_{j}} \in \Nc_\wp(\wp_{S_1, \ldots, S_k})$ for any $1 \leq j \leq k-1$, and we already knew that $\1_{S_1 \sqcup \cdots \sqcup S_k} = \1_I \in \Nc_{\wp}(\wp_{S_1, \ldots, S_k})
$ as well. Therefore 
\[
\Bc_{S_1, \ldots, S_k} = \cone\{\1_{S_1}, \1_{S_1 \sqcup S_2}, \ldots, \1_{S_1 \sqcup \cdots \sqcup S_k} \} \subseteq \Nc_{\wp}(\wp_{S_1, \ldots, S_k}).
\]
It follows that every cone $\Bc_{S_1, \ldots, S_k}$ of the braid arrangement is contained in a cone of the normal fan $\Nc_\wp$. By definition, this means that $\wp$ is a generalized permutahedron, as we wished to show. 

Since $\rP$ and $\rGP$ have the same product and coproduct, it now follows that $\rP$ is a Hopf submonoid of $\rGP$, as desired.
\end{proof}

Similar statements hold for Hopf monoids of possibly unbounded polytopes. We leave the details to the reader.

\section{{The antipode of $\rGP$}}\label{s:antipode}

In this section we show a remarkably simple formula for the antipode of the Hopf monoid of generalized permutahedra. This is the best possible formula since it involves no cancellation or repeated terms.
We will see throughout the paper that this formula generalizes numerous results in the literature and answers several open questions. 

\begin{theorem}\label{t:antipode}
The antipodes of the Hopf monoids $\wGP, \wGP_+$ of generalized permutahedra are given by the following \textbf{cancellation-free} and \textbf{grouping-free} formula. If $\wp \in \Rb^I$ is a generalized permutahedron, then
\[
\apode_I(\wp) = (-1)^{|I|}\sum_{\wq \leq \wp} (-1)^{\dim \wq} \, \wq,
\]
where we sum over all the nonempty faces $\wq$ of $\wp$. The same formula holds for the quotients $\wbGP, \wbGP_+, \wbbGP,$ and $\wbbGP_+$, where it is still cancellation-free. 
\end{theorem}

\begin{proof} 
Takeuchi's formula (\ref{e:takeuchi}) gives 
\[
\apode_I(\wp) \ =\ \sum_{(S_1,\ldots,S_k) \vDash I \atop k\geq 0} (-1)^k\, \mu_{S_1,\ldots,S_k}\circ\Delta_{S_1,\ldots,S_k}(\wp)
\]
summing over all compositions of $I$. 
Let us examine each summand. 

To compute $\Delta_{S_1,\ldots,S_k}(\wp)$ we look for the face 
$\wp_{S_1, \ldots, S_k}$ that maximizes any direction $y$ in the open cone $\Bc^\circ_{S_1, \ldots, S_k}$ of the braid arrangement $\Bc_I$. If there is no such face because $\wp$ is unbounded in those directions, then $\Delta_{S_1,\ldots,S_k}(\wp) = 0$. If there is an $y$-maximum face, then by Proposition \ref{p:face} it factors as a product of $k$ generalized permutahedra 
\[
 \wp_{S_1, \ldots, S_k} = 
 \wp_1 \times \cdots \times \wp_k
\]
where 
\[ 
\wp_i = (\wp|_{S_1 \sqcup \cdots \sqcup S_{i}})/_{S_1 \sqcup \cdots \sqcup S_{i-1}} \in \wGP[S_i] 
\] 
for $1 \leq i \leq k$. 
We then have  $\Delta_{S_1,\ldots,S_k}(\wp) = (\wp_1, \ldots, \wp_k)$, and therefore
$\mu_{S_1,\ldots,S_k}\circ\Delta_{S_1,\ldots,S_k}(\wp) =  \wp_1 \times \cdots \times \wp_k = 
\wp_{S_1, \ldots, S_k}$. We conclude that

\begin{equation}\label{eq:alternating}
\apode_I(\wp) \ =\ \sum_{(S_1,\ldots,S_k) \vDash I \atop k\geq 0} (-1)^k\, \wp_{S_1,\ldots,S_k} =: \sum_{\wq \leq \wp} \alpha_{\wq} \, \wq
\end{equation}
is indeed a linear combination of the non-empty faces of $\wq \leq \wp$.

Now let $\wq$ be a face of $\wp$, and let us compute the coefficient $\alpha_{\wq}$ of $\wq$ in the right hand side of (\ref{eq:alternating}). 
For a composition $(S_1, \ldots, S_k)$ we have $\wp_{S_1,\ldots,S_k} = \wq$ if and only if $\Bc^\circ_{S_1, \ldots, S_k} \subseteq \Nc^\circ_\wp(\wq)$. Recalling that the normal fan $\Nc_\wp$ refines the braid arrangement, we define
\begin{eqnarray*}
\C_\wq &=& \{\Bc_{S_1, \ldots, S_k} \, : \, \Bc^\circ_{S_1, \ldots, S_k} \subseteq \Nc^\circ_\wp(\wq)\}\\
\overline{\C_\wq} &=& \{\Bc_{S_1, \ldots, S_k} \, : \, \Bc^\circ_{S_1, \ldots, S_k} \subseteq \Nc_\wp(\wq)\}
\end{eqnarray*}

Then the coefficient of $\wq$ in the right hand side of (\ref{eq:alternating}) is
\[
\alpha_\wq = \sum_{C \in \mathcal{\C_\wq}} (-1)^{\dim C}.
\]
We would like to interpret this as an Euler characteristic, but the set of polyhedra $\C_\wq$ is not a polyhedral complex, since it is not closed under taking faces. To remedy this, we observe that $\overline{\C_\wq}$ and $\overline{\C_\wq} - \C_\wq$ \textbf{are} polyhedral complexes, and we may rewrite the previous equation as
\[
\alpha_\wq = \sum_{C \in \overline{\mathcal{C_\wq}}} (-1)^{\dim C}
- \sum_{C \in \overline{\mathcal{C_\wq}} - \mathcal{C_\wq}} (-1)^{\dim C}.
\]








Let us intersect the cones in $\overline{\C_\wq}$ with the sphere $\mathcal{S} := \{x \in \Rb^I\, : \, \sum x_i = 0, \sum x_i^2=1\}$ to make them bounded. The resulting cells  
form a CW-decomposition of the $(\dim \Nc_\wp(\wq)-2)$-ball $\Nc_\wp(\wq) \cap \mathcal{S}$, while the cells in $\overline{\C}_\wq-\C_\wq$ form a CW-decomposition of the $(\dim \Nc_\wp(\wq)-3)$-sphere $\partial \Nc_\wp(\wq) \cap \mathcal{S}$.\footnote{The $0$-ball is a point and the $(-1)$-sphere is the empty set. Extra care is required in the trivial case that $\dim \Nc_\wp(\wq)=1$. In this case we must have that $\wq = \wp$ and $\dim \wp = |I|-1$, so $\C_\wq = \{\Bc_I\}$ and $\alpha_\wq=-1$.} 
Therefore
\begin{eqnarray*}
\alpha_\wq &=& \overline{\chi}\left(\Nc_\wp(\wq) \cap \mathcal{S}\right) - \overline{\chi}\Large(\partial \Nc_\wp(\wq) \cap \mathcal{S}\Large) \\
&=&0-(-1)^{\dim \Nc_\wp(\wq)-3} = (-1)^{|I|-\dim \wq}.
\end{eqnarray*}
where $\overline{\chi}$ denotes the reduced Euler characteristic; we must use $\overline{\chi}$ because both $\overline{\C}_\wq$ and $\overline{\C}_\wq-\C_\wq$ contain the one-dimensional ray $\Bc_I$, which becomes the empty face when we intersect it  with $\mathcal{S}$. The desired result follows.

\medskip

Combining this with (\ref{eq:alternating}) gives the desired formula, which is clearly cancellation-free and grouping-free in $\wGP$ and $\wGP_+$.

In the quotients $\wbGP$, $\wbbGP$, $\wbGP_+$ and $\wbbGP_+$, (quasi)normally equivalent faces of $\wp$ will lead to the grouping of like terms in this antipode formula.  Since (quasi)normally equivalent faces must have the same dimension, the antipode formula is still cancellation-free.
\end{proof}

\bigskip

\begin{LARGE}
\noindent 
\textsf{PART 2: Permutahedra, associahedra, and inversion. 
}
\end{LARGE}

\section{{Preliminaries 3: The group of characters of a Hopf monoid}\label{s:prelimcharacters}}

We now return to the general setting of Hopf monoids of Section~\ref{s:hopf}, to define the notion of characters on a Hopf monoid, and discuss how the characters assemble into a group. We will use this general construction  to settle  a question of Loday \cite{loday05} and a conjecture of Humpert and Martin \cite{humpert12} in Sections \ref{s:inversion} and \ref{s:G}, respectively.

\subsection{{Characters}}\label{ss:char}

\begin{definition}\label{d:char}
Let $\wH$ be a connected Hopf monoid in vector species. 
A \emph{character} $\zeta$ on $\wH$ is a collection of linear maps
\[
\zeta_I:\wH[I]\to \Kb,
\]
one for each finite set $I$, subject to the following axioms.

\begin{naturality}
For each bijection $\sigma:I\to J$ and $x\in\wH[I]$, we have $\zeta_J\bigl(\rH[\sigma](x)\bigr) = \zeta_I(x)$.
\end{naturality}

\begin{multiplicativity}
For each $I=S\sqcup T$, $x\in\wH[S]$ and $y\in\wH[T]$, we have $\zeta_I(x\cdot y) = \zeta_S(x)\zeta_T(y)$.
\end{multiplicativity}

\begin{unitality} The map
$\zeta_\emptyset:\wH[\emptyset]  \to \Kb$
sends $1\in\Kb=\wH[\emptyset]$ to $1\in\Kb$: we have $\zeta_\emptyset(1)=1$.

\end{unitality}
\end{definition}

In most examples that interest us, naturality and unitality are trivial, and we can think of characters simply as multiplicative functions. When $\wH$ is the linearization of a Hopf monoid $\rH$ over set species, the characters $\zeta$ are constructed easily: one chooses arbitrarily the value $\zeta_I(h)$ for each object $h \in \rH[I]$ that is indecomposable under multiplication, and then extend those values multiplicatively to all objects.

\subsection{{The character group}}\label{ss:char}

The characters of a connected Hopf monoid $\wH$ have the structure of a group, called the \emph{character group} $\Xb(\wH)$.

\begin{theorem}
Let $\wH$ be a connected Hopf monoid on vector species. The set  $\Xb(\wH)$ of characters of $\wH$ is a group under the convolution product, defined by 
\begin{equation}\label{e:conv}
(\varphi\psi)_I(x) = 
\sum_{I=S\sqcup T} \varphi_S(x|_S)\psi_T(x/_S) 
\end{equation}
for characters $\varphi$ and $\psi$. The identity $\epsilon$ is given by $\epsilon_I=0$ if $I\neq\emptyset$ and $\epsilon_\emptyset(1)=1.$ The inverse of a character $\zeta$ is $\zeta \circ \apode$, its composition with the antipode $\apode$ of $\wH$.
\end{theorem}

\begin{proof}
We need to check that the convolution product of characters $\varphi$ and $\psi$ is indeed a character. Let $I = S \sqcup T$ be a decomposition and $z = x \cdot y$ for $x \in \wH[S]$ and $y \in \wH[T]$. Then, using the notation of \eqref{e:4sets} and the compatibility of the product and coproduct, we get
\begin{align*}
(\varphi \psi)_I(x \cdot y) &= 
\sum_{I = S' \sqcup T'} \varphi_{S'}((x \cdot y)|_{S'}) \psi_{T'}((x \cdot y)/_{S'}) = \sum_{I = S' \sqcup T'} \varphi_{S'}(x|_A \cdot y|_C) \psi_{T'}(x/_A \cdot y/_C) \\
&= \sum_{\substack{S=A \sqcup B \\ T = C \sqcup D}} \varphi_A(x|_A) \varphi_C(y|_C) \psi_B(x/_A) \psi_D(y/_C)  = (\varphi \psi)_S(x)  \cdot    (\varphi \psi)_T(y)
\end{align*}
as desired. It is easy to check that $\epsilon$ is indeed the identity, and the description of the inverse follows from \cite[Definition 1.15]{am}. 
\end{proof}

We mentioned in Section \ref{ss:antipode} that the antipode of a Hopf monoid plays the role of the inverse function in a group. The previous theorem is a concrete manifestation of that analogy. The following is another fundamental question.

\begin{problem}\label{prob:chargp}
Find an explicit description for the character group of a given Hopf monoid.
\end{problem}

We will now answer Problem \ref{prob:chargp} for two Hopf monoids of permutahedra and associahedra in Sections \ref{s:Pi} and \ref{s:F}. This will establish the connection between these Hopf monoids and the inversion of power series, as described in the introduction.

\section{$\rbPi$: {Permutahedra and the multiplication of power series}}\label{s:Pi}

In this section we consider the Hopf monoid of permutahedra, and  show that its character group is the group of formal power series under multiplication.

Recall that $\pi_I$ is the standard permutahedron in $\Rb I$. Let $\rbPi$ be the Hopf submonoid of $\rbGP$ generated by the standard permutahedra. 

\begin{lemma}\label{l:rbPi}
The coproduct of $\rbPi$ is given by 
\[
\Delta_{S,T}(\pi_I) = (\pi_S\, ,\,  \pi_T).
\]
for each decomposition $I=S \sqcup T$.
\end{lemma}

 \begin{proof}
From the description of the faces of permutahedron $\pi_I \subset \Rb^I$ in Section \ref{ss:p} we know that the maximal face of $\pi_I$ in the direction of $\1_S$ is $\pi_{S,T} = \pi_I|_S \times \pi_I/_S$ where
$\pi_I|_S$ is a translation of $\pi_S$ and $\pi_I/_S$ is equal to $\pi_T$. The result follows.
\end{proof}

This implies, in particular, that 
\begin{equation}\label{e:Pibasis}
\wbPi[I] = \mathrm{span} \{\pi_{S_1} \times \cdots \times \pi_{S_k} \, : \, I = S_1 \sqcup \cdots \sqcup S_k\}
\end{equation}
We can now prove the main result of this section.

\begin{theorem}\label{t:charperm}
The group of characters $\Xb(\wbPi)$ of the Hopf monoid of permutahedra is isomorphic to the group of exponential formal power series
\[
\left\{1+a_1x + a_2\frac{x^2}{2!} + a_3\frac{x^3}{3!} + \cdots \, : \, a_1, a_2, \ldots \in \Kb\right\}
\]
under multiplication.
\end{theorem}

\begin{proof}
Since characters are multiplicative and invariant under relabeling, a character $\zeta$ of $\wbPi$ is uniquely determined by the sequence $(1,z_1,z_2,\ldots)$ of values that it takes on the standard permutahedra of order $0, 1, 2, \ldots$. Here
 $z_n=\zeta_I(\pi_I)$ for $|I|=n$. 
 (Recall that any character has $z_0 = \zeta_\emptyset(1)=1$.) We encode this sequence in the exponential generating function $\zeta(t) =  1 + z_1t + z_2{t^2}/{2!} + z_3{t^3}/{3!} + \cdots$. Conversely, any such formal power series determines a character of $\wbPi$.

Now suppose that two characters $\varphi$, $\psi$ and their convolution product $\varphi \psi$ give rise to  sequences $(1,a_1,a_2,\ldots)$, $(1,b_1,b_2,\ldots)$, and $(1, c_1, c_2, \ldots)$, respectively. Consider any $I$ with 
$|I|=n$. By (\ref{e:conv}) we have
\[
c_n = (\varphi \psi)_I(\pi_I) = 
\sum_{I=S\sqcup T} \varphi_S(\pi_S)\psi_T(\pi_T) =
\sum_{k=0}^n {n \choose k} a_kb_{n-k}.
\]
This is equivalent to 
\[
\varphi\psi(x):= 
\sum_{n \geq 0} c_n \frac{x^n}{n!} = 
\left(\sum_{k \geq 0} a_k \frac{x^k}{k!}\right) \left(\sum_{l \geq 0} b_l \frac{x^l}{l!}\right)
 =: \varphi(x)\psi(x),
 \]
 as desired.
\end{proof}

Using Lemma \ref{l:rbPi} it is not difficult to see that the Hopf monoid of permutahedra $\rbPi$ is isomorphic to the Hopf monoid of set partitions $\rPi$. 
Theorem \ref{t:antipode} then gives us a combinatorial formula for the antipode of the Hopf monoid of set partitions $\rPi$.
We will carry out this computation in Section \ref{s:Pirevisited}, and explain why the Fock functor $\Kcb$ takes the Hopf monoid or permutahedra $\wbPi$ to the Hopf algebra  of symmetric functions $\Lambda$.

\section{$\rbbA$: {Associahedra and the composition of power series.}}\label{s:F}

In this section we consider the Hopf monoid $\rA$ of Loday associahedra, and  show that the character group of $\wbbA$ is the group of formal power series under multiplication.

\subsection{Loday's associahedron}\label{ss:Loday}

The \emph{associahedron} is ``a mythical polytope whose face structure represents the lattice of partial parenthesizations of a sequence of variables" \cite{Haiman}. Stasheff \cite{stasheff63:_homot_h} constructed it as an abstract cell complex in the context of homotopy theory and Milnor suggested that it could be realized as a polytope. There are now many different polytopal realizations due to Tamari, Stasheff, Haiman, Lee, and others; see \cite{ceballos14} for a survey. We will focus on the following construction due to Loday \cite{loday04:_assoc} and, in this formulation, to Postnikov \cite{postnikov09}.

\begin{definition}
Let $I$ be a finite set and $\ell$ be a linear order on $I$.
 \emph{Loday's associahedron} $\wa_\ell$ is the Minkowski sum
\[
\wa_\ell = \sum_{i \leq j} \Delta_{[i,j]_\ell}
\]
where $[i,j]_\ell = \{m \in I \, : \, i \leq m \leq j \textrm{ in } \ell\}$ is the interval from $i$ to $j$ for $i\leq j$ in $\ell$.
\end{definition}

We let $\wa_n$ denote the Loday associahedron for the natural order of $[n]$.
Since every linear order $\ell$ on a  finite set $|I|$ has an order-preserving bijection into $\{1, \ldots, |I|\}$, every Loday associahedron in $\Rb I$ is quasinormally equivalent to the \emph{standard} Loday associahedron $\wa_n$ for $n=|I|$.
In any case,  to make these objects into a set species, we need to consider $\wa_\ell$ for every linear order $\ell$.

We state the following theorem for completeness, but the connection between the associahedron and  parenthesizations will be irrelevant for now. We will return to this connection and its combinatorial consequences in Section \ref{s:Frevisited}. 

\begin{theorem}\label{t:loday} (\cite{loday04:_assoc, postnikov09})
Loday's associahedron $\wa_\ell$ is a simple polytope whose face poset is isomorphic to the poset of partial parenthesizations of a sequence of $n+1$ variables ordered by refinement. In particular, the number of vertices is the \emph{Catalan number} $C_n = \frac{1}{n+1} {2n \choose n}$.
\end{theorem}

A key property of Loday's associahedron is the following.

\begin{lemma}\label{lemma:assocfaces}
Let $I$ be a finite set and $\ell$ a linear order on $I$. Let $I
=S \sqcup T$ be a decomposition and let $T = T_1 \sqcup \cdots \sqcup T_k$ be the decomposition of $T$ into maximal subintervals of $\ell$. Then 
\[
\wa_\ell\,|_S \equiv \wa_{\ell\,|S}, \qquad \wa_\ell/_S = \wa_{\ell\,|T_1} \times \cdots \times \wa_{\ell\,|T_k}
\]
where $\equiv$ denotes normal equivalence and for each subset $U \subseteq I$, $\ell\,|U$ denotes the restriction of the linear order $\ell$ to $U$. 
\end{lemma}

\begin{proof}
Let us write $[i,j]$ for $[i,j]_\ell$ for simplicity.
The maximal face of a Minkowski sum $P+Q$ in direction $v$ is $(P+Q)_v = P_v + Q_v$. \cite{gruenbaum67:_convex} Therefore the $\1_S$-maximal face of $\wa_\ell$ is 
\[
(\wa_\ell)_{S,T} = 
(\wa_\ell)_{\1_S}
= \sum_{i \leq j} \left(\Delta_{[i,j]}\right)_{\1_S} = \sum_{i \leq j \, : \, [i,j] \cap S \neq \emptyset} \Delta_{[i,j] \cap S} +  \sum_{i \leq j \, : \, [i,j] \subseteq T} \Delta_{[i,j]},
\]
where the first summand lives in $\Rb S$ and the second lives in $\Rb T$, so they are $(\wa_\ell)|_S$ and $(\wa_\ell)/_S$, respectively. In $\Rb T$ we have
\[
(\wa_\ell)/_S =  \sum_{i \leq j \, : \, [i,j] \subseteq T} \Delta_{[i,j]}
= \sum_{t=1}^k \,  \sum_{i \leq j \, : \, [i,j] \subseteq T_l} \Delta_{[i,j]} 
= \sum_{t=1}^k  \wa_{\ell\,|T_t} 
= \wa_{\ell\,|T_1} \times \cdots \times \wa_{\ell\,|T_k}
\]
as desired. In $\Rb S$ we get
\[
(\wa_\ell)|_S = \sum_{i \leq j \, : \, [i,j] \cap S \neq \emptyset} \Delta_{[i,j] \cap S}.
\]
Now notice that $[i,j] \cap S$ is always a subinterval of $S$ with respect to the induced order $\ell\,|S$, and every such subinterval equals $[i,j] \cap S$ for some choice of $i \leq j$ in $\ell$. It follows that the Minkowski sum above involves the same summands as the Minkowski sum defining $\wa_{\ell\,|S}$ -- possibly with different coefficients. 

We now recall the fact that the normal fan $\N(P+Q)$ is the common refinement of $\N(P)$ and $\N(Q)$, while $\N(\lambda P) = \N(P)$ for any $\lambda > 0$ \cite{gruenbaum67:_convex}. Therefore the normal fan of a  Minkowski sum of scaled polytopes $\sum_i \lambda_i P_i$ does not depend on the scaling factors $\lambda_i$ as long as they are all positive. This implies that $(\wa_\ell)|_S \equiv \wa_{\ell\,|S}$ as desired.
\end{proof}

\begin{figure}[h]
\centering
\includegraphics[scale=.24]{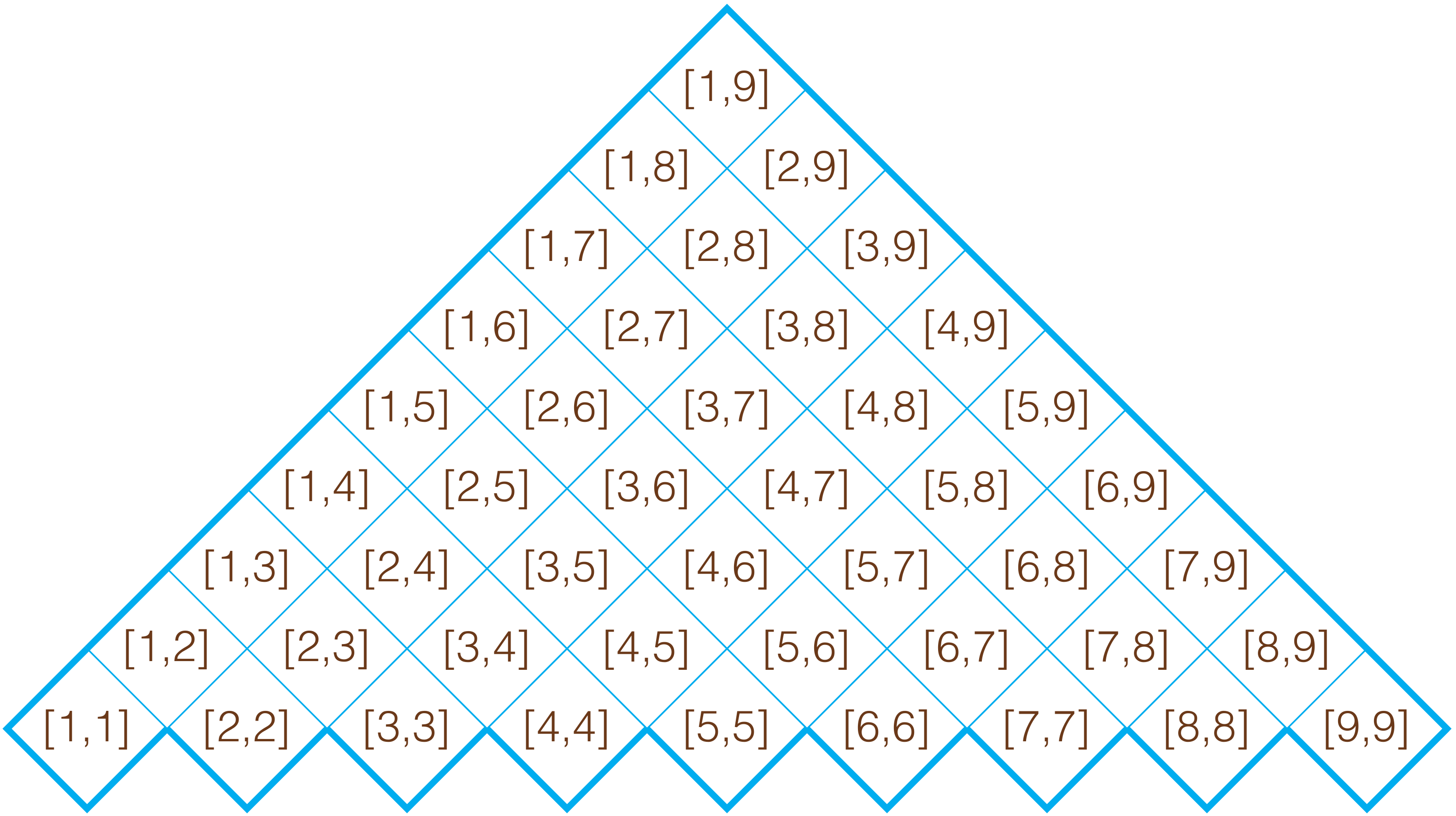}
\qquad  \includegraphics[scale=.24]{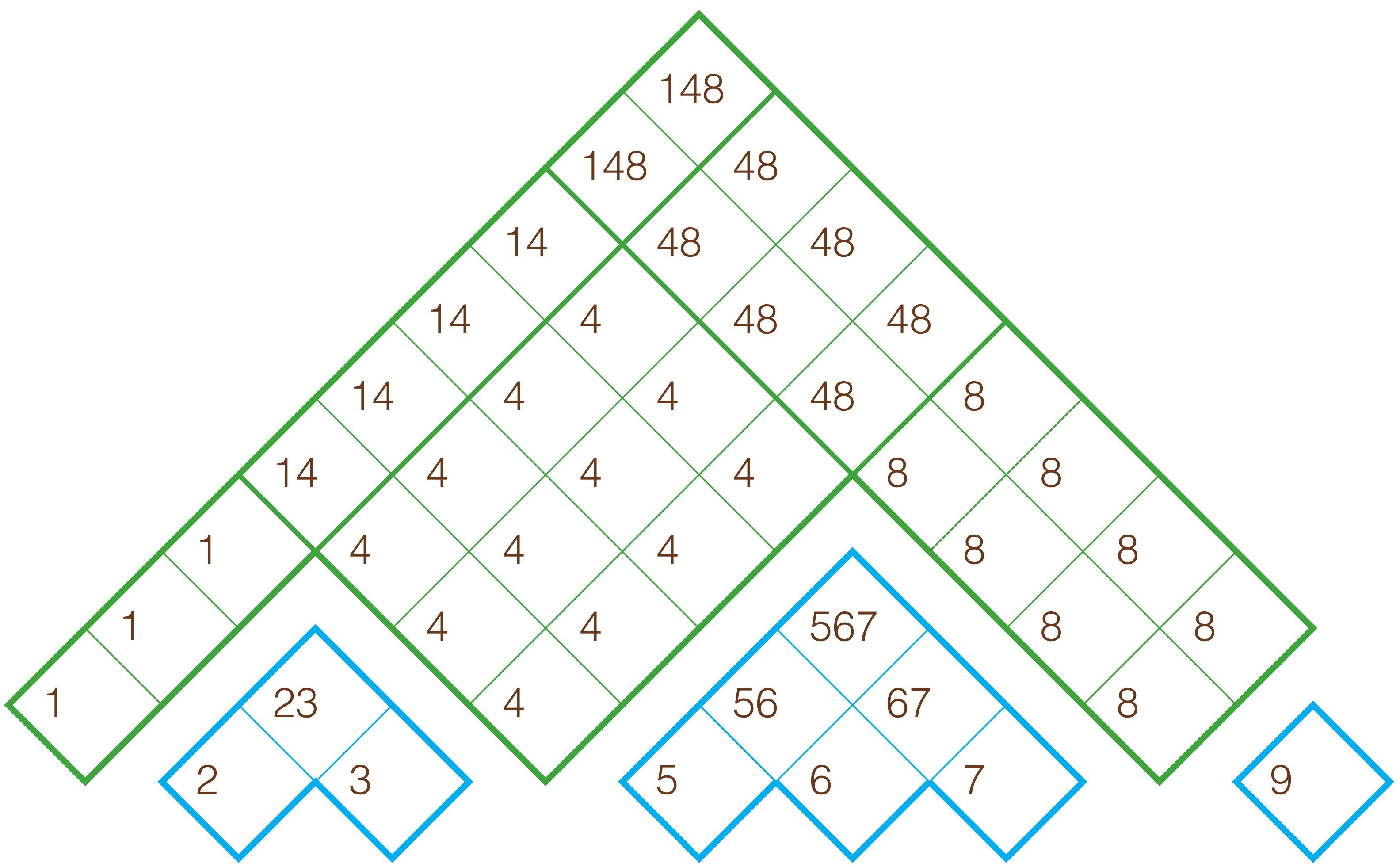}
\caption{The Minkowski sum decompositions of $\wa_9$ and $(\wa_9)_{148, 235679}$. \label{f:staircase}}
\end{figure} 

The above description of $(\wa_\ell)_{S,T} = (\wa_\ell)|_S \times (\wa_\ell)/_S$ has a nice pictorial description. It is natural to arrange the summands of  $\wa_n = \sum_{1 \leq i \leq j \leq n}\Delta_{[i,j]}$ into a staircase of size $n$, as shown in the left panel of Figure \ref{f:staircase} for $n=9$. To get the $\1_S$-maximal face $(\wa_n)_{S,T}$ we replace each summand $\Delta_{[i,j]}$ with $(\Delta_{[i,j]})_{\1_S}$. We can separate the resulting summands into a staircase above each one of the $T_i$s --  which give the associahedra $\wa_{\ell\,|T_1}, \ldots, \wa_{\ell\,|T_k}$ --  and a (fattened) staircase above $S$ which gives a polytope normally equivalent to $\wa_{\ell\,|S}$. This is illustrated in the right panel of Figure \ref{f:staircase} for the decomposition $[9] = \{1,4,8\} \sqcup \{2,3,5,6,7,9\}$.

\subsection{The Hopf monoid of associahedra and its character group}

Recall that $\rbGP$ and $\rbbGP$ are the Hopf monoids of generalized permutahedra modulo normal and quasinormal equivalence, as defined in Section \ref{s:rbbGP}.
Let $\rbA[I]$ and $\rbbA$ be the submonoids that the Loday associahedra generate in $\rbGP$  and $\rbbGP$, respectively. 
Lemma \ref{lemma:assocfaces} may be restated algebraically as follows.

\begin{corollary}\label{c:rbA}
The coproducts of $\rbA$ and $\rbbA$ are given by 
\[
\Delta_{S,T}(\wa_\ell) = (\wa_{\ell\,|S}\, ,\,  \wa_{\ell\,|T_1} \times \cdots \times \wa_{\ell\,|T_k}).
\]
for each linear order $\ell$ on $I$ and each decomposition $I=S \sqcup T$, where  $T = T_1 \sqcup \cdots \sqcup T_k$ is the decomposition of $T$ into maximal intervals of $\ell$. 
\end{corollary}

This implies, in particular, that 
\begin{equation}\label{e:Abasis}
\wbA[I] = \mathrm{span} \{\wa_{\ell_1} \times \cdots \times \wa_{\ell_k} \, : \, \ell_i \textrm{ is a linear order on } S_i \textrm{ for } I = S_1 \sqcup \cdots \sqcup S_k\}.
\end{equation}

We can now prove the main result of this section.

\begin{theorem}\label{t:charassoc} 
The group of characters $\Xb(\wbbA)$ of the Hopf monoid of associahedra is isomorphic to the group of ordinary formal power series
\[
\left\{x+a_1x^2 + a_2 x^3 + \cdots \, : \, a_1, a_2, \ldots \in \Kb\right\}
\]
under composition.
\end{theorem}

\begin{proof}
Recall from Section \ref{ss:Loday} that every Loday associahedron is quasinormally equivalent to one of the standard Loday associahedra $\wa_1, \wa_2, \wa_3, \ldots.$ Therefore, analogously to Theorem \ref{t:charperm}, a character $\zeta$ of $\wbbA$ is uniquely determined by the sequence $(1,z_1,z_2,\ldots)$ where 
$z_n=\zeta_{[n]}(\wa_n)$.
We encode that character in the formal power series $\zeta(t) =  t + z_1t^2 + z_2 t^3 + \cdots$. Conversely, any such formal power series gives a character of $\wbbA$.

Now suppose that two characters $\varphi$, $\psi$ and their convolution product $\varphi \psi$ give sequences $(1,a_1,a_2,\ldots)$, $(1,b_1,b_2,\ldots)$, and $(1, c_1, c_2, \ldots)$, respectively. By (\ref{e:conv}) and Corollary \ref{c:rbA}, 
\[
c_{n-1} = (\varphi \psi)_{[n-1]}(\wa_{n-1}) = 
\sum_{[n-1]=S\sqcup T} 
\varphi_S( \wa_S)\psi_{T_1}(\wa_{T_1}) \cdots \psi_{T_k}(\wa_{T_k}).
\]
where $T = T_1 \sqcup \cdots \sqcup T_k$ is the decomposition of $T$ into maximal subintervals of $[n-1]$, and $S, T_1, \ldots, T_k$ are listed in their standard linear order.

Each $(k-1)$-subset $S \subseteq [n-1]$ determines a ``gap sequence" $i_1, \ldots, i_k$ where $i_j=|T_j|$ is the number of elements of $[n-1]$ in the gap between the $(j-1)$th and the $j$th elements of $S$. These non-negative integers satisfy $i_1+\cdots+i_k+(k-1) = n-1$, and it is clear how to recover $S$ from them. Since $\wa_{n-1}|_S \equiv \wa_{k-1}$ and $\wa_{n-1}/_S \equiv \wa_{i_1} \times  \cdots \times \wa_{i_k}$ by Lemma \ref{lemma:assocfaces}, we may rewrite the above equation as
\[
c_{n-1} = 
\sum_{k=1}^{n} 
\sum_{i_1, \ldots, i_k \geq 0 \atop i_1+\cdots+i_k+(k-1)=n-1}
a_{k-1} b_{i_1} \cdots b_{i_k}
\]
which is equivalent to 
\[
\phi\psi(t): = \sum_{n \geq 1} c_{n-1}x^n =  
\sum_{k \geq 1} a_{k-1} \left(\sum_{i \geq 0} b_i x^{i+1}\right)^k =:
\phi(\psi(t)),
\]
as desired.
\end{proof}

A similar Hopf-theoretical result, without the connection to associahedra, is due to Doubilet, Rota, and Stanley. \cite{doubilet72}

In light of Corollary \ref{c:rbA}, 
Theorem \ref{t:antipode} gives us a combinatorial formula for the antipode of the Hopf monoid of paths $\rA$.
We will carry out this computation in Section \ref{s:Frevisited}, and explain why the Fock functor $\Kcb$ takes the Hopf monoid of associahedra $\wbbA$ to the Fa\'a di Bruno Hopf algebra $F$.








\section{{Inversion of formal power series and Loday's question}}\label{s:inversion}

In this section we will show how the formulas for multiplicative and compositional inverses of formal power series follow directly from the Hopf monoids $\rbPi$ and $\rbbA$ on permutahedra and associahedra, respectively.

\subsection{{Multiplicative Inversion Formulas}} 

As illustrated in the Introduction, the multiplicative inversion of power series is precisely given by the facial structure of permutahedra. We now explain this phenomenon.

\begin{theorem}\label{t:multinvpol} \emph{(Multiplicative Inversion, Polytopal Version)} 
The mutliplicative inverse of 
\[
A(x) = 1+a_1x + a_2\frac{x^2}{2!} + a_3\frac{x^3}{3!} + \cdots \quad \textrm{ is } \quad \frac1{A(x)} = B(x) = 1+b_1x + b_2\frac{x^2}{2!} + b_3\frac{x^3}{3!} + \cdots
\]
where
\[
b_n = \sum_{F \textrm{ face of } \pi_n} (-1)^{n - \dim F} a_F 
\]
and we write $a_F = a_{f_1}  \cdots  a_{f_k}$ for each face $F \cong \pi_{f_1} \times \cdots \times  \pi_{f_k}$ of the permutahedron $\pi_n$. 
\end{theorem}

\begin{proof}
Theorem \ref{t:charperm} allows us to identify the formal power series $A(x)=\sum a_nx^n/n!$  and $1/A(x) = B(x) = \sum b_nx^n/n!$ with the characters $\alpha$ and $\beta$ of the Hopf monoid $\wbPi$ determined uniquely by
\[
\alpha_{[n]}(\pi_n) = a_n,  \qquad 
\beta_{[n]}(\pi_n) = b_n,
\]
where $\pi_n$ is the standard permutahedron in $\Rb[n]$. 
By Theorem \ref{t:charperm}, since $B(x) = 1/A(x)$,
these characters are inverses of each other in the character group $\Xb(\rbPi)$.

Recall that the inverse in  the character group of any Hopf monoid is given by $\beta = \alpha \circ \apode$ where $\apode$ is the antipode. For $\rbPi$, this antipode is given by Theorem \ref{t:antipode}. Therefore
\[
b_n = \beta_{[n]}(\pi_n) =  (\alpha \circ \apode)_{[n]}(\pi_n) =  \alpha_{[n]}\left( \sum_{F \textrm{ face of } \pi_n} (-1)^{n - \dim F} F\right)  =  \sum_{F \textrm{ face of } \pi_n} (-1)^{n - \dim F} a_F,
\]
using the multiplicativity of the character $\alpha$.
\end{proof}

\begin{theorem}\label{t:multinvenum} \emph{(Multiplicative Inversion, Enumerative Version)} 
The multiplicative inverse of 
\[
A(x) = 1+a_1x + a_2\frac{x^2}{2!} + a_3\frac{x^3}{3!} + \cdots \quad \textrm{ is } \quad \frac1{A(x)} = B(x) = 1+b_1x + b_2\frac{x^2}{2!} + b_3\frac{x^3}{3!} + \cdots,
\]
where
\[
b_n =  \sum_{\langle 1^{m_1}2^{m_2} \cdots \rangle \vdash n} (-1)^{|m|} 
{n \choose \underbrace{1, 1, \ldots}_{m_1}, \underbrace{2, 2, \ldots}_{m_2}, \ldots}
{|m| \choose  m_1, m_2, \ldots} \,
a_1^{m_1}a_2^{m_2} \cdots
\]
summing over all partitions $\langle 1^{m_1}2^{m_2} \cdots \rangle =  \underbrace{1 1 \ldots}_{m_1} \underbrace{2 2 \ldots}_{m_2}\cdots$ of $n$, 
where $|m| = m_1+m_2+\cdots$. 
\end{theorem}

\begin{proof}
Recall from Section \ref{ss:p} that the faces of $\pi_n$ are in bijection with the compositions $(S_1, \ldots, S_k)$ of $[n]$, where the face $F=\pi_{S_1, \ldots, S_k} \cong
\pi_{S_1} \times \cdots \times  \pi_{S_k}$ corresponds to the composition $(S_1, \ldots, S_k)$. If we let 
$m_i$ be the number of $S_j$s of size $i$, then $n - \dim F = k = |m|$ and $a_F = a_1^{m_1}a_2^{m_2} \cdots$. Therefore the coefficient of this monomial is the number of compositions leading to block sizes $\langle 1^{m_1}2^{m_2} \cdots \rangle$. There are ${m_1 + m_2 + \cdots \choose m_1, m_2, \ldots}$ ways of assigning these sizes to the parts $S_1, \ldots, S_k$ in some order. Having fixed that order, there are then ${n \choose 1, 1, \ldots, 2, 2, \ldots}$ ways of partitioning the elements of $I$ into parts $S_1, \ldots, S_k$ of those respective sizes. The desired result follows.
\end{proof}

\subsection{{Compositional inversion formulas}}\label{s:compinv}
Just as the facial structure of permutahedra tells us exactly how to compute the multiplicative inverse of a formal power series, the facial structure of associahedra tell us how to compute the compositional inverse.

\begin{theorem}\label{t:Lagrinvpol} \emph{(Lagrange Inversion, polytopal version)}
The compositional inverse of  
\[
C(x) = x+c_1x^2 + c_2x^3 + \cdots \qquad \textrm{ is } \qquad C^{\langle -1 \rangle}(x) = D(x) = x+d_1x^2 + d_2x^3 + \cdots,
\]
where
\[
d_n = \sum_{F \textrm{ face of } \wa_n} (-1)^{n - \dim F} c_F 
\]
and we write $c_F = c_{f_1}  \cdots  c_{f_k}$ for each face $F \cong \wa_{f_1} \times \cdots \times  \wa_{f_k}$ of the associahedron $\wa_n$. 
\end{theorem}

\begin{proof}
We proceed exactly as in the proof of Theorem \ref{t:multinvpol}. We identify the formal power series $C(x)=\sum c_{n-1}x^n$  and $C^{\langle -1 \rangle}(x) = D(x) = \sum d_{n-1}x^n$ with the characters $\gamma$ and $\delta$ of the Hopf monoid $\rbbA$ determined uniquely by
\[
\gamma_{[n]}(\wa_n) = c_n,  \qquad 
\delta_{[n]}(\wa_n) = d_n
\]
for the standard Loday associahedron $\wa_n$.
By Theorem \ref{t:charassoc}, since $C^{\langle -1 \rangle}(x) = D(x)$, these characters are inverses in the character group $\Xb(\rbbA)$. Therefore $\delta = \gamma \circ \apode$, and the result now follows from the antipode formula of Theorem \ref{t:antipode}. \end{proof}

\begin{theorem}\label{t:Lagrinvenum} \emph{(Lagrange Inversion, enumerative version)}
The compositional inverse of  
\[
C(x) = x+c_1x^2 + c_2x^3 + \cdots \qquad \textrm{ is } \qquad C^{\langle -1 \rangle}(x) = D(x) = x+d_1x^2 + d_2x^3 + \cdots,
\]
where
\[
d_n = \sum_{\langle 1^{m_1}2^{m_2} \cdots \rangle \vdash n} \, 
 (-1)^{|m|}\frac{(n+|m|)!} {(n+1)! \, m_1! \, m_2! \cdots} c_1^{m_1} c_2^{m_2} \cdots
\]
summing over all partitions $\langle 1^{m_1}2^{m_2} \cdots \rangle$ of $n$, 
where $|m| = m_1+m_2+\cdots$. 
\end{theorem}

\begin{proof}
This follows from Theorem \ref{t:Lagrinvpol} and the known correspondence between faces of  associahedra and trees, which we reprove in a more general setting in Section \ref{s:W}. More precisely, the $(n-|m|)$-dimensional faces of the associahedron $\wa_n$ of type $\wa_{1}^{m_1} \times \wa_2^{m_2}\times \cdots$ are in bijection with the plane rooted trees that have $n+1$ leaves and $m_i$ vertices of down-degree $i$  for each $i \geq 1$. The result then follows from the fact \cite[Theorem 5.3.10]{stanley99:_ec2}
that there are 
$(n+|m|)!/((n+1)! \, m_1! \, m_2! \cdots)$
such plane rooted trees. 
\end{proof}

\subsection{{Loday's question}}

It has long been known that Lagrange inversion is closely related to the enumeration of trees (or, equivalently, parenthesizings). In turn, this enumeration is related to the associahedron; see for example \cite{ardila2015algebraic,stanley99:_ec2}. However, in 2005, Loday \cite{loday05} asked for a direct explanation of the connection between Lagrange inversion and the associahedra:

\begin{quote}
``There exists a short operadic proof of the [Lagrange inversion] formula which explicitly involves the parenthesizings, but it would be interesting to find one which involves the topological structure of the associahedron."
\end{quote}

The associahedral statement and proof of the Lagrange inversion formula in Theorem \ref{t:Lagrinvpol} may be regarded as an answer to Loday's question. It is a combinatorics-free approach. Aside from the basic Hopf monoid architecture, it relies only on two key ingredients:

$\bullet$ our topological proof for the antipode of the associahedron (Theorem \ref{t:antipode})

$\bullet$ the structure of Loday's associahedron with respect to the $\1_S$ directions (Lemma \ref{lemma:assocfaces})

Interestingly, there are many other realizations of the associahedron as a generalized permutahedron \cite{ceballos2015many, ceballos2012realizing, hohlweg2007realizations, hohlweg2011permutahedra, 
hohlweg2017polytopal,
lange2013associahedra, pilaud2016permutrees, pilaud2012brick}.  
These have isomorphic face posets, but they lead to different Hopf structures and different character groups. Surprisingly, to answer Loday's question within this algebro-polytopal context, Loday's realization of the associahedron is precisely the one that we need!

Relatedly, in the closing remarks to his 1987 paper \cite{schmitt1987antipodes}, Schmitt wrote about the cancellation of $1$s and $-1$s that leads to his Hopf algebraic proof of the Lagrange inversion formula:

\begin{quote}
``We believe that an understanding of exactly how these cancellations take place will not only provide a direct combinatorial proof of the Lagrange inversion formula, but may well yield analogous formulas for the antipodes of [...] other [...] Hopf algebras."
\end{quote}

Schmitt's suggestion is very close to the philosophy of this project, though our approach is more geometric and topological than combinatorial.  Applying the same point of view to other families of polytopes, we will obtain optimal formulas for the antipodes of many Hopf monoids  throughout the paper.

\newpage

\begin{LARGE}
\noindent \textsf{PART 3: Characters, polynomial invariants, and reciprocity}
\end{LARGE}

\section{$\rSF$: {Submodular functions: an equivalent formulation of $\rGP$}}\label{s:SF}

Generalized permutahedra arise in a multitude of settings, and can be used to model many combinatorial objects: graphs, matroids, posets, set partitions, paths,  and many others. In this section we present one reason for the ubiquity of these polyhedra: generalized permutahedra are equivalent to submodular functions, which are central objects in optimization. These functions occur in numerous mathematical and real-world contexts, since they are characterized by a \emph{diminishing returns} property that is natural in many settings.

%
%
%

\subsection{{Boolean functions}}\label{ss:bool}
Let $2^I$ denote the collection of subsets of a finite set $I$. 
A \emph{Boolean function} on $I$ is an arbitrary function 
$z:2^I\to\Rb$
such that 
$z(\emptyset)=0$.

Let $\rBF[I]$ denote the set of Boolean functions on $I$.  
To turn the species $\rBF$ into a connected Hopf monoid, we first notice that $\rBF[\emptyset]$ is indeed a singleton.
Now fix a decomposition $I=S\sqcup T$. We make the following definitions.

\noindent $\bullet$ The product of two Boolean functions $u\in\rBF[S]$ and $v\in\rBF[T]$ is the function $u\cdot v\in\rBF[I]$ given by
\begin{equation}\label{e:SFproduct}
(u\cdot v)(E):= u(E\cap S)+v(E\cap T) \text{ for $E\subseteq I$.}
\end{equation}

\noindent 
$\bullet$ The coproduct of a Boolean function $z\in\rBF[I]$ is $(z|_S, z/_S) \in \rBF[S] \times \rBF[T]$, where
\begin{equation}\label{e:SFcoproduct}
z|_S(E):=z(E) \text{ for $E\subseteq S$}
\qand
z/_S(E):=z(E\cup S)-z(S) \text{ for $E\subseteq T$.}
\end{equation}

The Hopf monoid axioms of Definition~\ref{d:hopfset} are easily verified. To illustrate this, we check the compatibility between products and coproducts. Consider two compositions $I = S \sqcup T$ and $I = S' \sqcup T'$ as described in~\eqref{e:4sets} and illustrated below, and  choose $u\in\rBF[S]$, $v\in\rBF[T]$.

\[
\begin{picture}(170,90)(0,0)
\put(80,40){\oval(160,80)}
  \put(0,40){\dashbox{2}(38,0){}}
  \put(54,40){\dashbox{2}(58,0){}}
   \put(126,40){\dashbox{2}(34,0){}}
  \put(80,0){\dashbox{2}(0,80){}}
  \put(40,55){$A$}
  \put(115,55){$B$}
  \put(40,15){$C$}
  \put(115,15){$D$}
  \put(45,40){\oval(30,15)}
  \put(120,40){\oval(50,24)}
  \put(40,35){$E$}
  \put(115,35){$F$}
  \put(165,50){$S = A \sqcup B$}
  \put(165,20){$T = C \sqcup D$}
  \put(10,85){$S' = A \sqcup C$}
  \put(90,85){$T' = B \sqcup D$}
      \end{picture}
\]
For any $E \subseteq S'$ we have
\begin{align*}
(u\cdot v)|_{S'}(E) & =(u\cdot v)(E) = u(E \cap S)+v(E\cap T) =u(E \cap A)+v(E\cap C) \\
&=u|_A(E \cap A)+v|_C(E\cap C) 
=\bigl(u|_A \cdot v|_C\bigr)(E), 
\end{align*}
and for any $F\subseteq T'$ we have
\begin{align*}
(u\cdot v)/_{S'}(F) & =(u\cdot v)(F\cup S')-(u\cdot v)(S')\\
&=u\bigl((F\cup S')\cap S\bigr)+v\bigl((F\cup S')\cap T\bigr)
-u(S'\cap S)-v(S'\cap T)\\
&=u\bigl((F\cap B)\cup A\bigr)-u(A)+v\bigl((F\cap D)\cup C\bigr)-v(C)\\
&=(u/_A)(F\cap B)+(v/_C)(F\cap D)=\bigl((u/_A)\cdot(v/_C)\bigr)(F).
\end{align*}
Thus $(u\cdot v)|_{S'}=(u|_A)\cdot(v|_C)$
and
$(u\cdot v)/_{S'}=(u/_A)\cdot(v/_C)$, as needed.

\subsection{{Submodular functions and diminishing returns}}\label{ss:submod}

A Boolean function $z$ on $I$ is \emph{submodular} if
\begin{equation}\label{e:submod}
z(A\cup B)+z(A\cap B)\leq z(A)+z(B)
\end{equation}
for every $A,B\subseteq I$. Submodular functions arise in many contexts in mathematics and applications, partly because submodularity is equivalent to a natural \emph{diminishing returns property} that we now describe.

Suppose the Boolean function $z$ measures some quantifiable \emph{benefit} $z(A)$ associated to each subset $A \subseteq I$. Then the contraction $z/_S$ has a natural interpretation: for $e \notin S$, 
\[
z/_S(e) = z(S \cup e) - z(S) = \textrm{\emph{marginal return} of adding $e$ to $S$.}
\]

\begin{theorem} \label{t:diminishing}\cite[Theorem 44.1]{schrijver03} \emph{(Diminishing returns)} A Boolean function $z$ on $I$ is submodular if and only if for every $e \in I$ we have

\begin{equation}\label{e:diminish}
z/_S(e) \geq z/_T(e) \qquad \textrm{ for } S \subseteq T \subseteq I-e
\qquad \textit{(diminishing returns)}
\end{equation}
that is,  the marginal return $z/_S(e)$ decreases as we add more elements to $S$.
\end{theorem}

From the algebraic point of view, submodular functions have a Hopf monoid structure because they are closed under products and coproducts.

\begin{theorem}\label{t:submod} 
Let $\rSF[I]$ denote the set of submodular functions on $I$. Then $\rSF$ is a Hopf submonoid of $\rBF$, with the product and coproduct given by (\ref{e:SFproduct}) and (\ref{e:SFcoproduct}).
\end{theorem}

\begin{proof}
It suffices to show that submodular functions are closed under the product (\ref{e:SFproduct}) and coproduct (\ref{e:SFcoproduct}) of Boolean functions as defined above. This is well known \cite{oxley92:_matroid} and follows from
Theorem \ref{t:diminishing}; the details are left to the reader.
%
%
\end{proof}

%
%
%
%
%

\subsection{{Submodular functions and generalized permutahedra}}\label{ss:submod-gp}

The \emph{base polytope}  of a given Boolean function $z:2^I\to\Rb$ 
is the set\footnote{It is worth remarking that, in Postnikov's work on generalized permutahedra \cite{postnikov09}, he writes the defining inequalities as $\sum_{i\in A}x_i\geq z'(A)$. The difference is unimportant thanks to the equality $\sum_{i\in I}x_i=z(I)$. Our convention affords a cleaner connection between generalized permutahedra and submodular functions.}
\begin{equation}\label{e:submod-gp}
\Pc(z):=\{x\in\Rb I \mid  \sum_{i\in I}x_i=z(I) \, \text{ and } \sum_{i\in A}x_i\leq z(A) \text{ for all }A\subseteq I
 \}.
\end{equation}

For $x \in \Rb I$ and $A \subseteq I$, we denote
\[
x(A) = \sum_{i \in A} x_i.
\]
We say the inequality $x(A) \leq z(A)$ is \emph{optimal} for $\Pc(z)$ if $z(A)$ is the minimum value for which this inequality holds; that is, if $z(A)$ equals the maximum value of $x(A)$ over all $x$ in the polytope $\Pc(z)$.


The following theorem collects several results from the literature, and plays a central role in this paper. 


\begin{theorem}\cite{derksenfink10, fujishige05, postnikov09, schrijver03}\label{t:submod-gp} For a polytope $\wp$ in $\Rb I$, the following conditions are equivalent.
\begin{enumerate}
\item
The polytope $\wp$ is a generalized permutahedron.
\item
The normal fan $\Nc_\wp$ is a coarsening of the braid arrangement $\Bc_I$.
\item
Every edge of $\wp$ is parallel to the vector $e_i - e_j$ for some $i, j \in I$.
\item
There exists a submodular function $z:2^I\to\Rb$ such that $\wp=\Pc(z)$. 
\end{enumerate}
Furthermore, when these conditions hold, the submodular function $z$ of part 4 is unique, and every definining inequality in (\ref{e:submod-gp}) is optimal. 
\end{theorem}

We will extend this result to possibly unbounded objects in Theorem \ref{t:submod-gp+}, and provide references and a complete proof there. We are now ready to prove an important result about the Hopf monoid $\rGP$.

\begin{theorem}\label{t:submod-gp} The collection of maps
\[
\rSF[I]\to\rGP[I], \quad z\mapsto \Pc(z)
\]
is an isomorphism of Hopf monoids in set species $\rSF\cong\rGP$.
\end{theorem}
\begin{proof}
Theorem \ref{t:submod-gp} shows that each one of those maps is bijective. It is not difficult to check that the products on $\rSF$ and $\rGP$ agree. To prove that the coproducts agree, we now check that restriction and contraction coincide in $\rSF$ and $\rGP$. 

Let $\wp = \Pc(z)$ be a generalized permutahedron in $\Rb I$ and let $I = S \sqcup T$ be a decomposition. We need to show that the maximal face in direction $\1_S$ is $\wp_{S,T} = \Pc(z|_S)  \times \Pc(z/_S)$. We prove the two inclusions.

\medskip

$\mathbf{\supseteq}$: First consider any point $x = (x_S, x_T)  \in \Pc(z|_S)  \times \Pc(z/_S)$. For any $A \subseteq I$ let $A_S = A \cap S$ and $A_T = A \cap T$, so that $A = A_S \sqcup A_T$. Then 
\begin{align*}
x(A) &= x_S(A_S) + x_T(A_T) \leq  z|_S(A_S) + z/_S(A_T)  \\
& = z(A_S) + z(A_T \cup S) - z(S) = z(A \cap S) + z(A \cup S) - z(S) \leq z(A)
\end{align*}
by submodularity. In particular, for $A=I$ we get
\[
x(I) = x_S(S) + x_T(T) = z|_S(S) + z/_S(T) = z(S) + z(T \cup S) - z(S) = z(I).
\]
Therefore $x \in \wp$. On the other hand, for $A=S$ we get
\[
x(S) = x_S(S) + x_T(\emptyset) = z|_S(S) + 0 = z(S)
\]
which, in view of (\ref{e:submod-gp}), implies that $x$ is $\1_S$-maximal in $\wp$, that is, $x \in \wp_{S,T}$.

\medskip

$\mathbf{\subseteq}$:
In the other direction, let $x \in \wp_{S,T}$. By Theorem \ref{t:submod-gp}, $x$ attains the $\1_S$-optimal value $x(S) = z(S)$. Letting $x = (x_S, x_T) \in \Rb^S \times \Rb^T$, we then have 
\begin{align*}
x_S(S) &= x(S) = z(S) = z|_S(S), \\
x_T(T) &= x(T) = x(I) - x(S) = z(I) - z(S) = z/_S(T).
\end{align*}
Furthermore, for any $A \subseteq S$ and $B \subseteq T$,
\begin{align*}
x_S(A) &= x(A) \leq z(A) = z|_S(A), \\
x_T(B) &= x(B) = x(B \cup S) - x(S) \leq z(B \cup S) - z(S) = z/_S(B).
\end{align*}
These observations imply that $x_S \in \Pc(z|_S)$ and $x_T \in  \Pc(z/_S)$ as desired.
\end{proof}

\subsection{$\wGP_+$: Extended generalized permutahedra and
extended submodular functions}\label{ss:infinity}

We now extend the previous constructions to allow for unbounded polyhedra. Most of the results of this section were obtained earlier by Fujishige \cite{fujishige05}. 

Let an \emph{extended} Boolean function be a function $z:2^I\to\Rb \cup \{\infty\}$ with $z(\emptyset) = 0$ and $z(I) \neq \infty$. We say $z$ is \emph{submodular} if 
\[
z(A \cup B) + z(A \cap B) \leq z(A) + z(B) \textrm{ whenever $z(A), z(B)$ are finite.}
\]
Extended submodular functions are also called \emph{submodular systems}. \cite{fujishige05}
The \emph{base polyhedron} of $z$ is
\begin{equation}\label{e:submod-gp+}
\Pc(z):=\{x\in\Rb I \mid  \sum_{i\in I}x_i=z(I) \, \text{ and } \sum_{i\in A}x_i\leq z(A) \text{ for all }A\subseteq I \textrm{ with $z(A)<\infty$}
 \}.
\end{equation}

Theorem  \ref{t:submod-gp} extends to this setting, providing a bijective correspondence between extended submodular functions and extended generalized permutahedra. We now survey this correspondence in Theorem \ref{t:submod-gp+}, providing proofs for some statements which we were not able to find in the literature.


Define a \emph{braid cone} to be a cone in $(\Rb I)^* = \Rb^I$ cut out by inequalities of the form $y(i) \geq y(j)$ for $i, j \in I$. Define a \emph{root subspace} of $\Rb I$ to be a subspace spanned by  vectors of the form $e_i - e_j$ for $i,j \in I$; these vectors are the \emph{roots} of the root system $A_I = \{e_i - e_j \, | \, i, j \in I\}$ in the sense of Lie theory. \cite{humphreys90:_reflec_coxet} Define an \emph{affine root subspace} of $\Rb I$ to be a translate of a root subspace.

\begin{theorem}\cite{fujishige05, postnikov09, schrijver03}\label{t:submod-gp+} For a polyhedron $\wp$ in $\Rb I$, the following are equivalent.
\begin{enumerate}
\item
The polyhedron $\wp$ is an extended generalized permutahedron.
\item
The normal fan $\Nc_\wp$ is a coarsening of $(\Bc_I)|_C$, the restriction of the braid arrangement $\Bc_I$ to some braid cone $C$.
\item
The affine span of every face of $\wp$ is an affine root subspace.
\item
There exists an extended submodular function $z:2^I\to\Rb \cup \{\infty\}$ such that $\wp=\Pc(z)$. 
\end{enumerate}
Furthermore, when these conditions hold, the extended submodular function $z$ of part 4 is unique, and every definining inequality in (\ref{e:submod-gp+}) is optimal. 
\end{theorem}

\begin{proof} We proceed in several steps.

$\mathbf{1 \Leftrightarrow 2}$: This is Definition \ref{def:GP}.

$\mathbf{3 \Leftrightarrow 4}$: This  is anticipated by Fujishige in \cite[Thms. 3.15, 3.18, 3.22]{fujishige05} and proved explicitly by Derksen and Fink in \cite[Prop. 2.9]{derksenfink10} for \emph{megamatroids}, where the function $z$ is integral; their proof works for general $z$. In condition 3 they include the additional hypothesis that the polyhedron $\wp$ lies on a hyperplane of the form $\sum_{i \in I} x_i = r$ for some $r \in \Rb$, but this follows from the assumption that the affine span of $\wp$ is an affine root subspace.

$\mathbf{2 \Rightarrow 3}$: Assume $\wp$ satisfies $2$. Since $\1 \in \Rb^I$ is in every braid cone, it is also in $\Nc_\wp(\wp)$, so $\1(x) = \sum_{i \in I} x_i$ is constant on $\wp$. 

Now let $\wq$ be any $d$-dimensional face of $\wp$ and write $\textrm{aff}(\wq) = v + W$ for a vector $v$ and a subspace $W$. We need to show that $W$ is a root subspace.
The normal face $\Nc_\wp(\wq)$ contains a face $F$ of the braid arrangement $\Bc_I$ of its same dimension, so $\espan (\Nc_\wp(\wq)) = \espan (F)$ is the intersection of $d$ independent hyperplanes $y(i_k) = y(j_k)$ for $1 \leq k \leq d$. 
We claim that $W = \espan\{e_{i_k} - e_{j_k} \, : \, 1 \leq k \leq d\}$. Since  both of these vector spaces are $d$-dimensional, it suffices to show that $e_{i_k}-e_{j_k} \in W$ for each $k$. 

We have the following inequality description of $\Nc_\wp(\wq)$:
\[
\Nc_\wp(\wq) = \{y \in \Rb^I \, : \, 
y(q_1) = y(q_2) \,\,\, \mathrm{ for } \,\,\, q_1, q_2 \in \wq, \,\, \,
y(q) \geq y(p)\,\,\,  \mathrm{ for }\,\,\,  q \in \wq, p \in \wp\}
\]
Since $y \in \Nc_\wp(\wq)$ implies that $y(i_k)=y(j_k)$,  $e_{i_k} - e_{j_k}$ must be a linear combination of vectors of the form $q_1 - q_2$ for $q_1, q_2 \in \wq$. But every such vector is in $W$, so $e_{i_k}-e_{j_k} \in W$ as desired.

$\mathbf{(3+4) \Rightarrow 2}$: Let $\wp$ satisfy 3 and 4. 

First we  show that the support of the normal fan $\Nc_\wp$
\[
C =  \supp(\Nc_\wp) = \{ y \in \Rb^I \, : \, \max_{p \in \wp} \, y(p) \textrm{ is finite } \},
\]
is a braid cone. 
Let $D = \Nc_\wp(\wq)$ be a codimension 1 face of $\Nc_\wp$ on the boundary of $\Nc_\wp$. Say $\wq$ is $d$-dimensional, and, in light of 3, let the affine span of $\wq$ be a translate of the subspace $W=\espan\{e_{i_1} - e_{j_1}, \ldots, e_{i_d}-e_{j_d}\}$. We claim that $\espan(D)$ is the intersection of the hyperplanes $y(i_k) = y(j_k)$ for $1 \leq k \leq d$. Since both subspaces have codimension $d$, it is enough to prove one inclusion. To do that, observe that if $y \in D$, then $y(q)$ is constant for $q \in \wq$, so $y(w)=0$ for $w \in W$ and therefore $y(i_k) = y(j_k)$. The same statement is then true for any $y \in \espan(D)$.
We conclude that $C$ can be described by inequalities of the form $y(i) \geq y(j)$, as desired. 

Now that we know that $\Nc_\wp$ is supported on a braid cone $C$, we need to show that it is refined by the braid arrangement; that is, that for $y \in C$, the relative order of the coordinates of $y \in \Rb^I$ is enough to determine the maximum face $\wp_y$. But condition 4 tells us that $\wp = \Pc(z)$ for an extended submodular function $z$, and Fujishige showed that this family of functions may be optimized using the greedy algorithm, which only pays attention to the relative order of the coordinates of $y$. \cite[Thms. 3.15, 3.18]{fujishige05} The result follows.

Having proved the equivalence of 1, 2, 3, and 4, it remains to remark that the uniqueness and optimality of the defining equations (\ref{e:submod-gp+}) of $\wp$ are implicit in \cite[Section 3]{fujishige05}.
\end{proof}

\begin{remark}
When $\wp$ is bounded, Theorem \ref{t:submod-gp+} reduces to  Theorem \ref{t:submod-gp}. Condition 3 looks different in these two statements, but in this setting, the seemingly weaker condition that every edge is parallel to a root $e_i-e_j$ implies that every face spans an affine root subspace. The reason for this is that in a bounded polytope, every face is spanned by its edges. This is not true in general; some unbounded polytopes do not even have one-dimensional faces.
\end{remark}

Let $\rSF_+[I]$ be the set of extended submodular functions on $I$. To construct a connected Hopf monoid, we use essentially the same operations as in $\rBF$ and $\rSF$. The only difference is that the contraction $z/_S$ of $z \in \rSF_+[I]$ is no longer defined when $z(S) = \infty$. Therefore, we need to modify the coproduct by defining
\[
\Delta_{S,T}(z) =  \begin{cases}
z|_S\otimes z/_S & \text{if $z(S) \neq \infty$}\\
0 & \text{if $z(S) = \infty$.}
\end{cases}
\]
for a decomposition $I = S \sqcup T$. This definition forces us to work in the context of vector species. It is now straightforward to extend Theorem \ref{t:submod-gp} to this context.

\begin{theorem}\label{t:submod-gp+} The collection of maps
\[
\wSF_+[I]\to\wGP_+[I], \quad z\mapsto \Pc(z)
\]
is an isomorphism of Hopf monoids in vector species $\wSF_+ \cong\wGP_+$.
\end{theorem}

%
%

%
%
%
%
%
%
%
%

\section{$\rG$: {Graphs, graphic zonotopes, and Humpert-Martin's conjecture}}\label{s:G} 

In this section we revisit the Hopf monoid of graphs of Section \ref{ss:graphs}, now taking a geometric perspective: we realize $\rG$ as a submonoid of $\rGP$. The key idea is that every graph $g$ is modeled by a generalized permutahedra $Z_g$ called its graphic zonotope, and this model respects the Hopf structure of graphs.
This geometric interpretation of the Hopf monoid $\wG$ readily gives us the optimal formula for its antipode -- obtained independently by Humpert and Martin \cite{humpert12} -- and allows us to prove their conjecture from \cite[Section 5]{humpert12}.

\subsection{{Graphic zonotopes}}\label{s:zono}
 
 Let $g$ be a graph with vertex set $I$. Given $A\subseteq I$ and an edge $e$ of $g$, we say that $e$ is \emph{incident} to $A$ if either endpoint of $e$ belongs to $A$. Consider the \emph{incidence} function
\begin{eqnarray*}
\inc_g & : & 2^I\to\Zb \\ 
 \inc_g(A)& = & \text{ number of edges and half-edges of $g$ incident to $A$}.
\end{eqnarray*}
 For example, the incidence function of the graph \,  \includegraphics[scale=0.5]{Figures/coproductgraphB.pdf} \,
 then
 \[
 \inc_g(\emptyset)=0, \quad
 \inc_g(\{x\})=3, \quad
 \inc_g(\{y\})=2, \qand
 \inc_g(\{x,y\})=3.
 \]
 
 The following result is well-known.
 
\begin{proposition}\label{p:graph-submod}
For any graph $g$, the incidence function $\inc_g$ is submodular.
\end{proposition}
\begin{proof}
By Theorem \ref{t:diminishing} it suffices to observe that the marginal benefit of adding $e$ to $S$:
\[
(\inc_g)/_S(e) = \# \textrm{ of edges of $g$ incident to $e$ and not to $S$}
\]
diminishes as we add elements to $S$. 
\end{proof}
%


By Theorem \ref{t:submod-gp} and (\ref{e:submod-gp}), the submodular function $\inc_g$ gives rise to a generalized permutahedron $\Pc(\inc(g)) = Z_g$ which is called the \emph{graphic zonotope} of $g$.

%

\begin{example}\label{eg:graph-zono} 
Revisiting Example \ref{eg:antipode-graph}, if 
$g$ is the graph \, \includegraphics[scale=0.5]{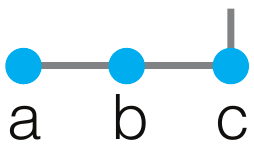} \, 
then the graphic zonotope $Z_g = \Pc(\inc_g)$ 
is given by 
\[
x_a+x_b+x_c=3, \,\,\,\,
x_a+x_b\leq 2,\,\,\,\,
x_b+x_c\leq 3,\,\,\,\,
x_a+x_c\leq 3, \,\,\,\,
x_a\leq 1, \,\,\,\,
x_b\leq 2,\,\,\,\,
x_c\leq 2.
\]
and is shown below. Note that the third and fifth inequalities are optimal but redundant. 
\[
\begin{picture}(50,140)(0,-20)
{\color{black}
 \put(0,30){\circle*{3}} \put(50,0){\circle*{3}}  \put(0,90){\circle*{3}} \put(50,60){\circle*{3}}
 \thicklines
  \put(0,30){\line(0,1){60}} \put(50,0){\line(0,1){60}} 
  \put(50,0){\line(-5,3){50}} \put(50,60){\line(-5,3){50}}
  \put(0,30){\line(0,-1){20}} \put(0,90){\line(0,1){20}}
  \put(50,0){\line(0,-1){20}} \put(50,60){\line(0,1){20}}
  \put(50,0){\line(5,-3){20}} \put(0,30){\line(-5,3){20}}
  \put(50,60){\line(5,-3){20}} \put(0,90){\line(-5,3){20}}
  \put(50,0){\line(5,3){20}}  \put(50,0){\line(-5,-3){20}}
  \put(0,90){\line(5,3){20}}  \put(0,90){\line(-5,-3){20}}
   }
    \put(-20,1){$x_a=1$} \put(38,84){$x_b+x_c=3$}
    \put(0,-24){$x_b=2$}  \put(20,105){$x_a+x_c=3$}
      \put(-75,46){$x_a+x_b=2$}   \put(70,40){$x_c=2$}
\end{picture}
\]
\end{example}

There is a useful alternative description of the zonotope of a graph. 

\begin{proposition}\label{prop:Zgzonotope} \cite[Proposition 6.3]{postnikov09}
The zonotope $Z_g \subseteq \Rb I$ of a graph $g$ on $I$ equals the Minkowski sum
\begin{equation}\label{eq:Zg}
Z_g = 
\sum_{\{i\} \textrm{ half-edge of } g}{\Delta_i}+ 
\sum_{\{i,j\} \textrm{ edge of } g}{\Delta_{\{i,j\}}}.
\end{equation}
In particular, the zonotope of the complete graph $K_I$ on the set $I$ is a translation of the standard permutahedron $\pi_I$:
\[
Z_{K_I} = \pi_I - e_I.
\]
\end{proposition}
\noindent Note that the right hand side of (\ref{eq:Zg}) may have repeated summands.

The facial structure of graphic zonotopes can be described combinatorially \cite{postnikov09, stanley73:_acycl} as we now recall. 
 A \emph{flat} $f$ of a graph $g$ is a set of edges with the property that for any cycle of $g$ consisting of edges $e_1, \ldots, e_k$, if $e_1, \ldots, e_{k-1} \in f$ then $e_k \in f$. 
 
For each flat $f$ of $g$ and each acyclic orientation $o$ of $g/f$, let $g(f,o)$ be the graph obtained from $g$ by keeping $f$ intact, and replacing each edge $\{i,j\}$ not in $f$ by the half-edge $\{i\}$ where $i \rightarrow j$ in the orientation $o$ of $g/f$. The following result is essentially known. \cite{stanley73:_acycl}.

\begin{lemma} \label{l:facesZg}
Let $g$ be a graph with vertex set $I$. The faces of the zonotope $Z_g \subset \Rb I$ are in bijection with the pairs of a flat $f$ of $g$ and an acyclic orientation $o$ of $g/f$. The face corresponding to flat $f$ and orientation $o$ is $Z_{g(f,o)}$, and it is a translation of $Z_f$.
\end{lemma}

\begin{proof}
By (\ref{eq:Zg}), the maximal face of $Z_g$ in the direction of $y \in \Rb^I$ is 
\begin{equation}\label{eq:faceZg}
(Z_g)_y = 
\sum_{\{i\} \in g}{\Delta_i}+ 
\sum_{\{i,j\} \in g \, : \, y(i) = y(j)}\Delta_{\{i,j\}}+
\sum_{\{i,j\} \in g \, : \, y(i) > y(j)}\Delta_i+
\sum_{\{i,j\} \in g \, : \, y(i) < y(j)}\Delta_j.
\end{equation}
The vector $y$ determines a flat $f_y$ consisting of the edges $\{i,j\}$ of $g$ such that $y(i)=y(j)$. It also determines an acyclic orientation $o_y$ of $g/f$ obtained by giving the edge $\{i,j\}$ the orientation $i \rightarrow j$ if $y(i) > y(j)$ or $i \leftarrow j$ if $y(i) < y(j)$. Clearly the maximal face $(Z_g)_y$ depends only on $f_y$ and $o_y$. 
Furthermore, different choices of $f_y$ and $o_y$ determine different faces of $(Z_g)_y$, and every  choice of a flat $f$ of $g$ and an acyclic orientation $o$ of $g/f$ can be realized by some vector $y$. This proves the desired one-to-one correspondence. It follows from (\ref{eq:faceZg}) that $(Z_g)_y = Z_{g(f_y, o_y)}$ and that this is a translation of $Z_{f_y}$, as desired.
\end{proof}

\subsection{{Graphs as a submonoid of generalized permutahedra}}\label{s:GinGP}
  
 Recall that $\rG$ is the Hopf monoid of graphs, where $\rG[I]$ is the set of graphs with vertex set $I$, where repeated edges and half-edges are allowed.
For a decomposition $I=S\sqcup T$, the product of two graphs $g_1\in\rG[S]$ and $g_2\in\rG[T]$ is their disjoint union. The coproduct of $g \in\rG[S]$ is $(g|_S, g/_S) \in \rG[S] \times \rG[T]$, where the restriction $g|_S\in\rG[S]$ is the induced subgraph on $S$, while the contraction $g/_S\in\rG[T]$ is obtained by keeping all edges incident to $T$, converting each edge from $T$ to $S$ into a half-edge on $T$.
  
Let $\rG^{cop}$ be the Hopf monoid co-opposite to $\rG$, as defined at the end of Section \ref{ss:hopf-set}.

\begin{proposition}\label{p:graph-submod2}
The map $\inc:\rG^{cop} \to\rSF  \map{\cong} \rGP$ is an injective morphism of Hopf monoids.
\end{proposition}
\begin{proof}
We first check that $\inc$ is a morphism of Hopf monoids.
Let $I=S\sqcup T$. Choose $g_1\in\rG[S]$ and $g_2\in\rG[T]$.
Since there are no edges connecting $S$ to $T$ in $g_1\cdot g_2$, an
edge of $g_1\cdot g_2$ incident to $A\subseteq I$ is either incident to $A\cap S$ or to $A\cap T$, but not both. Hence,
\[
\inc_{g_1\cdot g_2}(A)=\inc_{g_1}(A\cap S)+\inc_{g_2}(A\cap T)=(\inc_{g_1}\cdot\inc_{g_2})(A).
\]
Thus, $\inc$ preserves products. 

Let us now show that $\inc$ reverses coproducts. Choose $g\in\rG[I]$. If $A\subseteq T$, then for any edge $e$ of $g$ incident to $A$ there is a corresponding edge $e'$ of $g/_S$ incident to $A$ (possibly a half-edge, if the other endpoint of $e$ belongs to $S$). Since every edge of $g/_S$ arises in this manner from an edge of $g$, we have
\[
\inc_{g/_S}(A)=\inc_g(A)=(\inc_g)|_T(A).
\]
Now, if $A\subseteq S$, notice that an edge of $g$
incident to $A\cup T$ is either incident to $T$,  
or has both endpoints in $S$ (and at least one endpoint in $A$), in which case it is an edge of $g|_S$. Therefore $\inc_g(A\cup T) = \inc_g(T) + \inc_{g|_S}(A)$, so 
\[
\inc_{g|_S}(A)=\inc_g(A\cup T)-\inc_g(T)=(\inc_g)/_T(A).
\]
It follows that $\inc$ reverses restrictions and contractions, as desired.

To prove injectivity, note that if $a$ and $b$ are two distinct vertices of a graph $g$, then the number of edges of $g$ between $a$ and $b$ is
$
\inc_g(\{a\})+\inc_g(\{b\})-\inc_g(\{a,b\}).
$
Also, the number of half-edges at $a$ is 
$
\inc_g(I)-\inc_g(I\setminus\{a\})
$. These numbers determine $g$ entirely.
\end{proof}

\begin{remark}\label{r:cut-inc}
In graph theory one also considers the \emph{cut function} $\cut_g$ defined by
\[
\cut_g(A)=\text{ the number of edges of $g$ joining $A$ to $I\setminus A$},
\]
The map $g\mapsto \cut_g$ is not a morphism of Hopf monoids $\rG\to\rSF$: neither restrictions nor contractions are preserved. 
However, we do have $\cut_g(A) = 2\cdot \inc_g(A) - \sum_{i\in A}\deg_g(i)$, where the \emph{degree} $\deg_g(i)$ is the number of edges incident to vertex $i$. It follows from this that $\cut_g$ is submodular (a known result) and its generalized permutahedron
$\Pc(\cut_g)$ 
is a scaling of $\Pc(\inc_g) = Z_g$ followed by a translation by the vector $-\deg_g \in \Rb I$. 
Therefore the map $g\mapsto \cut_g$ does give a morphism of Hopf monoids $\rG\to\rbGP$; but since $\Pc(\inc_g)$ and $\Pc(\cut_g)$ are normally equivalent, this morphism does not teach us anything new about the Hopf monoid of graphs.
\end{remark}

\subsection{{The antipode of graphs}}

In view of Proposition \ref{p:graph-submod2} and Theorem \ref{t:antipode}, the antipode of $\wG$ is given by the facial structure of graphic zonotopes, as described in Lemma \ref{l:facesZg}.

\begin{corollary}\label{c:antipodeG}
The antipode of the Hopf monoid of graphs $\wG$ is given by the following \textbf{cancellation-free} and \textbf{grouping-free} expression. If $g$ is a graph on $I$ then
\[
\apode_I(g) = \sum_{f,o} (-1)^{c(f)} g(f,o),
\]
summing over all pairs of a flat $f$ of $g$ and an acyclic orientation $o$ of $g/f$, where $c(f)$ is the number of connected components of $f$.
\end{corollary}

\begin{proof}
This follows from Theorem \ref{t:antipode} and Lemma \ref{l:facesZg}, and the observation that the dimension of the zonotope $Z_f$ is $|I|-c(F)$.
\end{proof}

\begin{example} Let us revisit Example \ref{eg:antipode-graph}.  The formula
 \begin{figure}[H]
\centering
\includegraphics[scale=.5]{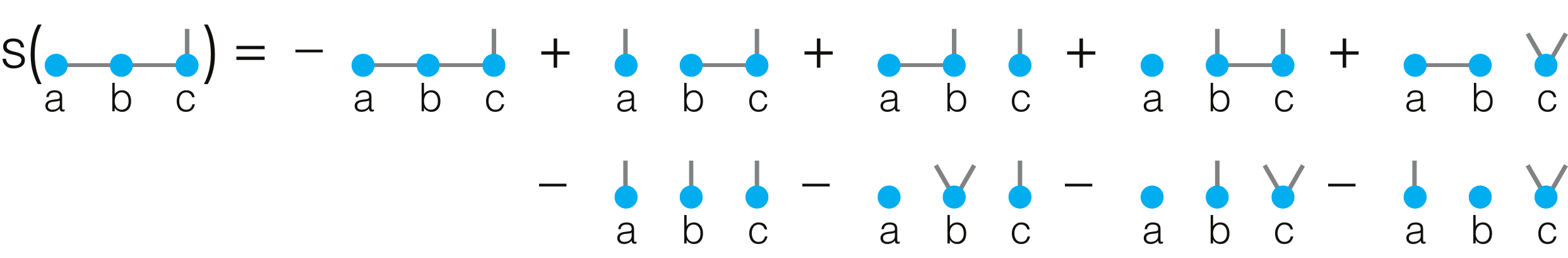} \qquad
\end{figure}
\noindent is the algebraic manifestation of the face structure of the graphic zonotope of Example \ref{eg:graph-zono} which consists of one parallelogram, four edges, and four vertices. These nine faces are the graphic zonotopes of the nine graphs occurring in the expression above.
\end{example}

\subsection{{Simple graphs}}\label{s:sgraph} 
A graph is \emph{simple} if it has no half-edges or multiple edges.
Let $\rSG[I]$ denote the set of all simple graphs with vertex set $I$.
Then $\rSG$ is a subspecies of $\rG$, but it not a Hopf submonoid because a contraction of a simple graph need not be simple. 

To remedy this situation, consider the \emph{simplification} map
\[
\rG[I]\to\rSG[I], 
\quad g\mapsto \simple{g}. 
\]
which removes half-edges and edge multiplicities: in $\simple{g}$ there is a unique edge joining two vertices $a$ and $b$ if and only if $a\neq b$ and there is at
least one edge joining $a$ and $b$ in $g$.
This defines a surjective morphism of species
\[
\rG\onto\rSG.
\]
The Hopf monoid structure of $\rG$ descends to $\rSG$ via this map,
so that $\rSG$ is a quotient Hopf monoid of $\rG$. In $\rSG$, products and
contractions have the same description as in $\rG$, while restrictions now coincide
with contractions. Therefore $\rSG$ is cocommutative.

The (linearization of the) Hopf monoid $\rSG$ appears (with different notation) in~\cite[Section~13.2]{am}. A closely related structure was first considered by
Schmitt~\cite[Example~3.3.(3)]{schmitt93:_hopf}. 

%
%

%

\begin{proposition}\label{p:graph-zono}
There is a commutative diagram of morphisms of Hopf monoids as follows.
\[
\xymatrix{
\rG^{cop} \ar@{->>}[d]  \xyinc[r] & \rGP\ar@{->>}[d] \\
\rSG^{cop}  \xyinc[r] & \rbGP
}
\]
\end{proposition}
%

\begin{proof}
Simplification gives the vertical map $\rG^{cop} \onto\rSG^{cop}$ while the map $\rGP \onto \rbGP$ identifies generalized permutahedra with the same normal fan.
The top map $\rG^{cop} \into \rGP$ is given by Proposition~\ref{p:graph-submod2}, while the bottom map $\rSG^{cop} \into \rbGP$ sends a simple graph to the normal equivalence class of its zonotope. To verify that the diagram commutes, we need to show that if $g$ is a graph and $g'$ is its simplification, then $Z_g$ and $Z_{g'}$ are normally equivalent.

By (\ref{eq:Zg}), the normal fan
$\N(Z_g)$ is the common refinement of the fans $\N({\Delta_{\{i\}}})$ for all half-edges $\{i\}$ and $\N({\Delta_{\{i,j\}}})$ for all edges $\{i,j\}$. This common refinement is unaffected by the removal of the former fans (which are trivial) and by the removal of repetitions of the latter fans. Therefore $\N(Z_g) = \N(Z_{g'})$ as desired.
\end{proof}

%
%
%

\begin{corollary}\label{c:antipodeG2}
The antipode of the Hopf monoid of simple graphs $\rSG$ is given by the following \textbf{cancellation-free} and \textbf{grouping-free} expression. If $g$ is a simple graph on $I$ then
\[
\apode_I(g) = \sum_{f \textrm{ flat of } g} (-1)^{c(f)} a(g/f) \,  f
\]
where $a(g/f)$ is the number of acyclic orientations of the contraction $g/f$ and $c(f)$ is the number of connected components of $f$.
\end{corollary}

\begin{proof}
This follows from Corollary \ref{c:antipodeG} and the observation that when $g$ is a simple graph, the simplification of $g(f,o)$ is $f$.
\end{proof}

An equivalent formula for Hopf algebras was also obtained by Humpert and Martin \cite{humpert12} through a clever inductive argument. In the context of Hopf algebras isomorphic graphs are identified, so to find the coefficient of a particular graph $h$ in $\apode_I(g)$ one has to overcome the additional problem of identifying all flats of $g$ isomorphic to $h$. This is one reason to prefer working with Hopf monoids instead of Hopf algebras in combinatorial contexts. See also Remark \ref{r:monoid}.

\subsection{{Characters of complete graphs and Humpert and Martin's conjecture}}

For each $k \in \Cb$ let $\xi_k$ be the character on $\wG$ given by $(\xi_k)_I(g) = k^{|I|}$ for any graph $g$ on vertex set $I$. Let $\zeta$ be the character on $\wG$ where $\zeta_I(g)$ equals $1$ if $g$ has no edges and $0$ otherwise.
For each $k \in \Cb$ and $c \in \Zb$ let $\xi_k \zeta^c$ denote the convolution product of $\xi_k$ and $\zeta^c$ in $\wG$.

Recall that a \emph{derangement} of $I$ is a permutation of $I$ without fixed points, and an \emph{arrangement} is a permutation of a subset of $I$. The following formulas were conjectured by Humpert and Martin \cite{humpert12}.

\begin{theorem}\cite[Conjecture (27)]{humpert12}\label{conj:hm} Let $K_n$ be the complete graph on $n$ vertices. Then
\[
\sum_{n \geq 0} (\xi_k \zeta^c)(K_n) \frac{x^n}{n!} = e^{kx}(1+x)^c
\]
for any complex number $k$ and integer $c$. In particular,
\[
(\xi_1 \zeta^{-1})(K_n) = (-1)^n D_n
\qquad
(\xi_{-1} \zeta^{-1})(K_n) = (-1)^n A_n
\]
where $D_n$ and $A_n$ are the numbers of derangements and arrangements of 
$[n]$ respectively.
\end{theorem}

\begin{proof}
Since the graphic zonotope of a complete graph $K_I$ is a translation of the standard permutahedron $\pi_I$ by Proposition \ref{prop:Zgzonotope}, the Hopf submonoid of $\wG$ generated by complete graphs is isomorphic to the Hopf submonoid $\wbPi$ of $\wGP$ generated by standard permutahedra, considered in Section \ref{s:Pi}. We may then regard $\xi_k$ and $\zeta$ as characters on $\wbPi$. This allows us to carry out the required computations in the character group $\Xb(\wbPi)$, where they become straightforward.

By Theorem \ref{t:charperm}, convolution of characters of $\wbPi$ corresponds to multiplication of their exponential generating functions; therefore
\[
\sum_{n \geq 0} (\xi_k \zeta^c)(\pi_n) \frac{x^n}{n!} = \left(\sum_{n \geq 0} \xi_k(\pi_n) \frac{x^n}{n!}\right)
\left(\sum_{n \geq 0} \zeta(\pi_n) \frac{x^n}{n!}\right)^c = 
e^{kx}(1+x)^c
\]
as desired. By comparing this with the generating functions
\[
\sum_{n \geq 0} (-1)^n D_n \frac{x^n}{n!} = \sum_{n \geq 0} \left[(-1)^n \left(\sum_{i =0}^n (-1)^i \frac{n!}{i!}\right)\frac{x^n}{n!}\right] = \left(\sum_{i \geq 0} \frac{x^i}{i!}\right) \left(\sum_{j \geq 0} (-1)^j x^j\right) = e^x (1+x)^{-1}
\]
and
\[
\sum_{n \geq 0} (-1)^n A_n \frac{x^n}{n!} = \sum_{n \geq 0} \left[ (-1)^n \left(\sum_{i =0}^n \frac{n!}{i!}\right)\frac{x^n}{n!} \right]= \left(\sum_{i \geq 0} \frac{(-1)^ix^i}{i!}\right) \left(\sum_{j \geq 0} (-1)^j x^j\right) = e^{-x} (1+x)^{-1}
\]
we obtain the remaining two formulas.
\end{proof}

\section{$\rM$: {Matroids and matroid polytopes}}\label{s:M}  

Similarly to graphs, matroids also have a polyhedral model called its \emph
{matroid polytope}, due to Edmonds \cite{edmonds70}; and this model respects the Hopf-algebraic structure of matroids, introduced in 1982 by Joni and Rota~\cite{joni82:_coalg} and further studied by Schmitt~\cite{schmitt93:_hopf}. We now employ the geometric perspective to compute, for the first time, the optimal formula for the antipode of matroids.

\subsection{{Matroid polytopes}}\label{ss:edmonds}

Let $m$ be a matroid with ground set $I$. The \emph{rank} of $A\subseteq I$ in $m$, denoted $\rank_m(A)$, is the cardinality of any maximal independent set of $m$ contained in $A$. The matroid axioms guarantee that this is well-defined and moreover, that the function
\[
\rank_m:2^I\to\Nb
\]
is submodular~\cite[Lemma 1.3.1]{oxley92:_matroid}; indeed, the marginal benefit of adding $e$ to $S$
\[
(\rank_m)/_S(e) = \begin{cases}
1 & \textrm{ if $e$ is independent of $S$,} \\ 
0 & \textrm{ if $e$ is dependent on $S$} 
\end{cases}
\]
weakly decreases as we add elements to $S$.

By Theorem \ref{t:submod-gp} and (\ref{e:submod-gp}), the submodular function $\rank_m$ gives rise to a generalized permutahedron $\Pc(\rank_m) = \Pc(m)$ which is called the \emph{matroid polytope} of $m$
This polytope has an elegant vertex description.

\begin{proposition}\label{prop:matroidpolytope} \cite{edmonds70,gelfand1987combinatorial}
The matroid polytope $\Pc(\rank_m) = \Pc(m)$ of a matroid $m$ on $I$ is given by
\[
 \Pc(m) = \textrm{conv} \, \{e_{b_1} + \cdots + e_{b_r} \, : \, \{b_1, \ldots, b_r\} \textrm{ is a basis of } m\} \subset \Rb I,
 \]
 where $\{e_i \, : \, i \in I\}$ is the standard basis. Furthermore, every basis gives a vertex of $\Pc(m)$.
\end{proposition}

This construction goes back to Edmonds~\cite{edmonds70} in optimization, and later to Gel'fand, Goresky, MacPherson, and Serganova~\cite{gelfand1987combinatorial} in algebraic geometry. 
In what follows, we will sometimes identify a matroid $m$ with its matroid polytope $\Pc(m)$.

\begin{example}\label{eg:matroid-poly} Revisiting Example \ref{eg:antipode-matroid}, let $m$ be the matroid of rank 2 on $\{a,b,c,d\}$ whose only non-basis is $\{c,d\}$. The matroid polytope $\Pc(m)$, shown in Figure \ref{f:matroidpolytope}, is given by the inequalities:
\[
x_a+x_b+x_c+x_d=2, \,\,\,\,
x_a+x_b, x_a+x_c, x_a+x_d,
x_b+x_c, x_b+x_d \leq 2,
x_c+x_d \leq 1, \,\,\,\,\,
x_a, x_b, x_c, x_d\leq 1
\]
\end{example}

\begin{figure}[h]
\label{f:matroidpolytope}
\centering
\includegraphics[scale=.8]  {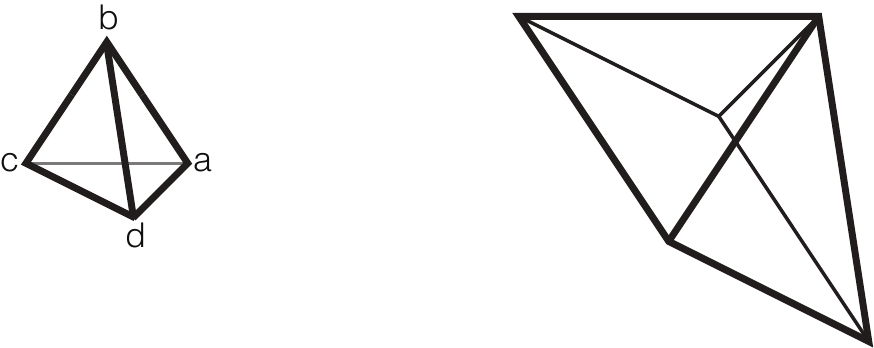} \qquad \caption{The matroid polytope of the matroid of Example \ref{eg:matroid-poly}.}
\end{figure}

There does not seem to be a simple and purely combinatorial indexing for the faces of the matroid polytope $\Pc(m)$. For a non-bijective description of these faces, see \cite[Proposition 2]{ardila2006bergman} or \cite[Problem 1.26]{borovik2003coxeter}.

\subsection{Matroids as a submonoid of generalized permutahedra}

Recall that $\rM$ is the Hopf monoid of matroids, where $\rM[I]$ is the set of graphs with ground set $I$. 
For  a decomposition $I=S\sqcup T$, the product of two matroids $m_1\in\rM[S]$ and $m_2\in\rM[T]$ is their direct sum $m_1 \oplus m_2 \in \rM[I]$. The coproduct of a matroid $m \in\rM[I]$ is $(m|_S, m/_S)$, where $m|_S \in  \rM[S]$ and $m/_S \in \rM[T] $ are the restriction and contraction of $m$ with respect to $S$, respectively.

\begin{proposition}\label{p:matroid-submod}
The map $\rank:\rM\to\rSF \map{\cong} \rGP$ is an injective morphism of Hopf monoids.
\end{proposition}
\begin{proof}
The descriptions for the rank function of the direct sum, restriction and 
contraction of matroids in~\cite[Prop. 3.1.5, 3.1.7, 4.2.17]{oxley92:_matroid}
imply that $\rank$ is a morphism of Hopf monoids.
Injectivity holds since the rank function determines the matroid uniquely.
\end{proof}

\subsection{The antipode of matroids}

To our knowledge, a cancellation-free formula for the antipode of matroids was only known for the special case of uniform matroids. \cite{bucher16}
Proposition \ref{p:matroid-submod} and Theorem \ref{t:antipode} now tell us that  the antipode of $\rM$ is given by the facial structure of matroid polytopes.

Every matroid $m$ has a unique maximal decomposition as a direct sum of smaller matroids. We let $c(m)$ be the number of summands, which are called the \emph{connected components} of $m$. \cite[Section 4]{oxley92:_matroid}

\begin{theorem}\label{t:antipodeM}
The antipode of the Hopf monoid of matroids $\rM$ is given by the following \textbf{cancellation-free} and \textbf{grouping-free} formula. If $m$ is a matroid on $I$, then
\begin{equation}\label{amnmonoid}
\apode_I(m) = \sum_{n \leq m} (-1)^{c(n)} \, n,
\end{equation}
where we sum over all the nonempty faces $n$ of the matroid polytope of $m$. 
\end{theorem}

\begin{proof}
This  is an immediate consequence of Theorem \ref{t:antipode}, taking into account that the dimension of a matroid polytope $\Pc(m)$ on $I$ equals $|I|-c(m)$. \cite{edmonds70}
\end{proof}

As mentioned earlier, there seems to be no simple combinatorial indexing of the faces of a matroid polytope, and hence no purely combinatorial counterpart of this formula.

The (discrete and algebraic) geometric point of view on matroids, initiated in \cite{edmonds70} and \cite{gelfand1987combinatorial}, has evolved into a central component of matroid theory thanks to the natural appearances of matroid polytopes in various settings in optimization, algebraic geometry, and tropical geometry. Theorem \ref{t:antipodeM} shows that this geometric point of view also plays an essential role here: if one wishes to fully understand the Hopf algebraic structure of matroids, it becomes indispensable to view them as polytopes.

\subsection{{The Hopf algebra of matroids.}} The Fock functor sends the Hopf monoid $\rM$ to the Hopf algebra of (isomorphism classes of) matroids defined by Joni and Rota
~\cite[Section~XVII]{joni82:_coalg} and also studied by Schmitt~\cite[Section~15]{schmitt94:_incid_hopf}. 
Theorem \ref{t:antipodeM} answers the open question of determining the optimal formula for the antipode of a matroid: it is simply the signed sum of the faces of the matroid polytope.

\begin{theorem}\label{t:antipodeM2}
In the Hopf algebra of (isomorphism classes of) matroids, the antipode of a matroid $m$ is 
\begin{equation}\label{amn}
\apode(m) = \sum_{n} (-1)^{c(n)} a(m:n) \, n,
\end{equation}
where $c(n)$ is the number of components of a matroid $n$ and $a(m:n)$ is the number of faces of the matroid polytope $\Pc(m)$ which are congruent to $\Pc(n)$.
\end{theorem}

\begin{proof}
This is an immediate consequence of Theorem \ref{t:antipodeM}.
\end{proof}

\begin{example} Let us revisit Example \ref{eg:antipode-matroid}. The formula
 \begin{figure}[h]
\centering
\includegraphics[scale=.5]{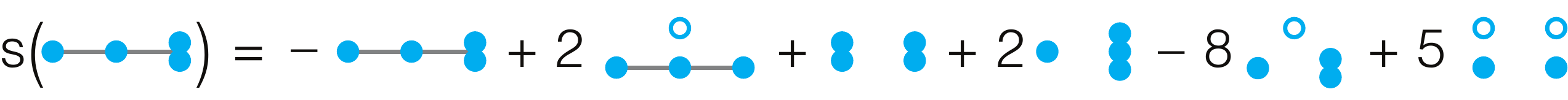} \qquad
\end{figure}
\noindent is the algebraic manifestation of the face structure of the corresponding matroid polytope, which is a square pyramid. It has one full-dimensional face, 5 two-dimensional faces (in matroid isomorphism classes of sizes 2, 1, 2), 8 edges (in one isomorphism class), and 5 vertices (in one isomorphism class).
\end{example}

\begin{remark}\label{r:monoid}
Theorems \ref{t:antipodeM} and \ref{t:antipodeM2}  illustrate an important advantage of working with Hopf monoids instead of Hopf algebras. 

To try to discover (\ref{amn}), we might compute a few small examples and try to find a pattern. After witnessing unexpected cancellations and unexplained groupings of equal terms, we are left with coefficients $a(m:n)$ that are very hard to  identify; in fact, we do not know any enumerative properties of these coefficients. 

If, instead, we work in the context of Hopf monoids,
a coefficient equal to 5 in (\ref{amn}) comes from a sum 1+1+1+1+1 in (\ref{amnmonoid}) where each 1 is indexed combinatorially; this additional granularity allows us to identify each term contributing to 
(\ref{amnmonoid}), and to then combine them to obtain (\ref{amn}).

However, for matroids, the geometric lens is crucial -- even in the context of Hopf monoids. It is not easy to identify the individual terms of (\ref{amnmonoid}) if one is not thinking about the matroid polytope, whose faces have no simple combinatorial description.
\end{remark}

\subsection{{Graphical matroids and another Hopf monoid of graphs}}\label{ss:cycle-matroid}

Any family of matroids which is closed under direct sums, restriction, and contraction forms a Hopf submonoid of $\rM$. Many important families of matroids satisfy these properties and have the structure of a Hopf monoid; for instance: linear matroids over a fixed field, graphical matroids, algebraic matroids over a fixed field, gammoids, and lattice path matroids. \cite{bonin2006lattice, oxley92:_matroid,white1986theory}.
In particular, the Hopf monoid of graphical matroids is closely related to a third Hopf monoid on graphs $\rGamma$, which we now describe.


For a finite set $I$, let $\rGamma[I]$ be the set of graphs with edges labeled by $I$, with unlabeled vertices, and without isolated vertices. 
To define a product and coproduct on $\rGamma$, let $I=S\sqcup T$ be a decomposition. 
The product $\gamma_1\cdot \gamma_2\in\rGamma[I]$ is
the (disjoint) union of the graphs $\gamma_1\in\rGamma[S]$ and $\gamma_2\in\rGamma[T]$.
The coproduct $\Delta_{S,T}(\gamma) = (\gamma|_S, \gamma/_S)$ is given by the standard notions of restriction and contraction from graph theory.
The restriction
$\gamma|_S\in\rGamma[S]$ is obtained from $\gamma\in\rGamma[I]$ by removing all edges in $T$ 
and all vertices not incident to $S$. 
The contraction
$\gamma/_S\in\rGamma[S]$ is obtained by 
contracting all edges in $S$ from $\gamma\in\rGamma[I]$, and removing any isolated vertices that remain. (To contract an edge, we identify the endpoints and remove the edge.) 

The two Hopf monoids of graphs $\rGamma$ and $\rG$ that we have discussed are not directly related; in fact, they differ already as species. 
Instead, we have a morphism of Hopf monoids 
\[
\rGamma \to \rM
\]
mapping each graph $\gamma$ to its graphical matroid, which is the set of spanning trees of $\gamma$. \cite{oxley92:_matroid} We do not know further properties of the Hopf monoid $\rGamma$, in particular, because we are not aware of any results on graphical matroid polytopes.

 \section{$\rP$: {Posets and poset cones}}\label{s:P}

Similarly to graphs and matroids, posets also have a polyhedral model that respects the Hopf algebra structure introduced by Schmitt in 1994. \cite{schmitt94:_incid_hopf} We use this geometric model to give an optimal combinatorial formula for the antipode of posets.

\subsection{{Poset cones}}\label{ss:posetcones}

A \emph{$\{0, \infty\}$ function} on $I$ is a Boolean function $z: 2^I \to \{0, \infty\}$ such that $z(\emptyset) = z(I) = 0$. Its \emph{support} is $\supp(z) = \{J \subseteq I \, | \, z(J)=0\}$. For a \emph{$\{0, \infty\}$ function} $z$ on $I$,
\begin{equation}\label{eq:support}
z \textrm{ is submodular } \Longleftrightarrow \textrm{if } A, B \in \supp(z) \textrm{ then } A\cup B, A \cap B \in \supp(z).
\end{equation}

For each poset $p$ on $I$ we define the \emph{lower set} function
\[
\low_p: 2^I \mapsto \Rb \cup \{\infty\}, \quad 
\low_p(J) = \left\{ \begin{array}{ll}
0 & \textrm{if $J$ is a lower set of $p$,} \\
\infty & \textrm{if $J$ is not a lower set of $p$.} \\
\end{array} \right.
\]
This is an extended submodular function since the family of lower sets of $p$ is closed under unions and intersections.

By Theorem \ref{t:submod-gp+} and (\ref{e:submod-gp+}), the submodular function $\low_p$ gives rise to an extended generalized permutahedron
\[
\Pc(p) := \Pc(\low_p) = \{x \in \Rb I \, : \, \sum_{i \in I} x_i = 0 \textrm{ and } \sum_{a \in A} x_a \leq  0 \textrm{ for every lower set $A$ of p}\}
\]
which we call the \emph{poset cone} of $p$.
This cone has an elegant description in terms of generators. Dobbertin proved an analogous result for a related polytope in \cite{dobbertin85}.

\begin{proposition}\label{prop:posetconegenerators}
The poset cone of a poset $p$ is given by
 \[ 
 \Pc(p) = \textrm{cone} \, \{e_i - e_j \, : \, i > j \textrm{ in }  p\}
 \]
 where $\{e_i \, : \, i \in I\}$ is the standard basis of $\Rb I$.
The generating rays of  $\Pc(p)$ are given by the roots $e_i-e_j$ corresponding to the cover relations $i \gtrdot j$ of $p$.
\end{proposition}

\begin{proof} Recall the notation
\[
x(A) = \sum_{i \in A} x_i
\]
for $x \in \Rb I$ and $A \subseteq I$. We prove both containments:

\medskip

\noindent $\supseteq$: Let $i>j$ in $p$. Every order ideal $A$ that contains $i$ must also contain $j$, so $e_i-e_j$ satisfies $x(A) \leq 0$. This implies that $\textrm{cone}\{e_i - e_j \, : \, i > j \in p\} \subseteq \Pc(p)$.

\medskip

\noindent $\subseteq$:  We will need the following lemma.

\begin{lemma}
Let $x \in \Pc(p)$. Let $i$ be a maximal element of $p$ such that $x_i \neq 0$, and let $i_1, \ldots, i_k$ be the elements covered by $i$ in $p$. We can write
\[
x = x' + \lambda(e_i - e_j)
\]
for some $x' \in \Pc(p)$, some element $j \lessdot i$, and some number  $\lambda>0$ which is a linear combination of the $x_i$s.
\end{lemma}

\begin{proof}[Proof of Lemma 15.2]
The maximality of $i$ and the fact that $p_{\geq i}:= \{j \in p \, : \, j \geq i\}$ is an upper set imply that 
\begin{equation}\label{eq:xi>0}
x_i = \sum_{j \in p_{\geq i}} x_j = x(p_{\geq i}) > 0.
\end{equation}
Let $i_1, \ldots, i_k$ be the elements covered by $i$ in $p$. 
We claim that
\begin{equation}\label{e:somex<0}
\textrm{there exists an index $1 \leq a \leq k$ with $x(I_a) < 0$ for every lower set $I_a \ni i_a$ }. 
\end{equation}

We prove this claim by contradiction. If that was not the case, then for every $1 \leq a \leq k$ we would have a lower set $I_a \ni i_a$ such that $x(I_a) = 0$. Now, we observe that
\begin{equation}\label{e:capcup}
\textrm{ if $A$ and $B$ are lower sets with } x(A) = x(B) = 0, 
\textrm{ then } x(A \cup B) = x(A \cap B) = 0. 
\end{equation}
This observation follows from the fact that $A \cup B$ and $A \cap B$ are lower sets, so they satisfy $x(A \cup B) \leq 0$ and $x(A \cap B) \leq 0$, while also satisfying $x(A \cup B) + x(A \cap B) = x(A) + x(B) = 0$.
Applying (\ref{e:capcup}) repeatedly, we see that $I_1 \cup \cdots \cup I_k$ is a lower set with $x(I_1 \cup \cdots \cup I_k) = 0$. But then we observe that $I_1 \cup \cdots \cup I_k \cup i$ is also a lower set, so we get
\[
x_i = x(I_1 \cup \cdots \cup I_k \cup i) \leq 0,
\]
contradicting (\ref{eq:xi>0}).

Having proved (\ref{eq:xi>0}) and  (\ref{e:somex<0}), let $1 \leq a \leq k$ be  as in (\ref{e:somex<0}). Then  
\[
\lambda := \min(\{x_i\} \cup \{-x(I_a) \, : \, I_a \textrm{ is a lower set containing } i_a\}) > 0,
\]
and define $x' = x - \lambda(e_i - e_{i_a})$ as required. To conclude, it remains to prove that $y \in \Pc(p)$. To do this, let $J$ be any lower set of $p$. If $J$ contains both ${i_a}$ and $i$, or if it contains neither ${i_a}$ nor $i$, then we have $x'(J) = x(J) \leq 0$. On the other hand, if $J$ contains ${i_a}$ but not $i$, then $x'(J) = x(J) + \lambda \leq 0$
by the definition of $\lambda$, since $J$ is a lower set containing $i_a$. It follows that $x' \in \Pc(p)$, concluding the proof of the lemma.
\end{proof}

Now we need to prove that any $x \in \Pc(p)$ is a positive linear combination of vectors of the form $e_i-e_j$ such that $i < j$ in $p$. Since the rationals are a dense subset of the reals and the cones we are considering are closed, it suffices to prove this when all entries of $x$ are rational. We proceed by induction on the number of positive entries of $x$. 

Let $i$ be a maximal element of $p$ with $x_i > 0$. Write $x = x' + \lambda(e_i - e_j)$ for $\lambda > 0$ and $i \gtrdot j$ as in the lemma, and note that $x'_i < x_i$. If $x'_i > 0$, use the lemma again to write $x' = x'' + \lambda'(e_i - e_{j'})$ for $\lambda' > 0$ and $i \gtrdot j'$, and note that $x''_i < x'_i < x_i$ . We can continue applying the lemma in this way while $x_i^{'' \cdots '} > 0$. In each step, the $i$th coordinate decreases by a positive linear combination of the original $x_i$s. Since the $x_i$s are rational, the $i$th coordinate is decreasing discretely, and must reach $0$ eventually. We will then have written $x = y + c$ for a linear combination $c \in \textrm{cone}\{e_i - e_j \, : \, i > j \in p\}$ and a vector $y \in \Pc(p)$ with one fewer positive entry, since $y_i = 0$. The induction hypothesis now gives $y \in \textrm{cone}\{e_i - e_j \, : \, i > j \in p\}$, which implies $x \in \textrm{cone}\{e_i - e_j \, : \, i > j \in p\}$ as well. The desired result follows by induction.

\smallskip

Having proved that $\Pc(p)$ is generated by the vectors $e_i - e_j$ where $i>j$, let us observe that if $i>j$ then there is a sequence of  cover relations $i \gtrdot k_1 \gtrdot \cdots \gtrdot k_r \gtrdot j$, which implies that $e_i-e_j = (e_i - e_{k_1}) + (e_{k_1}-e_{k_2}) + \cdots + (e_{k_r} - e_j)$. Therefore the vectors $e_i - e_j$ with $i \gtrdot j$ generate $\Pc(p)$. By a similar argument one sees that they generate $\Pc(p)$ irredundantly.
\end{proof}

The faces of poset polytopes were described (for the cones dual to poset cones) by Postnikov-Reiner-Williams \cite[Proposition 3.5]{prw08}  (for order polytopes) by Geissinger \cite{geissinger1981face} and Stanley \cite{stanley1986two}, and (for oriented matroids) by Las Vergnas [Prop. 9.1.2]\cite{bjorner1999oriented}. Our presentation follows Las Vergnas, interpreting his general criterion in this special case.

Define a \emph{circuit} of $p$ to be a cyclic sequence $i_1, \ldots, i_n$ of elements of $p$ where every consecutive pair is comparable in $p$.
Circuits consist of \emph{up-edges} where $i_j < i_{j+1}$ in $p$ and \emph{down-edges} where $i_j > i_{j+1}$ in $p$. We will say that a subposet $q$ of $p$ is \emph{positive}\footnote{this terminology comes from the theory of oriented matroids}  if the following conditions hold for every circuit $X$: 

(1)
if all the down-edges of a circuit $X$ are in $q$, then all the up-edges of $X$ are in $q$, and

(2)
if all the up-edges of a circuit $X$ are in $q$, then all the down-edges of $X$ is in $q$.

%

\begin{lemma} \label{l:facesPp}
Let $p$ be a poset on $I$. The faces of the poset cone $\Pc(p) \subset \Rb I$ are precisely the poset cones $\Pc(q)$ as $q$ ranges over the  positive subposets of $p$.
\end{lemma}

\begin{proof}
In this proof we will assume some basic facts about oriented matroid theory; see \cite{bjoerner93:_om,ardila2006positive}
for the relevant definitions. Let $\cal M$ be the (acyclic) oriented matroid of the set of vectors $\{e_i-e_j \, : \, i> j \textrm{ in } p\}$. The faces of the poset cone $\Pc(p)$ are the cones generated by the \emph{positive flats} of the Las Vergnas face lattice of $\cal M$. By \cite
[Prop. 9.1.2]{bjoerner93:_om}, these are the subsets $F$ of $\cal M$ such that for every signed circuit $X$ of $\cal M$, $X^+ \subseteq F$ implies $X^- \subseteq F$.

The oriented matroid $\cal M$ is isomorphic to the graphical oriented matroid of the graph of $p$ on $I$, whose directed edges $i \rightarrow j$ correspond to the order relations $i>j$ in $p$. Therefore the 
signed circuits of $\cal M$ correspond to the cycles of the graph; they are the sets of the form:
\[
X = \{e_{i_k} - e_{i_{k+1}} \, : \, i_1, \ldots, i_n \textrm{ is a circuit of } p\}
\]
where $i_{n+1} = i_1$. Each circuit $X$ comes with two orientations. One of them is given by $X^+ = \{e_{i_k} - e_{i_{k+1}} \, : i_k > i_{k+1} \textrm{ in } p\}$ and $X^- = \{e_{i_k} - e_{i_{k+1}} \, : i_k < i_{k+1} \textrm{ in } p\}$ 
and the other one is its reverse.


Now let $F = q \subset p$ be a subposet of $p$. In the first orientation of $X$, the condition that $X^+ \subseteq F$ implies $X^- \subseteq F$ says that if every down-edge is in $q$ then every up-edge must be in $q$. In the other orientation, this condition is reversed.
It follows that the positive flats of $\cal M$ are in bijection with the positive subposets of $p$, as desired.
\end{proof}

\begin{example} \label{ex:poset} Let $p$ be the poset on $\{a,b,c,d\}$ given by the cover relations $a<c, \, b<c, \, a<d, \,  b<d$. The poset cone of $p$ is shown 
below. The positive subposets $q \neq p$ are those which do not contain both vertical cover relations $a<c$ and $b<d$, and do not contain both diagonal cover relations $a<d$ and $b<c$. There are nine such subposets, corresponding to the nine proper faces of $\Pc(p)$.
 \begin{figure}[h]
\centering
\includegraphics[scale=.7]{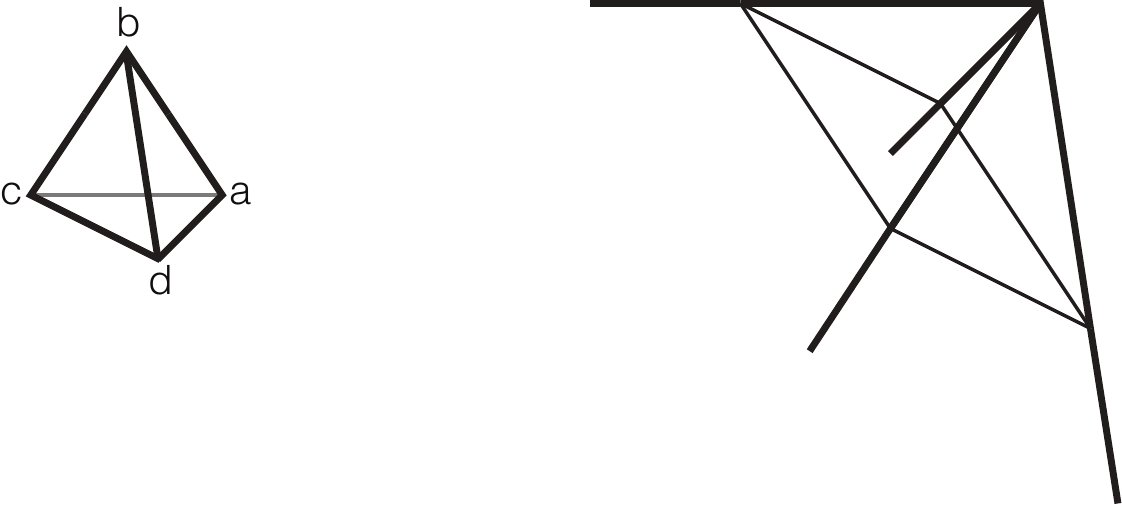}
\caption{The poset cone for the poset of Example 15.4. 
\label{f:posetcone} }
\end{figure}

\end{example}

\begin{remark}
Let us give some additional intuition for the definition of positive subposets. We will need  preposets; see Section \ref{ss:preposets} for a definition.

A \emph{poset contraction} is a preposet obtained from $p$ by successively \emph{contracting} order relations $i<j$ of $p$ and replacing them by equivalence relations $i \sim j$. Since we need to keep the preposet transitive, contracting the up-edges of a circuit forces us to also contract the down-edges, and viceversa. For instance, in Example \ref{ex:poset}, if we contract $a<c$ and $b<d$, we get the contradictory relations $a \sim c > b \sim d > a$; to remedy this, we are forced to contract $b<c$ and $a<d$ into $b \sim c$ and $a \sim d$ as well.

In conclusion, the positive subposets of $p$ are precisely the contracted subposets for the contractions of $p$.
 \end{remark}

\subsection{Posets as a submonoid of extended generalized permutahedra}  \label{ss:PtoGP}

Recall that $\wP$ is the Hopf monoid (in vector species) of posets. 
For $I=S\sqcup T$, the product of two posets $p_1$ on $S$ and $p_2$ on $T$ is their disjoint union $p_1 \sqcup p_2$ regarded as a poset on $I$. The coproduct $\Delta_{S,T}: \wP[I] \to  \wP[S] \otimes \wP[T]$ is
\[
\Delta_{S,T}(p) =  \begin{cases}
p|_S\otimes p|_T & \text{if $S$ is a lower set of $p$,}\\
0 & \text{otherwise.}
\end{cases}
\]

\begin{proposition}\label{p:PtoSF}
The map $\low:\wP \to \wSF_+ \map{\cong}\wGP_+$ is an injective morphism of Hopf monoids in vector species.
\end{proposition}

\begin{proof}
To check that $\low$ preserves the product, let $I = S \sqcup T$ be a decomposition. Let $p_1$ and $p_2$ be posets on $S$ and $T$, and $p_1 \sqcup p_2$ be their product. A subset $J \subseteq I$ is a lower set of $p_1 \sqcup p_2$ if and only if $J \cap S$ and $J \cap T$ are lower sets of $p_1$ and $p_2$, respectively. It follows that 
\[
\low_{p_1 \sqcup p_2}(J) = \low_{p_1}(J \cap S) + \low_{p_2}(J \cap T) = (\low_{p_1} \cdot \low_{p_2})(J),
\]
so $\low$ preserves products.

To check that $\low$ preserves the coproduct, let $I = S \sqcup T$ and let $p$ be a poset on $I$. We need to consider two cases:

1. Suppose $S$ is not a lower set of $p$. Then $\Delta_{S,T}(p) = 0$. 
In this case we also have $\low_p(S)=\infty$ so $\Delta_{S,T}(\low_p) = 0$ by the definition of the coproduct in $\wSF_+$. It follows that $\low$ trivially respects the coproduct in this case.

2. Suppose $S$ is a lower set of $p$. Then the restriction and contraction of $p$ with respect to $S$ are $p|_S$ and $p|_T$, respectively. Also $\low_p(S)=0$. To see that $\low$ is compatible with restriction,
notice that for $R \subseteq S$ we have $(\low_p)|_S(R) = \low_p(R)$, so
\[
\low_{p|_S}(R) = 
\begin{cases}
0 & \textrm{ if $R$ is a lower set of $p|_S$} \\
\infty & \textrm{ otherwise}
\end{cases},  \,\,
(\low_p)|_S(R) = 
\begin{cases}
0 & \textrm{ if $R$ is a lower set of $p$} \\
\infty & \textrm{ otherwise.}
\end{cases}
\]
Since $R$ is a lower set of $p|_S$ if and only if it is a lower set of $p$, we have $\low_{p|_S} = (\low_p)|_S$.

On the other hand, to see that $\low$ is compatible with contraction,
notice that for $R \subseteq T$ we have 
$\low_{p/_S}(R) = \low_{p|_T}(R) = \low_p(R)$ and 
$(\low_p)/_S(R) = \low_p(R \cup S)$, so
\[
\low_{p/_S}(R) = 
\begin{cases}
0 & \textrm{ if $R$ is a lower set of $p|_T$} \\
\infty & \textrm{ otherwise}
\end{cases}, \,\, 
(\low_p)/_S(R) = 
\begin{cases}
0 & \textrm{ if $R \cup S$ is a lower set of $p$} \\
\infty & \textrm{ otherwise.}
\end{cases}
\]
Since $R$ is a lower set of $p|_T$ if and only if it $R \cup S$ is a lower set of $p$, we have $\low_{p/_S} = (\low_p)/_S$.


We conclude that $\low$ is a morphism of monoids. Injectivity follows from the fact that we can recover a poset $p$ from its collection of lower sets as follows: two elements $i,j$ of $p$ satisfy $i<j$ if and only if every lower set containing $j$ also contains $i$.
\end{proof}

\subsection{{The antipode of posets}}

In view of Proposition \ref{p:PtoSF} and Theorem \ref{t:antipode}, the antipode of $\wP$ is given by the facial structure of poset polytopes, as described in Lemma \ref{l:facesPp}. This allows us to give the optimal combinatorial formula for the antipode of the Hopf monoid of posets.

Recall that the \emph{Hasse diagram} of a poset $p$ is the graph whose vertices correspond to the elements of $p$ and whose edges $x \rightarrow y$, which are always drawn with $x$ lower than $y$, correspond to the cover relations $x \lessdot y$ of $p$.

\begin{corollary}\label{c:antipodeP}
The antipode of the Hopf monoid of posets $\wP$ is given by the following \textbf{cancellation-free} and \textbf{grouping-free} expression. If $p$ is a poset on $I$ then
\[
\apode_I(p) = \sum_{q} (-1)^{c(q)} q,
\]
summing over all positive subposets $q$ of $p$, where $c(q)$ is the number of connected components of the Hasse diagram of $q$.

\end{corollary}

\begin{proof}
This follows from Theorem \ref{t:antipode} and Lemma \ref{l:facesPp}, and the observation that the dimension of the poset cone $\Pc(p)$ is $|I|-c(p)$.
\end{proof}

\begin{example} Let us revisit Example \ref{eg:antipode-poset}. This example takes place in the Hopf algebra of posets $P$, where isomorphic posets are identified. The formula

\begin{figure}[h]
\centering
\includegraphics[scale=.5]{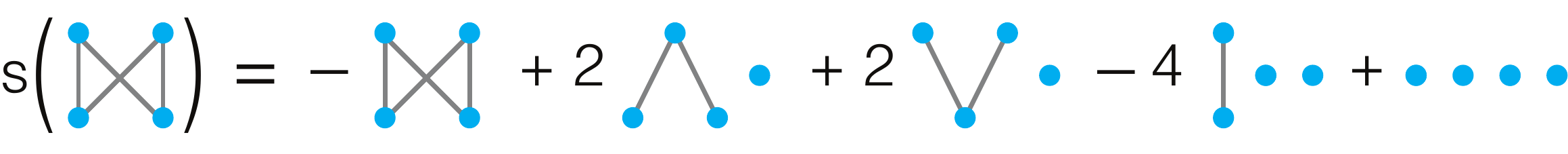} \qquad
\end{figure}
\noindent is the algebraic manifestation of the face structure of the corresponding poset cone, which is the cone over a square shown in Figure \ref{f:posetcone}. It has one full-dimensional face, 4 two-dimensional faces (in poset isomorphism classes of sizes 2 and 2), 4 rays (in one isomorphism class), and 1 vertex.
Combinatorially, the summands correspond to the positive subposets of the poset in question, as described in Example \ref{ex:poset}.
\end{example}

 \subsection{{Preposets and preposet cones}}\label{ss:preposets}


One may wonder whether there are other interesting submonoids of $\wGP$ consisting of cones, or (more or less equivalently) submonoids of $\wSF$ consisting of $\{0, \infty\}$ functions. In Theorem \ref{t:conespreposets} and Proposition \ref{p:PtoSF} we show that, essentially, there aren't. 
We prove that $\{0, \infty\}$ submodular functions are equivalent to the slightly larger class of preposets, which may be viewed as posets in their own right.

A \emph{preposet} on $I$ is a binary relation $q \subseteq I \times I $, denoted $\leq$, which is reflexive ($x \leq x$ for all $x \in q$) and transitive ($x \leq y$ and $y \leq z$ imply $x \leq z$  for all $x,y,z\in q$).
A preposet is not necessarily antisymmetric, and we define an equivalence relation by setting
\[
x \sim y \textrm{ when } x \leq y \textrm{ and } y \leq x.
\]
Let $p = q/\unsim$ be the set of equivalence classes of $p$. The relation $\leq$  induces a relation $\leq$ on $q/\unsim$ which is still reflexive and transitive, and is also antisymmetric; \emph{i.e.}, it defines a poset.

It follows that we may think of preposets as posets whose elements are labeled by non-empty and pairwise disjoint sets. More precisely, we may equivalently define a \emph{preposet} on $I$ to be a set partition $\pi = \{I_1, \ldots, I_k\}$ of $I$ together with a poset $p$ on $\pi$.

If $p'$ is a lower set of the poset $p=q/\unsim$, then we say $q'=\bigcup_{K \in p'} K$ is a \emph{lower set} of  the preposet $q$. As before, we define the lower set function of $q$ to be 
\[
\low_{q}: 2^I \mapsto \Rb \cup \{\infty\}, \quad 
\low_{q}(J) = \left\{ \begin{array}{ll}
0 & \textrm{if $J$ is a lower set of $q$,}\\
\infty & \textrm{otherwise.} \\
\end{array} \right.
\]

\begin{theorem}\label{t:conespreposets}
A Boolean function $z: I \rightarrow \{0, \infty\}$ is submodular if and only if $z= \low_q$ is the lower set function of a preposet $q$ on $I$.
\end{theorem}

\begin{proof}
The backward direction is straightforward:
If $q$ is a preposet then its collection of lower sets is closed under union and intersection. It follows from (\ref{eq:support}) that $\low_q$ is submodular.

\medskip

The forward direction will require more work. 
Suppose $z$ is a submodular $\{0, \infty\}$ function on $I$ and let 
\[
L:= \supp(z). 
\]
We need to show that $L$ is the collection of lower sets of a preposet $q$ on $I$.

Thanks to (\ref{eq:support}) we know that $L= \supp(z)$ is a lattice under the operations of union and intersection. These operations are  distributive, so Birkhoff's fundamental theorem of distributive lattices \cite[Theorem 3.4.1]{stanley11:_ec1} applies: If $L_{\irred}$ is the subposet of join-irreducible elements of $L$, and if $J(L_{\irred})$ is the poset of lower sets of $L_{\irred}$ ordered by inclusion, then 
\[
L \cong J(L_{\irred}).
\]

We reinterpret $L_{\irred}$ as a preposet on $I$ as follows. For each set $A \in L_{\irred}$ let
\[
\ess(A) = A - \bigcup_{B \in L_{\irred} \atop B < A} B.
\]
be the \emph{essential set} of $A$, consisting of the \emph{essential elements} which are in no lesser join-irreducible. Consider the collection of essential sets 
\[
q: = \{\ess(A)\, : \, A \in L_\irred\},
\]
endowed with the partial order inherited from $L_\irred$. We will now show that: 

1. $q$ is a preposet on $I$, and 

2. $L$ is the collection of lower sets of $q$. \\
These two statements will complete the proof.

Before we prove these two statements, let us illustrate this construction with an example. The left panel of Figure \ref{f:distributive} shows a distributive lattice $L$ of subsets of $I=\{a,b,c,d,e,f,g,h,i,j\}$. We only label the join-irreducible elements; the label of every other set is the union of the join-irreducibles less than it in $L$. The right hand side panel shows the subposet $L_\irred$. For each join-irreducible set $A \in L_\irred$ we have indicated its essential set $\ess(A)$ in bold. These essential sets partition $I$, allowing us to think of this object $q$ as a preposet on $I$.

\bigskip

\begin{figure}[h]
\centering
\includegraphics[scale=.65]{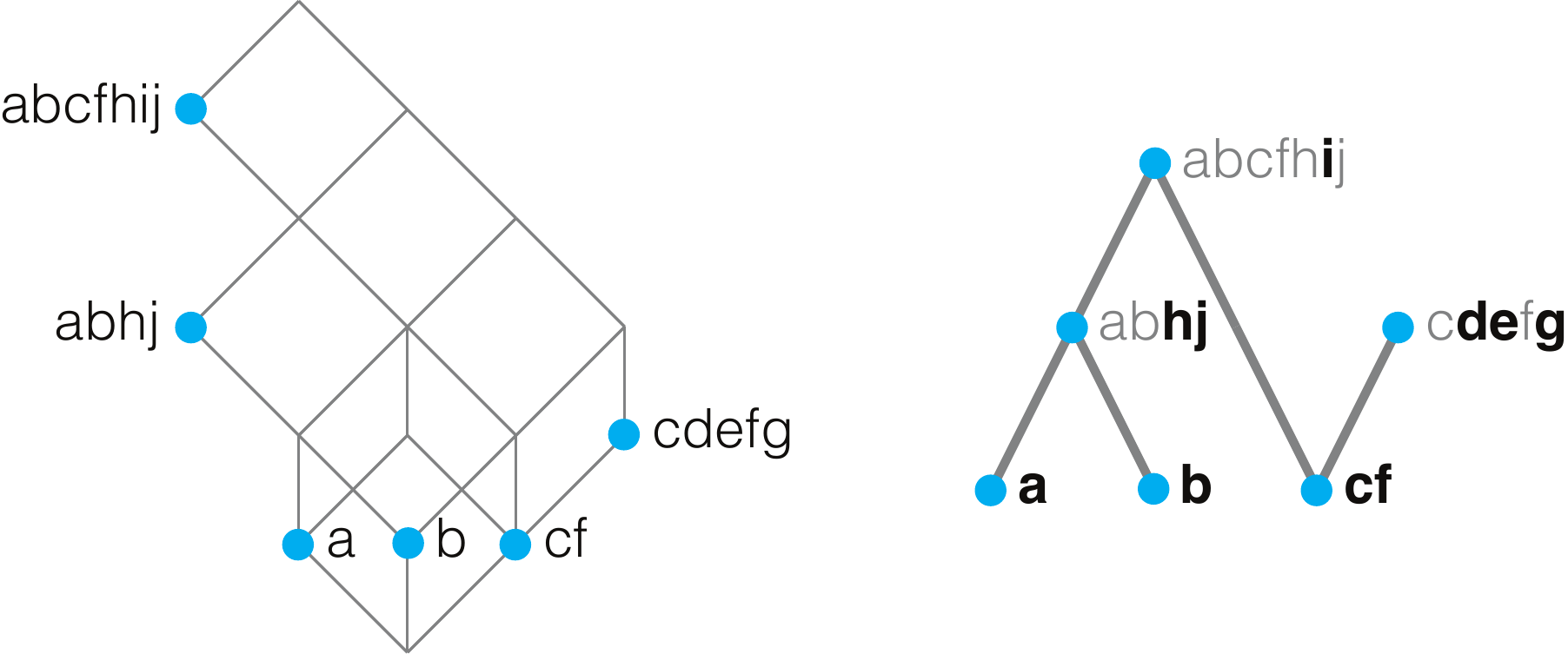} \caption{A distributive lattice $L$ of subsets of $I=\{a,b,c,d,e,f,g,h,i,j\}$ and its poset of join-irreducibles. The essential sets of $L_\irred$ are  shown in boldface; they give rise to a preposet $q$ on $I$, whose lower sets are precisely the sets in $L$. \label{f:distributive}}
\end{figure}

\medskip

\noindent 
\emph{Step 1. $q$ is a preposet on $I$}: 
We need to show that the sets in $q$ form a set partition of $I$. Each essential set $\ess(A)$ is non-empty because $A$ is join-irreducible. 
To prove that the essential sets are pairwise disjoint, assume contrariwise that $x \in \ess(A)$ and $x \in \ess(B)$ for some $A \neq B \in L_\irred$. Then $A \cap B \in L$ and $x \in A \cap B$, so $x \in C$ for some join irreducible $C \in L_\irred$ with $C \subseteq A \cap B \subsetneq A$. This contradicts the assumption that $x$ is an essential element of $A$.

The following lemma completes the proof of Step 1.

\begin{lemma} 
\label{l:esspartitions}
For all $A \in L$,
\[
A = \bigsqcup_{B \in L_{\irred} \atop B \leq A} \ess(B).
\]
In particular, $\{\ess(B) \, : \, B \in L_\irred\}$ is a partition of $I$.
\end{lemma}

\begin{proof}[Proof of Lemma \ref{l:esspartitions}]
First we prove that the lemma holds for each join-irreducible $A \in L_\irred \subseteq L $, proceeding by induction. This statement is clearly true for the minimal elements of $L_\irred$. Also, if it holds for all elements $B<A$ in $L_\irred$, then using the definition of $\ess(A)$ and the induction hypothesis,
\[
A = 
\ess (A) \sqcup \bigcup_{B \in L_{\irred} \atop B < A} B
= 
\ess (A) \sqcup \bigcup_{B \in L_{\irred} \atop B < A} \bigsqcup_{C \in L_{\irred} \atop C \leq B} \ess(C)  
= 
\bigsqcup_{C \in L_{\irred} \atop C \leq A}  \ess(C) 
\]
so the claim holds for $A$ as well. Therefore the lemma holds for all $A \in L_\irred$.

Now we can prove Lemma \ref{l:esspartitions} holds for all $A \in L$. The backward inclusion is clear. To prove the forward inclusion, let $x \in A$. Since $A$ is the union of the join-irreducibles less than it in $L$, we have $x \in C$ for some $C \in L_\irred$ with $C \leq A$. By the previous paragraph, $x \in \ess(D)$ for some $D \in L_\irred$ with $D \leq C$; but then $D \leq A$ also, so $x$ is in one of the essential sets on the right hand side. The desired result follows. 

The last statement follows by recalling that $z(I) = 0$ and applying the lemma to $A=I$, which is the maximum element of the lattice $L$. This completes the proof of Lemma \ref{l:esspartitions} and of Step 1 of this proof.
\end{proof}

\medskip

\noindent 
\emph{Step 2. $L$ is the collection of lower sets of $q$:} 
By Birkhoff's theorem and Lemma \ref{l:esspartitions}, $A \in L$ if and only if there is a down set $J \subseteq L_\irred$ with
\[
A = \bigcup_{B \in J} B = \bigsqcup_{B \in J} \ess(B);
\]
that is, if and only if $A$ is a  lower set of $q$.
\end{proof}

We now state an algebraic counterpart of 
Theorem \ref{t:conespreposets}. 
Let $\rQ[I]$ be the set of preposets on $I$. Preposets become a Hopf monoid in vector species $\wQ$ with the same operations of the Hopf monoid of posets $\wP$. 
Let $\wSF_{\{0,\infty\}}$ be the submonoid of $\wSF$ consisting of $\{0, \infty\}$ functions. Let $\wbGP_{\textrm{cone}}$ be the submonoid of $\wbGP$ consisting of cones.

\begin{proposition}\label{p:PtoSF}
The maps $\low:\wQ  \map{\cong} \wSF_{\{0,\infty\}} \map{\cong}\wbGP_{\textrm{cone}}$ are isomorphisms of Hopf monoids in vector species.
\end{proposition}

\begin{proof}
The first isomorphism is an immediate consequence of Theorem \ref{t:conespreposets}. For the second one, notice that every cone $c \in \wGP$ is a translate of a unique cone $c'$ that contains the origin. The submodular function $z_c$ such that $c' = \Pc(z_c)$ is a $\{0, \infty\}$ function, and the correspondence $c \mapsto z_c$ gives the desired isomorphism. 
\end{proof}

In the correspondence between preposets and generalized permutahedra which are cones, posets on $I$ correspond to cones of the maximum possible dimension $|I|-1$. The antipode formula for preposets $\wQ$ is essentially the same as the antipode formula for posets $\wP$.

\section{{Preliminaries 4: Invariants of Hopf monoids and reciprocity.}}\label{s:pol-inv}

Once again, we set aside the combinatorial examples of earlier sections 
and return to the general setting of Hopf monoids of Section~\ref{s:hopf}. This section shows that each character on a Hopf monoid gives rise to an associated polynomial invariant. 
There are two main results. Proposition~\ref{p:pol-inv} shows that the polynomial invariant is indeed polynomial and invariant. Proposition~\ref{p:reciprocity} relates the values of the invariant on an integer and on its negative by means of the antipode of the Hopf monoid.

This abstract framework has concrete combinatorial consequences. For instance, we will see in Section \ref{s:reciprocity} that the simplest non-zero characters on the Hopf monoids $\wG, \wP, \wM$ give rise to three important combinatorial polynomials: the chromatic polynomial of a graph, the strict order polynomial of a poset, and the BJR polynomial of a matroid. Furthermore, this Hopf-theoretic framework gives immediate proofs of the celebrated reciprocity theorems for these polynomials, due to Stanley and Billera-Jia-Reiner. 
%

\subsection{{The polynomial invariant of a character}}\label{ss:inv}

Recall from Section \ref{s:prelimcharacters} the notion of a character $\zeta$ on a Hopf monoid on vector species $\wH$. In the examples that interest us, $\wH$ is a Hopf monoid coming from a family of combinatorial objects, and $\zeta$ is a multiplicative function on our objects which is invariant under relabelings of the ground set.

Throughout this section, we fix a connected Hopf monoid $\wH$ and a character $\zeta: \wH \rightarrow \Kb$. 
Define, for each element $x\in\wH[I]$ and each natural number $n\in\Nb$, 
\begin{equation}\label{e:pol-inv}
\chi_I(x)(n)\ :=
\sum_{I=S_1\sqcup \cdots \sqcup S_n} (\zeta_{S_1}\otimes\cdots\otimes\zeta_{S_n})\circ \Delta_{S_1,\ldots,S_n} (x),
\end{equation}
summing  over \emph{all} decompositions of $I$ into $n$ disjoint subsets which are allowed to be empty.
For fixed $I$ and $x$, the function $\chi_I(x)$ is defined on $\Nb$ and takes values on $\Kb$. Note that 
\begin{equation}\label{e:pol-inv2}
\chi_I(x)(0)=
\begin{cases}
\zeta_\emptyset(x) & \text{if $I=\emptyset$,} \\
 0        & \text{otherwise,}
\end{cases}
\qquad \qquad 
\chi_I(x)(1)=\zeta_I(x).
\end{equation}

\begin{proposition}\label{p:pol-inv} \emph{(Polynomial invariants)}
Let $\wH$ be a connected Hopf monoid, $\zeta: \wH \rightarrow \Kb$ be a character, and $\chi$ be defined by (\ref{e:pol-inv}). Fix a finite set $I$ and
an element $x\in\wH[I]$.

\begin{enumerate}
\item 
For each $n \in \Nb$ we have 
\[
\chi_I(x)(n)\ = \sum_{k=0}^{\abs{I}}
\chi^{(k)}_I(x)\binom{n}{k}
\]
where, for each $k=0,\ldots,\abs{I}$,
\[
\chi^{(k)}_I(x) = \sum_{(T_1,\ldots, T_k)\vDash I} (\zeta_{T_1}\otimes\cdots\otimes\zeta_{T_k})\circ \Delta_{T_1,\ldots,T_k} (x) \,\, \in \Kb.
\]
summing over all compositions $(T_1, \ldots, T_k)$ of $I$.
Therefore, 
$\chi_I(x)$ is a polynomial function of $n$ of degree at most $\abs{I}$.
\item 
Let $\sigma:I\to J$ be a bijection, $x\in\wH[I]$ and
$y:=\wH[\sigma](x)\in\wH[J]$.
Then $\chi_I(x) = \chi_J(y)$.
\end{enumerate}
\end{proposition}
\begin{proof}
1. Given a decomposition $I=S_1\sqcup \cdots \sqcup S_n$, let $(T_1,\ldots,T_k)$
be the composition of $I$ obtained by removing the empty $S_i$s and keeping
the remaining ones in order. In view of unitality of $\Delta$ and $\zeta$, we have
\[
(\zeta_{S_1}\otimes\cdots\otimes\zeta_{S_n})\circ \Delta_{S_1,\ldots,S_n} (x)
= (\zeta_{T_1}\otimes\cdots\otimes\zeta_{T_k})\circ \Delta_{T_1,\ldots,T_k} (x).
\]
Note that $k \leq |I|$ and 
the number of decompositions $I=S_1\sqcup \cdots \sqcup S_n$ which give rise to a given composition $(T_1,\ldots,T_k)$ is
$\binom{n}{k}$. It follows that
\[
\chi_I(x)(n)\ = \sum_{k=0}^{\abs{I}}
\Biggl(\sum_{(T_1,\ldots T_k)\vDash I}  (\zeta_{T_1}\otimes\cdots\otimes\zeta_{T_k})\circ \Delta_{T_1,\ldots,T_k} (x)\Biggr)\,\binom{n}{k}.
\]
as desired.
Since each $\binom{n}{k}$ is a polynomial function of $n$ of degree $k$, $\chi_I(x)$ is polynomial of degree at most $|I|$.

2. This follows from the naturality of $\Delta$ and $\zeta$.
\end{proof}

Let $\Kb[t]$ denote the polynomial algebra.
Proposition~\ref{p:pol-inv} states that each character $\zeta$ gives rise to a family of 
 polynomials
$\chi_I(x)\in\Kb[t]$ associated to each structure $x\in\wH[I]$,
whose values on nonnegative integers $n$
are given by~\eqref{e:pol-inv}. Furthermore, it says that two isomorphic structures have the same associated polynomial. Thus, the function $\chi_I(x)$ is
a \emph{polynomial invariant} of the structure $x$ (canonically associated to 
the Hopf monoid $\wH$ and the character $\zeta$).

\subsection{{Properties of the polynomial invariant of a character}}\label{ss:propsinv}

We now collect some useful properties of these polynomial invariants.


\begin{proposition}\label{p:pol-inv3}
Let $\wH$ be a connected Hopf monoid, $\zeta: \wH \rightarrow \Kb$ be a character, and $\chi$ be the associated polynomial invariant, 
defined by (\ref{e:pol-inv}).
 Let $I$ be a finite set.
\begin{itemize}
\item[(i)] $\chi_I$ is a linear map from $\wH[I]$ to $\Kb[t]$. \item[(ii)] Let $I=S\sqcup T$ be a decomposition. For any $x\in\wH[S]$ and $y\in\wH[T]$, we have the equality of polynomials
\[
\chi_I(x\cdot y) = \chi_S(x)\chi_T(y)
\]
\item[(iii)] $\chi_\emptyset(1)=1$, the constant polynomial.
\item[(iv)] For any $x\in\wH[I]$ and scalars $n$ and $m$,
\[
\chi_I(x)(n+m) =\sum_{I=S\sqcup T} \chi_S(x|_S)(n)\chi_T(x/_S)(m).
\]
\end{itemize}
\end{proposition}
\begin{proof} Property (i) follows from the linearity of $\Delta$ and $\zeta$.

Property (ii) follows from the compatibility between $\mu$ and $\Delta$ and the multiplicativity of $\zeta$. We provide the details. First, decompositions 
$I=I_1\sqcup\cdots\sqcup I_n$
into $n$ parts are in bijection with  pairs of decompositions $S=S_1\sqcup\cdots\sqcup S_n$ and $T=T_1\sqcup\cdots\sqcup T_n$, where $S_i=I_i\cap S$ and
$T_i=I_i\cap T$.
\[
\begin{picture}(100,90)(20,0)
  \put(50,40){\oval(100,80)}
  \put(0,40){\dashbox{2}(100,0){}}
  \put(45,55){$S$}
  \put(45,15){$T$}
\end{picture}
\quad
\begin{picture}(100,90)(10,0)
  \put(50,40){\oval(100,80)}
  \put(20,0){\dashbox{2}(0,80){}}
   \put(80,0){\dashbox{2}(0,80){}}
  \put(5,35){$I_1$}
  \put(45,35){$\cdots$}
  \put(85,35){$I_n$}
\end{picture}
\quad
\begin{picture}(100,90)(0,0)
\put(50,40){\oval(100,80)}
  \put(0,40){\dashbox{2}(100,0){}}
   \put(20,0){\dashbox{2}(0,80){}}
   \put(80,0){\dashbox{2}(0,80){}}
    \put(45,55){$\cdots$}
     \put(45,15){$\cdots$}
  \put(5,55){$S_1$}
  \put(85,55){$S_n$}
  \put(5,15){$T_1$}
  \put(85,15){$T_n$}
\end{picture}
\]
The compatibility between $\mu$ and $\Delta$ and the associativity of the latter
imply that if we write
\[
\Delta_{S_1,\ldots,S_n}(x)=\sum x_1\otimes\cdots\otimes x_n
\qand
\Delta_{T_1,\ldots,T_n}(y)=\sum y_1\otimes\cdots\otimes y_n,
\]
in Sweedler's notation, as described in Section \ref{ss:Hopfmonoidinvectorspecies}, 
then
\[
\Delta_{I_1,\ldots,I_n}(x\cdot y) = 
\sum (x_1\cdot y_1)\otimes\cdots\otimes (x_n\cdot y_n)
\]
The above, together with the multiplicativity of $\zeta$, yield that $\chi_I(x\cdot y)(n)$ equals
\begin{align*}
& \hspace{.6cm} \sum_{I=I_1\sqcup \cdots \sqcup I_n} (\zeta_{I_1}\otimes\cdots\otimes\zeta_{I_n})\circ \Delta_{I_1,\ldots,I_n} (x\cdot y)\\
& =\sum_{S=S_1\sqcup \cdots \sqcup S_n \atop T=T_1\sqcup\cdots\sqcup T_n} 
\sum \zeta_{I_1}(x_1\cdot y_1) \cdots \zeta_{I_n}(x_n\cdot y_n) =\sum_{S=S_1\sqcup \cdots \sqcup S_n \atop T=T_1\sqcup\cdots\sqcup T_n} 
\sum \zeta_{S_1}(x_1)\zeta_{T_1}(y_1)\cdots \zeta_{S_n}(x_n)\zeta_{T_n}(y_n)\\
& =\Bigl(\sum_{S=S_1\sqcup \cdots \sqcup S_n} (\zeta_{S_1}\otimes\cdots\otimes\zeta_{S_n})\circ \Delta_{S_1,\ldots,S_n} (x)\Bigr)
\Bigl(\sum_{T=T_1\sqcup \cdots \sqcup T_n} (\zeta_{T_1}\otimes\cdots\otimes\zeta_{T_n})\circ \Delta_{T_1,\ldots,T_n} (y)\Bigr)\\
& =\chi_S(x)(n)\,\chi_T(y)(n).
\end{align*}
Thus $\chi_I(x\cdot y) = \chi_S(x)\chi_T(y)$ as polynomials, since they agree at every natural number $n$.

Property (iii) follows from unitality of $\Delta$ and $\zeta$.

For property (iv), note that decompositions of $I$ into $n+m$ parts are
in bijection with tuples 
\[
(S,S_1,\ldots,S_n,T,T_1,\ldots,T_m)
\]
where $I=S\sqcup T$, $S=S_1\sqcup\cdots\sqcup S_n$, and $T=T_1\sqcup\cdots\sqcup T_m$. In addition, associativity of $\Delta$ implies that
\[
\Delta_{S_1,\ldots,S_n,T_1,\ldots,T_m} =
\bigl(\Delta_{S_1,\ldots,S_n}\otimes\Delta_{T_1,\ldots,T_m}\bigr)\circ \Delta_{S,T}.
\]
Therefore, $\chi_I(x)(n+m)$ is equal to 
\begin{align*}
& \hspace{.6cm} \sum_{I=S_1\sqcup \cdots \sqcup S_n\sqcup T_1\sqcup\cdots\sqcup T_m} (\zeta_{S_1}\otimes\cdots\otimes\zeta_{S_n}\otimes\zeta_{T_1}\otimes\cdots\otimes\zeta_{T_m})\circ \Delta_{S_1,\ldots,S_n,T_1,\ldots,T_m} (x)\\
&=\sum_{I=S\sqcup T}\sum_{S=S_1\sqcup \cdots \sqcup S_n \atop T=T_1\sqcup\cdots\sqcup T_m} 
(\zeta_{S_1}\otimes\cdots\otimes\zeta_{S_n}\otimes\zeta_{T_1}\otimes\cdots\otimes\zeta_{T_m})\circ \bigl(\Delta_{S_1,\ldots,S_n}\otimes\Delta_{T_1,\ldots,T_m}\bigr)\circ \Delta_{S,T}(x)\\
&=\sum_{I=S\sqcup T}\sum_{S=S_1\sqcup \cdots \sqcup S_n \atop T=T_1\sqcup\cdots\sqcup T_m} 
\Bigl((\zeta_{S_1}\otimes\cdots\otimes\zeta_{S_n})\circ\Delta_{S_1,\ldots,S_n}(x|_S)\Bigr) \Bigl(
(\zeta_{T_1}\otimes\cdots\otimes\zeta_{T_m})\circ \Delta_{T_1,\ldots,T_m}(x/_S)\Bigr)\\
&=\sum_{I=S\sqcup T} \chi_S(x|_S)(n)\, \chi_T(x/_S)(m).
\end{align*}

The above yields the desired equality when $n$ and $m$ are nonnegative integers. Since both sides of the equation are polynomial functions of $(n,m)$ in view of Proposition~\ref{p:pol-inv}, the result then follows for arbitrary scalars $n$ and $m$. 
\end{proof}

The following result states that if two characters are related by a morphism of Hopf monoids, then the same relation holds for the corresponding polynomial invariants.

\begin{proposition}\label{p:pol-inv4}
Let $\wH$ and $\wK$ be two Hopf monoids. Suppose $\zeta^{\wH}$ is a character on $\wH$, $\zeta^{\wK}$ is a character on $\wK$, and $f:\wH\to\wK$ is a morphism
of Hopf monoids such that
\[
\zeta^{\wK}_I\bigl(f_I(x)\bigr) = \zeta^{\wH}_I(x)
\]
for every $I$ and $x\in\wH[I]$. 
Let $\chi^{\wH}$ and $\chi^{\wK}$ be the polynomial invariants corresponding to $\zeta^{\wH}$ and $\zeta^{\wK}$, respectively. Then
\[
\chi^{\wK}_I\bigl(f_I(x)\bigr) = \chi^{\wH}_I(x)
\]
for every $I$ and $x\in\wH[I]$.
\end{proposition}
\begin{proof} Since $f$ preserves coproducts, we have 
$\Delta_{S,T}\bigl(f_I(x)\bigr)=(f_S\otimes f_T)\bigr(\Delta_{S,T}(x)\bigr)$
and a similar fact for iterated coproducts. This and the hypothesis give the result.
\end{proof}

\begin{remark} Most of the results in this section hold under weaker hypotheses
(different ones for each result). For instance, Proposition~\ref{p:pol-inv} holds
for any collection of linear maps $\zeta_I:\wH[I]\to\Kb$ which is unital (with the
same proof). If $n$ and $m$ are nonnegative integers,
statement (iv) in Proposition~\ref{p:pol-inv3} holds for any collection of linear maps $\zeta_I:\wH[I]\to\Kb$. Proposition~\ref{p:pol-inv4} holds for any
morphism of comonoids which preserves the characters.
\end{remark}

\subsection{{From Hopf monoids to reciprocity theorems}}\label{ss:rec}

For a character $\zeta$ on a Hopf monoid $\wH$, the construction of
Section~\ref{ss:inv} produces
a polynomial invariant $\chi$ whose values on natural numbers 
are well understood in terms of $\wH$ and $\zeta$.
What about the values on negative integers?
The antipode provides an answer to this question.

 \begin{proposition} \emph{(Reciprocity for polynomial invariants)}
\label{p:reciprocity} 
Let $\wH$ be a connected Hopf monoid, $\zeta: \wH \rightarrow \Kb$ be a character, and $\chi$ be the associated polynomial invariant, 
defined by (\ref{e:pol-inv}).
 Let $\apode$ be
 the antipode of $\wH$. Then
 \begin{equation}\label{e:reciprocity}
 \chi_I(x)(-1)=\zeta_I\bigl(\apode_I(x)\bigr).
 \end{equation}
 More generally, for every scalar $n$,
 \begin{equation}\label{e:reciprocity2}
 \chi_I(x)(-n)=\chi_I\bigl(\apode_I(x)\bigr)(n).
 \end{equation}
\end{proposition}

\begin{proof}
Since $\binom{-1}{k}=(-1)^k$, Proposition~\ref{p:pol-inv} implies
\[
\chi_I(x)(-1)\ = \sum_{k=0}^{\abs{I}}
\Biggl(\sum_{(T_1,\ldots T_k)\vDash I}  (\zeta_{T_1}\otimes\cdots\otimes\zeta_{T_k})\circ \Delta_{T_1,\ldots,T_k} (x)\Biggr)\,(-1)^k.
\]
Using multiplicativity of $\zeta$ and Takeuchi's formula~\eqref{e:takeuchi}, this may be rewritten as
\begin{align*}
\chi_I(x)(-1)\ &= \sum_{k=0}^{\abs{I}}
\Biggl(\sum_{(T_1,\ldots T_k)\vDash I}  \zeta_I\circ(\mu_{T_1}\otimes\cdots\otimes\mu_{T_k})\circ \Delta_{T_1,\ldots,T_k} (x)\Biggr)\,(-1)^k\\
&= \zeta_I\Biggl(\sum_{k\geq 0} \,(-1)^k\sum_{(T_1,\ldots T_k)\vDash I} \mu_{T_1,\ldots,T_k}\circ \Delta_{T_1,\ldots,T_k} (x)\Biggr) =\zeta_I\bigl(\apode_I(x)\bigr),
\end{align*}
which proves~\eqref{e:reciprocity}.

To prove~\eqref{e:reciprocity2} one may assume that the scalar $n$ is a nonnegative integer, since both sides are polynomial functions of $n$.
We make this assumption and proceed by induction on $n\in\Nb$.

When $n=0$ the result holds in view of~\eqref{e:pol-inv2} and the fact that $\apode_\emptyset = \id$. When $n=1$ it follows from~\eqref{e:pol-inv2}
and~\eqref{e:reciprocity}. For $n\geq 2$ we apply Proposition~\ref{p:pol-inv3}(iv)  as follows:
\[
\chi_I(x)(-n) = \chi_I(x)(-n+1-1)
=\sum_{I=S\sqcup T} \chi_S(x|_S)(-n+1)\,\chi_T(x/_S)(-1).
\] 
Using the induction hypothesis, and then reversing the roles of $S$ and $T$, this equals
\[
\sum_{I=S\sqcup T} \chi_S\bigl(\apode_S(x|_S)\bigr)(n-1) \, \chi_T\bigl(\apode_T(x/_S)\bigr)(1) = \sum_{I=S\sqcup T}  \chi_S\bigl(\apode_S(x/_T)\bigr)(1)\, \chi_T\bigl(\apode_T(x|_T)\bigr)(n-1).
\]
Applying Proposition~\ref{p:pol-inv3}(iv) to $\apode_I(x)$, and using the fact~\eqref{e:apode-delta} that the antipode reverses coproducts, we see that this equals
\[
\chi_I\bigl(\apode_I(x)\bigr)(1+n-1)=\chi_I\bigl(\apode_I(x)\bigr)(n),
\]
as needed.
\end{proof}

Formulas~\eqref{e:reciprocity} and~\eqref{e:reciprocity2} are \emph{reciprocity results} of a very general nature. They gives us another reason to be interested in an explicit antipode formula: such a formula allows for knowledge of the values of \emph{all}
polynomial invariants at negative integers. 
The antipode acts as a universal link between the values of the
invariants at positive and negative integers. We now apply this approach to $\wGP$ in Section \ref{s:basic}. This will allow us to unify several important reciprocity results in combinatorics and to obtain new ones in Section \ref{s:reciprocity}.

\section{The basic character and the basic invariant of $\wGP$} \label{s:basic}


In this section we return to specifics, focusing on the Hopf monoids of generalized permutahedra $\wGP$ and $\wGP_+$. We will prove the results in this section for $\wGP$ but they also hold in $\wGP_+$; see Remark \ref{r:basicGP_+}. 

We introduce the (almost trivial) \emph{basic character} $\beta$  and its associated \emph{basic invariant} $\chi$ on the Hopf monoid of generalized permutahedra $\wGP$. We use the algebraic structure of  $\wGP$ and $\beta$ to obtain combinatorial formulas for $\chi(n)$ and $\chi(-n)$ for $n \in \Nb$ in Propositions \ref{prop:basicinv} and \ref{prop:basicrecip}; these were also obtained in \cite{billera06}.
In Section \ref{s:reciprocity} we will see that
several important combinatorial facts about graphs, posets, and matroids are straightforward consequences of this setup.


\begin{definition}\label{d:basic}
The \emph{basic character} $\beta$ of $\wGP$ is given by
\[
\beta_I(\wp) = 
\begin{cases}
1 & \textrm{ if $\wp$ is a point}\\
0 & \textrm{ otherwise}.
\end{cases}
\]
for a generalized permutahedron $\wp \in \Rb I$. 
The \emph{basic invariant} $\chi$ of $\wGP$ is the polynomial invariant associated to $\beta$ by Proposition \ref{p:pol-inv} and (\ref{e:pol-inv}).
\end{definition}

Note that $\beta$ is indeed a character because the product of two polytopes $\wp \times \wq$ is a point if and only if both $\wp$ and $\wq$ are points.

\subsection{A lemma on directionally generic faces.}

Given a generalized permutahedron $\wp \subset \Rb I$ and a linear functional $y \in \Rb^I$, say $\wp$ is \emph{directionally generic in the direction of $y$} if the $y$-maximal face $\wp_y$ is a point. If this is the case, we will also say that $y$ is $\wp$-generic and that $\wp$ is \emph{$y$-generic}. 

\begin{figure}[h]
\centering
\includegraphics[scale=.55]{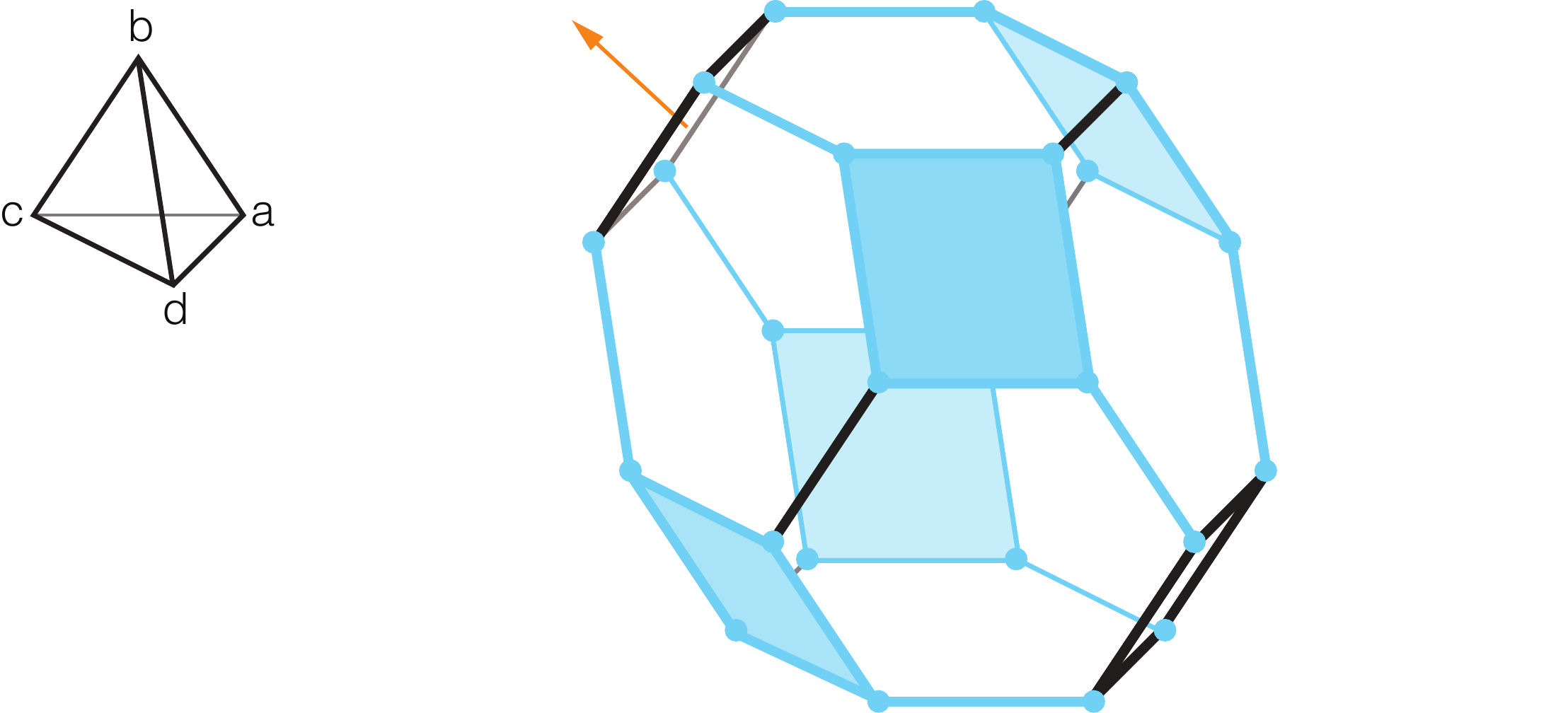} 
\caption{The $y$-generic faces of the permutahedron $\pi_4$ for $y=(0,1,1,0)$ are shaded.}
\label{f:hypergraphic}
\end{figure}

We will need the following technical lemma about directionally generic faces.

\begin{lemma}\label{lemma:wgenericfaces}
For any generalized permutahedron  $\wp \subset \Rb I$ and linear functional $y \in \Rb^I$, the following equations hold.
\begin{enumerate}
\item
\[
\sum_{\wq \leq \wp} (-1)^{\dim \wq} \wq_y =
\sum_{\wq \leq \wp_{-y}} (-1)^{\dim \wq} \wq
\]
\item
\[
\sum_{\wq \leq \wp:  \atop y \textrm{ is $\wq$-generic}}(-1)^{\dim \, \wq}
 = (-1)^{|I|} \textrm{ (number of vertices of  $\wp_{-y}$)}.
\]
\end{enumerate}
\end{lemma}

\begin{proof}
1. Let us express both sides of the equation Hopf-theoretically. 
Let $F$ be the face of the braid arrangement that $y$ belongs to, and say it corresponds to the decomposition $I = S_1 \sqcup \cdots \sqcup S_k$, as described in Section \ref{ss:gp}. 
Also recall from Section \ref{ss:antipodeproperties} that we denote
$\mu_F = \mu_{S_1,\ldots, S_k}$, 
$\Delta_F = \Delta_{S_1,\ldots, S_k}$, and
$\apode_F = \apode_{S_1} \otimes \cdots \otimes \apode_{S_k}$.

For any generalized permutahedron $\wr \subset \Rb I$ we have  $\wr_y = \wr_F = \mu_F \Delta_F(\wr)$ by Proposition \ref{p:face2}. 
 It then follows from the formula for the antipode of $\wGP$ in Theorem \ref{t:antipode} that 
\[
(-1)^{|I|}\mu_F\Delta_F\apode_I(\wp) =
\sum_{\wq \leq \wp} (-1)^{\dim \wq} \wq_y. 
\]
Now let $-F$ be the opposite face of $F$, corresponding to the decomposition $I = S_k \sqcup \cdots \sqcup S_1$. Then
$-F$ contains $-y$ so  $\mu_{-F}\Delta_{-F}(\wp) = \wp_{-y}$, and
\[
(-1)^{|I|} \apode_I \mu_{-F}\Delta_{-F}(\wp) = 
\sum_{\wq \leq \wp_{-y}} (-1)^{\dim \wq} \wq. 
\]
Now recall Proposition \ref{p:antipode1.5}, which holds for any Hopf monoid in vector species:
\[
\apode_I \mu_F = \mu_{-F} \apode_{-F} \sw_F , \qquad 
\Delta_F \apode_I = \apode_F \sw_{-F}  \Delta_{-F}. 
\]
Applying the second equation to $F$ and then the first equation to $-F$, we obtain
\[
\mu_F \Delta_F \apode_I = \mu_F \apode_F \sw_{-F} \Delta_{-F} = \apode_I \mu_{-F} \Delta_{-F},
\]
which gives the desired result.

\noindent 
2. This follows by applying the character $\chi_I$ to both sides of the equation of part 1.
%
\end{proof}

%
%


\subsection{{The basic invariant and the basic reciprocity theorem of $\wGP$}}


Recall that the basic invariant $\chi$ of $\wGP$ is the polynomial invariant that Proposition \ref{p:pol-inv} associates to the basic character $\beta$ of Definition \ref{d:basic}.

\begin{proposition} \label{prop:basicinv} \cite[Def. 2.3, Thm 9.2.(v)]{billera06}
At a natural number $n$, the \emph{basic invariant} $\chi$ of a generalized permutahedron $\wp \subset \Rb I$ is given by
\[
\chi_I(\wp)(n) = 
\textrm{ (number of $\wp$-generic functions $y: I \rightarrow [n])$.}
\]
\end{proposition}

\begin{proof}
First notice that each summand in (\ref{e:pol-inv}) comes from a decomposition $I=S_1\sqcup \cdots \sqcup S_n$, which bijectively corresponds to a function $y: I \rightarrow [n]$ defined  by $y(i)=k$ for each $i \in S_k$. The corresponding summand for $\chi_I(\wp)(n)$ is
\[
(\zeta_{S_1}\otimes\cdots\otimes\zeta_{S_n})\circ \Delta_{S_1,\ldots,S_n} (\wp) = \zeta_{S_1}(\wp_1) \cdots \zeta_{S_n}(\wp_n)
\]
where the $y$-maximal face $\wp_y$ factors as $\wp_y = \wp_1 \times \cdots \times \wp_n$ for $\wp_i \in \Rb{S_i}$. This term contributes to the sum if and only if every $\wp_i$ is a point, that is, if and only if $\wp_y$ is a point; and in that case, it contributes $1$. The desired result follows.
\end{proof}

\begin{proposition}\label{prop:basicrecip}
\cite[Thm. 6.3, Thm 9.2.(v)]{billera06} (Basic invariant reciprocity.) 
At a negative integer $-n$, the basic invariant $\chi$ of a generalized permutahedron $\wp \subset \Rb I$ is given by
\[
(-1)^{|I|}\chi_I(\wp)(-n) = \sum_{y: I \rightarrow [n]} \textrm{ (number of vertices of  $\wp_y$)}
\]
where $\wp_y$ is the $y$-maximum face of $p$.
\end{proposition}

\begin{proof}
Using the general reciprocity formula for characters of Proposition \ref{p:reciprocity} and the formula for the antipode of Theorem \ref{t:antipode} of $\rGP$ we obtain

\[
\chi_I(\wp)(-n) = \chi_I(\apode_I(\wp))(n) 
= (-1)^{|I|} \sum_{\wq \leq \wp} (-1)^{\dim \wq} \chi_I(\wq)(n). 
\]

Proposition \ref{prop:basicinv} and Lemma \ref{lemma:wgenericfaces} then give
\begin{eqnarray*}
\chi_I(\wp)(-n) 
&=& (-1)^{|I|} \sum_{\wq \leq \wp} (-1)^{\dim \wq} (\#  \textrm{ of $\wq$-generic functions $y:I \rightarrow [n]$}) \\
&=& (-1)^{|I|} \sum_{y:I \rightarrow [n]} 
\sum_{\wq \leq \wp: \atop y \textrm{ is $\wq$-generic}}(-1)^{\dim \, \wq} 
= 
\sum_{y:I \rightarrow [n]} 
(\textrm{number of vertices of }\wp_{-y}).
\end{eqnarray*}
This gives the desired result since $\wp_{-y} = \wp_{(n+1, \ldots, n+1) - y}$, and $(n+1, \ldots, n+1) - y$ maps $I$ to $[n]$ if and only if $y$ maps $I$ to $[n]$.
\end{proof}

\begin{remark}
Propositions \ref{prop:basicinv} and \ref{prop:basicrecip} were also obtained by Billera, Jia, and Reiner in \cite{billera06}; their proof of the basic invariant reciprocity of Proposition \ref{prop:basicrecip} relies on Stanley's combinatorial reciprocity theorem for $P$-partitions. 
Our approach is different: we choose to give Hopf-theoretic proofs of these results. This will allow us to  give straightforward derivations of 
various combinatorial reciprocity theorems,
using only the Hopf-theoretic structure of $\wGP$; we do this in the following section.
\end{remark}

\begin{remark}\label{r:basicGP_+} The results of this section also hold for 
the Hopf monoid $\wGP_+$ of possibly unbounded generalized permutahedra. In that setting, we must set $\wp_y=0$ whenever the polyhedron $\wp$ is unbounded above in the direction of $y$. For a linear functional $y$ to be $\wp$-generic, we must require that the polyhedron $\wp$ is bounded above in the direction of $y$, and that $\wp_y$ is a point. 
\end{remark}

\section{Combinatorial reciprocity theorems for graphs, matroids, and posets}
\label{s:reciprocity}


We now show how characters on Hopf monoids naturally give rise to numerous reciprocity theorems in combinatorics; some old, some new. We would like to emphasize one benefit of this approach: this algebraic framework allows us to discover and prove reciprocity theorems automatically. All we have to do is define a character on a Hopf monoid, and the general theory will produce a polynomial invariant and a reciprocity theorem satisfied by it. In this section we will use some of the simplest possible characters to obtain several theorems of interest. 

This section is closely related to combinatorial Hopf algebras  \cite{absottile:_comb_hopf} and to \cite{billera06}.

\subsection{{The basic invariant of graphs is the chromatic polynomial}}

Given a graph $g$, an $n$-coloring of the vertices of $g$ is an assignment of a color in $[n]$ to each vertex of $g$. A coloring is 
\emph{proper} if any two vertices connected by an edge have different colors.

\begin{proposition}
Let $\zeta$ be the character on the Hopf monoid of graphs $\wG$ defined by
\[
\zeta_I(g) = 
\begin{cases}
$1$ &  \textrm{ if $g$ has no edges, and} \\
$0$ & \textrm{ otherwise.}
\end{cases}
\]
The corresponding polynomial invariant is the \emph{chromatic polynomial}, which equals
\[
\chi_I(g)(n)\ = \textrm{ number of proper colorings of $g$ with $n$ colors.}
\]
for $n \in \Nb$. 
\end{proposition}

\begin{proof}
The zonotope $Z_g$ is a point if and only if $g$ has no edges. Therefore, thanks to the inclusion $\wG^{cop}  \xhookrightarrow{} \wGP$ of Proposition \ref{p:graph-submod2},
when we restrict the basic character $\beta$ of $\wGP$ to graphic zonotopes, we obtain the character $\zeta$ of graphs. It follows that  $\chi_I(g)$ is the basic invariant of the graphic zonotope $Z_g$, and
 Proposition \ref{prop:basicinv} then tells us that $\chi_I(g)(n)$ is the number of $Z_g$-generic functions $y: I \rightarrow [n]$. By (\ref{eq:faceZg}), a function $y: I \rightarrow [n]$ is $Z_g$-generic if and only if $y(i) \neq y(j)$ whenever $\{i,j\}$ is an edge of $g$; that is, if and only if $y$ is a proper coloring of $g$. The result follows.
\end{proof}

We say that an $n$-coloring $y$ of $g$ and an acyclic orientation $o$ of the edges of $g$ are \emph{compatible} if we have $y(i) \geq y(j)$ for every directed edge $i \rightarrow j$ in the orientation $o$.

\begin{corollary} \label{cor:coloringrecip} \emph{(Stanley's reciprocity theorem for graphs \cite{stanley73:_acycl})}
Let $g$ be a graph on vertex set $I$, and $n \in \Nb$. Then $(-1)^{|I|} \chi_I(g)(-n)$ equals the number of compatible pairs of an $n$-coloring and an acyclic orientation of $g$. In particular,  $(-1)^{|I|} \chi_I(g)(-1)$ is the number of acyclic orientations of $g$.
\end{corollary}

\begin{proof}
This result is a special case of Proposition \ref{prop:basicrecip}. To see this, regard an $n$-coloring $y$ of $g$ as a linear functional $y: I \rightarrow [n]$ on the zonotope $Z_g$. This coloring induces a partial orientation $o_y$ of the edges of $g$, assigning an edge $\{i,j\}$ the direction $i \rightarrow j$ whenever $y(i) > y(j)$. By (\ref{eq:faceZg}), the vertices of $(Z_g)_y$ correspond to the acyclic orientations that extend $o_y$; these are precisely the acyclic orientations of $g$ compatible with $y$. 
\end{proof}

\subsection{{The basic invariant of matroids is the Billera-Jia-Reiner polynomial.}}
Given a matroid $m$ on $I$, say a function $y:I \rightarrow [n]$ is \emph{$m$-generic} if $m$ has a unique \emph{$y$-maximum basis} $\{b_1, \ldots, b_r\}$ maximizing $y(b_1) + \cdots + y(b_r)$.

\begin{proposition}
Let $\zeta$ be the character on the Hopf monoid of matroids $\wM$ defined by
\[
\zeta_I(m) = 
\begin{cases}
$1$ &  \textrm{ if $m$ has only one basis, and} \\
$0$ & \textrm{ otherwise.}
\end{cases}
\]
The corresponding polynomial invariant is the \emph{Billera-Jia-Reiner polynomial} of a matroid, which equals
\[
\chi_I(m)(n)\ := \textrm{ number of $m$-generic functions $y: I \rightarrow [n]$}
\]
for $n \in \Nb$. 
\end{proposition}

\begin{proof}
The matroid polytope of $m$ is a point if and only if $m$ has only one basis. Therefore, thanks to the  inclusion $\wM   \xhookrightarrow{} \wGP$ of Proposition \ref{p:matroid-submod}, 
when we restrict the basic character $\beta$ of $\wGP$ to matroid polytopes, we obtain the character $\zeta$ of matroids. It follows that $\chi_I(m)$ The result now follows by applying Proposition \ref{prop:basicinv} to matroid polytopes.
\end{proof}

\begin{corollary} \label{cor:orderrecip} 
\emph{(Billera-Jia-Reiner's reciprocity theorem for matroids \cite{billera06})}
Let $m$ be a matroid on $I$ and $n \in \Nb$. Then 
\[
(-1)^{|I|} \chi_I(m)(-n)\ = \sum_{y: I \rightarrow [n]} 
 \textrm{ (number of $y$-maximum bases of $m$)}.
\]
\end{corollary}

\begin{proof}
This is the result of applying Proposition \ref{prop:basicrecip} to matroid polytopes.
\end{proof}

\subsection{{The basic invariant of posets is the strict order polynomial}}

Given a poset $p$, say a map $y:p \rightarrow [n]$ is \emph{order-preserving} if $y(i) \leq y(j)$ whenever $i < j$ in $p$. Say $y$ is \emph{strictly order-preserving} if $y(i) < y(j)$ whenever $i < j$ in $p$.

\begin{proposition}
Let $\zeta$ be the character on the Hopf monoid of posets $\wP$ defined by
\[
\zeta_I(p) = 
\begin{cases}
$1$ &  \textrm{ if $p$ is an antichain, and} \\
$0$ & \textrm{ otherwise.}
\end{cases}
\]
The corresponding polynomial invariant is the \emph{strict order polynomial}, which equals
\[
\chi_I(p)(n)\ := \textrm{ number of strictly order-preserving maps $p \rightarrow [n]$.}
\]
for $n \in \Nb$. 
\end{proposition}

\begin{proof}
The poset cone $\Pc(p)$ is a point if and only if $p$ is an antichain.
Therefore, thanks to the  inclusion $\wP   \xhookrightarrow{} \wGP_+$ of Proposition \ref{p:PtoSF} (see Remark \ref{r:basicGP_+}), 
when we restrict the basic character $\beta$ of $\wGP_+$ to poset cones, we obtain the character $\zeta$ of posets. It follows that $\chi_I(p)(n)$ is the number of $\Pc(p)$-generic functions $y: p \rightarrow [n]$.
Now, thanks to Proposition \ref{prop:posetconegenerators}, the normal fan to $\Pc(p)$ is a single cone cut out by the inequalities $y(i) \leq y(j)$ for $i > j$ in $p$, so the $p$-generic functions are precisely the strictly order-reversing maps. It remains to note that there is a natural bijection between order-reversing maps $I \rightarrow [n]$ and order-preserving maps $I \rightarrow [n]$.
\end{proof}

\begin{corollary} \label{cor:orderrecip} \emph{(Stanley's reciprocity theorem for posets \cite{stanley1970chromatic})}
Let $p$ be a poset on $I$ and $n \in \Nb$. Then $(-1)^{|I|} \chi_I(p)(-n)$ is the \emph{order polynomial} of $p$, that is,
\[
(-1)^{|I|} \chi_I(p)(-n)\ = \textrm{ number of order-preserving maps $p \rightarrow [n]$.}
\]
\end{corollary}

\begin{proof}
This is a consequence of Proposition \ref{prop:basicrecip} and the following observations.
The poset cone $\Pc(p)$ only has one vertex, namely, the origin. If $y: p \rightarrow [n]$ is order-reversing, then there is a $y$-maximum face $\Pc(p)_y$, and it contains that single vertex. If $y$ is not order-reversing, then $\Pc(p)$ is not bounded above in the direction of $y$. 
\end{proof}

\subsection{{The Bergman polynomial of a matroid.}}\label{ss:Bergmanpoly}

A \emph{loop} in a matroid is an element which is not contained in any basis. 

\begin{definition}
The \emph{Bergman character} $\gamma$ of the Hopf monoid of matroids $\wM$ is given by
\[
\gamma_I(m) = 
\begin{cases}
1 & \textrm{ if $m$ has no loops}\\
0 & \textrm{ otherwise}.
\end{cases}
\]
for a matroid $m$ on $I$. 
The \emph{Bergman polynomial} $B(m)$ of a matroid $m$ is the  invariant associated to $\gamma$ by Proposition \ref{p:pol-inv} and (\ref{e:pol-inv}).
\end{definition}

Note that $\gamma$ is indeed a character, because a direct sum of matroids $m \oplus n$ is loopless if and only if $m$ and $n$ are both loopless.
To study the Bergman polynomial, we need some definitions. 
A \emph{flat} is a set $F$ of elements such that $r(F \cup i) > r(F)$ for every $i \notin F$. When $m$ is the matroid of a collection of vectors $A$ in a vector space $V$, the flats correspond to the subspaces of $V$ spanned by subsets of $A$. The flats form a lattice $L$ under inclusion, and the M\"obius number $\mu_L(\widehat{0}, \widehat{1})$ of this lattice (see \cite{ardila2015algebraic}, \cite[Chapter 3]{stanley11:_ec1}) is also called the \emph{M\"obius number of the matroid} $\mu(m)$.

We call $B(m)$ the \emph{Bergman polynomial} because it is related to the \emph{Bergman fan} 
\[
\Bc(m) = \{y \in \Rb^I \, : \, m_y \textrm{ has no loops}\}
\]
where $m_y$ is the matroid  whose bases are the $y$-maximum bases of $m$. Notice that the matroid polytope of $m_y$ is the $y$-maximum face of the matroid polytope of $m$; that is, $\Pc(m_y) = \Pc(m)_y$. Therefore $\Bc(m)$ is a polyhedral fan: it is a subfan of the normal fan of the matroid polytope $\Pc(m)$, consisting of the faces $\Nc_m(n)$ normal to the loopless faces $n$ of $m$.

Note also that $\Bc(m)$ is invariant under translation by $\1$ and under scaling by a positive constant. Therefore, nothing is lost by intersecting it with the hyperplane $\sum_i x_i=0$ and the sphere $\sum_i x_i^2 = 1$, to obtain the \emph{Bergman complex} $\widetilde{\Bc}(m)$. 

Bergman fans of matroids are central objects in tropical geometry, because they are the tropical analog of linear spaces. \cite{ardila2006bergman,sturmfels02} Two central results are the following combinatorial and topological descriptions.

\begin{theorem} \label{t:Bergman} \cite{ardila2006bergman}
Let $m$ be a matroid of rank $r$ on $I$. The Bergman fan $\Bc(m)$ has a triangulation into cones of the braid arrangement $\Bc_I$, consisting of the cones $\Bc_{S_1, \ldots, S_r}$ such that $S_1 \sqcup  \cdots \sqcup S_i$ is a flat of $m$ for $i = 1, \ldots, r$.
\end{theorem}

\begin{theorem} \label{t:topBergman} \cite{ardila2006bergman}
The Bergman complex of a matroid $m$ of rank $r$ is homeomorphic to a wedge of $(-1)^{r} \mu(m)$ spheres of dimension $r-2$, where $\mu(m)$ is the M\"obius number of $m$.
\end{theorem}

We now describe some of  the combinatorial properties of the Bergman polynomial. The first one is essentially equivalent to \cite[Example 4.15]{breuer2016scheduling}. 
Define a \emph{flag of flats} of $m$ to be an increasing chain of flats under containment $\emptyset = F_0 \subsetneq F_1 \subsetneq F_2 \subsetneq \cdots \subsetneq F_{n-1} \subsetneq F_n = \widehat{1}$. We call $n$ the length of the flag. Similarly, a \emph{weak flag of flats} to be a weakly increasing chain of flats.

\begin{proposition} \label{prop:Bergmaninv} 
At a natural number $n$, the \emph{Bergman polynomial} $B(m)$ of a matroid $m$ is given by
\[
B(m)(n) = \textrm{number of weak flags of flats of $m$ of length $n$} 
= \sum_{k=0}^r c_d {n \choose d},
\]
where $c_d$ is the number of flags of flats of $m$ of length $d$. Its degree is the rank $r$ of $m$.
\end{proposition}

\begin{proof}
We use the inclusion  $\wM  \xhookrightarrow{} \wGP$ to proceed geometrically. Let $\wp = \Pc(m)$ be the matroid polytope of $m$. 
The summand of $B_I(\wp)(n)$ in (\ref{e:pol-inv}) corresponding to a decomposition $I=S_1\sqcup \cdots \sqcup S_n$ equals
\[
(\gamma_{S_1}\otimes\cdots\otimes\gamma_{S_n})\circ \Delta_{S_1,\ldots,S_n} (\wp) = \gamma_{S_1}(\wp_1) \cdots \gamma_{S_n}(\wp_n) = \gamma_I(\wp_{T_1, \cdots, T_d})
\]
where $I = T_1 \sqcup \cdots \sqcup T_d$ is the composition obtained by removing all empty parts, and $\wp_{T_1, \cdots, T_d}$ is the $y$-maximal face of $\wp$ for any  $y \in \Bc_{T_1, \ldots, T_d}$. This term contributes $1$ to the sum if $\wp_{T_1, \cdots, T_d}$ is loopless and $0$ otherwise.

By Theorem \ref{t:Bergman}, 
$\wp_{T_1, \cdots, T_d}$ is loopless if and only if $\Bc_{T_1, \ldots, T_d}$ is in the Bergman fan of $m$, and 
 this is the case if and only if $\emptyset \subsetneq T_1  \subsetneq T_1 \cup T_2 \subsetneq \cdots \subsetneq T_1 \cup \cdots \cup T_d=I$ is a flag of flats. For fixed $n$ and $d$ there are $c_d$ choices for that flag of flats, and ${n \choose d}$ ways to enlarge the resulting composition $I = T_1 \sqcup \cdots \sqcup T_d$ into a decomposition $I = S_1 \sqcup \cdots \sqcup S_n$ by adding empty parts. This results in a weak flag of flats $\emptyset \subseteq S_1  \subseteq S_1 \cup S_2 \subseteq \cdots \subseteq S_1 \cup \cdots \cup S_n=I$ of length $n$. 
 
Since ${n \choose d}$ is a polynomial in $n$ of degree $d$, the degree of $B(m)$ is the largest possible length of a flag of flats of $m$, which is the rank $r$ of $m$.
 \end{proof}




\begin{proposition}\label{prop:Bergmanrecip} \emph{(Bergman polynomial reciprocity.)} 
The Bergman invariant  of a matroid $m$ of rank $r$ satisfies
\[
B(m)(-1) = (-1)^{r} \mu(m)
\]
where $\mu(m)$ is the M\"obius number of $m$.
\end{proposition}

\begin{proof}
Using Proposition \ref{p:reciprocity} and Theorem \ref{t:antipode} we get 
\begin{eqnarray*}
B(m)(-1) &=& \gamma_I(\apode_I(m)) = \sum_{n \textrm{ face of } m} (-1)^{|I| - \dim n} \gamma(n) \\
&=& \sum_{n \textrm{ face of } m \atop n \textrm{ loopless }} (-1)^{|I| - \dim n} 
= \sum_{F=\Nc_m(n) \atop \textrm{ face of } \Bc(m)} (-1)^{\dim F} = \overline{\chi}(\widetilde{\Bc}(m)), 
\end{eqnarray*}
the reduced Euler characteristic of the Bergman complex of $m$. The result now follows from Theorem \ref{t:topBergman}.
\end{proof}

\bigskip\bigskip
\begin{LARGE}
\noindent \textsf{PART 4: Hypergraphs and hypergraphic polytopes}
\end{LARGE}

\section{$\rHGP$: Minkowski sums of simplices, hypergraphs, Rota's question} \label{s:HGP}

In this section we focus on a large family of generalized permutahedra which we call \emph{hypergraphic polytopes} or \emph{Minkowski sums of simplices}. The polytopes in this family conserve the Hopf algebraic structure of $\rGP$ while featuring additional combinatorial structure, which makes them very useful for combinatorial applications, as we will see in Sections \ref{s:HG}, \ref{s:SC}, \ref{s:BS}, \ref{s:W},  \ref{s:Pirevisited}, and \ref{s:Frevisited}. In fact, $\rHGP$ is a useful source of old and new Hopf monoids: 
we start with some important subfamilies of generalized permutahedra --  namely hypergraphic polytopes, graphic zonotopes, simplicial complex polytopes, nestohedra, graph associahedra, permutahedra, and associahedra -- 
and we let them give rise to several interesting (and mostly new)  Hopf monoids of a more combinatorial nature, denoted  $\rHG, \rSHG, \rG, \rSC, \rBS, \rWBS, \rW, \rPi, \rF$, which consist of hypergraphs, simple hypergraphs, graphs, simplicial complexes, building sets, graphical building sets, simple graphs, set partitions, and paths, respectively. As we will see in the upcoming sections, these Hopf monoids are related as follows:\\ 

\[
\xymatrix{
\rPi \ar[rd] \xyinc[rd] &  & \rG^{cop} \ar[r] \xyinc[r] & \rHG^{cop} \ar@{->>}[rd]^s  \ar[r]^{\cong} & \rHGP   \ar[r] \xyinc[r]  \ar@{->>}[d]^{\supp} & \rGP    \\
& \rW^{cop}  \ar[r]^{\cong} & \rWBS^{cop}\ar[r]  \xyinc[r] &  \rBS^{cop} \ar[r] \xyinc[r]  &  \rSHG^{cop}  \\
\rF \ar[ur] \xyinc[ur] & &  & \rSC^{cop} \ar[ru] \xyinc[ru]  
}
\]

%
%


\subsection{{Minkowski sums of simplices}}\label{ss:minkowski}

%

We briefly mentioned in earlier sections that permutahedra, Loday's associahedra, and graphic zonotopes may be expressed as Minkowski sums of simplices. We now place these statements into a broader context, following Postnikov \cite{postnikov09}. 

Recall that the \emph{Minkowski sum} of two polytopes $P$ and $Q\subseteq \Rb I$
is
\[
P+Q := \{p+q \mid p \in P,\ q \in Q\} \subset  \Rb I
\]
Normal fans of polytopes behave well under scaling and Minkowski sums: the polytopes $P$ and $\lambda P$ have the same normal fan for $\lambda >0$, while the normal fan of $P+Q$ (and hence of $\lambda P+ \mu Q$ for $\lambda, \mu >0$) is the coarsest common refinement of the normal fans of $P$ and $Q$~\cite{sturmfels02}. 
It follows that
 if $P$ and $Q$ are generalized permutahedra, then so is $\lambda P+ \mu Q$ for $\lambda, \mu \geq 0$. 

Recalling from Theorem \ref{t:submod-gp} that every generalized permutahedron $\wp$ is associated to a unique  submodular function $z$ such that $\wp = \Pc(z)$, the previous statement has the following counterpart. 
If $z$ and $z'$ are submodular functions, then so is $\lambda z+\mu z'$ for $\lambda, \mu \geq 0$, and
\begin{equation} \label{eq:z+z'}
 \lambda \Pc(z) + \mu \Pc(z') = \Pc(\lambda z + \mu z').
\end{equation}


Let $\Delta_I = \text{conv}\{e_i\, :\,  i \in I\}$
be the \emph{standard simplex} in $\Rb I$. Let 
\[
\Delta_J = \text{conv}\{e_i \, : \, i \in J\} \qquad \textrm{ for } J \subseteq I
\]
be the faces of $\Delta_I$; note that the face $\Delta_J$ is itself the standard simplex in $\Rb J$. The following proposition is a consequence of (\ref{eq:z+z'}). 

\begin{proposition} \label{p:minkowski} (\cite[Proposition~6.3]{postnikov09})
If $y: 2^I \to \Rb_{\geq 0}$ is a non-negative Boolean function then the Minkowski sum $\sum_{J \subseteq I} y(J) \Delta_J$ of dilations of faces of the standard simplex in $\Rb I$ is a generalized permutohedron. We have
\begin{equation} \label{e:y-positive}
\sum_{J \subseteq I} y(J)\Delta_J = \Pc(z),
\end{equation}
where $z$ is the submodular function given by 
\[
z(J) = \sum_{K \cap J \neq \emptyset} y(K) \quad \textrm{ for each } J \subseteq I.
\]
Furthermore, if a polytope can be written in the form (\ref{e:y-positive}), then there is a unique choice of $y$ that makes this equation hold.\footnote{In fact, every generalized permutahedron can be expressed uniquely as a \emph{signed Minkowski sum} $\sum_{J \subseteq I} y(J)\Delta_J$ where $y(J)$ is allowed to be negative, but the definitions become more subtle. We will not pursue this point of view here; for more information, see \cite[Proposition~2.3]{abd10}.}
\end{proposition}

\begin{definition} A generalized permutahedron $\wp$ is \emph{$y$-positive} if it is given by (\ref{e:y-positive}) for a non-negative Boolean function  $y: 2^I \to \Rb_{\geq 0}$. 
If, additionally, $y(J)$ is an integer for all $J \subseteq I$, we call $\wp$ a \emph{Minkowski sum of simplices} or a \emph{hypergraphic polytope}.
\end{definition}

We should say a word about this nomenclature. A \emph{hypergraph} $\H$ on $I$ is a collection of (possibly repeated) subsets of $I$, called the \emph{multiedges} of $\H$. Our convention will be that the empty set appears exactly once in $\H$. Then there is a natural bijection between  hypergraphs and hypergraphic polytopes: to a hypergraph $\H$ on $I$ containing $y(J)$ copies of the subset $J \subseteq I$, we associate the hypergraphic polytope $\Delta_\H = \sum_{H \in \H} \Delta_H  = \sum_{J \subseteq I} y(J) \Delta_J$.

\begin{remark}\label{r:y-positive}
We saw in Theorem \ref{t:submod-gp} that there is a one-to-one correspondence between generalized permutahedra in $\Rb^n$ and submodular functions, which naturally form a polyhedral cone in $\Rb^{2^n-1}$. The $y$-positive generalized permutahedra form a polyhedral subcone of this submodular cone, which is full-dimensional since it is parameterized by $2^n-1$ independent parameters. The inequalities defining this subcone will be given  in Proposition \ref{p:relation}.2. It would be interesting to compute the probability that a  generalized permutahedron in $\Rb^n$ is $y$-positive under a suitable probability measure, and to describe how that probability varies with $n$. 
\end{remark}

Many polytopes of interest are hypergraphic, although that is not always apparent at the outset. For example, graphic zonotopes, permutahedra, and associahedra turn out to be hypergraphic, but this is not clear from their definitions. We will see many other examples in the upcoming sections.



\subsection{{Relations, hypergraphic polytopes, and Rota's question.}}\label{ss:relations}

A relation $R \subseteq I \times J$ gives rise to a function $f_R : 2^I \rightarrow \mathbb{N}$ defined by
\[
f_R(A) = |R(A)| = |\{b \in B  \, | \, (a,b) \in R \textrm{ for some } a \in A\}| \qquad \textrm{ for } A \subseteq I.
\]
Let us call such a function \emph{relational}. One may verify that every relational function is submodular, and 
Rota \cite[Problem 2.4.1(d)]{rota09} asked for a characterization of these \emph{relational submodular functions}:

\begin{quote}
There is an interesting open question which ought to have been worked out, and that I ought to have worked out, but I haven't:  Characterize those submodular set functions that come from a relation in this way. \cite[Exercise 18.1]{rota18.315}
\end{quote}

It seems clear that Rota knew how to do this, and it is quite possible that others have carried out this computation, but we have not been able to find a precise statement in the literature. We offer the following characterizations.

\begin{proposition}\label{p:relation}
A submodular function $f: 2^I \rightarrow \Rb$ is relational if and only if either of the following conditions hold:
\begin{enumerate}
\item 
Its associated polytope $\Pc(f)$ is hypergraphic.
\item
$f(\emptyset) = 0$ and for all $A \subseteq I$ we have $f(A) \in \Zb$ and
\[
\sum_{K \supseteq A} (-1)^{|K-A|}f(K) \leq 0.
\]
\end{enumerate}
\end{proposition}

\begin{proof}
1. A relation $R \subseteq I \times J$ naturally gives rise to a hypergraph $\H_R$ on $I$ whose hyperedges $h_j = \{i \, : \, (i,j) \in R\}$ for $j \in J$ are given by the columns of $R$. Clearly any hypergraph on $I$ arises in this way from a relation. If $y_R(K)$ is the multiplicity of hyperedge $K$ in $\H_R$ then 
%
\begin{equation}\label{e:ytof}
f_R(A) = \sum_{K \cap A \neq \emptyset} y_R(K)
\end{equation}
for all $A \subseteq I$. Proposition \ref{p:minkowski} then gives
\[
\Pc(f_R) = \sum_{J \subseteq I} y_R(K)\Delta_K.
\]
which is a Minkowski sum of simplices. Conversely, given such a Minkowski sum, we can use its coefficients as the multiplicities of a hypergraph which gives rise to the desired relation.

%

\medskip

\noindent 2. The submodular function of a relation $R$ clearly satisfies $f_R(\emptyset) = 0$. We rewrite (\ref{e:ytof}) as
$
f_R(A) =  |J| - \sum_{K \subseteq I-A} y_R(K)
$
and use the inclusion-exclusion formula to obtain
\[
y_R(B) = \sum_{K \subseteq B} (-1)^{|B-K|}(|J| - f_R(I-K)) = - \sum_{K \subseteq B} (-1)^{|B-K|}f_R(I-K)
\]
for $B \neq \emptyset$. 
Therefore
\begin{equation}\label{e:ftoy}
y_R(I-A) = - \sum_{K \supseteq A} (-1)^{|K-A|} f_R(K) \geq 0
\end{equation}
Conversely, for any integral function $f$ satisfying the given inequalities, (\ref{e:ftoy}) gives us a non-negative function $y: 2^I \rightarrow \Zb$. We then construct the desired relation $R \subseteq I \times J$ as in part 1: for each $K \subseteq I$ we include $y(K)$ elements $j$ in $J$ such that $h_j=K$.
%
\end{proof}


We wish to study these objects further, following the yoga of Joni and Rota's paper \cite{joni82:_coalg}: we will  describe their Hopf algebraic structure in Sections \ref{ss:HopfHGP} and \ref{s:HG}. This will turn out to be a crucial ingredient for the rest of the paper. 


\subsection{The Hopf monoid of hypergraphic polytopes} \label{ss:HopfHGP}

\begin{proposition}\label{p:HGP}
The hypergraphic polytopes form a submonoid $\rHGP$ of the Hopf monoid of generalized permutahedra $\rGP$.
\end{proposition}

\begin{proof}
Let $I = S \sqcup T$ be a decomposition. To prove $\rHGP$ is a submonoid of $\rGP$ we need to prove two things:

\noindent $\bullet$ If polytopes $\wp$ and $\wq$ are hypergraphic in $\Rb S$ and $\Rb T$, then $\wp \times \wq$ is hypergraphic in $\Rb I$.

\noindent $\bullet$ If $\wp$ is hypergraphic in $\Rb I$,
then $\wp|_S$
and $\wp/_S$ are hypergraphic in $\Rb S$ and $\Rb T$, respectively.

For the first statement, if $\wp = \sum_{J \subseteq S} z_1(J) \Delta_J \subset \Rb S$ and  $\wq = \sum_{K \subseteq T} z_2(K) \Delta_K \subset \Rb T$ are Minkowski sums of simplices, then 
\begin{equation}\label{eq:y-possum}
\wp \times \wq = \wp + \wq =  \sum_{J \subseteq S} z_1(J) \Delta_J +  \sum_{K \subseteq T} z_2(K) \Delta_K \subset \Rb I
\end{equation}
is also a Minkowski sum of simplices.

For the second one, we use that $(P+Q)_v = P_v+Q_v$ for any polytopes $P, Q \subseteq\Rb I$ and any linear functional $v \in \Rb^I$. Now, the maximal face of the simplex $\Delta_J$ in direction $\1_S$ is
\[
(\Delta_J)_{S,T}=
\begin{cases}
\Delta_{J \cap S} & \textrm{ if } J \cap S \neq \emptyset \\
\Delta_{J} & \textrm{ if } J \cap S = \emptyset.
\end{cases}
\]
Therefore if
$\wp = \sum_{J \subseteq I} y(J)\Delta_J \subset \Rb I $ is a hypergraphic polytope, then its 
$\1_S$-maximal face is $\wp_{S,T} = \wp|_S + \wp/_S$ where
\begin{equation}\label{eq:y-pos}
\wp|_S = \sum_{J \cap S \neq \emptyset} y(J)\Delta_{J \cap S} \subset \Rb S, \qquad 
\wp/_S = \sum_{J \cap S = \emptyset} y(J)\Delta_J \subset \Rb T.
\end{equation}
Therefore $\wp|_S$ and $\wp/_S$ are hypergraphic, as desired.
\end{proof}

Since $\rHGP$ is a Hopf submonoid of $\rGP$, Theorem \ref{t:antipode} gives us a formula for the antipode of $\wHGP$. We write it down in Theorem \ref{t:antipodeHG} in terms of hypergraphs.

\section{$\rHG$: Hypergraphs}\label{s:HG}

Recall that a \emph{hypergraph}  with vertex set $I$ is a collection $\H$ of (possibly repeated) subsets of $I$. We will use the convention that there is always a single copy of $\emptyset$ in $\H$.\footnote{This is the opposite of the usual convention that $\emptyset \notin \H$.} We can think of each subset $H$ in $\H$ as a \emph{multiedge} which can now connect any number of vertices.

\subsection{{The Hopf monoid of hypergraphs}}\label{ss:HG}

Let $HG[I]$ be the set of all hypergraphs with vertex set $I$. Clearly $\rHG$ is a species, which we now turn into a Hopf monoid. 

Let $I = S \sqcup T$ be a decomposition. 

\noindent $\bullet$ For $\H_1 \in HG[S]$ and $\H_2 \in HG[T]$, define their product $\H_1 \cdot \H_2 \in HG[I]$ to be the disjoint union $\H_1 \sqcup \H_2$ as a hypergraph on $I$. 

\noindent $\bullet$ The coproduct of $\H \in HG[I]$ is $(\H|_S, \H/_S)$, where the restriction and contraction of $\H$ with respect to $S$ are the multisets
\begin{eqnarray*}
\H|_S & := &  \{H :\,  \, H \in \H, \,  H \subseteq S\} \\
\H/_S & := & \{H \cap T \, : \, H \in \H, \, H \nsubseteq S\} \cup \{\emptyset\}. 
\end{eqnarray*}
\noindent Each multiedge $H_S$ of $\H|_S$ has the same multiplicity that it had in $\H$, while the multiplicity of a non-empty multiedge $H_T$ of $\H/_S$ is the sum of the multiplicities of the edges $H \in \H$ such that $H \cap T = H_T$. 

\noindent The Hopf monoid axioms are easily verified.

\begin{example}
For the hypergraph $\H = \{\emptyset, 1, 2, 3, 12, 23, 123\}$ on $I=[3]$, we have 
\begin{eqnarray*}
\H|_{13} = \{\emptyset, 1, 3\},&& \qquad \H/_{13} = \{\emptyset, 2, 2, 2, 2\}, \\
\H|_{2} = \{\emptyset, 2\},&& \qquad \H/_{2} = \{\emptyset, 1, 1, 3, 3, 13\}.
\end{eqnarray*}
We  omit the brackets from the individual multiedges in $\H$ for clarity. 
\end{example}

\subsection{{Hypergraphs as a submonoid of generalized permutahedra}}\label{ss:HGPtoHG}

Recall that the \emph{hypergraphic polytope} of a hypergraph $\H$ on $I$ is the Minkowski sum
\[
\Delta_\H = \sum_{H \in \H} \Delta_H
\]
where $\Delta_H$ is the standard simplex in $\Rb H \subseteq \Rb I$.

\begin{example}\label{ex:123}
The hypergraphic polytope for the hypergraph $\H = \{\emptyset, 1, 2, 3, 12, 23, 123\}$ is $\Delta_\H = \Delta_1 + \Delta_2 + \Delta_3 + \Delta_{12} + \Delta_{23} + \Delta_{123}$, as shown in Figure \ref{f:hypergraphic}. 
\end{example}

\begin{figure}[h]
\centering
\includegraphics[scale=.45]{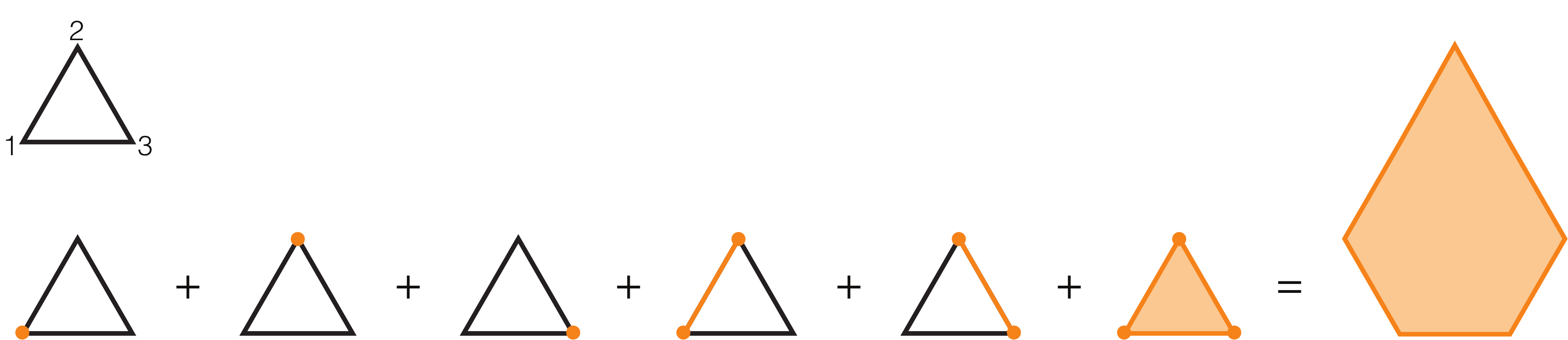} 
\caption{The hypergraphic polytope of the hypergraph $\H = \{\emptyset, 1, 2, 3, 12, 23, 123\}$.}
\label{f:hypergraphic}
\end{figure}


\noindent Let $\rHG^{cop}$ be co-opposite to the Hopf monoid  of hypergraphs $\rHG$, as defined in Section \ref{ss:hopf-set}.

\begin{proposition}\label{p:HGPtoHG}
The map $\H \mapsto \Delta_H$ gives an isomorphism $\rHG^\textrm{cop}  \map{\cong}  \rHGP$ between $\rHG^\textrm{cop}$ and the Hopf monoid of hypergraphic polytopes $\rHGP$.
\end{proposition}

\begin{proof}
We know that the map is bijective. The equation (\ref{eq:y-possum}) says that the map preserves the product and (\ref{eq:y-pos}), which may be rewritten as $(\Delta_\H)|_S = \Delta_{\H/_T}$ and $(\Delta_\H)/_S = \Delta_{\H|_T}$, says that the map reverses the coproduct. 
\end{proof}

%
%
%
%
%
%
%

\begin{example}\label{ex:123b}
For the hypergraphic polytope of Example \ref{ex:123} and Figure \ref{f:hypergraphic}, 
the northwest edge and southwest vertex are described by
\begin{eqnarray*}
(\Delta_\H)_{13,2} &=& (\Delta_1 +  \Delta_2 + \Delta_3 + \Delta_1 + \Delta_3 + \Delta_{13}) = \Delta_{\{\emptyset, 1, 1, 3, 3, 13\}}\times \Delta_{\{\emptyset, 2\}} = 
 \Delta_{\H/_{2}} \times \Delta_{\H|_{2}}\\
  (\Delta_\H)_{2,13} &=& (\Delta_1 +  \Delta_2 + \Delta_3 + \Delta_2 + \Delta_2 +\Delta_2) =  \Delta_{\{\emptyset, 2,2,2,2\}} \times  \Delta_{\{\emptyset, 1, 3\}}=
\Delta_{\H/_{13}} \times \Delta_{\H|_{13}},
\end{eqnarray*}
in (co-opposite) agreement with Example \ref{ex:123}.
\end{example}

\begin{theorem}\label{t:antipodeHG}
The antipode of the Hopf monoid of hypergraphs $\wHG$ is given by the following \textbf{cancellation-free} and \textbf{grouping-free} expression. If $\H$ is a hypergraph on $I$ then
\[
\apode_I(\H) = \sum_{\Delta_\G \leq \Delta_\H} (-1)^{c(\G)} \G,
\]
summing over all faces $\Delta_\G$ of the hypergraphic polytope $\Delta_\H$ of $\H$, where $c(\G)$ is the number of connected components of the hypergraph $\G$.
\end{theorem}

\begin{proof}
This is the result of applying Theorem \ref{t:antipode} to the submonoid $\wHGP$ of $\wGP$, taking into account the identification of $\wHGP$ and $\wHG$ of Proposition \ref{p:HGPtoHG} and the observation that $\dim \Delta_\G = |I| - c(\G)$. There is no cancellation or grouping in the right hand side of this equation because $\Delta_\G = \Delta_{\G'}$ implies $\G = \G'$.
\end{proof}

\begin{example}\label{e:antipodeHG}
The antipode of the hypergraph $\H= \{\emptyset, 1, 2, 3, 12, 23, 123\}$ in $\wHG$ is given by the hypergraphic polytope of Figure \ref{f:hypergraphic}, namely:
\[
\begin{array}{rl}
\apode_{[3]}(\H) =&
  \{\emptyset,1,2,3,12,23,123\} - \{\emptyset,1,2,3,1,23,1\} - \{\emptyset,1,2,3,1,3,13\} \\
& - \{\emptyset,1,2,3,12,3,3\} - \{\emptyset,1,2,3,2,23,23\}- \{\emptyset,1,2,3,12,2,12\}  \\
& + \{\emptyset,1,2,3,1,2,1\} + \{\emptyset,1,2,3,1,3,1\} + \{\emptyset,1,2,3,1,3,3\}\\
& + \{\emptyset,1,2,3,2,3,3\} + \{\emptyset,1,2,3,2,2,2\}.
\end{array}
\]
\end{example}

\subsection{{Graphs, revisited.}} \label{ss:graphsrev} 

We now give another explanation of the inclusion of $\rG^{cop}$ into $\rGP$ shown in Proposition \ref{p:graph-submod2}.

\begin{proposition}\label{p:GtoGP}
The map $g \mapsto Z_g$ is an injective morphism of Hopf monoids $\rG^{cop}  \into \rGP$.
\end{proposition}

\begin{proof}
Since the graph operations of $\rG$ defined in Section \ref{ss:graphs} are special cases of the hypergraph operations of $\rHG$ defined in Section \ref{ss:hyper}, we have an inclusion of Hopf monoids, $\rG \into \rHG$, which gives an inclusion $\rG^{cop} \into \rHG^{cop}$. Proposition \ref{p:HGPtoHG} tells us that the map $\H \mapsto \Delta_\H$
is an isomorphism $\rHG^{cop} \cong \rHGP$. By Proposition \ref{prop:Zgzonotope}, the composition of these maps is the map $\rG^{cop} \rightarrow \rHGP \into \rGP$ given by  $g \mapsto Z_g$. 
\end{proof}



\subsection{{Simple hypergraphs and simplification.}}\label{ss:hyper} A hypergraph is \emph{simple} if it has no repeated multiedges.\footnote{We allow simple hypergraphs to contain singletons, slightly against the usual convention.}
In the applications we have in mind, we are only interested in simple hypergraphs. Unfortunately, simple hypergraphs are not closed under the contraction map of $\rHG$, so the Hopf structure that we define on them requires a slightly different contraction map. 
Let $\rSHG[I]$ be the set of all simple hypergraphs with vertex set $I$.

\newpage
Let $I = S \sqcup T$ be a decomposition.

\noindent $\bullet$ The product of $\H_1 \in \rSHG[S]$ and $\H_2 \in \rSHG[T]$ is their disjoint union $\H_1 \sqcup \H_2$.

\noindent $\bullet$ The coproduct of $\H \in \rSHG[I]$ is $(\H|_S, \H/_S)$, where the restriction and contraction of $\H$ with respect to $S$ are:
\begin{eqnarray*}
\H|_S & := &  \{H \, : \, H \in \H,\,  H \subseteq S\} \\
\H/_S & := & \{H \cap T \, : \, H \in \H,\,  H \nsubseteq S\} \cup \{\emptyset\} =  \{ B \subseteq T \, : \, A \sqcup B \in \H \text{ for some } A \subseteq S\},
\end{eqnarray*}
now regarded as sets without repetition.

One easily verifies that the  \emph{simplification} maps, which remove any repetitions of multiedges in a hypergraph, give a morphism of Hopf monoids $s: \rHG \onto \rSHG$. We now show that this map behaves reasonably well with respect to the corresponding polytopes.
Define $\rbHGP \subseteq \rbGP$ to be the quotient of $\rHGP$ obtained by identifying hypergraphic polytopes with the same normal fan. 

\begin{proposition}\label{p:HGtoSHG}
We have a commutative diagram of Hopf monoids as follows.
\[
\xymatrix{
\rHG^{cop} \ar@{->>}[d]_{s}  \ar@{<->}[r]^-{\cong} & \rHGP \ar@{->>}[d] \\
\rSHG^{cop} \ar@{->>}[r]  & \rbHGP
}
\]
\end{proposition}

\begin{proof}
The two vertical maps are defined in the previous paragraph, while the top map is $\H \mapsto \Delta_\H$. It remains to verify that the bottom map that makes this diagram commute is well-defined: if $\H$ is a hypergraph, the normal fan $\Nc_{\Delta_{\H}}$ is the common refinement of $\Nc_{\Delta_H}$ as we range over all $H \in \H$; this only depends on the simplification of $\H$.
\end{proof}

\begin{remark}\label{r:normalfan}
The bottom map  $\rSHG^{cop}  \xhookrightarrow{} \rbHGP$ 
of Proposition \ref{p:HGtoSHG} is not an isomorphism. For example, 
 $\Delta_{\{\emptyset, 12, 13, 23\}}$ and $\Delta_{\{\emptyset, 12, 13, 23, 123\}}$
are hexagons with the same normal fan. More generally, 
for any simple hypergraph $\H$ on $I$ containing all pairs $\{i.j\}$ with $i,j \in I$, the hypergraphic polytope $\Delta_\H$ is normally equivalent to the standard permutahedron $\pi_I$. To see this, notice that the normal fan of $\Delta_{\H}$ coarsens the braid arrangement (since $\Delta_\H$ is a generalized permutahedron) and refines the braid arrangement (since it has $\pi_I = \sum_{\{i, j\}\subseteq I} \Delta_{\{i,j\}}$ as a Minkowski summand).
\end{remark}

\subsection{{The support maps.}} \label{ss:support} 

The \emph{support maps} $\supp_I: \rHGP[I] \to \rSHG[I]$ will be an important tool in what follows; they take a hypergraphic polytope $\wp = \Delta_\H = \sum_{J \subseteq I} y(J) \Delta_J \subseteq \Rb I$ to the simple hypergraph supporting it:
\[
\text{supp}_I (\wp) := \{J \subseteq I \, : \, y(J) > 0\} \cup \{\emptyset\}.
\]
Under the isomorphism $\rHGP \cong \rHG^{cop}$ of Proposition \ref{p:HGPtoHG} which identifies $\wp$ with its corresponding hypergraph $\H$, the support $\text{supp}_I (\wp)$ is the \emph{simplification} of $\H$.

\begin{theorem}\label{t:HGPtoSHG}
The support maps $\supp_I: \rHGP[I] \to \rSHG[I]$ give a surjective morphism of Hopf monoids $\text{supp}: \rHGP \onto \rSHG^{cop}$.
\end{theorem}

\begin{proof}
This morphism is the composition of the top isomorphism with the simplification map $s$ in Proposition \ref{p:HGtoSHG}. 
%
\end{proof}

\begin{theorem}\label{t:antipodeSHG}
The antipode of the Hopf monoid of simple hypergraphs $\wSHG$ is given by the following \textbf{cancellation-free} expression. If $\H$ is a simple hypergraph on $I$ then
\[
\apode_I(\H) = \sum_{F \leq \Delta_\H} (-1)^{c(F)} \supp_I(F),
\]
summing over all faces $F$ of the hypergraphic polytope $\Delta_\H$ of $\H$, where $c(F) = |I|-\dim F$ is the number of connected components of the hypergraph $\supp_I(F)$.
\end{theorem}

\begin{proof}
Thanks to Proposition \ref{p:antipode2}, the surjective maps $\supp$ turn Theorem \ref{t:antipodeHG}, our formula  for the antipode of $\wHG^{cop} \cong \wHGP$, into a formula for the  antipode of $\wSHG$.
The formula is cancellation free because faces of different dimension must have different support.
\end{proof}

\begin{example}\label{e:antipodeSHG}
The antipode of the hypergraph $\H= \{\emptyset, 1, 2, 3, 12, 23, 123\}$ in $\wSHG$ is also given by the hypergraphic polytope of Figure \ref{f:hypergraphic}, but the result is now the simplification of the one in  Example \ref{e:antipodeHG}:
\[
\begin{array}{rl}
\apode_{[3]}(\H) =&
\H - 2 \{\emptyset,1,2,3,23\} - 2\{\emptyset,1,2,3,12\} - \{\emptyset,1,2,3,12\} + 5 \{\emptyset, 1, 2, 3\}
  \end{array}
\]
\end{example}

As in the case of matroids, we have no simple combinatorial labeling of the faces of a general hypergraphic polytope, so we do not have a way of simplifying the formula of Theorem \ref{t:antipodeSHG}. This shows that hypergraphic polytopes are fundamental in the Hopf structure of hypergraphs.

However, we do know a few families of hypergraphic polytopes whose combinatorial structure we can describe more explicitly; they give rise to interesting combinatorial families which inherit Hopf monoid structures  from their polytopes. In the remaining sections of the paper, we will describe the resulting Hopf monoids and use Theorem \ref{t:antipodeSHG} to describe their antipodes.

%
%

%
%

\section{$\rSC$: {Simplicial complexes, graphs, and Benedetti et al.'s formula}} \label{s:SC} 

Benedetti, Hallam, and Machacek \cite{benedetti2016combinatorial} constructed a combinatorial Hopf algebra of simplicial complexes, and obtained a formula for its antipode through a clever combinatorial argument. Surprisingly, the formula is almost identical to Humpert and Martin's formula for the antipode of the Hopf algebra of graphs \cite{humpert12}. In this section, by modeling simplicial complexes polytopally, we are able to offer a simple geometric explanation of this phenomenon.

A(n abstract) \emph{simplicial complex} on a finite set $I$ is a collection $\C$ of subsets of $I$, called \emph{faces}, such that any subset of a face is a face; that is, if $J \in \C$ and $K \subseteq J$ then $J \in \C$. 
For a subset $J \subseteq I$, the \emph{induced simplicial complex} $\C|_J$ consists of the faces of $\C$ which are subsets of $J$.

\subsection{The Hopf monoid of simplicial complexes} \label{ss:SC}

Let $\rSC[I]$ denote the set of all simplicial complexes on $I$. 
We turn the set species $\rSC$  into a commutative and cocommutative Hopf monoid with the following structure. 

Let $I=S\sqcup T$ be a decomposition. 

\noindent $\bullet$
The product of two simplicial complexes $\C_1\in\rSC[S]$ and $\C_2\in\rSC[T]$ is their disjoint union. 

\noindent $\bullet$
The coproduct of a simplicial complex $\C \in \rSC[I]$ is $(\C|_S, \C|_T)$. 

\noindent The Hopf monoid axioms are easily verified.

At first sight, this Hopf monoid -- which is cocommutative -- does not seem related to the Hopf monoids of hypergraphs -- which are not cocommutative. However, it turns out that $\rSC$ lives inside the cocommutative part of $\rSHG$.

\begin{proposition} \label{p:SHGtoSC}
The Hopf monoid of simplicial complexes $\rSC$ is a submonoid of the Hopf monoid of simple hypergraphs $\rSHG$.
\end{proposition}

\begin{proof}
Simplicial complexes are simple hypergraphs, and the product and restriction operations for these two families coincide. The contraction operations are defined slightly differently. However, if $\C$ is a simplicial complex and  $I=S \sqcup T$ is a decomposition, one may verify that the contraction $\C/_S$ in the sense of simple hypergraphs coincides with the restriction $\C|_T$ in the sense of simplicial complexes. 
\end{proof}

\subsection{Simplicial complex polytopes} \label{ss:SCpolytopes} 

Each simplicial complex $\C$, being a hypergraph, has a corresponding hypergraphic polytope $\Delta_\C := \sum_{C \in \C} \Delta_C$. Unlike
general hypergraphic polytopes, this family of polytopes have a simple combinatorial facial structure.

Recall that the \emph{one-skeleton} $\C^{(1)}$ of a simplicial complex on $I$ is the graph on $I$ whose edges are the sets in $\C$ of size $2$.

\begin{lemma}\label{l:SCGP}
For any simplicial complex $\C$, the hypergraphic polytope $\Delta_\C$ is normally equivalent to the graphic zonotope $Z_{\C^{(1)}}$ of its one-skeleton $\C^{(1)}$.
\end{lemma}

\begin{proof}
We use the central fact from Proposition \ref{p:HGtoSHG} that the normal equivalence class of a hypergraphic polytope $\Delta_\H$ depends only on the support $\supp (\Delta_\H)$.

Let $\C$ be a simplicial complex on $I$. 
Since they have the same support, the simplicial complex polytope 
$\Delta_\C = \sum_{F \in \C} \Delta_F$ is normally equivalent to the polytope
\[
P_1 = 
\sum_{G \in \C} \sum_{F \subseteq G} \Delta_F =
\sum_{G \in \C} \pi'_G.
\]
where we define $\pi'_G := \sum_{F \subseteq G} \Delta_F$ for each set $G \in \C$. By Remark \ref{r:normalfan}, $\pi'_G$ is normally equivalent to the standard permutahedron $\pi_G$ in $\Rb I$. Therefore the polytope $P_1$ is normally equivalent to
\[
P_2 = \sum_{G \in \C} \pi_G = 
\sum_{G \in \C} \sum_{\{i,j\} \subseteq G} \Delta_{\{i,j\}}
\]
using (\ref{e:Pizonotope}). In turn, $P_2$ is normally equivalent to $
Z_{\C^{(1)}} = \sum_{\{i,j\} \in \C} \Delta_{\{i,j\}}
$
since they have the same support.
\end{proof}

As a consequence, the simplicial complex polytope $\Delta_\C$ has the same facial structure as the zonotope $Z_g$ for $g=\C^{(1)}$, as described by Lemma \ref{l:facesZg}.
It would be interesting to further study these simplicial complex polytopes.

\subsection{The antipode of simplicial complexes} \label{ss:antipodeSC} 

Since simplicial complexes form a submonoid of simple hypergraphs by Proposition \ref{p:SHGtoSC}, we may use Theorem \ref{t:antipodeSHG} to compute the antipode of $\wSC$, thus recovering the formula of Benedetti, Hallam, and Machacek. \cite{benedetti2016combinatorial} We now carry this out.

%
%
%

Let $\C$ be a simplicial complex $\C$ on $I$ and let $f$ be a flat of the $1$-skeleton $\C^{(1)}$ of $\C$. The flat $f$ is a subgraph of $\C^{(1)}$, and its connected components form a partition $\pi = \{\pi_1, \ldots, \pi_k\}$ of its vertex set $I$. As before, we let $c(f)=k$ denote the number of connected components of $f$. We define $\C(f) = \C|_{\pi_1}  \sqcup \cdots \sqcup \C|_{\pi_k}$ 
to be the subcomplex of $\C$ consisting of the faces which are contained in a connected component of $f$.

\begin{corollary}\cite{benedetti2016combinatorial} \label{c:antipodeSC}
The antipode of the Hopf monoid of simplicial complexes $\wSC$ is given by the following \textbf{cancellation-free} and \textbf{grouping-free} expression. If $\C$ is a simplicial complex on $I$ then
\[
\apode_I(\C) = \sum_{f} (-1)^{c(f)} a(g/f) \, \C(f),
\]
summing over all flats $f$ of the $1$-skeleton $g = \C^{(1)}$ of $\C$, where $a(g/f)$ is the number of acyclic orientations of the contraction $g/f$. 
\end{corollary}

\begin{proof}
By Theorem \ref{t:antipodeSHG}, the antipode of $\C$ is given by the face structure of the polytope $\Delta_\C$, which is equivalent to the face structure of the zonotope $Z_g$ by Lemma \ref{l:SCGP}. 
Lemma \ref{l:facesZg} tells us that the faces of these polytopes are in bijection with the pairs of a flat $f$ of $g$ and an acyclic orientation $o$ of $g/f$. 
Recall from that proof that the maximal face $(\Delta_\C)_y$ in a direction $y \in \Rb^I$ depends only on a flat $f=f_y$ of $g$ and an orientation $o=o_y$ of $g/f$ determined by $y$. 
The flat $f$ of $g$ consists of the edges $ij$ such that $y(i)=y(j)$;
the acyclic orientation $o$ of $g/f$ will be irrelevant here. 

The corollary will now follow from the claim that the support of the $(|I|-c)$-dimensional face $(\Delta_\C)_y$ equals $\C(f)$, independently of the choice of $o$.
To prove this claim, we will use the following expressions:
\begin{equation}\label{e:DeltaC}
(\Delta_\C)_y = \sum_{C \in \C} (\Delta_C)_y, \qquad
\Delta_{\C(f)} = \sum_{C \in \C \, : \atop y \textrm{ is constant on } C} \Delta_C.
\end{equation}
We will show that they have the same summands, possibly with different multiplicities.
 
\noindent $\longrightarrow$: 
For each $C \in \C$ we have $(\Delta_C)_y = \Delta_{C_{max}}$ where $C_{max} = \{c \in C \, | \, y(c) \textrm{ is maximum}\}$. Clearly $y$ is constant on $C_{max}$, so this is a summand of $\Delta_{\C(f)}$. 
 
\noindent $\longleftarrow$: 
For any summand $\Delta_C$ of $\Delta_{\C(f)}$, $C$ is a face of the simplicial complex $\C$ where $y$ is constant, so $\Delta_C = \Delta_{C_{max}} = (\Delta_C)_y$ is a summand of $(\Delta_{\C})_y$.
 
\noindent  
This proves the claim that $\supp (\Delta_\C)_y = \C(f)$, and the desired result follows.
\end{proof}

The proof above gives a simple geometric explanation for the striking similarity between the antipode formulas for the Hopf algebra of graphs $\wG$ and the Hopf algebra of simplicial complexes $\wSC$: these formulas have the same combinatorial structure because they  are controlled by polytopes that are normally equivalent.

\section{$\rBS$: {Building sets and nestohedra}}\label{s:BS}

In this section we study \emph{building sets}, a second family of hypergraphs whose hypergraphic polytope has an elegant combinatorial structure. This allows us to describe the Hopf theoretic structure of building sets very explicitly.

Building sets were introduced independently and almost simultaneously in two very different contexts by De Concini and Procesi \cite{deconcini95} in their construction of the wonderful compactification of a hyperplane arrangement, and by Schmitt \cite{schmitt95:_hopf} (who called them \emph{Whitney systems}) in an effort to abstract the notion of connectedness. We follow \cite{postnikov09}; see also  \cite{feichtner04, feichtner05,grujic2012hopf}.

\begin{definition} \label{d:buildingset}  A collection $\B$ of subsets of a set $I$ is a \emph{building set} on $I$ if it satisfies the following conditions:

\noindent $\bullet$
If $J, K \in \B$ and $J \cap K \neq \emptyset$ then $J \cup K \in \B$

\noindent $\bullet$
For all $i \in I$, $\{i\} \in \B$.

\noindent We call the sets in $\B$ \emph{connected}.
\end{definition}

We call the maximal sets of a building set $\B$ its \emph{connected components}; one may show that they form a partition of $I$. If $I \in \B$ then we say $\B$ is \emph{connected}.

One prototypical example of a building set comes from a graph $w$ on vertex set $I$. The connected sets are the subsets $J \subseteq I$ for which the induced subgraph of $w$ on $J$ is connected. This family of \emph{graphical building sets} is the subject of Section \ref{s:W}. 

\vspace{-.5cm}

\begin{example} \label{e:buildingset}
The graphical building set for the path $\wzero$ on $[3]$ 
is the hypergraph $\{\emptyset, 1, 2, 3, 12, 23, 123\}$ of Example \ref{ex:123}.
\end{example}

Another example of a building set comes from a matroid $m$ on $I$. The connected sets of $m$ form a building set on $I$. We recall that a subset $J \subseteq I$ of a matroid is \emph{connected} if for every pair of elements $x, y \in J$ there exists a circuit $C$ (a minimal set with $r(C) < |C|$) such that $\{x,y\} \subseteq C \subseteq J$.

\subsection{The Hopf monoid of building sets} \label{ss:B}

Let $\rBS[I]$ denote the species of buldling sets on $I$. The species $\rBS$ becomes a Hopf monoid with the following additional structure. 

Let $I = S \sqcup T$ be a decomposition. 

\noindent $\bullet$ 
The product of two building sets $\B_1\in\rBS[S]$ and $\B_2\in\rBS[T]$ is their disjoint union.

\noindent $\bullet$ 
The coproduct of a building set $\B \in \rBS[I]$ is $(\B|_S, \B/_S) \in \rBS[S] \times \rBS[T]$, where the restriction and contraction of $\B$ with respect to $S$ are defined as
 \begin{eqnarray*}
 \B|_S & = & \{B \, : \, B \in \B, \, B \subseteq S\} \\
 \B/_S & = & \{B \subseteq T \, : \, A \sqcup B \in \B \textrm{ for some } A \subseteq S\}.  
 \end{eqnarray*}

One may check that these two collections are indeed building sets, and that the operations defined above satisfy the axioms of Hopf monoid.

\begin{proposition}\label{p:BSintoSHG}
The Hopf monoid of building sets $\rBS$ is a submonoid of the Hopf monoid of simple hypergraphs $\rSHG$.
\end{proposition}

\begin{proof}
Building sets are simple hypergraphs, and the product, restriction, and contraction operations for these two families are defined identically.
\end{proof}

Note that this Hopf structure is essentially the same as the one defined by Gruji{\'c}
in \cite{grujic2014quasisymmetric}, but different from the (cocommutative) Hopf algebras of building sets defined in \cite{grujic2012hopf,schmitt95:_hopf}.

\subsection{Nestohedra} \label{ss:nestohedra}

Since each building set $\B$ is a hypergraph, we can model it polytopally using its hypergraphic polytope, which is called the \emph{nestohedron}
\[
\Delta_\B = \sum_{J \in \B} \Delta_J.
\]
Unlike general hypergraphic polytopes, there is an explicit combinatorial description of the faces of the nestohedron $\Delta_\B$; they are in bijection with the \emph{nested sets} for $\B$ and with the \emph{$\B$-forests}, two equivalent families of objects which we now define.

\begin{definition} \label{d:nestedset} \cite{feichtner05, postnikov09}
A \emph{nested set} $\N$ for a building set $\B$ is a subset $\N \subseteq \B$ such that: \\
(N1) If $J,K \in \N$ then $J \subseteq K$ or $K \subseteq J$ or $J \cap K = \emptyset$. \\
(N2) If $J_1, \ldots, J_k \in \N$ are pairwise incomparable and $k \geq 2$ then $J_1 \cup \cdots \cup J_k \notin \B$. \\
(N3) All connected components of $\B$ are in $\N$. \\ 
The nested sets of $\B$ form a simplicial complex, called the \emph{nested set complex} of $\B$.
\end{definition}

\begin{example}
The collection $\N=\{3,4,6,7,379,48,135679,123456789\}$ is a nested set for the graphical building set of the graph shown in Figure \ref{f:B-tree}(a); see also Figure \ref{f:tubing}. 
\end{example}

\begin{figure}[h]
\centering
\includegraphics[scale=.2]{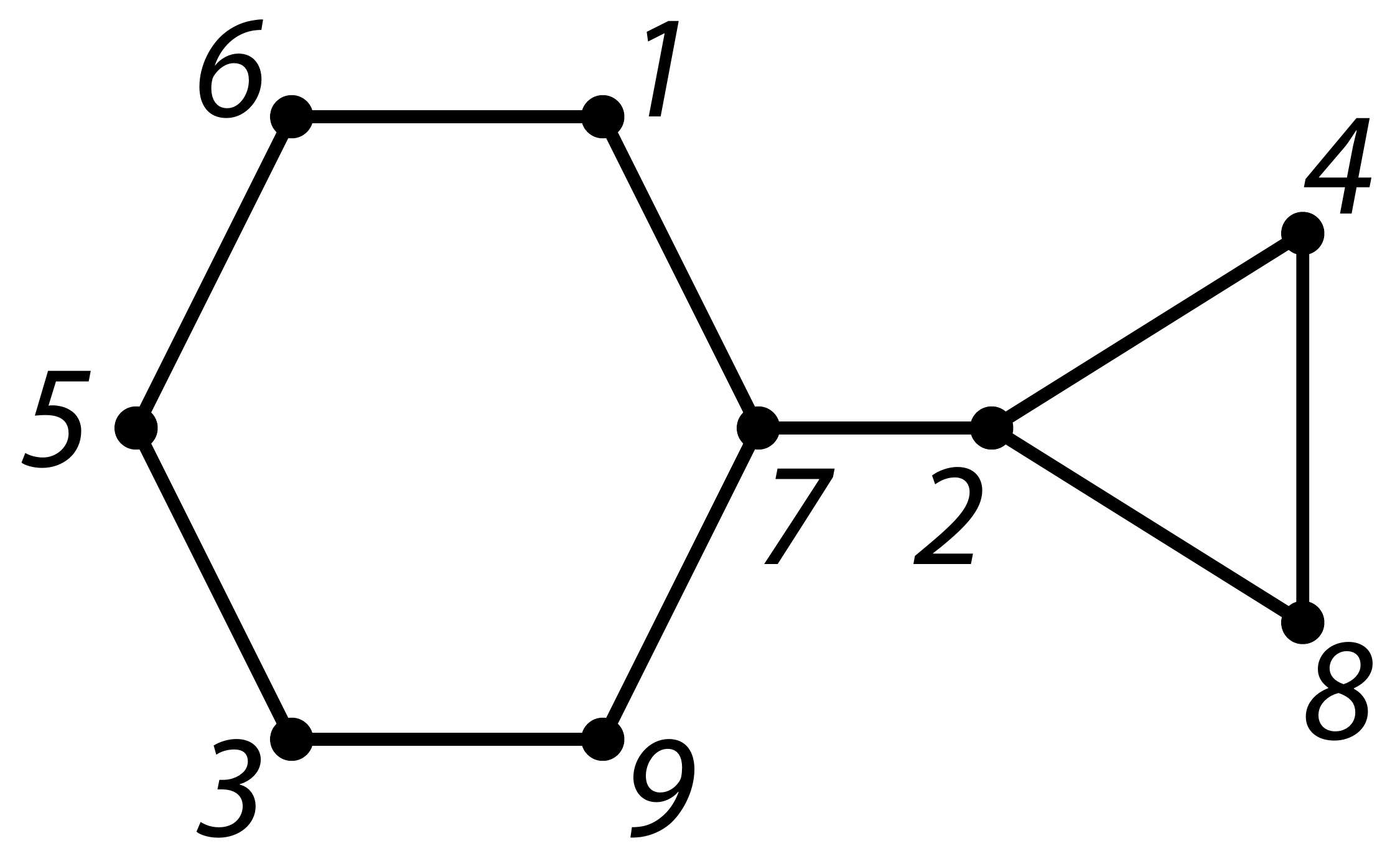}  \qquad  \qquad 
\includegraphics[scale=.25]{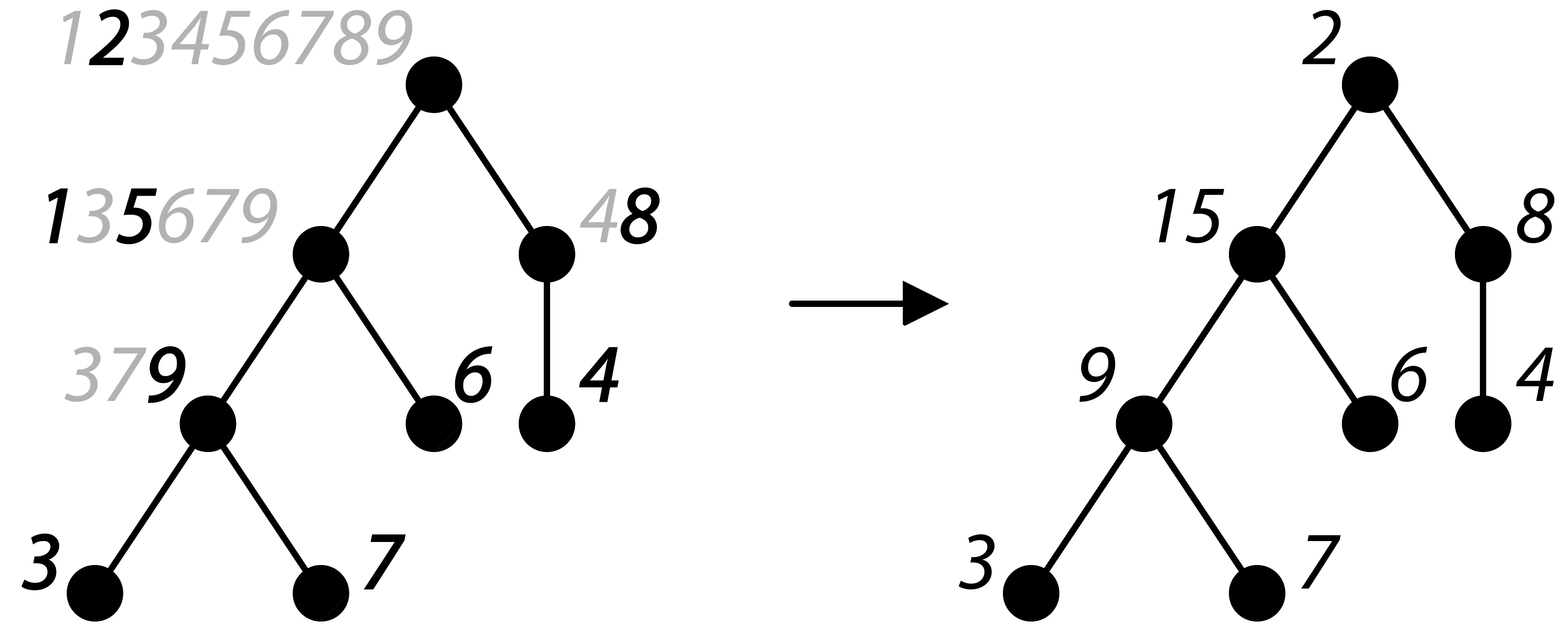} 
\caption{(a) A graph $w$. (b) A nested set for the graphical building set $\B$ of $w$ and the corresponding $\B$-forest.}\label{f:B-tree}
\end{figure} 

As shown in \cite{feichtner05, postnikov09} and 
illustrated in Figure \ref{f:B-tree}(b), nested sets for $\B$ are in bijection with a family of objects called $\B$-forests, as follows. We may regard a nested set $\N$ as a poset ordered by containment. We then  relabel each node by removing all elements which appear in nodes below it; the result is the corresponding $\B$-forest. We now define these objects more precisely.

\begin{definition} \cite{feichtner05, postnikov09}
Given a building set $\B$ on $I$, a \emph{$\B$-forest} $\N$ is a rooted forest whose vertices are labeled with non-empty sets 
partitioning $I$ such that: \\
(F1) For any node $S$,  $\N_{\leq S} \in \B$. \\
(F2) If $S_1, \ldots, S_k$ are pairwise incomparable and $k \geq 2$, $\bigcup_{i=1}^k \N_{\leq S_i} \notin \B$. \\
(F3) If $R_1, \ldots, R_r$ are the roots of $\N$, then the sets $\N_{\leq R_1}, \ldots, \N_{\leq R_r}$ are precisely the connected components of $\B$. 
\end{definition}

Here $\leq$ denotes the partial order on the nodes of the forest where all branches are directed up towards the roots. Also we denote  $\N_{\leq S} := \bigsqcup_{T \leq S} T$. 

\begin{proposition} \cite{feichtner05, postnikov09}
For any building set $\B$ on $I$, there is a bijection between the nested sets for $\B$ and the $\B$-forests.
\end{proposition}

As the notation suggests, we will make no distinction between a nested set and its corresponding $\B$-forest.

Each $\B$-forest $\N$ gives rise to a building set 
\begin{equation}\label{e:B(N)}
\B(\N) := \bigsqcup_{S \textrm{ node of } \N} \B[\N_{< S}, \N_{\leq S}]
\end{equation}
where for $X \subseteq Y \subseteq I$ we define $\B[X, Y] := (\B|_{Y})/X = (\B/X)|_{Y-X}$
on $Y-X$.

\begin{theorem} \label{t:nestofaces} \cite{feichtner05, postnikov09}
Let $\B$ be a building set. There is an order-reversing bijection between the faces of the nestohedron $\Delta_\B$ and the nested sets of $\B$. If $\N$ is a nested set of $\B$ and $F_\N$ is the corresponding face of $\Delta_\B$, then $\dim F_\N = |I| - |\N|$ and $\supp_I(F_\N) = \B(\N)$.
\end{theorem}

\begin{proof}
This is implicit in the proofs of \cite[Proposition 3.5]{ardila2006bergmanCoxeter} and \cite[Theorem 7.4, 7.5]{postnikov09}.
\end{proof}

In other words, the nestohedron $\Delta_\B$ is a simple polytope whose dual simplicial complex is isomorphic to the nested set complex of $\B$. An example is illustrated in Figure \ref{f:nesto}.

\begin{figure}[h]
\centering
\includegraphics[scale=.6]{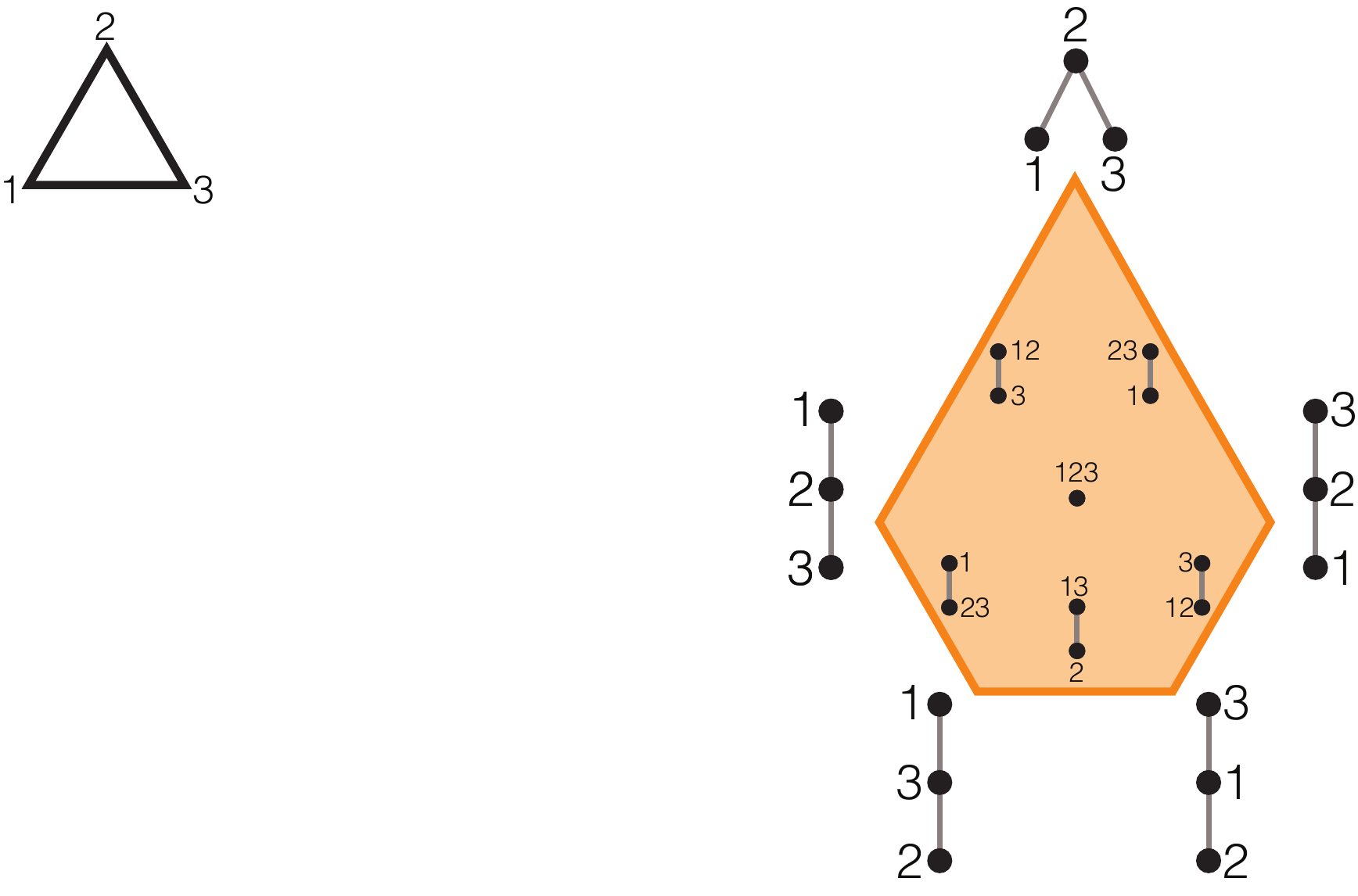} \hfill 
\caption{The hypergraphic polytope of Figure \ref{f:hypergraphic} is the nestohedron for the building set $\B=\{\emptyset, 1,2,3,12,23,123\}$; its faces are labeled by the $\B$-forests.}\label{f:nesto}
\end{figure}

\subsection{The antipode of building sets} \label{ss:antipodeBS}

Since building sets form a submonoid of simple hypergraphs by Proposition \ref{p:BSintoSHG}, we may use Theorem \ref{t:antipodeSHG} to compute the antipode of $\wBS$.

\begin{corollary}\label{c:antipodeB}
The antipode of the Hopf monoid of building sets $\wBS$ is given by the following \textbf{cancellation-free} expression. If $\B$ is a building set on $I$ then
\[
s_I(\B) = \sum_{\B-\textrm{forests } \N} (-1)^{|\N|} \B(\N)
\]
where for each $\B$-forest $\N$, $|\N|$ is the number of vertices of $\N$ and $\B(\N)$ is defined in (\ref{e:B(N)}).
\end{corollary}

\begin{proof}
By Theorem \ref{t:antipodeSHG}, the antipode of $\wBS$ is given by the face structure of the nestohedron $\Delta_\B$. 
It remains to invoke Theorem \ref{t:nestofaces} which tells us the dimension and the  building set supporting  each face of $\Delta_\B$.
The formula is cancellation-free since faces of different dimensions have different supports.
\end{proof}

Note that the formula of Corollary \ref{c:antipodeB} is not grouping-free. For example, all vertices of $\Delta_\B$ map to the trivial building set $\{\{i\}, i \in I\} \cup \{\emptyset\}$.

\begin{example}
Let us return to the building set $\{\emptyset, 1, 2, 3, 12, 23, 123\}$ of Example \ref{e:buildingset}. We computed  its antipode in Example \ref{e:antipodeSHG}:
\[
\begin{array}{rl}
\apode_{[3]}(\H) =&
\H - 2 \{\emptyset,1,2,3,23\} - 2\{\emptyset,1,2,3,12\} - \{\emptyset,1,2,3,12\} + 5 \{\emptyset, 1, 2, 3\}
  \end{array}
\]
and we now encourage the reader to compare this with the expression in Corollary \ref{c:antipodeB}.
\end{example}

\section{$\rW$: {Simple graphs, ripping and sewing, and graph associahedra}}\label{s:W}

In Section \ref{s:BS} we briefly mentioned how connectivity in graphs was one of the motivations to study building sets. In this section we focus on the \emph{graphical building sets} that arise in this way, which give rise to a new Hopf monoid $\rW$ on graphs. This \emph{ripping and sewing} Hopf monoid should not be confused with the monoids $\rG$, $\rSG$, and $\rGamma$ of Sections \ref{s:G} and \ref{ss:cycle-matroid}.

\begin{definition}
Let $w$ be a simple graph whose vertex set is $I$.
A subset $J \subset I$ is a \emph{tube} if the induced subgraph of $w$ on $J$ is connected. The set of tubes of $w$ is a building set; we denote it $\tubes(w)$ and call it the \emph{graphical building set of $w$}.
\end{definition}

Let $\rWBS[I]$ be the set of graphical building sets on $I$. We will see in Proposition \ref{p:WintoWB} that graphical building sets form a submonoid of $\rBS$, which we now describe directly in terms of the graphs.


\subsection{The ripping and sewing Hopf monoid of simple graphs}

\begin{definition}
Given a simple graph $w$ whose vertex set is $I$, and a partition $I = S \sqcup T$,  an \emph{$S$-thread} is a path in $w$ whose initial and final vertices are in $T$, and all of whose intermediate vertices (if any) are in $S$. 

Define the operations of \emph{ripping and sewing} as follows.

\noindent $\bullet$
 \emph{ripping out $T$}: $w|_S$ is the induced subgraph on $S$, obtained by ``ripping out"  every vertex of $T$ and every edge incident to $T$. 

\noindent $\bullet$
\emph{sewing through $S$}: $w/_S$ is the simple graph on $T$ where we add or ``sew in" an edge $uv$ between vertices $u,v \in T$ if the graph $w$ contains an \emph{$S$-thread} from $u$ to $v$.
Note that this includes all edges of $w|_T$.
\end{definition}


For example, let
$I=\{a,b,c,d,e,f,g\}, S=\{a,b,c,d\},$ and $T=\{e,f,g\}$.
 \begin{figure}[h]
\centering
If $w = $
\includegraphics[scale=.5]{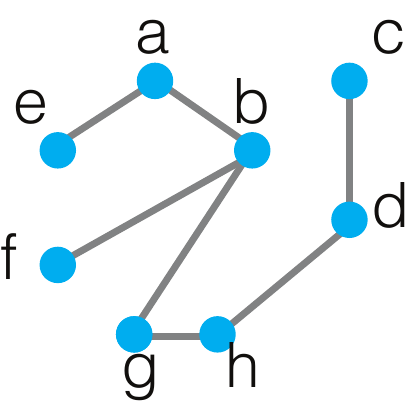} 
\qquad then \qquad $w|_S = $ 
\includegraphics[scale=.5]{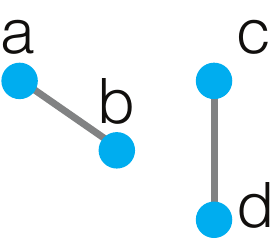}
\qquad and \qquad $w/_S = $
\includegraphics[scale=.5]{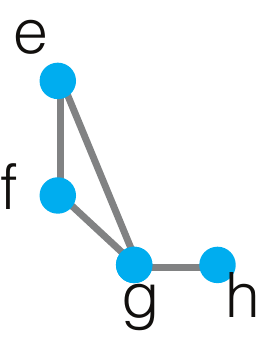}
\,\, .
\end{figure}


Let $\rW[I]$ be the set of simple graphs on vertex set $I$. We turn the species $\rW$ into the \emph{ripping and sewing  Hopf monoid} with the following operations.

Let $I = S \sqcup T$ be a decomposition.

\noindent $\bullet$ The product of two simple graphs $w_1 \in \rW[S]$ and $w_2 \in \rW[T]$ is their disjoint union.

\noindent $\bullet$ The coproduct of a simple graph $w \in \rW[I]$ is $(w|_S, w/_S) \in \rW[S] \times \rW[T]$ where $w|_S$ and $w/_S$ are obtained from $w$ by ripping out $T$ and sewing through $S$, respectively.

\noindent One easily checks that this is indeed a Hopf monoid.

%
%

\begin{proposition}\label{p:WintoWB} 
The species $\rWBS$ of graphical building sets is a submonoid of the Hopf monoid of building sets.
Furthermore, the tube maps $w \mapsto \tubes(w)$ give an isomorphism of Hopf monoids $\rW \cong \rWBS \into \rBS$.
\end{proposition}

\begin{proof}
We first prove that the map $\tubes:\rW \rightarrow \rBS$ is a morphism of Hopf monoids. 
We do know that the set $\tubes(w)$ is a building set for any $w$.
Also $\tubes$ preserves products because  $\tubes(w_1 \sqcup w_2) = \tubes(w_1) \sqcup \tubes(w_2)$ for  $w_1 \in \rW[S]$ and $w_2 \in \rW[T]$. It remains to check that the map $\tubes$ preserves coproducts; that is, 
\[
\tubes(w)|_S = \tubes(w|_S), \qquad \tubes(w)/_S = \tubes(w/_S) \qquad
\]
 for any simple graph $w$ on $I$ and any subset $S \subseteq I$.
 
The first statement is clear: the connected sets in $w$ which are subsets of $S$ are precisely the connected sets in $w|_S$, the induced subgraph on $S$. Let us prove the second one.
%

$\subseteq:$ Suppose $B \in \tubes(w)/_S$, so $A \sqcup B$ is a tube of $w$ for some subset $A \subseteq S$. To show $B \in \tubes(w/_S)$, we need to show that for any $u,v \in B$ there is a path from $u$ to $v$ in $w/_S$.

We do have a path $P$ from $u$ to $v$ inside the induced subgraph $A \sqcup B$ of $w$, since this is a tube in $w$. This path may contain vertices of $S$ and $T$; let $u=t_0, t_1, \ldots, t_{k-1}, t_k=v$ be the vertices of $T$ that it visits, in that order. Now, for each $0 \leq i \leq k-1$, the path $P$ contains an  $S$-thread $t_i s_1 \ldots s_l t_{i+1}$ from $t_i$ to $t_{i+1}$ for some $l \geq 0$, so $t_it_{i+1}$ is an edge of  $w/_S$. It follows that $t_0 t_1 \ldots t_{k-1} t_k$ is our desired path from $u$ to $v$ in $w/_S$. We conclude that $B \in \tubes(w/_S)$. 

$\supseteq:$ 
Conversely, suppose $B \in \tubes(w/_S)$. For each edge $uv$ in $w/_S$, choose an $S$-thread from $u$ to $v$; let $S_{uv}\subseteq S$ be the set of vertices on that $S$-thread other than $u$ and $v$. Let $A \subseteq S$ be the union of the sets $S_{uv}$ as we range over all edges $uv$ of $w/_S$. We claim that $A \sqcup B$ is a tube in $w$. To show this, first note that any two vertices $u, v$ of $B$ are connected by an $S$-thread inside $A \sqcup B$ by construction. Furthermore, any vertex of $A$ belongs to the set $S_{uv}$ for some $u,v \in B$, and hence is connected to $u$ and $v$  by a path in $A \sqcup B$.  It follows that $A \sqcup B$ is a tube of $w$ and  $B \in \tubes(w)/_S$ as desired. 

\medskip

Thus we have proved that $\tubes: \rW \rightarrow \rBS$ is a morphism of Hopf monoids, and hence that its image $\rWBS$ is a submonoid of $\rBS$. It remains to prove that the surjective map $\tubes:\rW \onto \rWBS$ is also injective. To see this, notice that we can easily recover a simple graph $w \in \rW[I]$ from its graphical building set $\tubes(w)$: the edges of $w$ are precisely the tubes of size $2$.
\end{proof}

\subsection{Graph associahedra}

For a simple graph $w$ on $I$ we define the \emph{graph associahedron} $\Delta_w \subset \Rb I$ to be
\[
\Delta_w := \sum_{\tau \in \tubes(w)} \Delta_\tau.
\]
Graph associahedra are the nestohedra corresponding to graphical building sets. Let us recall their combinatorial structure, as described in \cite{carr2006coxeter, postnikov09}.

\begin{definition} \label{d:tubing} Let $w$ be a simple graph. A \emph{tubing} is a set $t$ of tubes such that:

\noindent $\bullet$ any two tubes $\tau_1$ and $\tau_2$ in $t$ are disjoint or nested: we have $\tau_1 \subseteq \tau_2, \tau_1 \supseteq \tau_2, $ or $\tau_1 \cap \tau_2 = \emptyset$.

\noindent $\bullet$ if $\tau_1, \ldots, \tau_k$ are pairwise disjoint tubes in $t$, then $\tau_1 \cup \cdots \cup \tau_k$ is not a tube of $w$.

\noindent $\bullet$ every connected component of $w$ is a tube in $t$.
%
\end{definition}

\noindent Comparing this with Definition \ref{d:nestedset} we see that the tubings of $w$ are precisely the nested sets for the graphical building set $\tubes(w)$. An example is shown in Figure \ref{f:tubing}.

\begin{figure}[h]
\centering
\includegraphics[scale=.2]{Figures/graph.pdf}  \qquad \qquad \includegraphics[scale=.2]{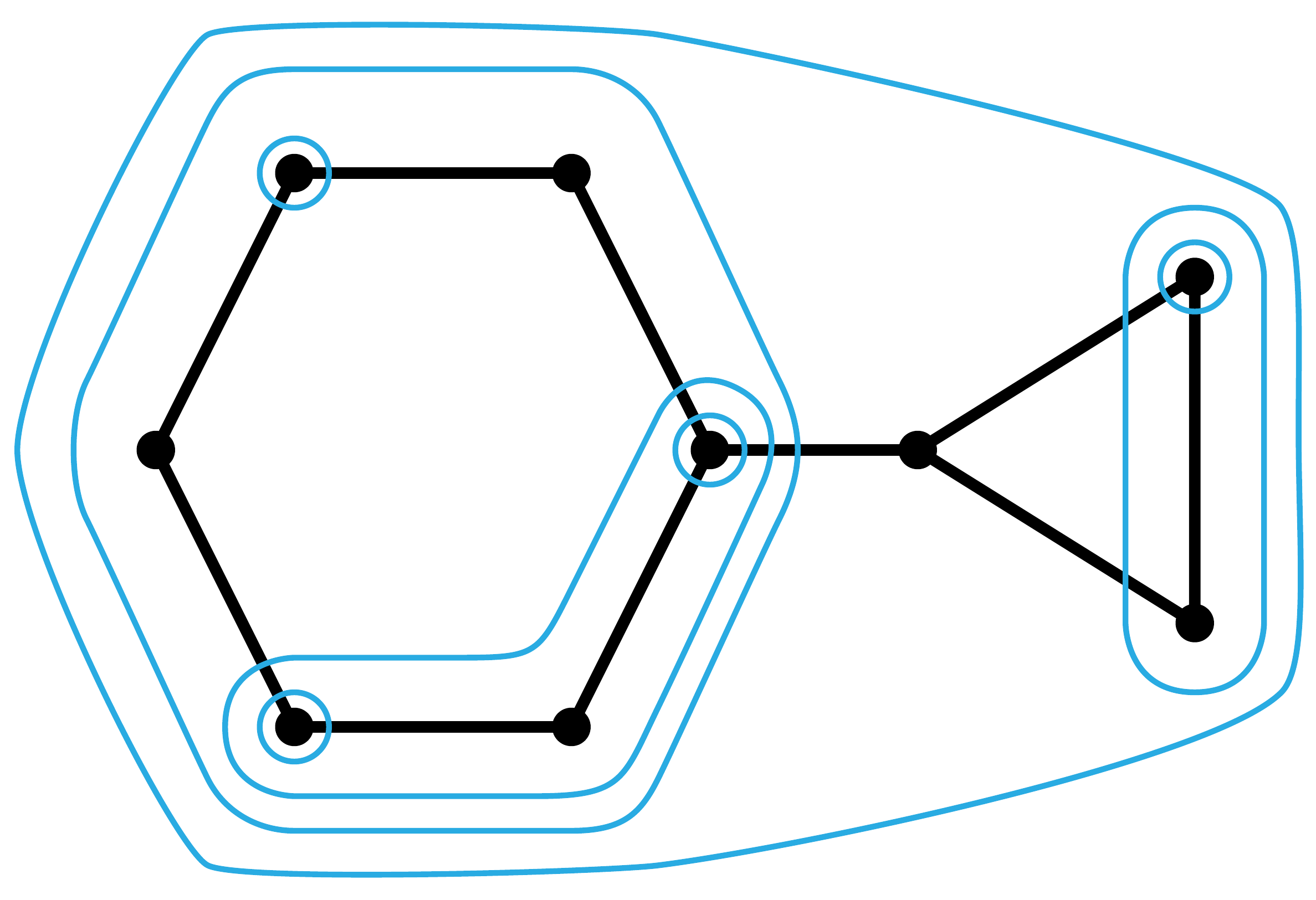}  
\caption{The nested set $\N=\{3,4,6,7,379,48,135679,123456789\}$ of Figure \ref{f:B-tree}, now drawn as a tubing.}\label{f:tubing}
\end{figure}

For each tube $\tau$ in a tubing $t$, let $t_{<\tau}$ be the union of the tubes of $t$ that are strictly contained in $\tau$, and let the \emph{essential set} of $\tau$ be $\ess(\tau) = \tau - t_{<\tau}$. As $\tau$ ranges over the tubes of $t$, the essential sets $\ess(\tau)$ partition $I$.

Each tubing $t$ of $w$ gives rise to a simple graph  
\begin{equation}\label{e:w(t)}
w(t) := \bigsqcup_{\tau \textrm{ tube of } t} 
w[t_{<\tau},\,  \tau], 
\end{equation}
where $w[t_{<\tau}, \, \tau] := (w|_\tau)/_{t_{<\tau}}$ is the simple graph on $\ess(\tau)$ obtained by restricting $w$ to $\tau$ and then sewing through the tubes strictly inside of $\tau$. Since the essential sets of $\tau$ partition $I$, $w(t)$ is a simple graph on $I$.

\begin{theorem} \label{t:graphassofaces} 
\cite{carr2006coxeter, postnikov09}
Let $w$ be a simple graph. There is an order-reversing bijection between the faces of the graph associahedron $\Delta_w$ and the tubings of $w$. If $t$ is a tubing of $w$ and $F_t$ is the corresponding face of $\Delta_w$, then $\dim F_t = |I|-|t|$ and $\supp_I(F_t) = w(t)$.
\end{theorem}

\begin{proof}
This is the result of specializing Theorem \ref{t:nestofaces} to graphical building sets and graph associahedra.
\end{proof}

An example is illustrated in Figure \ref{f:graphasso}.

\begin{figure}[h]
\centering
\includegraphics[scale=.58]{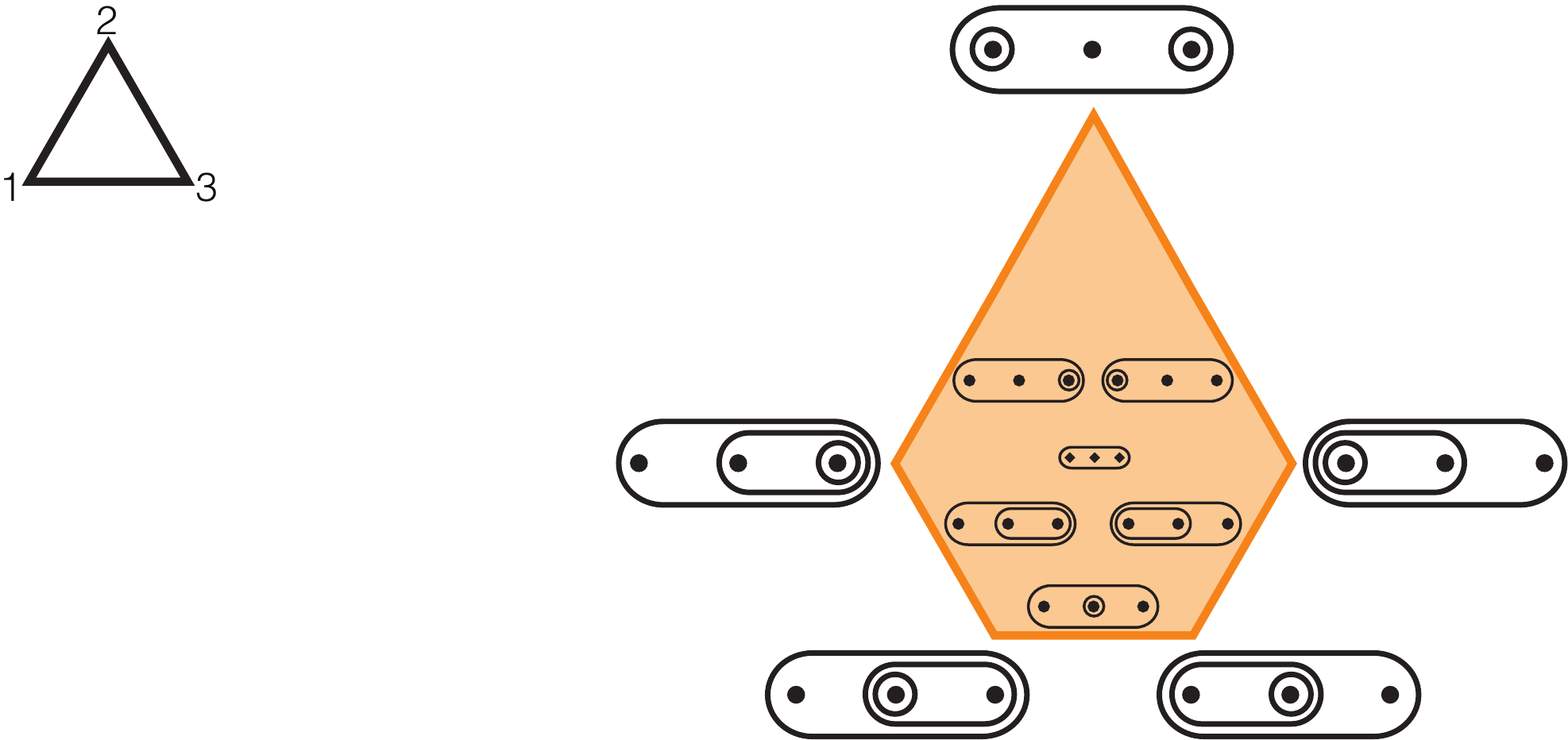}  \hfill 
\caption{The nestohedron of Figures \ref{f:hypergraphic} and \ref{f:nesto} is the graph associahedron for the path of length 3; its faces are labeled by the tubings of the path.} \label{f:graphasso}
\end{figure}

\subsection{The antipode of the ripping and sewing Hopf monoid} \label{ss:antipodeW}

\begin{theorem}\label{t:antipodeW}
The antipode of the ripping and sewing Hopf monoid of simple graphs $\wW$ is given by the following \textbf{cancellation-free} expression. If $w$ is a simple graph on $I$ then:
\[
s_I(w) = \sum_{t \textrm{ tubing}} (-1)^{|t|} w(t)
\]
where $|t|$ is the number of tubes of $t$ and $w(t)$ is defined in (\ref{e:w(t)}).
\end{theorem}

\begin{proof}
Since $\wW$ is 
isomorphic to the Hopf monoid of graphical building sets $\wWBS$, which is a submonoid of the Hopf monoid of simple hypergraphs $\wSHG$, its antipode is given by Theorem \ref{t:antipodeSHG}. It remains to invoke Theorem \ref{t:graphassofaces}, and to remark again that faces of different dimension map to different supports.
\end{proof}

Note that the formula above is not grouping-free. For example, for every maximal tubing $t$, $w(t)$ is the graph with no edges.

\begin{example}\label{e:antipodeW}
The antipode of the path of length $3$ in $\wW$ is dictated by its graph associahedron, which again is the polytope of Figures \ref{f:hypergraphic}, \ref{f:nesto}, and \ref{f:graphasso}. The result is now: 
\begin{figure}[h]
\centering
\includegraphics[scale=.6]{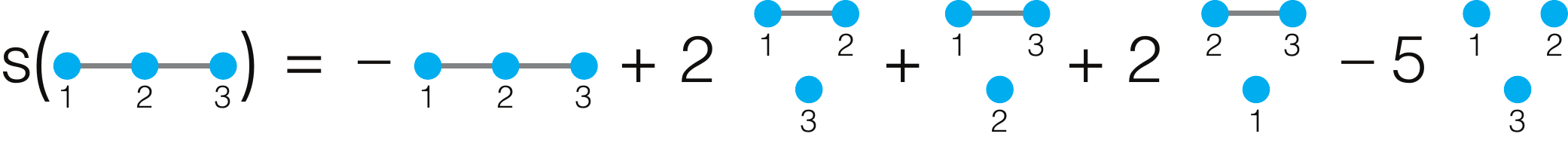}  
\caption{The antipode of a path of length 3 in $\wW$.}\label{f:antipodepath}
\end{figure} 
 \end{example}


\section{$\rPi$: {Set partitions and permutahedra, revisited.}}\label{s:Pirevisited}

\begin{definition}
A \emph{clique} is a complete graph. 
A \emph{cliquey graph} is a disjoint union of complete graphs. 
\end{definition}

Let $\rK[I]$ be the set of cliquey graphs on $I$.
There is a natural bijection between cliquey graphs on $I$ and set partitions of $I$: the cliquey graph $w$ on $I$ corresponds to the set partition $\pi(w)$ formed by its connected components.


\begin{proposition}\label{p:cliques}
The species $\rK$ of cliquey graphs is a submonoid of the ripping and sewing Hopf monoid of simple graphs $\rW$. Furthermore, $\rK$ is isomorphic to the Hopf monoid of set partitions $\rPi$.
\end{proposition}

\begin{proof}
Since the disjoint union of cliquey graphs is cliquey, $\rK$ is closed under multiplication. Also if $K_I$ is the clique on $I$ then $(K_I)|_S = K_S$ and $(K_I)/_S = K_T$, so $\rK$ is also closed under comultiplication, proving the first assertion. The map $\pi: \rK \rightarrow \rPi$ sending a cliquey graph $w$ to $\pi(w)$ gives the desired isomorphism; it clearly preserves products, and since
\begin{align*}
& \pi(K_I)|_S= \{I\}|_S = \{S\} = \pi(K_S) = \pi(K_I|_S) \textrm{ and } \\
&\pi(K_I)/_S= \{I\}/_S = \{T\} = \pi(K_T) = \pi(K_I/_S),
\end{align*}
it also preserves coproducts.
%
%
%
\end{proof}

Since $\Pi$ is cocommutative, we also have $\Pi \cong K \into W^{cop}$, as shown in the commutative diagram at the beginning of Section \ref{s:HGP}.

\subsection{The antipode of set partitions}

\begin{theorem}\cite{am}\label{t:antipodePi}
The antipode of the Hopf monoid of set partitions $\wPi$ is given by the following \textbf{cancellation-free} and \textbf{grouping-free} expression. If $\pi$ is a set partition on $I$, 
\[
s_I(\pi) = \sum_{\rho \, : \, \pi \leq \rho} (-1)^{b(\rho)} (\pi : \rho)! \,  \rho
\]
summing over all partitions $\rho$ that refine $\pi$. Here $b(\rho)$ denotes the number of blocks of $\rho$, and $(\pi:\rho)! = \prod_{p_i \in \pi} n_{i}!$ where $n_{i}$ is the number of blocks of $\rho$ that partition the block $p_i$ of $\pi$. 

\end{theorem}

\begin{proof}
Let $w$ be a cliquey graph and $\pi=\{p_1, \ldots, p_k\}$ be the corresponding set partition. A tube on $w$ is a subset of one of the parts $p_i$. A tubing $t$ on $w$ cannot contain two disjoint subsets of the same $p_i$; thus $t$ consists of a flag $t^i_\bullet$ of subsets $\emptyset = \tau^i_0 \subset \cdots \subset \tau^i_{n_i} = p_i$ for each part $p_i$. The flag $t^i_\bullet$ gives rise to a composition $p_i = \rho^i_1 \sqcup \cdots \sqcup  \rho^i_{n_i}$ where $\rho^i_j = \tau^i_j - \tau^i_{j-1}$. If we let $\rho(t)=\{\rho^i_j \, | \, 1 \leq i \leq k, 1 \leq j \leq n_i\}$ as an unordered set partition, then $\rho(t)$ is the partition corresponding to the graph $w(t)$ of (\ref{e:w(t)}). Clearly $\rho(t) \geq \pi$ and $|t| = b(\rho(t))$. 

It remains to observe that the map from a tubing $t$ to the partition $\rho(t)$ is a $(\pi:\rho)!$-to-1 map, because there are $n_i$! linear orders for the partition $\{\rho^i_1, \cdots,   \rho^i_{n_i}\}$ of $p_i$ for $1 \leq i \leq k$, which give rise to different choices of the tubing $t$.
\end{proof}

As an example, let us revisit the cancellation-free formula for the antipode of the set partition $\{ab,cde\}$ shown in the introduction.

 \begin{figure}[h]
\centering
\includegraphics[scale=.5]{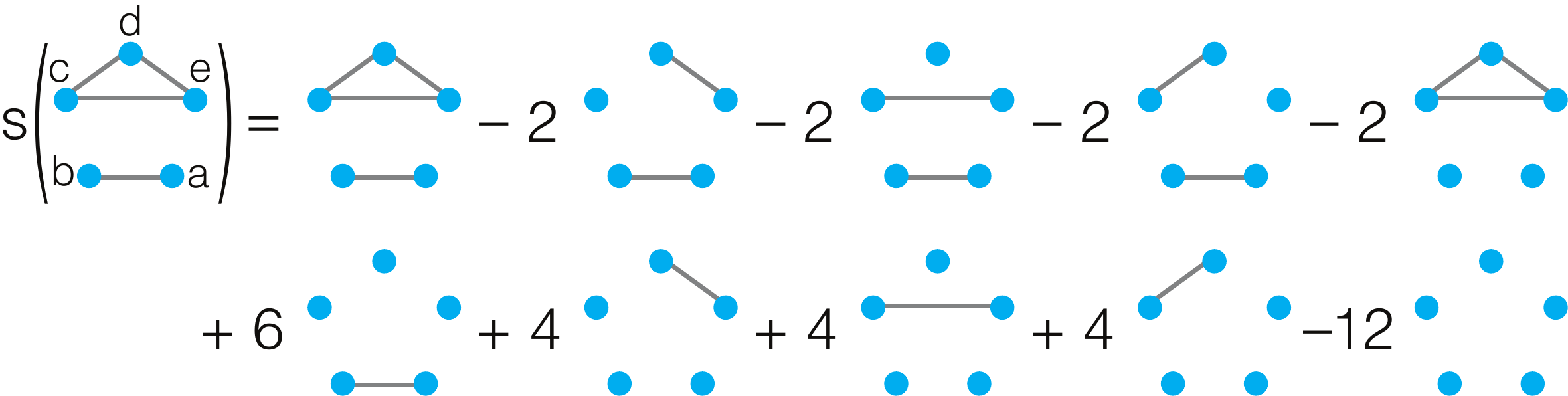} 
\end{figure}

As should be clear by now, our derivation of Theorem \ref{t:antipodePi} is controlled by a polytope; for the set partition $\pi$ with blocks $p_1, \ldots, p_k$, it is the graph associahedron
\[
\Delta_\pi = \pi'_{p_1} \times \cdots \pi'_{p_k} \equiv \pi_{p_1} \times \cdots \pi_{p_k},
\]
where $\pi'_I := \sum_{J \subseteq I} \Delta_J$  is normally equivalent to the standard permutahedron $\pi_I$.

Thus the antipode of $\pi = \{ab,cde\}$ is an algebraic shadow of the face structure of the hexagonal prism $\pi'_{\{a,b\}} \times \pi'_{\{c,d,e\}}$: it has one 3-face, eight 2-faces (in normal equivalence classes of size 2, 2, 2, 2), eighteen edges (in equivalence classes of sizes 6, 4, 4, 4) and twelve vertices (in one equivalence class of size 12).

\begin{figure}[h]
\centering
\qquad 
\qquad
\includegraphics[scale=.4]{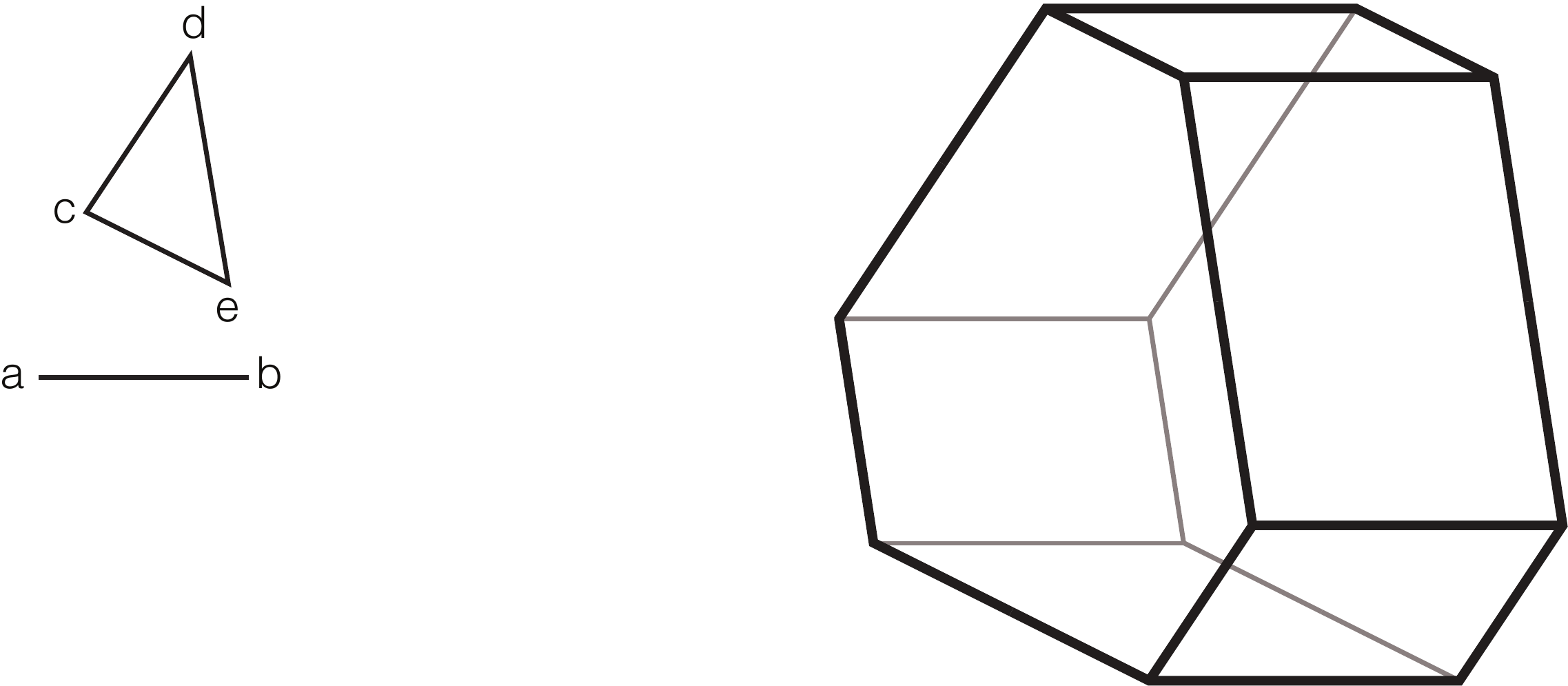} \hfill
\caption{The product $\pi'_{\{a,b\}} \times \pi'_{\{c,d,e\}}$ in $\Rb^{\{a,b,c,d,e\}}$.}
\end{figure}

\subsection{Permutahedra, set partitions, and the Hopf algebra of symmetric functions}\label{ss:sym}

We conclude this section by precisely stating  connections between permutahedra, set partitions, and symmetric functions

\begin{proposition}\label{prop:PibPi}
The Hopf monoid of permutahedra $\rbPi$ is isomorphic to the Hopf monoid of set partitions $\rPi$.
\end{proposition}

\begin{proof}
The Hopf monoid $\rbPi$ is generated multiplicatively by the standard permutahedra $\pi_I$, with coproduct given by $\Delta_{S,T}(\pi_I) = (\pi_S, \pi_T)$ as observed in Lemma \ref{l:rbPi}. Comparing this with the definition of the Hopf monoid $\rPi$ gives the isomorphism.
\end{proof}

Recall that $\Kcb$ is the Fock functor that associates a Hopf algebra $\Kcb(\wH)$ to any Hopf monoid on vector species $\wH$.

\begin{proposition}
The Hopf algebra of permutahedra $\Kcb(\wbPi)$ is isomorphic to the Hopf algebra of symmetric functions $\Lambda$.
\end{proposition}

\begin{proof}
This proof requires some basic facts about symmetric functions; see \cite{macdonald95:_symmet_hall} and \cite[Section 7]{stanley99:_ec2}. 
The Hopf algebra of symmetric functions $\Lambda = \Kb[x_1, x_2, \ldots]^{S_\infty}$ is most easily described in terms of the \emph{homogeneous} and \emph{elementary} symmetric functions:
\[
h_n = \sum_{i_1 \leq \cdots \leq i_n} x_{i_1}\cdots x_{i_n}, \qquad
e_n = \sum_{i_1 < \cdots < i_n} x_{i_1}\cdots x_{i_n}
\]
As an algebra, $\Lambda= \Kb[h_1, h_2, \ldots]$ 
is simply the polynomial algebra on the $h_i$, while the coproduct and antipode of $\Lambda$ are
\[
\Delta(h_n) = \sum_{i+j=n} h_i \otimes h_j, \qquad 
\apode(h_n) = (-1)^n e_n.
\]
for $n \geq 0$, where $h_0=1$.

The Fock functor $\Kcb$ maps $\wbPi$ to the graded Hopf algebra $\Kcb(\wbPi)$; let it take the permutahedron $\pi_I \in \wbPi[I]$ to the element $n!g_n \in \Pi_n$ where $n=|I|$. Then (\ref{e:Pibasis}) tells us that as an algebra $\Kcb(\wbPi) = \Kb[g_1, g_2, \ldots]$
while Lemma \ref{l:rbPi} tells us that the coproduct of $\Kcb(\wbPi)$ is given by
\[
\Delta(g_n) = \sum_{i+j=n} g_i \otimes g_j.
\]
It follows that the map $g_n \mapsto h_n$ preserves the product and coproduct. Since the antipode of a graded Hopf algebra is unique, this map also also preserves the antipode. This gives the desired isomorphism $\Kcb(\wbPi) \cong \Lambda$. 
\end{proof}

It is instructive to compare the antipodes of  $\Kcb(\wbPi)$ and $\Lambda$. In $\wbPi$ the antipode of $n! g_n$ is given by the face structure of the permutahedron $\pi_n$, as described in Section \ref{ss:p}:
\[
\apode(g_n) = \sum_{\lambda_1 + \cdots + \lambda_k = n} (-1)^k g_{\lambda_1}\cdots g_{\lambda_k},
\]
while the antipode of $\Lambda$ is given by $\apode(h_n) = (-1)^n e_n.$ Comparing these expressions, we obtain a polyhedral algebraic proof of the expression of the elementary symmetric function $e_n$ in the homogenous basis:
\[
e_n = 
\sum_{\lambda_1 + \cdots + \lambda_k = n} (-1)^{n-k} h_{\lambda_1}\cdots h_{\lambda_k}.
\]

\section{$\rF$: {Paths and associahedra, revisited.}}\label{s:Frevisited}

Recall that a \emph{set of paths} on $I$ is a graph whose connected components are paths, and $\rF[I]$ denotes the collection of sets of paths on $I$. Recall the Hopf monoid $\rF$ defined in Section \ref{ss:F}. The product of two sets of paths is their disjoint union. If $s$ is a path and $I=S \sqcup T$ is a decomposition, then $s|_S$ is the path on $S$ with the order inherited from $s$, whereas $s/_S$ is the induced subgraph on $T$.

\begin{proposition}\label{prop:FintoW}
The Hopf monoid $\rF$ of paths is a submonoid of the co-opposite $\rW^{cop}$ of the ripping and sewing Hopf monoid $\rW$. 
\end{proposition}

\begin{proof}
This follows readily from the observation that the product operations on $\rF$ and $\rW$ coincide, while the coproducts are co-opposite.
\end{proof}

In light of this statement and the fact that $\rW$ and $\rW^{cop}$ share the same antipode by Proposition \ref{p:antipode2}, Theorem \ref{t:antipodeW} immediately gives us a combinatorial formula for the antipode of the Hopf monoid of path $\rF$. This formula has several interesting combinatorial variants, which we explore in the remaining sections.


\subsection{The antipode of paths}\label{ss:antipodeF}

If $l$ is a linear graph and $t$ is a tubing of $l$, define the \emph{linear graph of $t$}, denoted $l(t)$, as follows. Each tube $\tau$ of $t$ gives a path $l(\tau)$ consisting of the vertices which are in $\tau$ and in no smaller tube of $t$, in the order they appear in $\tau$. The union of these paths is $l(t)$. This procedure is illustrated in Figure \ref{f:tubingtolineargraph}. 

\begin{figure}[h]
\centering
\includegraphics[scale=.6]{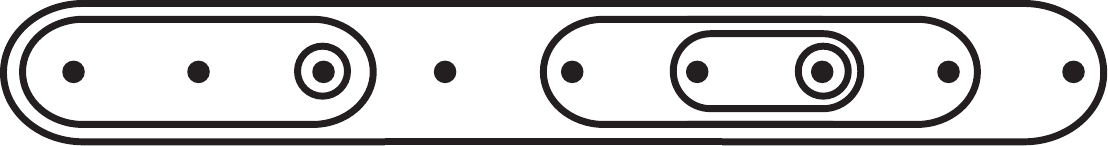} 
\caption{A tubing $t$ of the path $123456789$; its linear graph is $l(t) = 12|3|49|58|6|7$. The labels and edges of the path have been omitted for clarity.\label{f:tubingtolineargraph}}
\end{figure}

\begin{proposition}\label{p:antipodeF1}
The antipode of the Hopf monoid of paths $\wF$ is given by the following \textbf{cancellation-free} expression. If $l$ is a linear graph on $I$ then
\[
s_I(l) = \sum_{t \textrm{ tubing}} (-1)^{|t|} \, l(t)
\]
summing over all tubings $t$ of $l$, where $l(t)$ is the linear graph of $t$.
\end{proposition}

\begin{proof}
This is a direct consequence of Theorem \ref{t:antipodeW} because for a linear graph $w=l$, the graph $w(t)$ given by (\ref{e:w(t)}) is the linear graph $l(t)$.
\end{proof}
%
%
%
%

There are natural bijections between tubings on a path $p_n$ of length $n$, valid parenthesizations of the expression $x_0x_1\cdots x_n$, and plane rooted trees with $n+1$ unlabeled leaves.
\cite{carr2006coxeter} \cite[Chapter 6]{stanley99:_ec2} This bijection allows us to state Proposition \ref{p:antipodeF1} in terms of parenthesizations or plane rooted trees as well. We leave the details to the interested reader.

\medskip

We can obtain a more useful formula by grouping equal terms in Proposition \ref{p:antipodeF1} as follows. 
As we range over the tubes $\tau$ of a tubing $t$, the components of the linear graph $l(t)$ form a set partition of $I$, which we call $\pi=\pi(t)$. We also write $l(\pi) = l(t)$.
 
Notice that $\pi=\pi(t)$ is a \emph{noncrossing partition} of $l$; that is, if we let $<$ denote (either of) the (two) linear order(s) on $I$ imposed by $l$, then $\pi$ does not contain blocks $p_i \neq p_j$ and elements $a<b<c<d$ such that $a,c \in p_i$ and $b,d \in p_j$. 
It remains to describe the coefficient of $l(\pi)$ 
for each noncrossing partition $\pi$ in the expression of Proposition \ref{p:antipodeF1}. 

Let $NC(l)$ be the set of noncrossing partitions of $l$. If $|l|=n$, then
\[
|NC(l)| = C_n = \frac1{n+1} {2n \choose n}
\]
is the $n$-th \emph{Catalan number}. \cite{kreweras1972partitions}. We define
the \emph{linear graph of a noncrossing partition} $\pi \in NC(l)$ to be the graph on $I$ containing one path for each part of $\pi$ with the order induced by $l$.

To simplify the discussion we let $I = [n]$ and $l$ be the path $12\cdots n$.
For a noncrossing partition $\pi$ of $I$, let the \emph{adjacent closure} $\overline{\pi}$ be the  partition obtained from $\pi$ by successively merging any two \emph{adjacent} blocks $S_1$ and $S_2$ such that $\max S_1 = b$ and $\min S_2 = b+1$ for some $b$. 

\begin{example}
The adjacent closure of the noncrossing  partition $\pi = 1|26|3|45|78$ in $NC(8)$ is $\overline{\pi} = 12678|345$.
\end{example}

\begin{theorem}\label{t:antipodeF2}
The antipode of the Hopf monoid of paths $\rF$ is given by the following \textbf{cancellation-free} and \textbf{grouping-free} expression. If $l$ is a path on $I$,
\[
s_I(l) = \sum_{\pi \in NC(l)} (-1)^{|\pi|} C_{(\overline{\pi} : \pi)} \, l(\pi)
\]
summing over all the noncrossing partitions $\pi$ of $l$. 
Here $l(\pi)$ denotes the linear graph of $\pi$, $\overline{\pi}=\{p_1, \ldots, p_k\}$ is the adjacent closure of $\pi$, and $C_{(\overline{\pi}:\pi)} = \prod_{p_i \in \overline{\pi}} C_{n_i}$ where $n_i$ is the number of blocks of $\pi$ refining block $p_i$ of $\overline{\pi}$.
\end{theorem}

\begin{proof}
For a noncrossing partition $\pi$, the coefficient of $l(\pi)$
in the expression of Proposition \ref{p:antipodeF1} is equal to the number of tubings $u$ of $l$ with $\pi(u)=\pi$. We claim that this number equals $C_{(\overline{\pi}:\pi)}$.

Let $\pi$ be a noncrossing partition of $[n]$, and consider the set $t$ of tubes $\tau_i = [\min p_i, \max p_i]$ for all blocks $p_i$ of $\pi$. Notice that $\tau_i \subset \tau_j$, $\tau_i \supset \tau_j$, or $\tau_i \cap \tau_j = \emptyset$ for $i \neq j$; if that were not the case, without loss of generality we would have  $\min p_i < \min p_j < \max p_i  < \max p_j$, which would contradict the assumption that $\pi$ is noncrossing. However, $t$ is not necessarily a tubing because it may contain adjacent tubes. 

Let $\overline{t}$ be the tubing obtained from $t$ by successively merging any two adjacent tubes of the form  $[a,b]$ and $[b+1,c]$. It follows from the definitions that the noncrossing partition associated to $\overline{t}$ is $\overline{\pi}$.

For each tube of $\overline{t}$, let us remember the tubes in $t$ that constituted it by drawing vertical dotted lines separating them. This process is shown in Figure \ref{f:noncrossingtubes}. Notice that if part $p_i$ of $\overline{\pi}$ contains $n_i$ parts of $\pi$, then the corresponding tube $\overline{t_i}$ of $\overline{t}$ contains  $n_i$ tubes of $t$.

\begin{figure}[h]
\centering
\includegraphics[scale=.55]{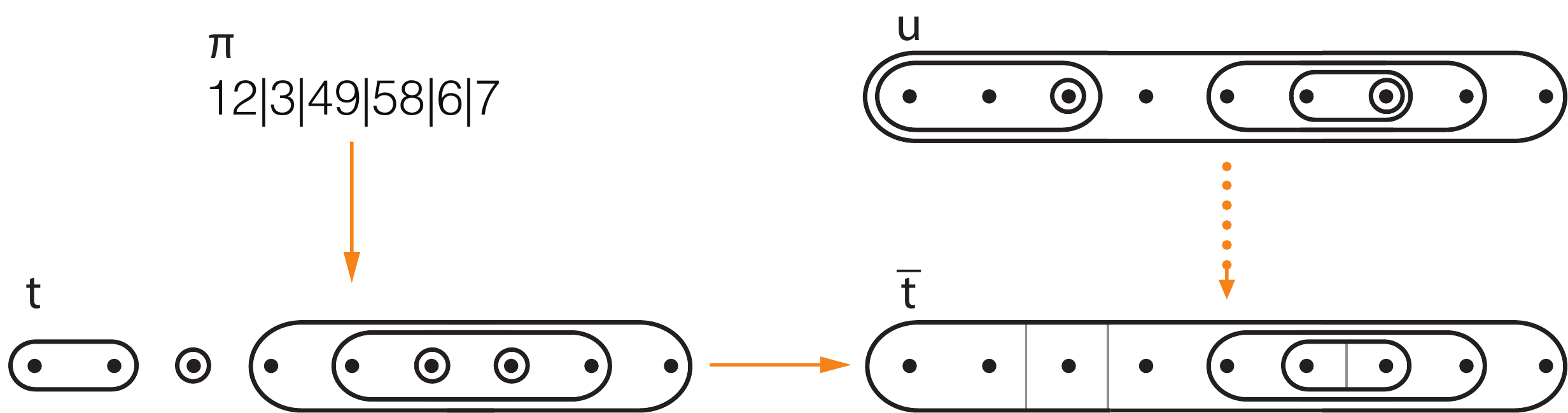} 
\caption{The process to go from a noncrossing partition $\pi=12|3|49|58|6|7$ to a tubing $u$ such that $\pi(u) = \pi$. The step $\pi \mapsto t$ is bijective and the map $t \mapsto t'$ is defined uniquely; we draw the vertical lines in $t'$ are a visual aid, but they are not part of $\overline{t}$. The partial tubing $\overline{t}$ has  $\prod_{p_i \in \overline{\pi}} C_{n_i} = C_3C_2 = 10$ possible preimages $u$, corresponding to resolving the two tubes having $3$ and $2$ vertical compartments, respectively. \label{f:noncrossingtubes}
}
\end{figure}

Any tubing $u$ such that $\pi(u)=\pi$ is obtained from the set $t$ of tubes -- which is usually not a tubing -- by ``resolving" any maximal sequence of adjacent tubes, making them nested. To do this, we consider each tube $\overline{\tau_i}$ of $\overline{t}$, treat the $n_i$ tubes of $t$ that it contains as singletons, and replace them with a maximal tubing of size $n_i$; there are $C_{n_i}$ such tubings for each $i$. This explains why there are $C_{(\overline{\pi}:\pi)}$ tubings $u$ of $l$ with $\pi(u)=\pi$, completing the proof.
\end{proof}

Since $\rF$ is commutative, its antipode is multiplicative. This gives a similar cancellation-free and grouping-free formula for $\apode_I(\alpha)$ for any set of paths $\alpha$ on $I$.

\newpage 

\begin{example} For the path $abcd$, Theorem \ref{t:antipodeF2} gives the formula from the introduction:

 \begin{figure}[H]
\flushleft
\includegraphics[scale=.5]{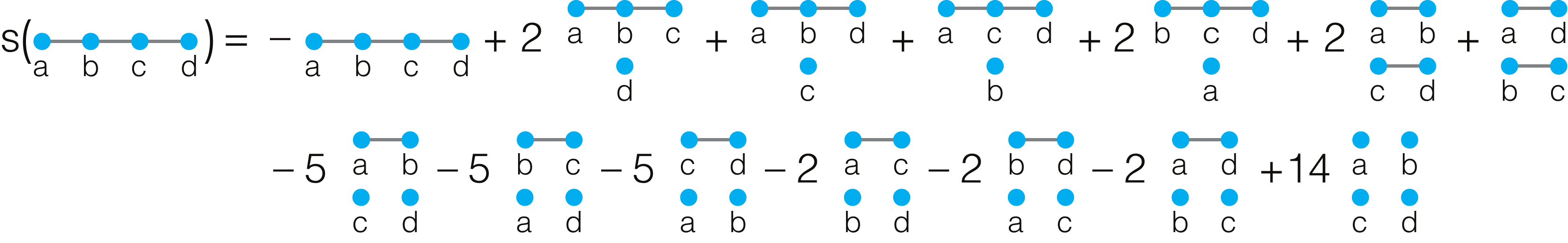} \qquad
\end{figure}
\end{example}


Theorem \ref{t:antipodeF2} explains the double appearance of Catalan numbers in the formula for the antipode of a linear graph: each coefficient is a products of Catalan numbers, and the number of terms (14 in this case) is the number of noncrossing partitions, which is also a Catalan number.

\subsection{Associahedra and paths}\label{ss:assopaths}

As we have already anticipated, our formulas for the antipode of the Hopf monoid of paths $\rF$ are controlled by Loday's associahedra. We now make this connection precise. 

We begin with a technical lemma. Recall that the Loday associahedron $\wa_\ell$ of a linear order $\ell$ of $I$ is the Minkowski sum $\wa_\ell = \sum_{J}\Delta_J$, where we sum over all the intervals $J$ of the linear order $\ell$. 

\begin{lemma}\label{l:associahedra}
If $\ell_1 \neq \ell_2$ are linear orders on $I$, then $\wa_{\ell_1}$ and $\wa_{\ell_2}$ are normally equivalent if and only if $\ell_2$ is the reversal of $\ell_1$.
\end{lemma}

\begin{proof} 
If $\ell_2$ is the reversal of $\ell_1$ then $\ell_1$ and $\ell_2$ have the same intervals, so $\wa_{\ell_1} = \wa_{\ell_2}$. 

Conversely, suppose we know the normal fan  $\Nc := \Nc(\wa_{\ell})$ of the associahedron of a linear order $\ell$. Then we know which hyperplanes of the form $y(i) = y(j)$ for $i, j \in I$ are contained in (the codimension 1 subcomplex of) $\Nc$. The hyperplane $y(i)=y(j)$ can only arise if $\wa_\ell$ has $\Delta_{ij}$ as a Minkowski summand. In turn, that summand appears if and only if $i$ and $j$ are adjacent in the linear order $\ell$. It follows that $\Nc$ determines the set of adjacent pairs of $\ell$, and these completely determine the linear order $\ell$ up to reversal. The desired result follows.
\end{proof}



\begin{proposition}\label{prop:AbA}
The Hopf monoid  of sets of paths $\rF$ is isomorphic to the 
Hopf monoid of associahedra $\rbA$.
\end{proposition}

\begin{proof} The injective maps $\rF \into \rW^{cop} \cong \rWBS^{cop} \into \rBS^{cop} \into \rSHG^{cop}$ of Propositions \ref{prop:FintoW}, \ref{p:WintoWB}, and \ref{p:BSintoSHG}  
allow us to identify a path $l \in \rF[I]$ with the set $\tubes(l) \in \rSHG[I]$.
Together with the surjection $\rSHG^{cop} \onto \rbHGP$ of Proposition \ref{p:HGtoSHG}, this gives a map $\wa: \rF \rightarrow \rbHGP$ which sends a path $l$ to the associahedron $\wa_l$. The image of this map  is $\rbA \subset \rbHGP$. Furthermore, 
$\wa$ is injective thanks to Lemma \ref{l:associahedra}, keeping in mind that a path and its reverse are identified in $\rF$. The desired result follows.
\end{proof}

\subsection{Associahedra and Fa\`a di Bruno}\label{ss:assoFaa}

The \emph{Fa\`a di Bruno Hopf algebra} $\Fc$, introduced by Joni and Rota \cite{joni82:_coalg} but anticipated by many others, appears naturally in several areas of mathematics and physics \cite{ebrahimi2015faa, figueroa05}. 
In this section we show that the Fock functor relates the Hopf monoid of associahedra $\rbA$
(or equivalently the Hopf monoid of paths $\rF$) to the Fa\`a Bruno Hopf algebra $\Fc$.

As an algebra, the Fa\`a di Bruno Hopf algebra $\Fc$ is freely generated as a graded commutative algebra by $\{x_2, x_3,  \ldots\}$ with $\deg x_n = n-1$. It is convenient to write $x_1=1$. The coproduct is given by
\[
\Delta(x_n) = \sum_{k=1}^n \sum_\lambda 
\frac{n!}{\lambda_1! \lambda_2! \cdots 1!^{\lambda_1} 2!^{\lambda_2} \cdots} x_1^{\lambda_1} x_2^{\lambda_2} \cdots \otimes x_k
\]
summing over all sequences $\lambda = (1, 1, \ldots; 2, 2, \ldots; \ldots) = (1^{\lambda_1}, 2^{\lambda_2}, \ldots)$ of length $k$ and total sum $n$, so $
\lambda_1 + \lambda_2 + \lambda_3 + \cdots = k$ and $\lambda_1 + 2\lambda_2 + 3 \lambda_3 +  \cdots = n.$

The grading and the formulas are cleaner when we present $F$ in terms of the generators $a_{n-1} = x_n/n!$; it is useful to write $a_0 = 1$. Then we have

\[
\Delta(a_{n-1}) = \sum_{k=1}^n \sum_\mu
{k \choose \mu_0, \mu_1, \mu_2, \ldots} a_1^{\mu_1} a_2^{\mu_2} \cdots \otimes a_{k-1}
\]
summing over all sequences $\mu = (0, 0, \ldots;  1, 1, \ldots; 2, 2, \ldots; \ldots) = (0^{\mu_0}, 1^{\mu_1}, 2^{\mu_2}, \ldots)$ of length $k$ and total sum $n-k$, so
$\mu_0 + \mu_1+ \mu_2 + \mu_3 + \cdots = k$ and $\mu_1 + 2 \mu_2 + 3 \mu_3 + \cdots = n-k.$

\begin{proposition}\label{p:FtoF}
The Fock functor $\Kcb$ maps the co-opposite $\rbA^{cop}$ of  the Hopf monoid of associahedra $\rbA$ to the Fa\`a di Bruno Hopf algebra $\Fc$.
\end{proposition}

\begin{proof}
Let the Fock functor $\Kcb$ take the associahedron $\wa_\ell$ to the element $a_n$ where $n=|\ell|$. Then (\ref{e:Abasis}) tells us that as an algebra $\Kcb(\rbA^{cop}) = \Kb[a_0, a_1, \ldots]$ while Lemma \ref{c:rbA} tells us that the coproduct of $\Kcb(\rbA^{cop})$ is given by
\[
\Delta(a_{n-1}) = \sum_{[n-1] = S \sqcup T} a_{|T_1|} \cdots a_{|T_k|} \otimes a_{|S|}
\]
where if $S = \{s_1, \ldots, s_{k-1}\}$ then $T_i$ is the interval of integers strictly between $s_i$ and $s_{i+1}$, with the convention that $s_0=0$ and $s_k=n$.

A decomposition $[n-1] = S \sqcup T$ contributes to the term $a_1^{\mu_1} a_2^{\mu_2} \cdots \otimes a_{k-1}$ in $\Delta(a_{n-1})$ when $|S|=k-1$ and the $k$ gaps $|T_1|, \ldots, |T_k|$ between consecutive elements of $S$, including the initial and final gap, have sizes $0, 0, \ldots$ ($\mu_0$ times), $1, 1, \ldots$ ($\mu_1$ times), $2, 2, \ldots$ ($\mu_2$ times), etcetera.
For example, for the decomposition $[12] = \{1,2,4,7,8,12\} \sqcup \{3,4,5,9,10,11\}$, the gaps between consecutive elements of $S = \{1,2,4,7,8,12\}$ have sizes $0,0,1,2,0,3,0$ in that order.

Now it remains to observe that there are ${k \choose \mu_0, \mu_1, \mu_2, \ldots}$ different ways of assigning the gap sizes $0, 0, \ldots$ ($\mu_0$ times), $1, 1, \ldots$ ($\mu_1$ times), $2, 2, \ldots$ ($\mu_2$ times), etcetera to their $k$ slots accordingly. Furthermore, these determine the possible choices for $S$ and $T$ that contribute to the term $a_1^{\mu_1} a_2^{\mu_2} \cdots \otimes a_{k-1}$ in $\Delta(a_{n-1})$, as desired.
\end{proof}

\subsection{Three antipode formulas for the associahedron}\label{ss:asso}

At this point we have given formulas for the antipode of Loday's associahedron $\wa_n$ in three different Hopf monoids: $\wGP, \wbGP,$ and $\wbbGP$.

In $\wGP$ Theorem \ref{t:antipode} gives
\begin{equation}\label{e:antipodeassoc1}
\apode(\wa_n) = \sum_{F \textrm{ face of } \wa_n} \, 
 (-1)^{n - \dim F} F
\end{equation}
where every face $F$ of $\wa_n$ is normally equivalent to a product of Loday associahedra.

In $\wbA \subset \wbGP$, thanks to the isomorphism $\rF \cong \rbA$, Theorem \ref{t:antipodeF2}  gives
\begin{equation}\label{e:antipodeassoc2}
\apode(\wa_n) = 
\sum_{\pi \in NC(n)} (-1)^{|\pi|} C_{(\overline{\pi} : \pi)} \, \wa_{p_1} \cdots \wa_{p_k}
\end{equation}
summing over the noncrossing partitions $\pi$ of $[n]$; here $\overline{\pi} = \{p_1, \ldots, p_l\}$ is the adjacent closure of $\pi$, and $C_{(\overline{\pi}:\pi)} = C_{n_1}\cdots C_{n_l}$
where $n_i$ is the number of blocks of $\pi$ refining block $p_i$ of $\overline{\pi}$.

In $\wbbA \subset \wbbGP$ the proofs of Theorems \ref{t:Lagrinvpol} and \ref{t:Lagrinvenum} give
\begin{equation}\label{e:antipodeassoc3}
\apode(\wa_n) = \sum_{\langle 1^{m_1}2^{m_2} \cdots \rangle \vdash n} \, 
 (-1)^{|m|}\frac{(n+|m|)!} {(n+1)! \, m_1! \, m_2! \cdots} \wa_1^{m_1} \wa_2^{m_2} \cdots
\end{equation}
summing over all partitions $\langle 1^{m_1}2^{m_2} \cdots \rangle$ of $n$, 
where $|m| = m_1+m_2+\cdots$.

\begin{figure}[h]
\centering
\qquad \qquad
\includegraphics[scale=.3]{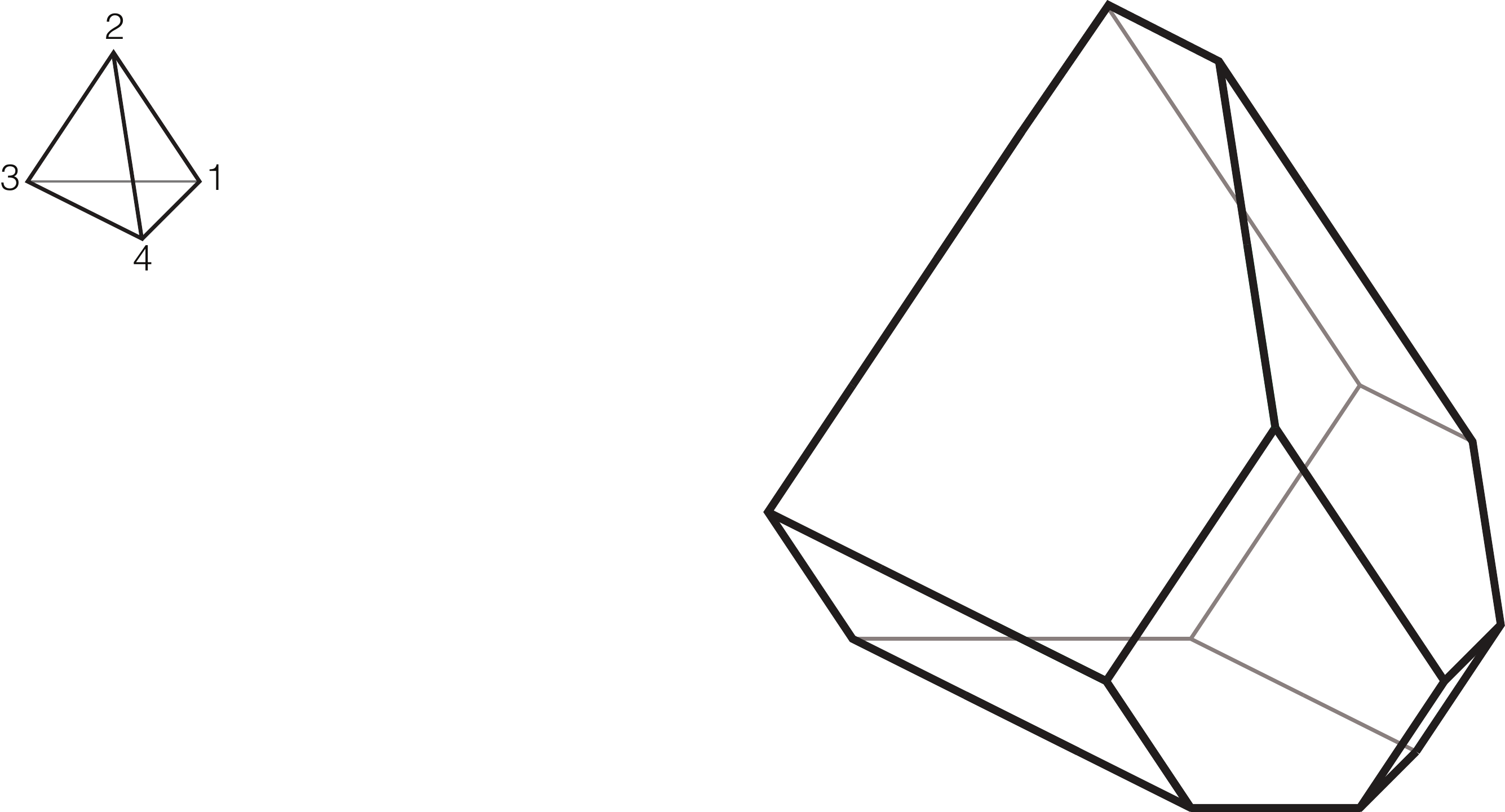} \hfill
\caption{The three-dimensional associahedron $\wa_4$.\label{f:assoc3D}}
\end{figure} 

Each formula coarsens the previous one under the projection maps $\wGP \onto \wbGP \onto \wbbGP$. In the first formula all faces of the associahedron are distinct. In the second formula, faces of the associahedron are grouped together according to their normal equivalence classes, which in turn correspond  to their combinatorial type and position with respect to the axes. In the third formula, faces of the associahedron are grouped according to their quasinormal equivalence classes, which correspond to their combinatorial type.

\begin{example}
Let us consider the contribution of the $6$ pentagonal faces of the associahedron $\wa_4$ to the three versions of the antipode $\apode(\wa_4)$:\\
$\bullet$ In $\wGP$, each one of these six pentagonal faces is a separate term of $\apode(\wa_4)$.\\
$\bullet$ In $\wbGP$, these six faces group into four normal equivalence classes:
 the noncrossing partitions $\{123,4\}$ and $\{1, 234\}$ contribute two pentagons each, 
 while the noncrossing partitions $\{134,2\}$ and $\{124,3\}$ contribute one pentagon each.\\
$\bullet$ In $\wbbGP$, these six faces are all grouped together into the coefficient of $\wa_3\wa_1$, which equals $(-1)^2(4+2)!/(4+1)!1!1! = 6$. 
\end{example}


These observations have two interesting enumerative corollaries. 

\begin{corollary}
The number of normal equivalence classes of faces of Loday's associahedron $\wa_n$ is the Catalan number $C_n$. 
\end{corollary}

\begin{proof}
The projection $\wGP \onto \wbGP$ takes (\ref{e:antipodeassoc1}) to (\ref{e:antipodeassoc2}), mapping the faces of $\wa_n$ onto their normal equivalence classes. The result follows from the fact that 
the terms of (\ref{e:antipodeassoc2}) are in bijection with the noncrossing partitions of $[n]$ which are counted by the Catalan number $C_n$.
\end{proof}

\begin{corollary}
Let $\mu = \langle 1^{m_1}2^{m_2} \cdots \rangle$ be a partition of $n$ and write $|m| = m_1 + m_2 + \cdots$. Let $NC(\mu)$ be the set of noncrossing partitions of $n$ having type $\mu$; that is, having $m_i$ blocks of size $i$ for $i=1,2,\ldots$. Then, in the notation of Theorem \ref{t:antipodeF2},
\[
\sum_{\pi \in NC(\mu)} C_{(\overline{\pi} : \pi)} =  \frac{(n+|m|)!}{(n+1)! \, m_1! \, m_2! \, \cdots}.
\]
\end{corollary}

\begin{proof}
The projection $\wbGP \onto \wbbGP$ takes (\ref{e:antipodeassoc2}) to (\ref{e:antipodeassoc3}). It maps each normal equivalence class of faces, which is labeled by a noncrossing partition of $[n]$, to its combinatorial type, which is the corresponding partition of $n$. It then remains to observe that the noncrossing partitions of type $\mu$ are the ones that map to the partition $\mu$, so their contributions to (\ref{e:antipodeassoc2}) must add up to the contribution of $\mu$ to (\ref{e:antipodeassoc3}).
\end{proof}

\appendix

\section{Future directions and open questions}

This project suggest several research directions which will be the subject of upcoming papers.

\begin{itemize}
\item The formula for the antipode of $\wGP$ is reminiscent of McMullen's polytope algebra, where the alternating sum of the faces of a polytope $\wp$ is equal to its relative interior $\wp^\circ$. Clarify the relationship between these two algebraic structures on polytopes.

\item Motivated by Brion's theorem, which expresses the lattice point enumerator of a polytope in terms of those of its vertex cones, there is a \emph{Brion map} of Hopf monoids  $\wGP \rightarrow \wP$. Explore the consequences of this map.

\item
For $\rGP$ or some of its interesting submonoids $\rH$:
\begin{itemize}
\item Describe the character group $\Xb(\rH)$.

\item Describe the Lie monoid $\Pc(\wH^{cop})$ of primitive elements of $\wH^{cop}$, which determines the Hopf monoid $\wH$ via a variant of the Cartier--Milnor--Moore theorem \cite[Prop. 11.45]{am} since $\wH^{cop}$ is cocommutative.

\item The Brion map of Hopf monoids $\wH \rightarrow \wP$ gives rise to a dual Brion map of Lie monoids  $B^{cop}: \mathcal{P}(\wP^{cop}) \rightarrow \mathcal{P}(\wH^{cop})$. Describe this map explicitly.
\end{itemize}

\item Extend the results of this paper to \emph{generalized Coxeter permutahedra}, the deformations of the Coxeter permutahedra $\pi_W$ corresponding to a finite reflection group $W$. Connect them to Zaslavsky's theory of signed graphs \cite{zaslavsky1982signed}, Borovik-Gelfand-Serganova-White's theory of Coxeter matroids \cite{borovik2003coxeter}, Reiner's theory of signed posets \cite{reiner93:_signed}, and Fomin and Zelevinsky's Coxeter associahedra. \cite{fomin2003systems}
This will involve an extension of  the theory of Hopf monoids -- which is inherent to the Coxeter group $W=S_n$ -- to any finite Coxeter group $W$, which is being developed by Aguiar and Mahajan. 

\item More generally, extend the results of this paper to the deformations of any simple polytope. This will involve a further extension of the theory of Hopf monoids to any hyperplane arrangement, which is also being developed by Aguiar and Mahajan. 
\end{itemize}

\noindent 
The following are a few more questions raised by this project which may be of interest:

\begin{itemize}


\item
Further study the connection between the Hopf monoid $\rGP$ and the valuative invariants on generalized permutahedra, which were described by Derksen and Fink. \cite{derksenfink10}

\item
Is there a Hopf algebraic answer to Rota's question of Section \ref{ss:relations}? Is there an intrinsic characterization of the Hopf monoid $\wHG \cong \wHGP$ of hypergraphs and hypergraphic polytopes?

\item
As explained in Remark \ref{r:y-positive}, under the suitable probability measure, an integer generalized permutahedra in $\Rb I$ is hypergraphic with positive probability. What is that probability, and how does it behave as the dimension of the ambient space goes to infinity?

\item
It would be interesting to study the combinatorial consequences (characters, invariants, reciprocity theorems) of our results on the Hopf monoids of hypergraphs, simplicial complexes, building sets, simple graphs, and paths. For some progress in these directions, see \cite{benedetti2016combinatorial,grujic2014quasisymmetric,grujic2017counting}.

\item
It would be interesting to further study the simplicial complex polytopes of Section \ref{ss:SCpolytopes}. For instance, are there formulas for their volumes or Ehrhart polynomials, at least in some special cases?

\end{itemize}

\section{Acknowledgments}

The main constructions in this paper were discovered in 2008 and announced in 2009. \cite{ardilaAMS} Throughout these years we have benefitted greatly from conversations with Carolina Benedetti, Lou Billera, Laura Escobar,  Alex Fink, Rafael Gonz\'alez d'Le\'on, Carly Klivans, Swapneel Mahajan, Jeremy Martin, Alex Postnikov, and Vic Reiner, among others. 
In the meantime, some of the results in Sections \ref{s:GP}, \ref{s:G}, \ref{ss:Bergmanpoly}, \ref{s:SC}, and \ref{ss:B} were discovered independently in 
\cite{derksenfink10,humpert12,breuer2016scheduling, benedetti2016combinatorial
grujic2014quasisymmetric}, respectively. We thank the authors of these papers for their patience while we published this work, and for their open communication with us, which has strengthened our understanding of this project. 

We do not thank the Portland, OR thieves who set the project back in 2013 by stealing a folder containing five years of work: results, proofs, writeups, pictures, examples, and counterexamples. We do thank them for not publishing our results in their name. 

\textsf{FA}: I would like to thank Gian-Carlo Rota for generously encouraging me to publish Proposition \ref{p:relation} when I was a first-semester graduate student in 1998. I did not understand the significance of this result at the time and did not publish it, but this project has served as yet another reminder that I still have much to learn from the brief but influential lessons I received from Rota. I was happily surprised to see his question reappear in a central role in this project, much of which was motivated by his ideas. I am glad to finally keep my word.

\makeatletter
\renewcommand{\@biblabel}[1]{\hfil#1.}
\makeatother
\bibliographystyle{amsplain}
\bibliography{gp}

\end{document}